\documentclass[11pt]{article}

\usepackage[margin=1in]{geometry}
\usepackage[utf8]{inputenc}
\usepackage[T1]{fontenc}
\usepackage{tikz}
\usepackage{multirow}
\usepackage{booktabs}
\usepackage{float}
\usepackage{amsmath,amssymb,amsthm,mathtools,amsfonts,bm, booktabs, siunitx}
\usepackage{natbib}
\usepackage{url}
\usepackage{enumitem,calc}
\usepackage{mathtools}
\mathtoolsset{showonlyrefs}
\usepackage{appendix} 
\usepackage{comment}

\usepackage[colorlinks=true,
linkcolor=blue,
citecolor=blue,
urlcolor=blue]{hyperref}

\theoremstyle{plain}
\newtheorem{theorem}{Theorem}
\newtheorem{proposition}{Proposition}
\newtheorem{lemma}{Lemma}
\newtheorem{corollary}{Corollary}

\newtheorem{remark}{Remark}

\theoremstyle{definition}
\newtheorem{definition}{Definition}

\newcommand{\bT}{{\rm \textbf T}}
\newcommand{\bZ}{{\rm \textbf Z}}
\newcommand{\bfT}{\mathrm F_{\rm \textbf T}}
\newcommand{\T}{\mathcal T}
\newcommand{\E}{\mathbb E}
\newcommand{\hh}{\mathfrak h}
\newcommand{\KM}{\mathcal K_M}
\newcommand{\Kc}{|\mathcal K_M|}
\newcommand{\HM}{\mathcal H_M}
\newcommand{\PP}{\mathbb P}
\newcommand{\Ln}{\mathcal L_n}
\newcommand{\In}{\mathcal I_n}
\newcommand{\Cov}{\operatorname{Cov}}
\newcommand{\Fclass}{\mathcal{F}}
\newcommand{\KL}{\mathrm{KL}}

\newcommand{\U}{\mathcal{U}}

\newcommand{\Ht}{\widetilde H_n}
\newcommand{\ee}{\mathfrak e_1}
\newcommand{\bXi}{\boldsymbol\Xi}
\newcommand{\Dn}{\mathcal D_n}
\newcommand{\vv}{\mathsf v}
\newcommand{\Lto}[1]{L^{2}(#1)}
\newcommand{\norm}[1]{\lVert #1\rVert}

\newtheorem{assumption}{Assumption}
\renewcommand{\theassumption}{\arabic{assumption}}

\title{Optimal estimation and goodness-of-fit testing of the mean for sparse longitudinal functional data}

\author{Valentin Patilea\footnote{Univ. Rennes, Ensai, CREST-UMR 9194, F-35000, France.  Corresponding author: valentin.patilea@ensai.fr  } \qquad \qquad Mahdi Saidi\footnote{Ensai, Campus de Ker-Lann, Bruz, France.}}
\date{\today}

\begin{document}

\maketitle

 \begin{abstract}
We study the mean function of longitudinal functional data, where each subject contributes a small number of complete profiles over a general domain, observed at random visit times. The mean is projected onto an orthonormal basis in the time direction, and each coefficient function is estimated by a weighted average of the profiles. We consider deterministic and random weights, including closed-form, data-driven weights that dispense with the design density entirely. For each weighting scheme we derive non-asymptotic bounds on the integrated quadratic risk and an explicit optimal truncation level. All schemes share the same convergence rate, which we show to be minimax optimal over the corresponding coefficient-decay class, only the leading constants differ. We also provide a goodness-of-fit test for the mean function, and derive non-asymptotic bounds for the Gaussian approximation of its distribution under both the null and alternative hypotheses. Our estimation and testing procedures are easy to implement, fast, and perform well in applications.

 \end{abstract}
 
 
 \textbf{MSC2020 subject classification:} 62G05, 62G10, 62R10, 62M40

\tableofcontents

\section{Introduction}\label{sec:intro}

Profiles recorded at a few irregularly spaced times, for each of many subjects,
arise across applied fields: diffusion tensor imaging curves along a brain tract
at the successive visits of a patient \cite{GCCR2010, PS2015}; a degradation
signature at each inspection of an industrial unit \cite{LM1993, ZSG2011}; a
$24$-hour activity profile at each wave of an aging study \cite{XHSFZC2015}; a
ground reaction force or center of pressure curve at each session in a movement
laboratory \cite{HSSS2021, WBLHRGH2021}; a velocity profile at each race across
an athlete's competitive career \cite{EBOS-FW2025}. Everywhere the structure is
the same: a subject $i$, a few visit times $T_{ij}\in\mathcal T$, and a profile
$u\mapsto Y_{ij}(u)$ observed, with noise, practically over the whole of
$\mathcal U$. The model we consider represents each subject by a random
surface on $\mathcal U\times\mathcal T$, written as the sum of a common
nonparametric mean function $\mu(u,t)$, the object of interest, and a
subject-specific zero-mean random field capturing individual variation. Sparsity
is longitudinal only, the observed profiles being sections of this surface at the
random visit times.

The standard route smooths in both arguments, either by local polynomial smoothing of the pooled scatterplot \cite{YMW2005, ZW2016, CM2012} or by penalized tensor-product bases \cite{GCCR2010, PS2015}. Both call for two smoothing parameters, have been studied only when $\mathcal U$ is a compact interval, and come with asymptotic guarantees whose constants are rarely explicit. Since in our setup each profile is observed over all of $\mathcal U$, or ultra-densely, no smoothing in $u$ is needed: we expand $\mu(u,\cdot)$ in an orthonormal basis of $L^2(\mathcal T)$ and estimate each coefficient function by a weighted average of the observations, for both deterministic and random weights. This yields non-asymptotic, rate-optimal risk bounds and an explicit optimal truncation. We also propose a spacings variant which dispenses with the design density. To the best of our knowledge, a goodness-of-fit test for
the mean surface in the longitudinal model we consider here, for equality to a given function, or for structural features such as constancy in the visit time, has not been considered. We provide one, based on the series coefficients.

A similar model was introduced by \cite{PS2015}; the same sparse random design is treated by \cite{CM2012} following \cite{YMW2005}, and a deep-neural-network estimator has been proposed by \cite{WCS2021} under a common equally spaced grid rather than random visit times. Our approach departs from these in its methodology, a series expansion in the time argument in place of splines, kernels, or networks, and in two structural respects: the functional argument $u$ may live on a general metric space $\mathcal U$, such as a sphere or a manifold, rather than an interval, and the visit-time density can be removed altogether by means of a closed-form, spacings-based estimator.

On the theoretical side the existing record is thinner still: \cite{PS2015} give no risk
analysis, \cite{CM2012} establish uniform consistency of their kernel estimator
but do not investigate non-asymptotic bounds for the quadratic risk, and the
optimal rate available for the network estimator \cite{WCS2021} is confined to a
fixed common grid and does not extend to random visit times. We are not aware of
any non-asymptotic risk theory for the mean function under a sparse random
design. We close this gap: for each weighting scheme we obtain non-asymptotic
bounds on the integrated quadratic risk over $\mathcal U\times\mathcal T$, an
explicit optimal truncation level, and the optimal rate $M^{-2s/(2s+1)}$, governed
by the total number $M$ of visits and by a regularity index $s$ of the mean in the
time direction. Using the augmented half-cosine basis, this gives the rate
$M^{-4/5}$ for twice-differentiable means, matching the optimal nonparametric rate
of kernel and spline smoothing, while the recommended estimator remains free of
the design density and available in closed form.

Concerning the goodness-of-fit test, we provide a non-asymptotic bound for the Gaussian approximation error of the distribution of the test statistic under \emph{both} the null hypothesis and the alternative. The test is consistent against nonparametric alternatives, and the critical values are easily approximated by multiplier bootstrap in finite samples. The closest existing proposals target either a different hypothesis or a design without repeated within-subject measurements. The likelihood-ratio test of \cite{SLCR2014} addresses a one-sample mean-structure
hypothesis, but through a penalized-spline mixed-model formulation with a fixed
parametric null. It does not cover a growing basis in $t$, a general domain
$\mathcal U$, or the sparse longitudinal random design considered here.
\cite{PSS2016} proposed a $U$-statistic test, based on univariate nearest-neighbor
smoothing, for a related but different problem, namely the effect of a predictor on a functional
response, under i.i.d. observations. \cite{CD2013} tested a nonparametric regression
function on a compact rectangular domain from i.i.d. data, using a quadratic functional
of that function, as in our approach. However, both \cite{PSS2016}
and \cite{CD2013} could be adapted to our setting only in the limit case
of one visit time per subject, since they rely on i.i.d. single observations and do not accommodate the
within-subject dependence induced by repeated visits. Finally,
several contributions consider two-sample tests for functional data, see \cite{W2021},
\cite{ZSQ2025} and the references therein, but none directly applies to the one-sample mean
problem in the sparse longitudinal design.

The paper is organized as follows. In Section~\ref{sec:model} we introduce our model and the assumptions on the data-generating process. In Section~\ref{sec:series} we define a general class of series estimators for the mean of the profiles. The coefficient functions of the expansion are estimated by weighted averages, under high-level conditions on the weights, which may be random or deterministic. We derive non-asymptotic risk bounds for estimators with deterministic weights, for mean functions satisfying a Sobolev-ellipsoid-type coefficient-decay condition. As an example of random weights, we study those defined by a fast integration procedure using one-nearest-neighbor (1-NN) control variates inspired by the Monte Carlo literature. We conclude Section~\ref{sec:series} with a data-driven estimator that dispenses with knowledge of the visit-time density, using a simple integral quadrature based on spacings. Section~\ref{sec:basis_choice} is devoted to the choice of basis: the half-cosine system enriched with low-degree monomials in $t$ and orthonormalized by the Gram-Schmidt procedure. The resulting basis functions have closed-form expressions. We focus on the augmentation by $t$ and $t^2$, for which we also give closed forms for the spacings-based quadrature. In Section~\ref{sec:test} we introduce the new goodness-of-fit test, based on $U$-statistics, for the mean in the context of sparse longitudinal functional data. The null hypothesis is formulated as the nullity of a set of coefficient functions. As a consequence of a general result of independent interest concerning the Gaussian approximation of degenerate and non-degenerate second-order $U$-statistics, we derive a Berry-Esseen bound for the Gaussian approximation of the test statistic. We also prove the validity of a bootstrap version of the test. The study of the test is completed by proving a Gaussian approximation result for the test statistic version using integral quadrature based on spacings. In Section~\ref{sec:numerics}, the finite-sample performance of the spacings-based estimator and test is illustrated by simulations. Our estimator is competitive with approaches based on local-linear kernel smoothing and splines. In simulations, the test based on the spacings approach has an accurate level and appears more powerful than the test using the true density of the visit times. The new estimation and testing approaches are applied to diffusion tensor imaging (DTI) data. The test does not reject the null hypothesis of a mean function constant in the visit time $t$. Section~\ref{sec:conclusion} summarizes the contributions and practical recommendations, and outlines possible extensions left for future work. The proofs are collected in the Appendix.

Below, we will use the following notation: $\lesssim$ means that the left side is bounded by the right side times a positive  constant independent of the sample size; for $a,b\in\mathbb R$, $a \asymp b $ means  $a \lesssim b $ and $b \lesssim a $; $\|\cdot\|_\infty$ is the uniform norm; $\lfloor a \rfloor$ (resp. $\lceil a \rceil$) denotes the largest integer not larger than $a$ (resp. the smallest integer not smaller than $a$); for a set $\mathcal A$, $|\mathcal A|$ denotes its cardinality and $\mathbb I_{\mathcal A}$ is the indicator function of $\mathcal A$; $\Phi(\cdot)$ stands for the standard normal distribution function and $z_{\mathfrak a}$ for its  $\mathfrak a$-order quantile. We denote by $C,C',C_0,C_1, \mathfrak c,\ldots$  constants that may differ from line to line.

\section{Model setup}\label{sec:model}

Let $\mathcal T$ be a compact interval on the real line, typically $\mathcal T =[0,1]$, equipped with the Lebesgue measure. Unless otherwise specified, $(\mathcal U, d_{\mathcal U}) $ is a separable metric space endowed with its Borel $\sigma$-field and a \emph{finite} Borel measure $\nu$. Typical examples include a compact manifold, such as the sphere $\mathbb{S}^{2}$ (where $d_{\mathcal{U}}$ is the geodesic distance and $\nu$ is the surface measure), and the unit hypercube in Euclidean space (where $d_{\mathcal{U}}$ is the Euclidean distance and $\nu$ is the Lebesgue measure). A finite set $\mathcal U$ is another trivial example. Let $L^2(\mathcal U \times \mathcal T)$, $L^2(\mathcal T)$ and $L^2(\mathcal U)$ be the spaces of square-integrable functions defined on $\mathcal U \times \mathcal T$, $\mathcal T$ and $\mathcal U$, respectively.

We consider longitudinal functional data observed for $n$ subjects.
For each subject $i = 1,\ldots,n$, the data consist of $\{(T_{ij}, Y_{ij}(\cdot)): 1\leq j\leq m_i\}$, where $T_{ij}\in \mathcal T$ and the $Y_{ij}(\cdot)$ are functional profiles 
recorded at $m_i$ visit times $T_{i1},\ldots,T_{im_i}$, according to the
model
\begin{equation}\label{eq:model0}
Y_{ij}(u) = \mu(u, T_{ij}) + X_i(u, T_{ij}) + \varepsilon_{ij}(u),
\qquad u \in \mathcal{U},\; T_{ij} \in \mathcal{T}.
\end{equation}
Here, $\mu:\mathcal U \times \mathcal T \rightarrow \mathbb R$ is the mean function, which is  the quantity of interest; the $X_i(\cdot,\cdot)$ are subject-specific, zero-mean stochastic processes (random fields, see~\cite{AT2007}) indexed by the elements of the domain $\mathcal U \times \mathcal T$; and $\varepsilon_{ij}(\cdot)$ are zero-mean  error processes (random fields) defined on $\mathcal U$.  Let 
$
M = m_1 +\cdots+ m_n
$
be the total (non random) number of observed pairs $(T_{ij}, Y_{ij}(\cdot))$. Although several results below hold with $m_i$ growing with $n$, our main interest is the case of \emph{bounded} $m_i$, \emph{i.e.},
\begin{equation}\label{cd_m_max}
 1\leq m_i\leq m_{\max}\quad\text{for all }1\leq i\leq n,
\end{equation}
for a constant $m_{\max}$ independent of $n$. We indicate explicitly when Condition~\eqref{cd_m_max} is used.

\begin{assumption}[Data generating process]\label{Ass_gen}
		
	\begin{enumerate} [	label=(\roman*),ref=\theassumption(\roman*)]
		
		\item\label{ass:process}   $\{X_i\}$ are independent copies of a random field $X$ on the domain $\mathcal U \times \mathcal T$ such that
		$ \mathbb E \bigl[\, \|X\|^2_{L^2(\mathcal{U}\times \mathcal T)\,}\bigr]<\infty$, and $\mathbb E[X(u,t)]=0,\; \forall (u,t) \in \mathcal{U}\times \mathcal{T}$.\vspace{-.1cm}

		\item\label{ass:design} $\{T_{ij} : 1\leq i \leq n,\; 1 \leq j \leq m_i\}$
		are i.i.d.\ realizations of a variable $T$ with Lipschitz
		continuous density $g$ (with Lipschitz constant $L_g$) that is positive on $\mathcal{T}$. Let $g_{\min}=\inf_{\mathcal T} g$, $g_{\max}=\sup_{\mathcal T} g$.\vspace{-.1cm}

		\item\label{ass:noise} 	 $\varepsilon_{ij}(u)
		= \tau(T_{ij})\,\eta_{ij}(u)$, where $\tau : \mathcal{T} \to
		\mathbb{R}_+$ is a bounded, deterministic standard deviation function, and $\eta_{ij}(\cdot)$ are i.i.d. random fields with
		$\mathbb{E}[\eta_{ij}(u)] = 0$ and $\mathrm{Var}(\eta_{ij}(u)) = 1$, $\forall u\in\mathcal U$.\vspace{-.1cm}
		
		\item\label{ass:indep} 	The collections $\{T_{ij}\}$, $\{X_i\}$ and $\{\eta_{ij}\}$
are mutually independent.
	\end{enumerate}
\end{assumption}

Our model and its underlying assumptions also cover situations where $\mathcal U$ is a finite set (grid) and $\nu$ is the uniform distribution on this set, with each profile $Y_{ij}$ recorded at every point of the grid. This is common when profiles are measured by sensors that record data at fixed frequencies during visits. However, we do not address here the discretization errors that a finite $\mathcal U$ may entail. 

In the following, the $\sigma$-algebra generated by all the $T_{ij}$ is denoted $\mathcal{F}^{\mathrm{obs}}$. In our setup, each noisy profile $Y_{ij}(\cdot)$ is observed everywhere on $\mathcal{U}$.  Sparsity is longitudinal only, as the number $m_i$ of visit times $T_{ij}\in\mathcal{T}$ will be assumed bounded. The error term is allowed to be heteroscedastic, with a nonparametric scale that depends only on the visit time. The case without error corresponds to $\tau(\cdot)\equiv 0$. For each $u\in\mathcal U$, the covariance structure of $X$ in the direction of $t$ is described  by
\begin{equation}\label{def:gammaX}
\gamma_X(u;\,t,s)
= \mathbb{E}\!\left[X(u,t)\,X(u,s)\right],
\qquad
\gamma_X(u;\,t) = \gamma_X(u;\,t,t)
= \mathbb{E}\!\left[X^2(u,t)\right],
\end{equation}
and, after integration over the functional argument $u$, by
\begin{equation}\label{def:GammaX}
\!\!\!\Gamma_X(t,s)
= \!\! \int_{\mathcal{U}} \!\!\gamma_X(u;\,t,s)\mathrm{d}\nu (u)
\!= \mathbb{E}\!\left[\!
\langle X(\cdot,t),X(\cdot,s)\rangle_{\!L^2(\mathcal{U})}
\right]\!, \,\,\,   \Gamma_X(t) :=\Gamma_X(t,t) = \mathbb{E}\bigl[
\|X(\cdot,t)\|_{L^2(\mathcal{U})}^2\bigr].
\end{equation}

\section{Mean estimation by series expansion}\label{sec:series}

A natural way to estimate the mean function $\mu$ from the data as described above is by series expansion in the argument $t$. Let $\{\phi_k\}\subset L^2(\mathcal{T})$ be an orthonormal basis. Throughout,  we assume: 
\begin{equation}\label{eq:Lip_phi}
\text{$\forall k\geq 1$, the  function $\phi_k$ is Lipschitz continuous  with Lipschitz constant $L_{\phi,k}$.}
\end{equation} 
For each $u \in \mathcal{U}$ and a truncation level $K$, we can write 
\[
\mu(u, t) = \sum_{k=1}^{K} \beta_k(u)\,\phi_k(t) + r_K(u,t),
\qquad u \in \mathcal{U},\; t \in \mathcal{T},
\]
where the projection coefficients are $\beta_k(u) = \int_{\mathcal{T}} \mu(u, t)\,\phi_k(t)\,\mathrm{d}t$,
$ k \geq  1$, and $r_K(u,t) = \mu(u,t) - \sum_{k=1}^K \beta_k(u)\phi_k(t)$ is the truncation remainder. If $  \mu(u,\cdot)\in L^2(\mathcal{T})$, Parseval's identity guarantees that $\|r_K(u,\cdot)\|_{L^2(\mathcal{T})} \to 0$ as $K \to \infty$. Moreover, in model \eqref{eq:model0} under Assumption~\ref{Ass_gen}, for $\nu$-a.e. $u\in\mathcal U$ and $ k\geq 1$ it holds 
\begin{equation}\label{eq:rep_bk}
 \beta_k(u)	= \mathbb{E}\bigg[Y_{ij}(u)\,\frac{\phi_k(T_{ij})}{g(T_{ij})}\bigg], \qquad \forall (i,j).
\end{equation}
Identity \eqref{eq:rep_bk} is the starting point for the construction of our estimators of the coefficient functions $ \beta_k(\cdot)$ which we present below. Note that \eqref{eq:rep_bk} involves the design density $g$. In general, for the construction of estimators of $\beta_k(u)$ based on \eqref{eq:rep_bk}, $g$ has to be replaced by a suitable estimator. An alternative way avoiding the estimation of $g$ will be presented in Section~\ref{sec:spacings}. For now,    $g$ is given.

\subsection{Estimation of the coefficients}

Let $\{\omega_{ij}\}$ be a system of weights. We consider a general weighted estimator of $\beta_k(u)$ defined as
\begin{equation}\label{eq:generic_est}
\widehat{\beta}_k(u) = \sum_{i=1}^{n}\sum_{j=1}^{m_i} \omega_{ij}\,W_{ijk}(u), \quad \text{ where } \quad W_{ijk}(u) =Y_{ij}(u)\,\frac{\phi_k(T_{ij})}{g(T_{ij})}.
\end{equation}

\begin{assumption}[Admissible weights]\label{def:admissible_weights}
	The weights $\{\omega_{ij}\}$, possibly random but  $\mathcal F^{\rm obs}$-measurable, are such that $\sum_{i=1}^{n}\sum_{j=1}^{m_i} \omega_{ij}=1$ and: \vspace{-.1cm}
	\begin{enumerate} [	label=(\roman*),ref=\theassumption(\roman*)]
		
		\item\label{omeganobias} for any integrable function $\psi(T)$, it holds $\sum_{i=1}^{n}\sum_{j=1}^{m_i} \mathbb E [\omega_{ij}\psi(T_{ij})] = \mathbb E [\psi(T)] $;\vspace{-.2cm}
		
		\item\label{omegastability} there exists a constant $C_0 \geq 0$ such that $\sum_{i=1}^{n}\sum_{j=1}^{m_i} \omega_{ij}^2 \leq 1 +  C_0 M^{-1}$.
	\end{enumerate}
\end{assumption}

Assumption~\ref{omeganobias} guarantees that $\widehat{\beta}_k(u)$ is unbiased, \emph{i.e.,} $\mathbb E \big[\widehat{\beta}_k(u)\big]=\beta_k(u)$. Assumption~\ref{omegastability} is a very mild technical condition which disappears with non-negative weights. There are three main types of weights satisfying Assumption~\ref{def:admissible_weights} which appear in the literature:\vspace{-.1cm}
\begin{enumerate}[leftmargin=*,align=left,labelsep=1ex,labelwidth=\widthof{\bfseries Subject-balanced weights},itemindent=1pt]
	\item[\textbf{\hspace{.6cm}Uniform weights (OBS):}] $\omega_{ij} = 1/M$, with $M=m_1+\cdots+m_n$;\vspace{-.2cm}
	\item[\textbf{\hspace{.6cm}Subject-balanced weights (SUBJ):}] $\omega_{ij} = 1/(nm_i)$;\vspace{-.2cm}
	\item[\textbf{\hspace{.6cm}Random weights (MC):}] control-neighbors Monte Carlo weights.\vspace{-.1cm}
\end{enumerate}

The OBS and SUBJ weights are the most common; see \cite{ZW2016} for a detailed discussion. The MC weights were recently introduced in FDA by \cite{PW2026}, using an idea of fast linear integration via 1-NN control variates inspired by the Monte Carlo literature (see \cite{LPSZ2025}; see also Appendix~\ref{app_sec:MC} for the detailed definitions). The MC weights are defined as  $\omega_{ij} (= \omega_{ij}^{\mathrm{MC}} ):= M^{-1}  \{1 + \widehat{c}_{ij} - \widehat{d}_{ij}\}$, where $\widehat{d}_{ij}$ is the degree of $T_{ij}$, \emph{i.e.,} the number of indices $(i',j')\neq (i,j)$ for which $T_{ij}$ 
is the leave-one-out nearest neighbor of $T_{i'j'}$, and $\widehat{c}_{ij}$ is the corresponding cumulative Voronoi volume. In our setting, a Voronoi cell is a compact interval and the Voronoi volume is the integral of $g$ over that interval. By definition, the MC weights sum to 1, and Assumption~\ref{omeganobias} is proved in \cite{KP2025}; see also 
\cite{LPSZ2025}. The justification of Assumption~\ref{omegastability} is provided in Lemma~\ref{prop:cn_admissible} in Appendix~\ref{app_sec:MC}. The advantage of MC weights is that they are more accurate than Riemann integration or simple averages for approximating integrals. Note that any convex combination of weight systems satisfying Assumption~\ref{def:admissible_weights} still satisfies it.

We now investigate the quadratic risk of $\widehat{\beta}_k(u)$ using the decomposition
$$
	\widehat{\beta}_k(u) - \beta_k(u) = \{B_k(u) -  \beta_k(u) \} + V_k(u),
$$
where 
\begin{equation}\label{eq:bias_variance}
B_k(u) = \!\sum_{i=1}^{n}\sum_{j=1}^{m_i} \omega_{ij}\mu (u,T_{ij})\frac{\phi_k(T_{ij})}{g(T_{ij})}, \qquad  V_k(u) =	\sum_{i=1}^{n}\sum_{j=1}^{m_i} \omega_{ij}  \{X_i (u,T_{ij}) + \tau(T_{ij} ) \eta_{ij}(u)\}\frac{\phi_k(T_{ij})}{g(T_{ij})}.
\end{equation}
The term $B_k(u)$ is a \emph{signal} term with mean $\beta_k(u) $ under Assumption~\ref{Ass_gen}, $V_k(u)$ is a stochastic term.

For \emph{deterministic weights} $\omega_{ij}$ and because the $T_{ij}$ are independent, it holds
$$
\mathrm{Var}\big(B_k(u) \big)  = \Bigl(\sum_{i,j}\omega_{ij}^{2}\Bigr)\mathrm{Var}\left(\mu(u,T)\frac{\phi_k(T)}{g(T)}\right).
$$
Moreover, by the definition of $V_k(u)$ and Assumption~\ref{Ass_gen}, the conditional expectation of $V_k(u)$ given $\mathcal F^{\rm obs}$ is equal to zero, so that 
$
\mathrm{Var}\big(V_k(u) \big)  = \mathbb E \left[\mathrm{Var}\big(V_k(u)\mid \mathcal F^{\rm obs}\big) \right].
$
Calculating the conditional variance of $V_k(u)$ and combining with the variance of $B_k(u)$, we get
		\begin{equation}\label{eq:var_decomp}
		\mathrm{Var}\big(\widehat{\beta}_k(u)\big) = \mathrm{Var}\big(B_k(u) \big)  +\mathrm{Var}\big(V_k(u) \big) 
		= \Bigl(\sum_{i,j}\omega_{ij}^{2}\Bigr)D_k(u)
		+ \Bigl(\sum_{i=1}^{n}\sum_{1\leq j\neq j'\leq m_i}\omega_{ij}\omega_{ij'}\Bigr)E_k(u) ,
	\end{equation}
		where 
	\[
	D_k(u) = \mathbb E\!\left[\frac{\phi_k^2(T)}{g^2(T)}\bigl\{\gamma_X(u;T)+\tau^2(T)\bigr\}\right]+ \mathrm{Var}\left(
	\mu(u,T)
	\frac{\phi_k(T)}{g(T)}\right) \geq 0,
	\]
	and 
	\[
	E_k(u) = \mathrm{Var}\!\left(\int_{\mathcal T}\phi_k(t)\,X(u,t)\,\mathrm{d}t\right) \geq 0 .
	\]
With bounded counts $m_i$ the variances of $\widehat{\beta}_k(u)$ with OBS and SUBJ have the same rate, but the sign of their difference depends on the data generating process. Details are provided in Appendix~\ref{app_sec:var_OS}.

According to \cite{LPSZ2025}, in the case of the \emph{random MC weights}, for each $u$, the variance $\mathrm{Var}\big(B_k(u) \big)$ decreases to zero faster than $M^{-1}$, provided $t\mapsto \mu(u,t)$ is Hölder continuous. To formalize this fact in our framework, we consider the following assumption. 

\begin{assumption}[\textnormal{\textbf{Hölder regularity of $\mu$}}]\label{ass:holder}
	There exist $\alpha \in (0,1]$ and a function $L_\mu : \mathcal{U} \to \mathbb{R}_+$ with $\int_{\mathcal{U}} L^2_\mu(u)\,\mathrm{d}\nu(u) < \infty$, such that $	|\mu(u,t) - \mu(u,t')| \leq L_\mu(u)\,|t - t'|^\alpha$,
$\forall\,u \in \mathcal{U}$, $t,t' \in \mathcal{T}$.
\end{assumption}

Assumption~\ref{ass:holder} holds with $\alpha=1$ for any mean function that is partially differentiable in $t$, with  $\sup_t |\partial_t \mu(u,t)|$ square-integrable in $u$. The square-integrability of $L_\mu (u)$ will be used for studying the  risk of $\widehat{\beta}_k(u)$ integrated over $\mathcal U$. For now, we have the following result on the rate of $\mathrm{Var}\big(B_k(u) \big)$ for which we omit the proof as it is a direct consequence of \cite[Th.~1 and 2]{LPSZ2025}; see also \cite[Proposition~3]{KP2025}.

\begin{proposition}\label{lem:cn_error} 
Let Assumptions~\ref{Ass_gen}, \ref{ass:holder} and Condition \eqref{eq:Lip_phi} hold and $\omega_{ij} = M^{-1}  \{1 + \widehat{c}_{ij} - \widehat{d}_{ij}\}$ as defined in Appendix~\ref{app_sec:MC}. Then,  for $M \geq 4$, there exist constants $C_{1,\mathrm{MC}}(u), C_{2,\mathrm{MC}}(u;k) $ such that
$$
\mathrm{Var}\big(B_k(u) \big) \leq  C_{1,\mathrm{MC}}^2(u)\,\|\phi_k\|^2_\infty M^{-1-2\alpha}+C_{2,\mathrm{MC}}^2(u;k)\,M^{-3},
$$	 
with $C_{1,\mathrm{MC}}(u)$ and $C_{2,\mathrm{MC}}(u;k)$ proportional to $L_\mu(u)$ and $\{L_{\phi,k}+L_g \|\phi_k\|_\infty\}\|\mu(u,\cdot)\|_\infty$, respectively. The proportionality constants depend on $L_g$,  $g_{\min}$ and $g_{\max}$, but not on $M$ and $k$. 
\end{proposition}

Note that, whatever the choice of the weights satisfying Assumption~\ref{def:admissible_weights}, the risk of $\widehat{\beta}_k(u)$ is the sum of  $\mathrm{Var}\big(B_k(u) \big) $ and  $\mathrm{Var}\big(V_k(u) \big) $, and only the former depends on $\mu$. Moreover, the rate of $\mathrm{Var}\big(V_k(u) \big) $ cannot be faster than $M^{-1}$. Thus bounding $\mathrm{Var}\big(B_k(u) \big) $ by a constant multiplied by $M^{-1-2\alpha}$ allows one, on the one hand, to derive a sharper bound for the risk of $\widehat{\beta}_k(u)$, and, more importantly, the term depending on $\mu$ in this bound becomes negligible. However, when aggregating the bounds on $\mathrm{Var}\big(B_k(u) \big) $ in the risk of the mean function estimator, the constants $C_{2,\mathrm{MC}}(u;k)$ will matter as they depend on $k$. Proposition~\ref{lem:cn_error} reveals that the choice of the basis $\{\phi_k\}$ matters, as we will have to control the uniform norm of the $\phi_k$ and their derivatives which grow with $k$.

\subsection{Risk bounds for the mean estimators: OBS, SUBJ and MC weights}\label{sec:upper_bounds}

Given the density $g$,  the orthonormal basis $\{\phi_k\}$ and an integer $1\leq K\leq M$, our mean function estimator is 
\begin{equation}\label{eq:mean_est}
	\widehat{\mu}(u, t)=\widehat{\mu}_K(u, t) := \sum_{k=1}^{K} \widehat{\beta}_k(u)\,\phi_k(t),
	\qquad u \in \mathcal{U},\; t \in \mathcal{T}.
\end{equation}
When there is no risk of confusion, we simply write $\widehat \mu$. We consider the integrated quadratic risk of $\widehat{\mu}$ as
\[
\mathcal{R}(\widehat{\mu};K) := \mathbb{E}\!\left[\|\widehat{\mu} - \mu\|_{L^2(\mathcal{U}\times\mathcal{T})}^2\right],
\]
for which the following identity holds: 
\begin{equation}\label{eq:full_decomp}
	\mathcal{R}(\widehat{\mu}; K) = \sum_{k=1}^{K}\mathcal{R}_k + \sum_{k>K}\|\beta_k\|_{L^2(\mathcal{U})}^2=: \sum_{k=1}^{K}\mathcal{R}_k + \mathcal{B}_K^2,
\end{equation}
with
\begin{equation}\label{eq:IRk}
	\mathcal{R}_k
	= \int_{\mathcal{U}} \operatorname{Var}
	\bigl(\widehat{\beta}_k(u)\bigr)\mathrm{d}\nu(u)
	= \mathbb{E}\!\left[
	\|\widehat{\beta}_k - \beta_k\|_{L^2(\mathcal{U})}^2\right].
\end{equation}
To control the rate of decrease of the squared truncation  bias $\mathcal{B}_K^2$, we adapt the usual polynomial decay condition on the projection coefficients to our case where they are functions of $u$.

\begin{assumption}[\textnormal{\textbf{Coefficient decay}}]\label{ass:decay}
There exist a regularity index $s > 0$ and a finite constant $C_s > 0$ such that
$
	\sum_{k=1}^{\infty} k^{2s}\,\|\beta_k\|_{L^2(\mathcal{U})}^2 \leq C_s.
$
\end{assumption}

The exponent $s$ plays the role of an effective regularity index. Coefficient-decay conditions are common in nonparametric estimation  \citep{T2009, E1999}, and they determine the optimal rates for the risk bounds. When the $\beta_k$ are scalars, the class defined by Assumption~\ref{ass:decay} is the classical \emph{Sobolev ellipsoid}.

\begin{remark}\label{rem:regularities}
	Assumption~\ref{ass:decay} is a joint condition on $\mu$ and on the basis
	$\{\phi_k\}$. Consequently, the relationship between $s$ and the highest-order derivative of
	$\mu$ with respect to $t$, the latter being the common notion of regularity
	imposed in kernel-smoothing or spline approaches, depends on the basis and need
	not be the identity. This issue will be discussed in Section~\ref{sec:basis_choice}.
\end{remark}

We derive the risk bounds for $\widehat \mu$ separately for deterministic and random weights. The proofs are given in Appendix~\ref{app:risk_gknown}. Let $C_{1,\mu} = \min \big\{\bar\mu^2,\, \|\mu\|^2_{L^2(\mathcal{U\times \mathcal T}) }\sup_{k\geq 1\!}\|\phi_k\|^2_\infty  \big\}$ with $\bar\mu = \sup_{t\in\T}\!\|\mu(\cdot,t)\|_{L^2(\mathcal{U}) }$. 

\begin{theorem}[Integrated risk bound under deterministic weights]\label{thm:det_risk_bound} 
Under  Assumptions~\ref{Ass_gen}, \ref{ass:decay}, and if the covariance function $\Gamma_X$ in \eqref{def:GammaX} is continuous, for any deterministic admissible weights satisfying Assumption~\ref{def:admissible_weights}, and each $K\geq 1$, it holds (with $[x]_+ = \max(x,0)$)
	\begin{equation}
		\mathcal{R}(\widehat{\mu};K)
		\leq \Bigl(\sum_{i,j}\omega_{ij}^2\Bigr) \Bigl( C_1 K - \|\mu\|_{L^2(\mathcal{U}\times\mathcal{T})}^2 + C_s\,K^{-2s}\Bigr) 
		+ \Bigl[ \sum_{i}\sum_{j\neq j'}\omega_{ij}\omega_{ij'}\Bigr]_+  C_2 + C_s\,K^{-2s},
	\end{equation}
	with 
$C_2 = \mathbb E \bigl[\, \|X\|^2_{L^2(\mathcal{U}\times \mathcal T)\,}\bigr] $, $C_1 = ( C_{1,\mu}+ \|\Gamma_X\|_{\infty}+ \|\tau\|_\infty ^2 \nu(\mathcal U))/g_{\min}$. 
\end{theorem} 

As a consequence of this non-asymptotic bound on the risk of $\widehat \mu$, we deduce the optimal choice of the truncation $K$ with deterministic weights. For simplicity, we only give the optimal truncation choice, minimizing the bound, and the corresponding rate for the OBS and SUBJ weights.  Let  $\bar m_A$ (resp. $\bar m_H$) be the arithmetic (resp. harmonic) mean of the per-subject visit counts $m_i$.

\begin{corollary}\label{cor:case1and2}
	Under the assumptions of Theorem~\ref{thm:det_risk_bound}, the optimal truncations for OBS and SUBJ weights are
	\[
	K_1^* = \bigg\{\frac{2sC_s(n\bar m_A+1)}{C_1}\bigg\}^{\!\frac{1}{2s+1}} \quad \text{and} \quad K_2^* = \bigg\{\frac{2sC_s(n\bar m_H+1)}{C_1}\bigg\}^{\!\frac{1}{2s+1}},
	\]
	respectively. If Condition~\eqref{cd_m_max} holds, then the risk bounds associated with $K_1^*$ and $K_2^*$ are
	\begin{equation}
		\mathcal{R}^*(\widehat{\mu})
		\leq  C(n) \left[\{(2s+1)/(2s)\}\,\bigl(2s\,C_s\, C_1^{2s}\bigr)^{\frac{1}{2s+1}}\right]
		\{n\bar {\mathfrak m}\}^{-\frac{2s}{2s+1}} 
		+ \frac{C_2\,m_{\max}}{M}
		- \frac{\|\mu\|_{L^2(\mathcal{U}\times\mathcal{T})}^2}{n\bar {\mathfrak m}},
	\end{equation}
	with $C(n) = (1+\{n\bar {\mathfrak m}\}^{-1})^{1/(2s+1)}\downarrow 1$,  and $\bar{\mathfrak m}=\bar{m}_A$ and $\bar{\mathfrak m}=\bar{m}_H$, respectively. In particular, with bounded $m_i$ and OBS or SUBJ weights, with $K^*$ any of $K_1^*$ and $K_2^*$, 
	\[
	\mathcal{R}^*(\widehat{\mu})=\mathcal{R}(\widehat\mu; K^*)
	\lesssim (n\bar{m}_H)^{-\frac{2s}{2s+1}} \asymp (n\bar{m}_A)^{-\frac{2s}{2s+1}} \asymp n^{-\frac{2s}{2s+1}} \asymp M^{-\frac{2s}{2s+1}}.
	\]
\end{corollary}

 Since $\bar m_H\le\bar m_A$, we have $K_2^*\le K_1^*$. However, with bounded $m_i$ it holds $\bar m_A\asymp\bar m_H$, so the ratio
\[
\frac{K_1^*}{K_2^*}=\left(\frac{n\bar m_A+1}{n\bar m_H+1}\right)^{\frac{1}{2s+1}}=\left(\frac{\bar m_A}{\bar m_H}\right)^{\frac{1}{2s+1}}\{1+o(1)\}
\]
stays bounded, and the two truncations are of the same order. In practice the truncation should be an integer, so one can take $K=\lceil K^* \rceil$ or the nearest integer to $K^*$. 

Next, we study the risk bound and the optimal truncation for random MC weights. For this, we need to control the  constants $C_{1,\mathrm{MC}}(u)$ and  $C_{2,\mathrm{MC}}(u;k)$ from Proposition~\ref{lem:cn_error} as $k$ increases. See the proofs of \cite[Theorem~1 and 2]{LPSZ2025} and \cite[Proposition~3]{KP2025}. The following assumption is therefore essentially a condition on the orthonormal basis $\{\phi_k\}$.

\begin{assumption}[\textnormal{\textbf{Polynomial growth}}]\label{ass:growth}
For $C_{1,\mathrm{MC}}(u), C_{2,\mathrm{MC}}(u;k)$ the terms in the bound of Proposition~\ref{lem:cn_error}, there exist $q \geq 0$ and a constant $\bar{C}_{\mathrm{MC}} > 0$ such that, with $\tilde \mu^2 = \int_{\mathcal U}\sup_{\T}\mu^2(u,t) {\rm d}\nu(u)$,
$$
\int_{\mathcal{U}} \bigl\{C_{1,\mathrm{MC}}(u)\bigr\}^{\!2}\mathrm{d}\nu (u) \leq \bar{C}_{\mathrm{MC}}\int_{\mathcal U}\!L^2_\mu {\rm d}\nu   \quad \text{ and } \quad \int_{\mathcal{U}} \bigl\{C_{2,\mathrm{MC}}(u;k)\bigr\}^{\!2}\mathrm{d}\nu (u) \leq \bar{C}_{\mathrm{MC}}\tilde \mu^2 k^{2q},\quad \forall k \geq 1.
$$  
\end{assumption}

Assumption~\ref{ass:growth} is satisfied by all standard orthonormal bases on a compact interval, with different values $q$. The trigonometric bases (Fourier, cosine, \dots) are uniformly bounded with $L_{\phi,k}\asymp k$, thus the condition holds with  $q=1$; the Legendre polynomials have uniform norm of order $k^{1/2}$ and $L_{\phi,k}\asymp k^{5/2}$, thus   the condition holds with  $q=5/2$.

	\begin{theorem}[Integrated risk bound with MC weights]\label{thm:cn_risk_bound} 
		Let $\widehat \mu$ be defined with the  MC weights. Under
		Assumptions~\ref{Ass_gen} to~\ref{ass:growth} and Conditions~\eqref{cd_m_max} and~\eqref{eq:Lip_phi}, if the 	covariance function $\Gamma_X$ in \eqref{def:GammaX} is continuous, then for every $M \geq 4$ and every
		$K \geq 1$, it holds
		\begin{equation}\label{eq:bound55}
			\mathcal{R}(\widehat{\mu};K)
			\leq C_s\,K^{-2s}+ 	C_1^{\mathrm{MC}}\frac{K}{M}
			+\bar C_{\mathrm{MC}}\frac{K}{M} \left\{\mathfrak c_1(K)\int_{\mathcal U}L^2_\mu {\rm d}\nu + \mathfrak c_2(K)\tilde \mu^2\right\},
		\end{equation}
with $C_1^{\mathrm{MC}}\! = \kappa \{\|\Gamma_X\|_{\infty}m_{\max} + \|\tau\|_\infty^2\,\nu(\mathcal U)\}g^{-1}_{\min}$ and $\kappa$ the constant of Lemma~\ref{lem:cn_weighted_second_moment} in Appendix~\ref{app_sec:MC}, and
$$
\mathfrak c_1(K) =M^{-2\alpha}K^{-1}\bigg(\sum_{k=1}^K\|\phi_k\|_\infty^2\bigg) \quad \text{and} \quad \mathfrak c_2(K)=\bigg(\frac{K^{q}}{M}\bigg)^{\!2}.
$$
	\end{theorem}
	
	Unlike $C_1$ appearing in the risk bound with deterministic weights from Theorem~\ref{thm:det_risk_bound}, the  constant $C_1^{\mathrm{MC}}$ does not involve $\mu$. Therefore, when the terms $\mathfrak c_1(K)$ and $\mathfrak c_2(K)$ are negligible, the influence of $\mu$ on the bound vanishes asymptotically.

	\begin{corollary}[Risk bound with MC weights]\label{thm:cn_optim}   
		Let the assumptions of Theorem~\ref{thm:cn_risk_bound} hold. Let $K^*_{\mathrm{MC}} = \bigl\{2s\,C_s\,M\,  \{C_1^{\mathrm{MC}}\}^{-1}\bigr\}^{1/(2s+1)}$ be the minimizer of the map
		$K \mapsto C_s\,K^{-2s}+ C_1^{\mathrm{MC}} K M^{-1}$. Then, 
				\begin{multline}
			\mathcal{R}^*(\widehat\mu)=\mathcal{R}(\widehat\mu; K^*_{\mathrm{MC}})
			\leq \left[\{(2s+1)/(2s)\}\,\bigl(2s\,C_s\,  \{C_1^{\mathrm{MC}}\}^{2s}\bigr)^{\frac{1}{2s+1}}\right] M^{-\frac{2s}{2s+1}} \\
			+ \bar C_{\mathrm{MC}}\bigl(2s\,C_s\,  \{C_1^{\mathrm{MC}}\}^{-1}\bigr)^{\frac{1}{2s+1}} \left\{\mathfrak c_1(K^*_{\mathrm{MC}})\int_{\mathcal U}L^2_\mu {\rm d}\nu + \mathfrak c_2(K^*_{\mathrm{MC}})\tilde \mu^2\right\} M^{-\frac{2s}{2s+1}} .
		\end{multline}
	\end{corollary}
	
The mean estimator constructed with MC weights attains the rate of Corollary~\ref{cor:case1and2} with a leading constant asymptotically free of $\mu$ as soon as $\mathfrak c_1(K^*_{\mathrm{MC}})+\mathfrak c_2(K^*_{\mathrm{MC}})\rightarrow 0$. For any $\alpha,s >0$, the terms $\mathfrak c_1(K^*_{\mathrm{MC}})$ and $\mathfrak c_2(K^*_{\mathrm{MC}})$  tend to zero for a trigonometric basis, possibly augmented by a few monomials in $t$ as it will be done in Section~\ref{sec:basis_choice}. For the Legendre orthonormal basis, $\mathfrak c_1(K^*_{\mathrm{MC}})\asymp M^{-\{2\alpha -1/(2s+1)\}} $ and $\mathfrak c_2(K^*_{\mathrm{MC}})\asymp M^{-\{2-5/(2s+1)\}} $ and they tend to zero if $s>3/4$ and  $2\alpha(2s+1)>1$.

\subsection{A design-density-free optimal estimator of the mean}\label{sec:spacings}

	The estimators above are built on \eqref{eq:rep_bk} and involve the ratio
	$\phi_k(T_{ij})/g(T_{ij})$. The Voronoi construction of the MC weights still uses $g$. We now follow
	\cite[Sec.~4.2]{E1999} and replace $\phi_k /g $ by an empirical quadrature based
	on the pooled design, thereby avoiding the use of $g$. We call this the \emph{spacings} approach, which we eventually recommend based on its theoretical properties and practical effectiveness.

	Pool the observations into a single ordered sample $T_{(1)}\leq\cdots\le T_{(l)}\leq \cdots \leq T_{(M)}$, let
	$\varrho$ be the random bijection mapping a pair $(i,j)$ into the rank $l$, and write $Y_{l}(\cdot):=Y_{\varrho^{-1}(l)}(\cdot)$. Note that reordering leaves $\mathcal F^{\rm obs}$ unchanged. Extend the sample by symmetric reflection,
	$T_{(l)}:=2T_{(1)}-T_{(2-l)}$ for $l\leq 0$ and $T_{(l)}:=2T_{(M)}-T_{(2M-l)}$ for
	$l>M$, so that a window is defined around every rank; no observation is attached to a
	reflected point. Following
	\cite{E1999} we take $\mathfrak h=\lceil\{1+\log\log(M+20)\}/2\rceil$, equal to $2$ for
	every $M$ up to $M\approx 5.3\times10^{8}$. With
	$\chi_l(t):=\mathbf 1\{T_{(l-\mathfrak h)}\leq t\leq T_{(l+\mathfrak h)}\}$, set
	\begin{equation}\label{eq:sp_est}
		\widehat\Phi_{kl}:=\frac{1}{2\mathfrak h}\int_{\mathcal T}\phi_k(t)\chi_l(t)\,\mathrm{d}t,
		\qquad
		\widehat\beta^{\rm sp}_k(u):=\sum_{l=1}^{M}Y_{l}(u)\,\widehat\Phi_{kl},
	\end{equation}
	and let $\widehat\mu^{\rm sp}$ be the estimator \eqref{eq:mean_est} built with the
	$\widehat\beta^{\rm sp}_k$. Where $g$ is large the design is dense and the windows are short, and conversely where $g$ is small the windows are large, so that the geometry of the order statistics reproduces the correction $1/g(T_{ij})$ from the data alone. Let $\varrho^{-1}(l)=(i,j)$. Writing $\widehat\beta^{\rm sp}_k(u)=\sum_{i,j}\widetilde\omega_{ij,k}Y_{ij}(u)$ with $\widetilde\omega_{ij,k}:=\widehat\Phi_{k,\varrho(i,j)}$ shows that the implicit
	weights are random, $\mathcal F^{\rm obs}$-measurable and, unlike all previous schemes,
	$k$-dependent, having absorbed $\phi_k/g$. Alternatively, one can define 
		\begin{equation}
		\widetilde\Phi_{l}:=\frac{1}{2\mathfrak h}\int_{\mathcal T} \chi_l(t)\,\mathrm{d}t,
		\qquad
		\widetilde\beta^{\rm sp}_k(u):=\sum_{l=1}^{M}Y_{l}(u)\phi_k(T_{(l)}) \widetilde\Phi_{l},
	\end{equation}
	to avoid having weights changing with $k$. However, as discussed in \cite[p.~129]{E1999}, including $\phi_k$ in the integral in~\eqref{eq:sp_est} is expected to provide more accurate empirical quadrature for highly oscillatory $\phi_k$. Moreover, this does not make the computations more complicated, as it will be shown in Section~\ref{sec:basis_choice}.

	The weights $\widetilde\omega_{ij,k}=\widehat\Phi_{k,\varrho(i,j)}$ fall outside
	Assumption~\ref{def:admissible_weights}, and $\widehat\beta^{\rm sp}_k(u)$ is no longer
	unbiased: with $B^{\rm sp}_k(u):=\sum_l\mu(u,T_{(l)})\widehat\Phi_{kl}
	=\mathbb E[\widehat\beta^{\rm sp}_k(u)\mid\mathcal F^{\rm obs}]$ and
	$V^{\rm sp}_k(u):=\widehat\beta^{\rm sp}_k(u)-B^{\rm sp}_k(u)$,
	\begin{equation}\label{eq:sp_decomp}
		\mathcal R\bigl(\widehat\mu^{\rm sp};K\bigr)
		=\sum_{k=1}^{K}\int_{\mathcal U}\Bigl\{\mathrm{Var}\bigl(B^{\rm sp}_k(u)\bigr)
		+\mathbb E\bigl[\mathrm{Var}\bigl(V^{\rm sp}_k(u)\!\mid\!\mathcal F^{\rm obs}\bigr)\bigr]
		+\bigl[\mathbb E[B^{\rm sp}_k(u)]-\beta_k(u)\bigr]^{2}\Bigr\}\mathrm{d}\nu(u)
		+\mathcal B^2_K ,
	\end{equation}
	the last term under the integral being specific to this scheme. It will be shown that the role of
	$\sum_{i,j}\omega^2_{ij}$ is now played by $\sum_{l=1}^{M}\mathbb E\bigl[\widehat\Phi^{\,2}_{kl}\bigr]$ for which we derive a $k$-uniform bound.

	\begin{lemma}
		\label{lem:sp_second_moment}
		Under Assumption~\ref{ass:design}, for every $k\geq1$ and every $M\geq 2$, it holds
		\[
		\sum_{l=1}^{M}\mathbb E\bigl[\widehat\Phi^{\,2}_{kl}\bigr]
		\;\leq\;\frac{2\mathfrak h+1}{2\mathfrak h\,(M+1)\,g_{\min}} .
		\]
	\end{lemma}

	\begin{theorem}[Integrated risk bound and optimal truncation, spacings estimator]\label{thm:sp_risk_optim}
		Under Assumptions~\ref{Ass_gen}, \ref{ass:holder} and \ref{ass:decay}, if the covariance
 $\Gamma_X$ in \eqref{def:GammaX} is continuous and Condition~\eqref{cd_m_max} holds, then $\forall K\geq1$,
		\begin{equation}
			\mathcal R\bigl(\widehat\mu^{\rm sp};K\bigr)
			\leq \mathcal S^{\rm sp}\,C_1^{\rm sp}\,K
			+ C_2^{\rm sp}\Bigl(\frac{2\mathfrak h}{M+1}\Bigr)^{\!2\alpha}
			+ C_3^{\rm sp}\,\frac{2\mathfrak h+1}{M+1}
			+ C_s\,K^{-2s},
			\qquad
			\mathcal S^{\rm sp}:=\frac{2\mathfrak h+1}{2\mathfrak h\,(M+1)},
		\end{equation}
		with $C_1^{\rm sp}= (m_{\max}\|\Gamma_X\|_\infty 
		+ \|\tau(\cdot)\|^2_\infty\nu(\mathcal U))/g_{\min}$, $C_2^{\rm sp}= 24\,\|L_\mu\|^2_{L^2(\mathcal U)}/g_{\min}^{2\alpha+1}$, and $C_3^{\rm sp}= 8\bar \mu^2/g_{\min}$. The risk bound is minimized at
$
		K^*_{\rm sp}=\bigl\{2s\,C_s\,(\mathcal S^{\rm sp}C_1^{\rm sp})^{-1}\bigr\}^{\!1/(2s+1)},
$
		and the resulting risk satisfies
		\begin{equation}
		 \quad 	\mathcal R^*(\widehat\mu^{\rm sp})
			\leq \left[\{(2s+1)/(2s)\}\,\bigl(2s\,C_s\,\{ C_1^{\rm sp}\}^{2s}\bigr)^{\frac{1}{2s+1}}\right]M^{-\frac{2s}{2s+1}}
			+ C_2^{\rm sp}\Bigl(\frac{2\mathfrak h}{M+1}\Bigr)^{\!2\alpha}
			+ C_3^{\rm sp}\,\frac{2\mathfrak h+1}{M+1} .
		\end{equation}
		In particular, it holds $\mathcal R^*(\widehat\mu^{\rm sp})\lesssim M^{-2s/(2s+1)}$, provided $\alpha >s/(2s+1)$.
	\end{theorem}
	
	The proofs of Theorem~\ref{thm:sp_risk_optim} and Lemma~\ref{lem:sp_second_moment} can be found in Appendix~\ref{app_sec:sp}.	The term in $C_3^{\rm sp}$ comes from the two boundary blocks, of expected length $O(\mathfrak h/M)$, on which the windows do not cover $\mathcal T$ with full
	multiplicity. It is negligible for every $s>0$. The estimator attains the rate of Corollaries~\ref{cor:case1and2} and~\ref{thm:cn_optim} without ever using
	$g$, at the price of the condition $\alpha> s/(2s+1)$. This condition is very mild: if the mean function is continuously
	differentiable with respect to $t$ and its partial derivative satisfies a mild
	integrability condition over $\mathcal U$, then Assumption~\ref{ass:holder} holds
	with $\alpha=1$, whereas $s/(2s+1)<1/2$,  $\forall s>0$. Finally, note that by \eqref{eq:sp_est}, the estimator only requires the integral of $\phi_k$ over a window; these are given in closed form for the bases of Section~\ref{sec:basis_choice}, so that no numerical quadrature is needed.

\begin{remark}
	We derived in this section non-asymptotic risk bounds and the optimal truncation for several types of series-expansion estimators of the mean. The resulting optimized bounds share the same rate, determined by a standard decay-rate assumption on the coefficients and the basis choice. We provide explicit bounds for the constants in front of the optimal rate. Our results extend theory from one-dimensional nonparametric regression to longitudinal functional data with sparse random design; see \cite{T2009}. Next, we will show that the bounds are minimax optimal. 
\end{remark}


\subsection{Minimax lower bound and optimality}\label{sec:lower_bound}

Let $\T=[0,1]$. Let $\Fclass(s,C_s)$ be the set of $\mu\in\Lto{\U\times\T}$ satisfying Assumption~\ref{ass:decay} for a given orthonormal basis $\{\phi_k\}\subset L^2(\T)$. For $\overline g\geq 1$, $\tau_0>0$, let $\mathcal P=\mathcal P(s,C_s,\overline g ,\tau_0)$ be the set of distributions of the data $\{(T_{ij},Y_{ij})\}$ generated by model~\eqref{eq:model0} under Assumption~\ref{Ass_gen} with noise standard deviation such that $\inf_{t\in\T}\tau(t)\geq \tau_0$, a design density $g\leq \overline g$, the mean $\mu\in\Fclass(s,C_s)$, and with $X_i$ and $\eta_{ij}$ zero-mean Gaussian fields. Note that no Gaussianity assumption is imposed for upper bounds.

\begin{theorem}\label{thm:lb}
	Let $s>0$, $C_s>0$, and let Condition~\eqref{cd_m_max} hold. Then, there exists a constant $c>0$ depending on $s$, $C_s$, $\overline g $, $\tau_0$ and $\nu(\U)$ such that, for all $n$ sufficiently large,
	\[
	\inf_{\widehat\mu}\;
	\sup_{\mathcal P}
	\E\norm{\widehat\mu-\mu}_{L^2(\U\times\T)}^2
	\ge c\, M^{-\frac{2s}{2s+1}}
	\ge c\,m_{\rm max}^{- 2s/(2s+1)} n^{-\frac{2s}{2s+1}},
	\]
	the infimum being over all estimators measurable in $\{(T_{ij},Y_{ij})\}$ and the supremum over all distributions in $\mathcal P$. 
\end{theorem}

Theorem~\ref{thm:lb} is a version of the lower bound in~\cite[Th.~3.1]{CY2011} where the rate also contains a parametric term $n^{-1}$, which is negligible here under Condition~\eqref{cd_m_max}. Here, the rate of the lower bound, like the rates of the bounds derived in Section~\ref{sec:upper_bounds}, is free of the domain $\U$. The reason is that the profiles are observed fully in $u$, so the problem is one-dimensional in $t$ and driven by the total count $M$. 

To establish the minimax optimality, it remains to match the lower rate of Theorem~\ref{thm:lb} with the risk upper bounds of Section~\ref{sec:upper_bounds}. In this subsection, we consider the estimators with deterministic weights, OBS and SUBJ; those with MC weights and the spacings estimator are considered in Section~\ref{sec:minmax_random_w}, after the choice of the basis. The upper bounds involve constants depending on the distribution of the data. To obtain bounds that are uniform over a class of distributions, we restrict $\mathcal P$ as follows. Let $\underline g\in(0,1]$, $\overline\Gamma\geq0$ and $\overline\tau\geq\tau_0$, and let $\mathcal P^+\subset\mathcal P$ denote the set of distributions in $\mathcal P$ for which, in addition, the covariance function $\Gamma_X$ is continuous with $\|\Gamma_X\|_\infty\leq\overline\Gamma$, $\|\tau\|_\infty\leq\overline\tau$ and $\underline g\le g$. Finally, let $\widetilde K^*=\lceil M^{1/(2s+1)}\rceil$.

\begin{corollary}[Minimax optimality]\label{cor:minimax}
Let $s>0$, $C_s>0$,  Condition~\eqref{cd_m_max} hold, and assume $\{\phi_k\}$ is uniformly bounded  and Condition \eqref{eq:Lip_phi} holds. Then for $\widehat\mu_\bullet$ defined with OBS or SUBJ weights and $K=\widetilde K^*$, there exist constants $0<c\leq C$ such that, for all $n$ sufficiently large,
		\[
		c\, M^{-\frac{2s}{2s+1}}
		\le \inf_{\widehat\mu}\;\sup_{\mathcal P^+}\E\norm{\widehat\mu-\mu}_{L^2(\U\times\T)}^2
		\le \sup_{\mathcal P^+}\mathcal R(\widehat\mu_\bullet ;\widetilde K^*)
		\le C\, M^{-\frac{2s}{2s+1}} .
		\]
The constant $c$ depends on $s$, $C_s$, $\overline g$, $\tau_0$ and $\nu(\U)$; the constant $C$ depends, in addition, on $m_{\max}$, $\underline g$, $\overline\Gamma$, $\overline\tau$ and  $\sup_{k\geq1}\|\phi_k\|_\infty$.
\end{corollary}

The truncation $\widetilde K^*$ does not depend on the unknown constants, the minimax rate is attained without knowledge of the distribution of the data. The optimal truncation of Corollary~\ref{cor:case1and2} may provide a better theoretical constant.

To compare our optimal rate to those expected with kernel smoothing or spline approaches, we choose a suitable basis and build, in the next section, a bridge between the decay index $s$ from Assumption~\ref{ass:decay} and the smoothness (differentiability) of $\mu$ as a function of $t$. 

\section{Basis choice}\label{sec:basis_choice}

Thus far, our methodology has been presented for a generic basis. It was shown that the random weights have some advantage from a theoretical standpoint, and this will also be illustrated by simulations in Appendix~\ref{append_simus_est}. A bounded orthonormal basis appears preferable for the theory with random weights, and the half-cosine system on $\mathcal T$ is a natural candidate; see \cite{E1999}. It is, however, imperfect in two related respects, both traceable to its behavior at the endpoints of the interval. First, its approximation power saturates: however smooth the mean function, its coefficients decay no faster than a fixed rate, so in general regularity such as twice differentiable or more is left unexploited. Second, and for the same reason, attaining the optimal rate (\emph{i.e.}, the upper bound of the regularity $s$ reaching the order of the  highest derivative  of $\mu$ in $t$) would require the mean function to satisfy derivative conditions at the endpoints, conditions which may not hold in practice.  See \cite[pp.~32,~52--56]{E1999} for a discussion. Both limitations are removed at once by enriching the system with low-degree monomials in $t$, the degree matched to the assumed smoothness. For example, a customary assumption in nonparametric statistics is that the underlying target function is twice differentiable. Accordingly, we propose to augment the half-cosine basis with the two monomials $t$ and $t^2$ and to orthonormalize the enlarged family by the Gram-Schmidt procedure. The resulting augmented half-cosine basis remains uniformly bounded, leads to closed-form weights for the spacings approach, while now accommodating the full range of twice-differentiable mean functions in $t$, with no endpoint restriction, and leading to optimal rate $O(M^{-4/5})$ for our mean estimator; see~\cite{S1982}. The augmented-basis construction is related to the endpoint-subtraction idea underlying the so-called Krylov approximants from approximation theory; see, e.g., \cite{BDT1995}. As shown below and in Section~\ref{sec:numerics}, the augmented basis is also simple to implement and effective in practice.

\subsection{Augmented half-cosine orthonormal system}\label{subsec:aug}

To formalize our choice, without loss of generality set $\mathcal T = [0,1]$, and consider the so-called
\emph{half-cosine} (or simply \emph{cosine}) orthonormal basis
\[
c_1 \equiv 1,
\qquad
c_k(t) = \sqrt{2}\,\cos(\lambda_k t),
\quad \lambda_k := (k-1)\pi,
\qquad k \ge 2.
\]
Next, for $r\in\mathbb N$, we consider the ordered dictionary $\mathcal D_{r,K}=(c_1,c_2,\ldots\,c_K\big)$ if $r=0$ and
\begin{equation}\label{eq:orth_system}
\mathcal D_{r,K} = \big(\,c_1,\;\underbrace{p_1,\ldots,p_r}_{\text{monomials in $t$}},\;
c_2, c_3,\ldots\,c_K\big), \quad \text{with } \ p_{r'}(t)=t^{r'},  \quad 1\leq r' \leq r.
\end{equation}
This dictionary can be shown to be linearly independent. The augmented basis, denoted 
$$
\mathcal B_{r,K} = \{\phi_1,\ldots,\phi_{r+K}\},
$$  
is  obtained by Gram-Schmidt orthogonalization of $\mathcal D_{r,K} $, and the $\phi_k$'s have closed-form expressions. 

Hereafter, we focus on the case $r=2$. The first three functions of $\mathcal B_{2,K}$ are
$\phi_1(t) = 1$, $ \phi_2(t) = \sqrt3\,(2t-1)$, $\phi_3(t) = \sqrt5\,(6t^2-6t+1)$.
 Next, for $K\geq 2$ and $4\leq k\leq K+2$, it holds
 \[
 \phi_k = \frac{\psi_k}{\|\psi_k\|},
 \qquad
 \psi_k = c_{k-2} - \sum_{m=1}^{k-1}\langle c_{k-2},\phi_m\rangle\,\phi_m,
 \qquad
 \|\psi_k\|^2 = 1 - \sum_{m=1}^{k-1}\langle c_{k-2},\phi_m\rangle^2. 
 \]
Specifically, by elementary calculations,
 \[
 \phi_4 
 = \frac{\sqrt2\big[\cos(\pi t) + 12(2t-1)/ \pi^2\big]}{\sqrt{1-96/\pi^4}},
\qquad \phi_5 
 = \frac{\sqrt2\big[\cos(2\pi t) - 15(6t^2-6t+1)/ \pi^2\big]}{\sqrt{1-90/\pi^4}}.
 \]
 The recurrence formulae for the next basis functions are given in Appendix~\ref{app_sec:basis2}. We also show there that
 \[
 \|\phi_k\|_\infty < 4\sqrt2
 \qquad\text{and}\qquad
 \|\phi_k^\prime\|_\infty\lesssim k ,\quad \forall k\geq 1.
 \]
The two appealing properties of the half-cosine basis, uniformly bounded and derivative increasing linearly with $k$,  are therefore preserved
 after the augmentation with $t$ and $t^2$  ($r=2$).

\subsection{Coefficient decay rate}
 
 Let $s^\star$ be the supremum of the set of indices $s$ for which Assumption~\ref{ass:decay} holds; $s^\star$  may or may not be attained. We now discuss the relationship between the smoothness of $\mu$ as a function of $t$ and $s^\star$. This is helpful for building a bridge between the optimal risk rates $M^{-2s/(2s+1)}$ derived above and the optimal rates derived under smoothing assumptions, as used for kernel-smoothing or splines approaches. For simplicity, we adopt the common twice continuously differentiable assumption, though a similar bridge can also be built for smoother functions. More precisely, let us consider:
\begin{equation}\label{cd:mu_twice}
\text{(Smoothness condition) }\, \,  \mu(u,\cdot)\in C^2([0,1]) \text{ (a.e. $u$)}   \, \, \text{ and } \, \, \int_{\mathcal U}\sup_{t\in\mathcal T}|\partial_t^{2}\mu(u,t)|^2\,\mathrm{d}\nu(u)<\infty.
\end{equation}
In this case, the expected nonparametric rate under smoothing assumptions is $O(M^{-4/5})$.

Let 
$$
S:=\partial_t\mu(\cdot,0)+\partial_t\mu(\cdot,1) \qquad \text{ and } \qquad 
\Delta:=\partial_t\mu(\cdot,1)-\partial_t\mu(\cdot,0).
$$
Table~\ref{table:clear} depicts the full landscape of situations where Assumption~\ref{ass:decay} holds with $s=2$,
 and thus indicates that augmenting the half-cosine basis by $t$ and $t^2$ allows achieving the $O(M^{-4/5})$ rate for mean functions satisfying~\eqref{cd:mu_twice}. More precisely, for Assumption~\ref{ass:decay} to hold with index $s =2$, 
 no augmentation requires $\partial_t\mu(\cdot,0)=\partial_t\mu(\cdot,1)=0$; augmentation by $t$ requires only equal endpoint slopes $\partial_t\mu(\cdot,0)=\partial_t\mu(\cdot,1)$; augmentation by $t,t^2$ does not require any restriction. When a row's condition in  Table~\ref{table:clear}  fails, the corresponding basis has $s^\star = 3/2$ (and the value $3/2$ is not attained). The detailed justification of the entries in Table~\ref{table:clear} is provided in Appendix~\ref{app_sec:basis}.
 

\begin{table}[ht!]
 \footnotesize
 	\centering
	\renewcommand{\arraystretch}{1}
	\begin{tabular}{@{}lcl@{}}
		\toprule
		Augmentation & Monomials added &  Condition for $s=2$ \\
		\midrule
		None (simple)   & ---                   & $S=0$ and $\Delta=0$ \\
		Linear          & $t$                    & $\Delta=0$ \\
		Quadratic       & $t,\,t^2$    & none \\
		\bottomrule
	\end{tabular}
	\caption{\small Endpoint-derivative conditions  under which the augmented half-cosine basis
	allows Assumption~\ref{ass:decay} to hold with index $s=2$, for $\mu $ satisfying the condition~\eqref{cd:mu_twice}. All equalities hold in $L^2(\mathcal U)$.}		
\label{table:clear}
\end{table}


\subsection{Minimax optimality with the augmented cosine basis}\label{sec:minmax_random_w}

Here, we complete Corollary~\ref{cor:minimax} and prove the minimax optimality of the mean estimators with MC weights and spacings, constructed with $\{\phi_k\}$ the cosine basis augmented by $t$ and $t^2$ ($r=2$), over a class of mean functions that includes those satisfying Condition~\eqref{cd:mu_twice}. For $L>0$ and $\overline L_g\geq0$, and with $\mathcal P^+$ introduced in Section~\ref{sec:lower_bound} with $s=2$ and this basis, let $\mathcal P^+_{L}\subset\mathcal P^+$ denote the set of distributions in $\mathcal P^+$ for which, in addition, $L_g\leq\overline L_g$ and $\mu$ satisfies Assumption~\ref{ass:holder} with $\alpha=1$ and $\int_{\U}L^2_\mu\,\mathrm d\nu\leq L^2$. If $\mu\in\Fclass(2,C_s)$ satisfies Condition~\eqref{cd:mu_twice}, then Assumption~\ref{ass:holder} holds with $\alpha=1$, $L_\mu(u)=\sup_{t\in\T}|\partial_t\mu(u,t)|$ and
\[
\int_{\U}L^2_\mu(u)\,\mathrm d\nu(u)\leq 216\,C_s+2\int_{\mathcal U}\sup_{t\in\mathcal T}|\partial_t^{2}\mu(u,t)|^2\,\mathrm{d}\nu(u);
\]
see the proof of Corollary~\ref{cor:minimax_sp}. Thus, $\mathcal P^+_{L}$ contains all the distributions in $\mathcal P^+$ with a mean function satisfying Condition~\eqref{cd:mu_twice} with $\int_{\U}\sup_{t\in\T}|\partial_t^{2}\mu(u,t)|^2\,\mathrm d\nu(u)\le \overline L_\partial ^2$, as soon as $L^2\geq 216\,C_s+2\overline L_\partial^2$.

\begin{corollary}\label{cor:minimax_sp}
Let Condition~\eqref{cd_m_max} hold. Then, for $\widehat\mu_\bullet $ defined with MC weights or equal to $\widehat\mu^{\rm sp}$, with $K=\widetilde K^*=\lceil M^{1/5}\rceil$, there exist constants $0<c\leq C$ such that, for all $n$ sufficiently large,
		\[
		c\, M^{- 4/5}
		\le \inf_{\widehat\mu}\;\sup_{\mathcal P_{L}^+}\E\norm{\widehat\mu-\mu}_{L^2(\U\times\T)}^2
		\le \sup_{\mathcal P_{L}^+}\mathcal R(\widehat\mu_\bullet;\widetilde K^*)
		\le C\, M^{-4/5} .
		\]
\end{corollary}

The constants of Corollary~\ref{cor:minimax_sp} depend on the parameters defining the set $\mathcal P^+_{L}$, on $m_{\max}$ and $\nu(\U)$, and on the basis through $\sup_k\|\phi_k\|_\infty$ and $\sup_k\|k^{-1}\phi_k'\|_\infty$; see the proof. As for Corollary~\ref{cor:minimax}, the truncation $\widetilde K^*$ does not depend on the unknown constants, while the optimal truncations of Corollary~\ref{thm:cn_optim} and Theorem~\ref{thm:sp_risk_optim} may provide better theoretical constants.


\subsection{Explicit weights for the spacings approach}\label{sec:explicit_spc}

The closed form of the augmented basis elements as derived in Section~\ref{subsec:aug} makes the spacings coefficients elementary. With $\widehat \Phi_{kl}$ defined in~\eqref{eq:sp_est}, for an integer $p>0$ and a frequency $\varpi$,
write
\[
\delta_{l,p} := T_{(l+\mathfrak h)}^{\,p}-T_{(l-\mathfrak h)}^{\,p},
\qquad
\Sigma_l(\varpi) := \sin\!\big(\varpi T_{(l+\mathfrak h)}\big)-\sin\!\big(\varpi T_{(l- \mathfrak h)}\big).
\]
Then, we get the explicit expressions $\widehat \Phi _{1l} =  \delta_{l,1}/(2\mathfrak h)$, $\widehat \Phi _{2l} = \sqrt{3}\big(\delta_{l,2}-\delta_{l,1}\big)/(2\mathfrak h)$, and 
\[
\widehat \Phi _{3l} = \frac{\sqrt5}{2\mathfrak h}\big(2\delta_{l,3}-3\delta_{l,2}+\delta_{l,1}\big),
\qquad 
\widehat \Phi _{4l} = \frac{1}{2\mathfrak h}\,\frac{\sqrt2}{\sqrt{1-96/\pi^4}}
\bigg[\frac{\Sigma_l(\pi)}{\pi}
+ \frac{12}{\pi^2}\big(\delta_{l,2}-\delta_{l,1}\big)\bigg].
\]
Next, for $k\geq 4$, with $d=d(k) = \lfloor( k-2)/2\rfloor $ and $\varpi_\ell = (2\ell-1)\pi$ for even $k$ and $\varpi_\ell \!= 2\ell\pi$ for odd $k$,
\[
\widehat \Phi _{kl}
= \frac{1}{2 \mathfrak h}\bigg[\int_{T_{(l-\mathfrak h)}}^{T_{(l+ \mathfrak h)}} P_k(t)\,\mathrm{d}t
+ \sqrt2\sum_{\ell =1}^{d}\theta_{k,\ell}\,\frac{\Sigma_l(\varpi_\ell )}{\varpi_\ell}\bigg],
\]
where $P_k$ is a polynomial of degree at most 2, so that the integral is a combination of $\delta_{l,1},\delta_{l,2},\delta_{l,3}$. The explicit formulae for $P_k$ and the coefficients $\theta_{k,\ell}$ are provided in Appendix~\ref{app_sec:basis2}. In conclusion, the calculations of the $\widehat \Phi_{kl}$'s defined in~\eqref{eq:sp_est} do not require any numerical quadrature.

\section{Testing the mean function}\label{sec:test}

In this section we assume that Condition~\eqref{cd_m_max} holds. Let $\mathcal T=[0,1]$. Given a mean function $\mu_0$, we want to construct a simple test for the adequacy of this function in the sparse longitudinal design of Section~\ref{sec:model}, that is, a test of $H_0 : \mu \equiv \mu_0$ against
nonparametric alternatives. If $\mu_0$ is fully specified, in the sense that
it is a known function or can be estimated sufficiently fast, then after centering the observations as
$Y_{ij}(\cdot) - \mu_0(\cdot,T_{ij})$, the problem reduces to testing $H_0 : \mu \equiv 0$ against $H_1 : \mu \neq 0$.

Using the series expansion of Section~\ref{sec:series}, a natural way to test the mean, and
in particular its nullity, is to check whether the squared norms of the
coefficients $\beta_k$ vanish over a chosen index set $\mathcal K$, \emph{i.e.},
\begin{equation}\label{eq:H0_full_b}
\mathcal S(\mathcal K;\mu)	= \sum_{k\in\mathcal K}
	\bigl\| \beta_k\bigr\|_{L^2(\mathcal U)}^2 = 0,
	\qquad \mathcal K \subseteq \mathbb N^* .
\end{equation} 
The nullity of $\mu$ is equivalent to \eqref{eq:H0_full_b} considered with $\mathcal K = \mathbb N^*$. More generally,
\eqref{eq:H0_full_b} also allows us to test structural (composite) hypotheses, such as
whether $\mu$ depends at most quadratically on $t$ ($\mathcal K = \{4,5,\ldots\}$ with the cosine basis augmented by $t,t^2$). We therefore reformulate the null hypothesis under the more general form, allowing for composite hypotheses:  for a given (augmented) cosine basis  $\{\phi_k\}$ and a set of indices $\mathcal K \subseteq \mathbb N^*$, 
\begin{equation}\label{eq:H0_full}
	H_0 :\mathcal S(\mathcal K;\mu) = 0
	\qquad\text{against}\qquad
	H_1 :\mathcal S(\mathcal K;\mu)\neq 0 .
\end{equation}
With finite samples, only finite subsets  ${\mathcal K}_M \subset  \mathcal K$ can be considered. Then, we  build a test statistic estimating $\mathcal S( {\mathcal K}_M;\mu)$, which is defined as in~\eqref{eq:H0_full_b} with the testing set $ {\mathcal K}_M$ instead of $ {\mathcal K}$, and allow $ {\mathcal K}_M$ to increase with the sample size $n$. 

Note that augmenting the cosine basis is unnecessary for testing whether the mean function is constant in $t$, in particular whether $\mu\equiv 0$. Indeed, the augmentation by $t,t^2$ controls boundary bias in estimation, which plays no role here. Moreover, a sufficiently large set $\mathcal K_M$ already spans a subspace rich enough to detect deviations from $H_0$. Therefore, unless testing a composite hypothesis such as whether the mean depends at most quadratically on $t$, the test can use the half-cosine basis. 

Let
\[
\Xi_i:=\bigl(m_i,\,X_i,\,\{T_{ij}\}_{1\le j\le m_i},\,\{\eta_{ij}\}_{1\le j\le m_i}\bigr),
\qquad i\ge1.
\]
Hence, the $\Xi_i$'s are the per-subject observation units. We also introduce a  mild technical condition.

\begin{assumption}\label{ass:random_mi}
	The $m_i$, $i\geq 1$ are positive i.i.d. integer variables distributed as $\mathfrak m$, and bounded by a constant  $m_{\max}$. Moreover, $\{m_i\}$ are  independent of $\{T_{ij}\}$, $\{X_i\}$, $\{\eta_{ij}\}$, and thus $\{\Xi_i\}$ are i.i.d.
\end{assumption}

Assumption~\ref{ass:random_mi} allows us to apply the theory of degenerate $U$-statistics to our statistic which will be defined below as a function of the i.i.d. variables $\Xi_1,\ldots,\Xi_n$. This technical condition has very little impact for the applications as the distribution of $\mathfrak m$ is arbitrary.  Note that $M$ is now random. Moreover, we redefine the $\sigma$-field $\mathcal F^{\rm obs}$ from Section~\ref{sec:model} as being generated by  $\{T_{ij}\}$ and $\{m_i\}$.

We now introduce the quantities that will be used to construct our test. Let $l=(i,j)$ and $l^\prime = (i^\prime,j^\prime)$, and assuming $g$ is known, define
\[
\Omega(l,l';{\mathcal K}_M)
:= \frac{\Psi (T_{ij},T_{i'j'};{\mathcal K}_M)}{g(T_{ij})\,g(T_{i'j'})},
\quad\text{ with }\ 
\Psi(t,s;{\mathcal K}_M) := \sum_{k\in\mathcal K_M}\phi_k(t)\phi_k(s),
\]
and
\begin{equation}\label{eq:def_Hn_g}
H_n(\Xi_i , \Xi_{i'}) := \sum_{j=1}^{m_i}\sum_{j'=1}^{m_{i'}}
\Omega(l,l';{\mathcal K}_M) \int_{\mathcal U} Y_{ij}(u)\,Y_{i'j'}(u)\,\mathrm{d}\nu (u).
\end{equation}
With $H_n(\cdot,\cdot)$ symmetric in the arguments, the statistic that we use to estimate $\mathcal S( {\mathcal K}_M;\mu)$ is 
\begin{equation}
	Q_n = Q_n({\mathcal K}_M) =\frac{1}{n(n-1)} \sum_{1\leq i \neq i^\prime \leq n} H_n(\Xi_i , \Xi_{i'}).
\end{equation}
Unlike the estimation part, $Q_n$ involves no specific weighting scheme.

\begin{lemma}\label{lem:deg_Ustat2}
	If Assumptions~\ref{Ass_gen} and~\ref{ass:random_mi} hold, then $\mathbb E [Q_n] =  \mathbb E [\mathfrak m]^2 \mathcal S( {\mathcal K}_M;\mu)$. Moreover, $Q_n$ is a degenerate $U$-statistic of order 2 if $\mathcal S( {\mathcal K}_M;\mu)=0$.
\end{lemma}

The proof of Lemma~\ref{lem:deg_Ustat2} as well as all the other proofs for this section are in Appendix~\ref{app_sec:test}. The $U$-statistic $Q_n$ will define our test. Its asymptotic distribution under the null hypothesis~\eqref{eq:H0_full} can be derived by the Central Limit Theorem (CLT)  for degenerate $U$-statistics (see \cite[Th.~1]{H1984}, \cite[Th.~2.1]{dJ1987}): if $H_0$ in~\eqref{eq:H0_full} holds with $|\mathcal K_M|\to\infty$ and $|\mathcal K_M|/n\to 0$, then, under suitable conditions, 
\begin{equation}\label{eq:clt_classic}
\frac{n\,Q_n}{ \mathbb E[\mathfrak m]|\mathcal K_M|^{1/2}}   \;\xlongrightarrow{d}\; \mathcal N(0,V).
\end{equation}
The formal expression of $V$ is provided below in~\eqref{eq:def_V_test}. A significant refinement of the classical CLT~\eqref{eq:clt_classic} can be achieved using the sharp Berry-Esseen (BE) bound for the Gaussian approximation provided in \cite[Theorem~2.2]{LSS2025} and proved by Stein's method applied to degenerate $U$-statistics. See also~\cite[Theorem~1.4]{RR1997} and~\cite{Z2022} for related results. We follow this idea  aiming to derive the Gaussian approximation error bound for our test statistic under both the null and the alternative hypotheses.

Formally, if the density $g$ is known, our test is constructed with a standardized version of $Q_n$, that is
\begin{equation}\label{eq:def_t_testS} 
	T_n=\frac{n\,Q_n}{ \mathbb E[\mathfrak m]V^{1/2}|\mathcal K_M|^{1/2}} 
\end{equation}
where
\begin{equation}\label{eq:def_V_test}
	V =	2 \int_{\mathcal U\times \mathcal U}\!\int_{\mathcal T}
	\frac{\overline \rho^2(u,u';t)}{g^2(t)}\,dt\,d\nu(u)\,d\nu(u'),
\end{equation}
with
$$
\overline \rho(u,u';t)=\E\bigl[Y_{ij}(u)Y_{ij}(u')\mid T_{ij}=t\bigr]=\rho(u,u';t)+\mu(u,t)\mu(u',t),
$$
$\rho(u,u';t)=c_X(u,u';t,t)+\tau^2(t)\,\varrho_\eta(u,u')$, 
\begin{equation}\label{eq:spatial_covs}
	c_X(u,u';t,t')=\mathbb E[X(u,t)X(u',t')] \quad \text{ and } \quad \varrho_\eta(u,u')=\mathbb E[\eta_{ij}(u)\eta_{ij}(u')].
\end{equation}
Note that $c_X(u,u;t,t')=\gamma_X(u;t,t')$, with $\gamma_X$ defined in~\eqref{def:gammaX}. Thus, the variance $V$ depends  on the density $g$, on the noise conditional variance $\tau^2$, and on the covariance function of the random field $X$ and the  spatial covariance function of the noise. It also depends on the mean, unless $\mu\equiv 0$. In practice, the multiplier bootstrap below avoids the computation of $V$.

We provide a non-asymptotic Gaussian approximation result for $T_n$, for any $\mathcal S( {\mathcal K}_M;\mu)$. In particular, it holds regardless of whether the null hypothesis holds. For this we use the following additional mild assumption.

\begin{assumption}\label{ass:CLT_mild} 
	
\begin{enumerate} [	label=(\roman*),ref=\theassumption(\roman*)]

 \item\label{ass:CLT_mild_i}  It holds  $\mathbb E\|X\|_{L^2(\mathcal U\times\mathcal T)}^4+\mathbb E\|\eta_{ij}\|_{L^2(\mathcal U)}^4<\infty$, $ \|c_X\|_\infty <\infty$ and
\begin{equation}\label{eq:non_def_Y}
	\|\Gamma_X\|_{L^2(\mathcal T)}+\|\tau\|_{L^2(\mathcal T)}>0,
\end{equation}
with $c_X(\cdot,\cdot;\cdot,\cdot)$ and $\Gamma_X(\cdot)$ defined in \eqref{eq:spatial_covs} and \eqref{def:GammaX}, respectively. Moreover, the noise conditional standard deviation  $\tau(\cdot)$ and the map $t\mapsto c_X(u,u';t,t)$ are Hölder continuous with exponent $\varkappa\in(0,1]$ and Hölder constant $L_{\tau,c}>0$ not depending on $u,u'$.

\item\label{ass:profile_kurtosis} 
	There is a constant $\kappa_4 $ such that, for every $t\in\mathcal T$ and every
	$\mathfrak  y \in L^2(\mathcal U)$,
	\[
	\mathbb E\bigl[\langle \mathfrak  y ,\widetilde Y_{ij}\rangle_{L^2(\mathcal U)}^4 \,\big|\, T_{ij}=t\bigr]
	\;\le\;\kappa_4\,\langle \mathfrak  y ,\mathcal R_t\,\mathfrak  y \rangle_{L^2(\mathcal U)}^2
	=\kappa_4\bigl(\mathbb E[\langle \mathfrak  y,\widetilde Y_{ij}\rangle^2\mid T_{ij}=t]\bigr)^2 ,
	\]
	where $\widetilde Y_{ij}(u)=X_i(u,T_{ij})+\tau(T_{ij})\eta_{ij}(u)$ is the centered single-visit profile
	and $\mathcal R_t$ its conditional covariance operator on $L^2(\mathcal U)$ given $T_{ij}=t$. 
		
\end{enumerate}
\end{assumption}

Condition~\eqref{eq:non_def_Y} rules out the degenerate case of non-random profiles, so that the variance $V$ is positive and $T_n$ is well defined. Since $\|c_X\|_\infty<\infty$, the map $t\mapsto c_X(u,u';t,t)$ is uniformly H\"older continuous with exponent $\varkappa=b/2$ whenever there exist constants $c>0$, $b\in(0,2]$, such that $\mathbb E[|X(u,t)-X(u,t')|^2]\le c|t-t'|^b$. Such a condition is mild. For example, it does not exclude random fields with discontinuous paths. The fourth-order condition on the directional conditional moment is a mild restriction specific to the longitudinal functional data setup. If the profiles do not depend on $u$, Assumption~\ref{ass:profile_kurtosis} reduces to a uniform bound on the conditional kurtosis of the scalar $\widetilde Y_{ij}$ by a constant times its squared conditional variance, uniformly in $t$.

We also need additional notation.  Following the idea of the H\'ajek projection of $U$-statistics, let 
\begin{equation}
	\widetilde H_n(\Xi_i,\Xi_{i'}) : =  H_n(\Xi_i,\Xi_{i'}) - \mathbb E \left[H_n(\Xi_i , \Xi_{i'}) \mid \Xi_{i}\right] 
	- \mathbb E \left[H_n(\Xi_i , \Xi_{i'}) \mid \Xi_{i'}\right] + \mathbb E \left[H_n(\Xi_i , \Xi_{i'})\right] ,
\end{equation}
and 
$$
g_1(\Xi):=  \mathbb E \left[ H_n(\Xi_i , \Xi_{i'}) \mid \Xi_{i}\right] - \mathbb E \left[H_n(\Xi_i , \Xi_{i'})\right],
$$
so that the Hoeffding decomposition of $Q_n$ into a degenerate $U$-statistic and a linear term is
\begin{equation}\label{eq:dec_Hoeff_Qn}
Q_n - \E[Q_n] = Q_n - \E[\mathfrak m]^2\mathcal S(\KM;\mu) = \frac 2{n(n-1)}  \sum_{1\leq i < i'\leq n} 	\widetilde H_n(\Xi_i,\Xi_{i'}) + \frac 2n 
\sum_{1\leq i \leq n}  g_1(\Xi_i).	
\end{equation}
Note that $g_1\equiv0$ under $H_0$. Define 
$$
\vv_1=\E\big[g_1^2(\Xi_i)\big], \qquad \vv_2= \E\big[\widetilde H^2_n(\Xi_i,\Xi_{i'})\big],
$$ 
and consider the variance and the normalised variance of $Q_n$:
\[
\vv=\operatorname{Var}(Q_n)=\frac4n\,\vv_1+\frac{2\,\vv_2}{n(n-1)}
\quad\text{and}\quad
\varsigma_n^2(\mu)=\frac{n^2\vv}{\ee^2V\Kc}=\frac{4n\vv_1+2n\vv_2/(n-1)}{\ee^2V\Kc}.
\]
Finally, define the moment ratios, with the convention $0/0=0$:
\[
\Theta_1:=\frac{(\E g_1^4)^{1/2}}{\vv_1},\qquad
\Theta_{12}:=\Bigl(\frac{\E[g_1^2(\Xi_1)\Ht^2(\Xi_1,\Xi_2)]}{\vv_1\,\vv_2}\Bigr)^{1/2}.
\]
Under the null hypothesis,  $\Theta_1=\Theta_{12}=0$. Finally, define the vector of functions of $u$
\begin{equation}\label{eq:def_tildeL}
	\widetilde \Lambda_{i}  = \big(\widetilde \Lambda_{ik}\big)_{k\in\mathcal K_M}, \quad \text{ with } \quad \widetilde \Lambda_{ik}(u)= \Lambda_{ik}(u) -\E[\mathfrak m]\beta_k(u) \,\, \text{ and } \	
\Lambda_{ik}(u)=\!	\sum_{j=1}^{m_i} \frac{\phi_k(T_{ij})}{g(T_{ij})}\,Y_{ij}(u), 	
\end{equation}
so that $\mathbb E \big[\, \widetilde \Lambda_{i}(\cdot)\big]=0\in\mathcal H_M=\bigoplus_{k\in\mathcal K_M}L^2(\mathcal U)$. 
Here, $\mathcal H_M$ is a Hilbert space with inner product
$\langle a,b\rangle_{\mathcal H_M}=\sum_{k\in\mathcal K_M}\int_{\mathcal U}a_k b_k\,d\nu$.
Note that
\begin{equation}
	\widetilde H_n(\Xi_i,\Xi_{i'}) = \big\langle \,\widetilde \Lambda_{i}, \widetilde \Lambda_{i'} \big\rangle_{\mathcal H_M} = \sum_{k\in \mathcal K_M}\big\langle \,\widetilde \Lambda_{ik}, \widetilde \Lambda_{i'k} \big\rangle_{L^2(\mathcal U)}.
\end{equation}
Let $\mathcal C_M:\mathcal H_M\to\mathcal H_M$ be the covariance operator of
$\tilde\Lambda_i$, \emph{i.e.}, the operator with kernel
\begin{equation}
	\mathcal C_M[(k,u),(k',u')]=\mathbb E \big[\, \widetilde\Lambda_{ik}(u)\, \widetilde\Lambda_{ik'}(u')\big].
\end{equation}
Its eigenvalues are denoted $\lambda_{1}\ge\lambda_2\ge\cdots\ge 0$. The eigenvalue $\lambda_1$ will appear in Theorem~\ref{prop:dist_tstat} below. 

The following is the main result of this section.

\begin{theorem}\label{prop:dist_tstat}
	Suppose that Assumptions~\ref{Ass_gen}, \ref{ass:random_mi} and Condition~\eqref{eq:non_def_Y} hold, and that $\E\big[ H_n^4(\Xi_1,\Xi_2)\big]<\infty$. Then $V>0$, $\vv_2>0$, and, with $\delta_n = \E[\mathfrak m]V^{-1/2} n |\mathcal K_M|^{-1/2}$, for an absolute constant $C$, 
	\[
	\sup_{z\in\mathbb R}\Bigl|\mathbb P(T_n\le z)-\Phi\Bigl(\frac{z-\delta_n\mathcal S(\KM;\mu)}{\varsigma_n(\mu)}\Bigr)\Bigr|
	\le C\Bigl\{\frac{1+\Theta_1+\Theta_{12}}{\sqrt n}+\frac{\lambda_1}{\sqrt{\vv_2}}+\Bigl(\frac{\E\Ht^4}{n\,\vv_2^2}\Bigr)^{1/2}\Bigr\}.
	\]
If, in addition, Assumptions~\ref{ass:holder} and~\ref{ass:CLT_mild} hold and $\KM$ is a set of consecutive indices, then the bound becomes 
	\[
	C'\Bigl\{\Kc^{-1/2}+\Bigl(\frac{\Kc}{n}\Bigr)^{1/2}+\frac{1+\Theta_1}{\sqrt n}\Bigr\}.
	\]
\end{theorem}

\smallskip

To the best of our knowledge, Theorem~\ref{prop:dist_tstat} is a new result in the context of nonparametric testing, in particular for functional data. \cite[Section~3]{LSS2025} derived a related result for distance correlation-based statistics used to check independence between random vectors. However, their result only considers the null hypothesis. Including the alternatives in the statement requires a non-trivial extension of their Theorem 2.2 to the case where the effect of the linear term (defined by $g_1$) in the Hoeffding decomposition~\eqref{eq:dec_Hoeff_Qn} is not negligible compared with that of the degenerate $U$-statistic of order 2.  This extension, of independent interest, is stated and proved in  Appendix~\ref{app:sec_BE_new}. The condition that $\mathcal K_M$ is a set of consecutive indices in Theorem~\ref{prop:dist_tstat} is a mild simplifying assumption to derive a simplified BE bound. The constant $C'$ absorbs several other constants used to simplify the bound, and the way $C'$ depends on the characteristics of the data, including the mean through $\bar\mu$ and $\|L_\mu\|_{L^2(\nu)}$ when $\mu\not\equiv0$, is explicit and can be traced in the proofs in Appendix~\ref{app_sec:test}.

In our result, the Gaussian approximation is centered at $\delta_n\mathcal S( {\mathcal K}_M;\mu)$, which vanishes under $H_0$. When $H_0$ fails, the sequence $\delta_n$ calibrates how small a departure from $H_0$, measured by $\mathcal S( {\mathcal K}_M;\mu)$, the test can detect. Thus, the asymptotic level and the consistency of our test against local alternatives are direct consequences of Theorem~\ref{prop:dist_tstat}, as formalized in the following result.

\begin{corollary}\label{corr:test_level}
	Assume the conditions of Theorem~\ref{prop:dist_tstat} (second part) hold, and that
	$\Kc\to\infty$ and $\Kc/n\to0$. Then:
	\begin{enumerate}[label=(\roman*),leftmargin=*, ref=\thecorollary(\roman*)]
		\item\label{corr:test_level(i)} \emph{(Level)} Under $H_0$ in~\eqref{eq:H0_full}, with
		$\bar\varkappa=\min\{\varkappa,\alpha\}$, it holds
		$\varsigma_n(\mu)=1+O(\Kc^{-\bar\varkappa/(1+\bar\varkappa)}+n^{-1})$ and
		\[
		\sup_z|\mathbb P(T_n\le z)-\Phi(z)|\le C''\bigl\{\{\Kc/n\}^{1/2}+\Kc^{-\bar\varkappa/(1+\bar\varkappa)}\bigr\},
		\]
		for some constant $C''$, and the test defined by the rule
		$\mathbb I\{T_n>z_{1-\mathfrak a}\}$ has asymptotic level $\mathfrak a$.
		
		\item\label{corr:test_level(ii)} \emph{(Consistency)} Consider the local alternatives $H_{1,n}:\mu=\mu_n$, where
$\mu_n=\mu_0+r_n\Upsilon$, with fixed $\mu_0,\Upsilon\in L^2(\mathcal U\times\mathcal T)$ satisfying
		Assumption~\ref{ass:holder} with the same exponent $\alpha$,
		$\mathcal S(\mathcal K;\mu_0)=0$, and $\{r_n\}$ a bounded sequence, so that
		$\mathcal S(\KM;\mu_n)=r_n^2\,\mathcal S(\KM;\Upsilon)$. Then the rejection probability
		of the test $\mathbb I\{T_n>z_{1-\mathfrak a}\}$ tends to one provided
		$n\Kc^{-1/2}r_n^2\,\mathcal S(\KM;\Upsilon)\to\infty$.
	\end{enumerate}
\end{corollary}

\smallskip

The condition $n|\mathcal K_M|^{-1/2}r_n^2\,\mathcal S(\mathcal K_M;\Upsilon)\to\infty$ requires the alternative direction $\Upsilon$ to have a non-zero projection on $\operatorname{span}\{\phi_k:k\in\mathcal K_M\}$. As $|\mathcal K_M|$ grows this span increases, so the test becomes sensitive to a widening class of local and fixed alternatives.

\begin{remark}\label{rem:unif_meanrn}
	The quantities controlling the rejection probability in
	Corollary~\ref{corr:test_level(ii)}, namely $V$, $\lambda_1$ and the bounds on
	$\vv_1$ and $\vv_2$, depend on the mean only through $\bar\mu$. Consequently, the
	consistency in Corollary~\ref{corr:test_level}(ii) holds uniformly over any class
	of local alternatives $H_{1,n}$ for which
	$\sup_{t\in\mathcal T}\{\|\mu_0(\cdot,t)\|_{L^2(\mathcal U)}+\|\Upsilon(\cdot,t)\|_{L^2(\mathcal U)}\}$
	is bounded and $n|\KM|^{-1/2}r_n^2\inf\mathcal S(\KM;\Upsilon)\to\infty$, the infimum
	being taken over the class, with $\{r_n\}$ as in Corollary~\ref{corr:test_level(ii)}.
\end{remark}

\begin{remark}\label{rem:unif_K_M}
	The bounds in Theorem~\ref{prop:dist_tstat} are non-asymptotic with constants independent of the set $\mathcal K_M$, so they hold for every index set $\mathcal K_M$, with the same constant. Consequently, the Gaussian approximation, and with it the asymptotic level of Corollary~\ref{corr:test_level(i)}, holds \emph{uniformly} over any deterministic class of testing sets with $|\mathcal K_M|\to\infty$, and the consistency of the corollary holds uniformly over any such class with $n|\mathcal K_M|^{-1/2}r_n^2\to\infty$ (provided $\mathcal S(\mathcal K_M;\Upsilon)>0$ for $\mathcal K_M$ sufficiently large). This uniformity is a step towards a data-driven choice of $\mathcal K_M$, but does not by itself justify one: a selector computed from the same data is correlated with $T_n$, a dependence that must be accounted for, e.g. by sample splitting. In particular, a truncation chosen by cross-validation in the \emph{estimation} step targets prediction error and, under $H_0$, need not satisfy $|\mathcal K_M|\to\infty$, so it is not directly useful to the test. Finally, the detection boundary for local alternatives is $r_n^2\asymp|\mathcal K_M|^{1/2}/n$, so the minimal detectable deviation $r_n\asymp|\mathcal K_M|^{1/4}n^{-1/2}$ is increasing in $|\mathcal K_M|$. Thus, against local alternatives the most sensitive tests let $|\mathcal K_M|$ diverge slowly.  A data-driven $\mathcal K_M$ optimizing power against given classes of alternatives is a delicate
	problem that we leave for future work.
\end{remark}

In the more realistic case where $g$ is neither known nor estimable from external data, we propose to use the spacings idea to define a feasible test statistic. For this we redefine $Q_n$ using the order statistics $T_{(l)}$, $1 - \mathfrak h \leq l\leq M + \mathfrak h$ and the $\widehat\Phi_{kl}$'s introduced in Section~\ref{sec:spacings}. More precisely, with $l=\varrho(i,j)$, $l'=\varrho(i',j')$ the ranks of $T_{ij}$ and $T_{i'j'}$ among the order statistics of the visit times, let $\varrho_1^{-1}(l)$ be the subject index $i$ of the pair $(i,j)$ with rank $l$, and define
$$
\In  = \{(l,l'): 1\le l,l'\le M , \varrho_1^{-1}(l)\neq\varrho_1^{-1}(l')\},
$$ 
the set of pairs of pairs $(i,j), (i',j')$ with different subject indices $i,i'$. 
Then, we define
\begin{equation}\label{eq:def_Hn_sp} 
	Q^{\rm sp}_n= Q^{\rm sp}_n(\mathcal K_M)= \frac{(M+1)^2}{n(n-1)} \sum_{(l,l')\in\In}\sum_{k\in\mathcal K_M}\widehat\Phi_{kl}\widehat\Phi_{kl'}
	\int_{\mathcal U} Y_{l}(u)\,Y_{l'}(u)\,\mathrm{d}\nu (u),
\end{equation}
with $\widehat\Phi_{kl}$ defined in~\eqref{eq:sp_est}. There are $\sum_{i\neq i'}m_i m_{i'}$ pairs $(l,l')$ in $\In$. The factor $(M+1)^2/[n(n-1)]$ brings $Q^{\rm sp}_n$ to
the same scale as $Q_n$. Then, considering  $Q^{\rm sp}_n$ a feasible proxy of $Q_n$, we define
\begin{equation}\label{eq:def_t_test_Sp}
	T^{\rm sp}_n=\frac{n\,Q^{\rm sp}_n}{ \mathbb E[\mathfrak m]V^{1/2}|\mathcal K_M|^{1/2}},
\end{equation}
with $V$ as in~\eqref{eq:def_V_test}. Note that $T^{\rm sp}_n$ still involves the unknown
$V$ and is therefore not directly computable, it serves only to formulate the next result. 
A feasible test is obtained from $Q^{\rm sp}_n$
through the multiplier bootstrap described below, which removes the need for $V$.

We will prove that the test based on the rule $\mathbb I\{T^{\rm sp}_n>z_{1-\mathfrak a}\}$ has asymptotic level $\mathfrak a\in(0,1)$, by showing that  $T^{\rm sp}_n$ inherits
the Gaussian approximation of $T_n$ of Theorem~\ref{prop:dist_tstat}. Moreover, we show that under the local alternatives $H_{1,n}$ considered in Corollary~\ref{corr:test_level}, the probability of the event $\{T^{\rm sp}_n>z_{1-\mathfrak a}\}$ tends to one if that of the event $\{T_n>z_{1-\mathfrak a}\}$ does. For simplicity, for the level we only study the case of a simple hypothesis $\mu\equiv 0$. For the consistency part we will use the following technical condition which, in particular, implies that, for any $u\in\mathcal U$, the map $t\mapsto\Upsilon(u,t)$ is continuously differentiable with derivative $\partial_t\Upsilon(u,t)$ and it holds $\partial_t\Upsilon(u,t)=\sum_{k\ge1}\beta_{\Upsilon, k}(u)\phi_k'(t)$.

\begin{assumption}\label{ass:coeff-decay_H1n} 
	Let $\{\phi_k\}$ be the cosine or the augmented cosine basis with $t,t^2$. Consider the local alternative hypotheses  $H_{1,n}:\mu_n=r_n\Upsilon$, with $\mathcal S(\mathcal K_M;\Upsilon)>0$ for $\mathcal K_M$ sufficiently large. There exists a constant $C_\Upsilon<\infty$ such that
	\[
	0< \sum_{k\ge 1} k^{4}\|\beta_{\Upsilon,k}\|_{L^2(\mathcal U)}^2 \le C_\Upsilon ,
	\qquad \text{ where } \quad \beta_{\Upsilon,k}(u)=\int_{\mathcal T}\Upsilon(u,t)\phi_k(t)\,{\rm d}t .
	\]
\end{assumption}

\begin{theorem}[Level and consistency of the spacings test]\label{thm:spacings_H0H1}  
	Suppose the conditions of Corollary~\ref{corr:test_level} hold. Assume that $\hh\to\infty$ and $\hh^2|\KM| n^{-1} \to 0$. Then the spacings test based on $\mathbb I \{T^{\rm sp}_n>z_{1-\mathfrak a}\}$, $\mathfrak a \in (0,1)$, satisfies the following.
	
	\smallskip
	\emph{(I) (Level)} Under $H_0:\mu \equiv  0$, if  $\hh|\KM|\, n^{-4/5}\to0$ it holds 
	\[
	\sup_{z\in\mathbb R}\bigl|\PP(T^{\rm sp}_n\le z)-\PP(T_n\le z)\bigr|\le C_0 \big\{\hh^{-1/4}+|\KM|^{-1/2}+\{|\KM|/n\}^{1/2}\big\},
	\]
	for a constant $C_0$ independent of $\mathcal K_M$ and $\hh$. Hence, the test has  asymptotic level $\mathfrak a$.
	
	\smallskip
	\emph{(II) (Consistency)} In addition, suppose Assumption~\ref{ass:coeff-decay_H1n} holds, $\hh\,|\KM|^{-1}\to0$ and $ \hh\,|\KM|^{3/2} n^{-1}\rightarrow 0$. Then, $\PP\big(T^{\rm sp}_n>z_{1-\mathfrak a}\big)\rightarrow 1$ under the local alternatives $H_{1,n}$ of Corollary \ref{corr:test_level(ii)} with $\mu_0\equiv0$ (that is, $\mu_n=r_n\Upsilon$), if $r_n^2\, n\,|\KM|^{-1/2}\mathcal S(\mathcal K_M;\Upsilon) \rightarrow \infty$.
\end{theorem} 
  
\medskip

The validity of the feasible test is proved by controlling the difference $Q^{\rm sp}_n-Q_n$. Under the null hypothesis, therefore, the rate of the Berry-Esseen bound that we obtain for $T^{\rm sp}_n$ is much slower than the one derived for $T_n$. Furthermore, we have to constrain the rate of $|\mathcal K_M|$, although this has no impact when $|\mathcal K_M|$ diverges slowly. This is expected to correspond to the most sensitive choice of $\mathcal K_M$ against local alternatives. See also Remark~\ref{rem:unif_K_M}. The statement under the local alternatives requires  additional technical restrictions. Moreover, the proof of part (II) is detailed for the uniform design, which simplifies the spacings calculus. 
We conjecture that Theorem~\ref{thm:spacings_H0H1} holds under the conditions of Theorem~\ref{prop:dist_tstat}, but the proof would require extending the  Berry-Esseen bound to $U$-statistics based on spacings, which are no longer independent. This will be investigated in future work. 

\begin{remark}\label{rem:sp_composite}
	For composite null hypotheses $\mathcal S(\mathcal K;\mu)=0$ with $\mu\not\equiv0$, the
	statistic $Q^{\rm sp}_n$ is biased. Since the spacings weights of different subjects are
	dependent, removing the within-subject pairs in~\eqref{eq:def_Hn_sp} creates a negative
	bias of order $(\E[\mathfrak m]/n)\,|\KM|\,\|\mu\|^2_{L^2(\mathcal U\times\mathcal T)}$. The resulting test (including its bootstrap version) is then conservative and loses power. In practice, we therefore recommend centering the profiles by an estimate of $\mu$  under $H_0$  (using $\{\phi_k:k\notin\mathcal K\}$), before computing $Q^{\rm sp}_n$. For example, center the profiles by their empirical mean  when the null hypothesis is that $\mu$ does not depend on $t$ (in which case 	$1\notin\mathcal K$). A formal analysis of the centered statistic is left
	for future work.
\end{remark}

Finally, to avoid using estimates of $V$, we can use the multiplier bootstrap. Let $\zeta_1,\ldots,\zeta_n$ be i.i.d. random variables with mean zero and unit variance and finite fourth-order moments. The bootstrap version of  $Q_n$ is defined as
\begin{equation}\label{eq:def_Hn_Qn_boot}
	Q^*_n = Q^*_n({\mathcal K}_M) =\frac{1}{n(n-1)} \sum_{1\leq i \neq i^\prime \leq n} \zeta_i \zeta_{i'} H_n(\Xi_i , \Xi_{i'}).
\end{equation}
The conditional $(1-\mathfrak a)$-th quantile $q^*_{1-\mathfrak a}$ of $Q^*_n$ given the data and the choice of $\mathcal K_M$ can be approximated by Monte Carlo using $B$ independent samples of $\zeta_1,\ldots,\zeta_n$. Since, given the data, $Q_n^*$ mimics the null distribution of $Q_n$, including its possibly $\mu$-dependent scale, the standardization by $V$ is no longer needed, and the test can be defined by the rule $\mathbb I\{Q_n>q^*_{1-\mathfrak a}\}$. Common choices for drawing the multipliers $\zeta$, independent of the data, are the standard normal distribution and the two-point distribution with support on $\{(1\mp \sqrt{5})/2\}$ and probabilities $\{(\sqrt{5}\pm 1)/(2\sqrt{5})\}$.

The validity of the bootstrap critical values given by the quantiles of $Q^*_n$ is proved in the next result. Their validity when the spacings are used instead is investigated by simulations in Section~\ref{sec:numerics}.

\begin{theorem}\label{th:bootstrap}
Assume the conditions of Corollary~\ref{corr:test_level} hold. Let $\zeta_1,\dots,\zeta_n$ be i.i.d., independent of the data, with $\mathbb E\zeta_1=0$, $\mathbb E\zeta_1^2=1$ and $\mathbb E\zeta_1^4<\infty$. Let  $q^*_{1-\mathfrak a}$ denote the conditional $(1-\mathfrak a)$-quantile of $Q_n^*$ given the data. Then, the test $\mathbb I\{Q_n>q^*_{1-\mathfrak a}\}$ has asymptotic level $\mathfrak a$ under $H_0$ in~\eqref{eq:H0_full},  and is consistent against the local alternatives $H_{1,n}$ defined in Corollary~\ref{corr:test_level} provided $n|\mathcal K_M|^{-1/2}  r_n^2\mathcal S(\mathcal K_M;\Upsilon) \rightarrow \infty$.
\end{theorem}



\section{Implementation aspects and numerical experiments}\label{sec:numerics}

In this section we investigate the finite sample performance of our estimation and testing approaches. First we present a simulation study, and next an application to real data. All the code used in our study, which reproduces the results presented in what follows, is publicly available on  \href{https://github.com/mahdy-saidi/Optimal-estimation-and-goodness-of-fit-testing-of-the-mean-for-sparse-longitudinal-functional-data}{GitHub}.

For studying the finite sample performance of our test, we will mainly use the version based on spacings, with the half-cosine basis, denoted by $\mathcal B_0$.  For this we use $Q^{\rm sp}_n$ as defined in~\eqref{eq:def_Hn_sp}. As before, let $\zeta_1,\ldots,\zeta_n$ be i.i.d. random variables with mean zero and unit variance and finite fourth order moments. The bootstrap version of $Q^{\rm sp}_n$ is 
\begin{equation}\label{eq:def_Hn_sp_boot} 
	Q^{\rm sp,*}_n= Q^{\rm sp,*}_n({\mathcal K}_M) = \frac{(M+1)^2}{n(n-1)} \sum_{(l,l') \in\In }\sum_{k\in\mathcal K_M} \zeta_{\varrho_1^{-1}(l)} \zeta_{\varrho_1^{-1}(l')}  \widehat\Phi_{kl}\widehat\Phi_{kl'}
	\int_{\mathcal U} Y_{l}(u)\,Y_{l'}(u)\,\mathrm{d}\nu (u).
\end{equation}
In the construction of $Q^{\rm sp,*}_n$ all observations of the same subject share the multiplier $\zeta_i$.  The test is given by the rule $\mathbb I\{Q^{\rm sp}_n>q^{\rm sp,*}_{1-\mathfrak a}\}$ with $q^{\rm sp,*}_{1-\mathfrak a}$ the conditional $(1-\mathfrak a)$-th quantile of $Q^{\rm sp,*}_n$ given the data and the choice of $\mathcal K_M$, which can be approximated by Monte Carlo.

\subsection{Simulation study: estimation}\label{sec:simulation}

Here, we investigate what the new spacings estimator of Section~\ref{sec:spacings} achieves once it
has to choose, on the sample at hand, both the basis of
Section~\ref{sec:basis_choice} and the truncation $K$. 
With $\mathcal U=\mathcal T=[0,1]$ and $\nu$ the Lebesgue measure, we generate data from model~\eqref{eq:model0} with each of the three mean functions
\begin{gather*}
	\mu_1(u,t)=(a\,t+b\,t^{2})\sum_{p=1}^{P}\frac{(-1)^{p}}{p^{2}}
	\cos(\pi p\,t+c\pi u),
	\quad
	\mu_2(u,t)=(a\,t+b\,t^{2})\sum_{p=1}^{P}\frac{(-1)^{p}}{p^{2}}
	\cos(\pi p\,u+c\pi t),
	\\
	\mu_3(u,t)=a\,t+b\,t^{2}+\cos(c\pi t)\sum_{p=1}^{P}
	\frac{(-1)^{p}}{p^{2}}\cos(\pi p\,u),\qquad a,b,c\in\mathbb R,
\end{gather*}
displayed in Figure~\ref{fig:means}. With these choices, $\mu_1$ oscillates in $t$ at increasing frequencies, $\mu_2$ in $u$ with a single frequency in $t$, and $\mu_3$ is a sum of separable terms. The endpoint derivative values of these functions with respect to $t$ are different for the choices of the parameters $a,b,c$ we will consider, so that they require the cosine basis augmentations by $t,t^2$ (Table~\ref{table:clear}).
The visit counts are drawn from the Poisson distribution with mean $\mathfrak M$, conditioned on $2\leq m_i \leq m_{\max}$. The visit times are drawn from the non-uniform density $g(t)=1+3\cos(2\pi t)/5$, so that $g_{\min}=2/5$. The subject field and the error field used with the three means are
\[
X_i(u,t)=\sum_{l,l'=1}^{100}\frac{\gamma^{(i)}_{l,l'}}{l^{2}l'^{2}}
\cos(l\pi u)\cos(l'\pi t),
\qquad
\eta_{ij}(u)=v(u)^{-1/2}\sum_{\ell=1}^{50}\frac{\xi^{(ij)}_{\ell}}{\ell^{4}}
\cos(\ell\pi u), 
\]
with all $\gamma^{(i)}_{l,l'}$ and $\xi^{(ij)}_{\ell}$ i.i.d. $\mathcal N (0,1)$ and $v(u)=\sum_{1\le \ell\le50}\ell^{-8}\cos^{2}(\ell\pi u)$, so that
$\mathrm{Var}(\eta_{ij}(u))=1$ for every $u$. The noise scale is constant, $\tau\equiv 1/2$ (homoscedastic case). Figure~\ref{fig:data} shows one subject field and three of the profiles it generates, with the mean $\mu_1$. We take $n\in\{100,200, 500\}$ and $R=2000$ replications at each $n$. Within a replication, one draw of the design, of the $X_i$ and of the errors serves the three means, so that the comparisons are paired. The values of the parameters of this data-generating setup are: $a=1$, $b=-7$, $c=4$, $P=8$, $\mathfrak M = 7$, $m_{\max}=15$.

The estimator under study is $\widehat\mu^{\rm sp}$ of~\eqref{eq:sp_est},
with window half-width $\mathfrak h=2$ at every sample size considered,
expanded in one of the three candidate systems
$\mathcal B_r=\mathcal B_{r,K}$, $r\in\{0,1,2\}$, of
Section~\ref{sec:basis_choice}. In all three the window integrals
$\widehat\Phi_{kl}$ have closed-form expressions, and none of them uses $g$.

\begin{figure}[ht!]
	\centering
	\includegraphics[width=0.3\linewidth] {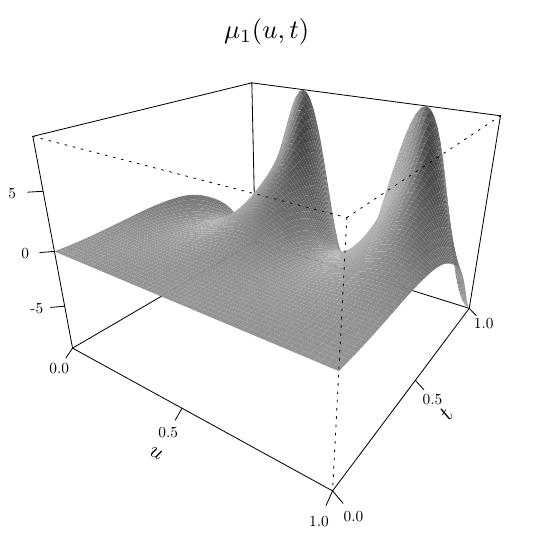}\hspace{-0.25cm}
	\includegraphics[width=0.3\linewidth] {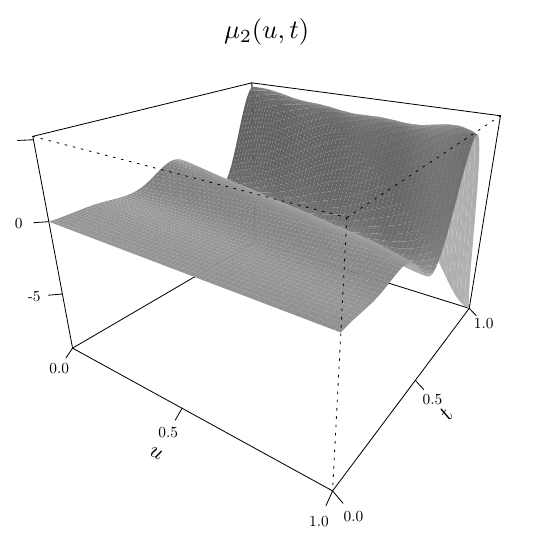}\hspace{-0.25cm}
	\includegraphics[width=0.3\linewidth] {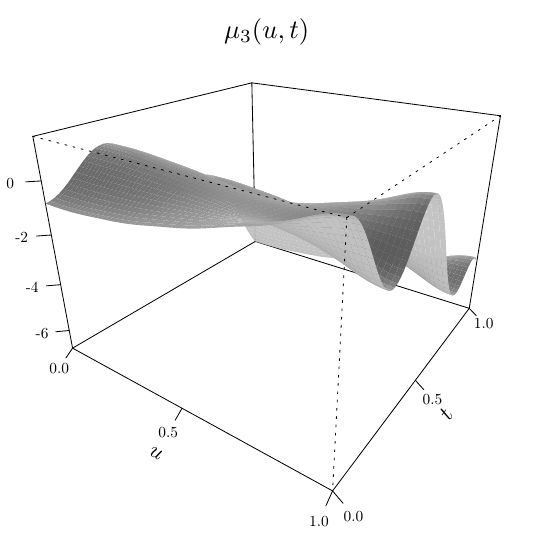}\vspace{-.2cm}
	\caption{\small The mean surfaces $\mu_1$, $\mu_2$ and $\mu_3$.}\vspace{-.2cm}
	\label{fig:means}
\end{figure}

\begin{figure}[ht!]
	\centering
	\includegraphics[width=0.3\linewidth]{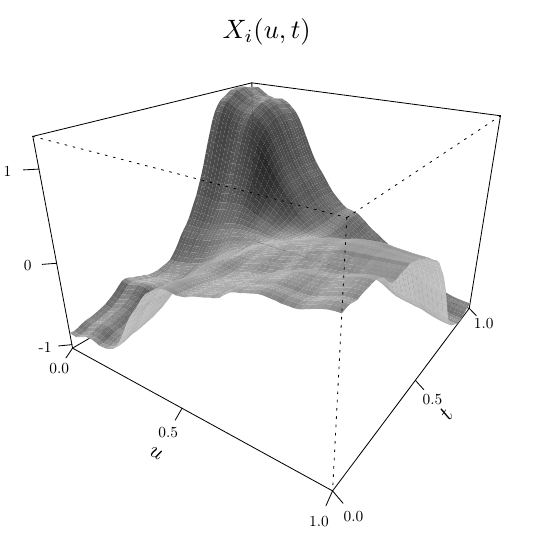}
		\includegraphics[width=0.3\linewidth]{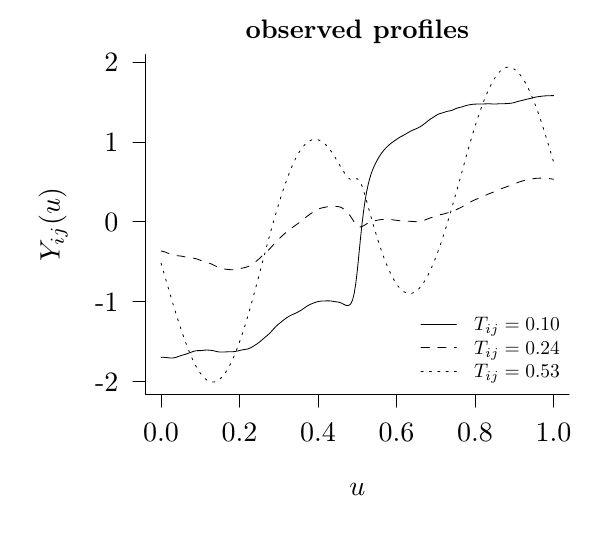}\vspace{-.2cm}
	\caption{\small One realization of the subject field $X_i$ (left) and three
		profiles $Y_{ij}(\cdot)$ of that subject (with the mean $\mu_1$) (right), each recorded on a 		regular grid of $1001$ points of $\mathcal U$ and labeled by its
		visit time $T_{ij}$.}
	\label{fig:data}
\end{figure}

Several benchmark competitors are considered. The first
competitor is the pooled local-linear smoother of
\cite[eq.~(8)]{CM2012}, applied in the visit-time direction only: at each $(u,t)$ it solves
\[
\min_{a_0,a_1}\sum_{i=1}^{n}\sum_{j=1}^{m_i}
\mathfrak K\!\left(\frac{T_{ij}-t}{h_t}\right)
\bigl\{Y_{ij}(u)
-a_0-a_1(T_{ij}-t)\bigr\}^{2},
\]
with $\mathfrak K$ the Epanechnikov kernel, and sets
$\widehat\mu^{\rm LL}(u,t)=\widehat a_0$. The second is a penalized
regression spline in $t$ applied for each $u$: with $\mathbf B$ the matrix of a cubic B-spline
basis with $20$ interior knots evaluated at the pooled visit times, 
$\mathbf P$ the second-order difference penalty, and $\mathbf Y(u)$ the vector of corresponding pooled $Y_{ij}(u)$, we define 
$$
\widehat\mu^{\rm spl}(u,\cdot)=\mathbf B(\cdot)^{\!\top}
(\mathbf B^{\!\top}\mathbf B+\lambda M\mathbf P)^{-1}\mathbf B^{\!\top}
\mathbf Y(u).
$$ 
The third competitor is the spacings estimator itself, with the augmented basis $\mathcal B_{2,K}$ and a given choice of $K$ suitable for twice differentiable functions in $t$, that is 
$K+2=\bigl\lceil C\,(n\mathfrak M)^{1/5}\bigr\rceil$,  $C=2$, where $n\mathfrak M$ is an approximation of the expected number of pooled observations, when 
$2\le m_i\le m_{\max}$ and the $m_i$ are drawn from a Poisson distribution with mean $\mathfrak M$. 
The constant $C$ is fixed and is not adjusted across means or sample sizes. We write $\widehat \mu ^{\rm rate} = \widehat \mu_{2,K} ^{\rm sp}$.

Accuracy is measured on a regular grid of $N_u\times N_t$ points $(u_l,t_{l'})\in \mathcal U \times \mathcal T$, with $N_u= N_t=50$, by the integrated squared error $\mathrm{ISE}=(2500)^{-1}\sum_{l,l'}\{\widehat\mu(u_l,t_{l'})- \mu(u_l,t_{l'})\}^{2}$, and the four estimators, being computed on the same data, are compared replication by replication through the paired log-ratio
\begin{equation}\label{eq:delta_e}
	\Delta_e:=\log\frac{\mathrm{ISE}(\widehat\mu^{e})}
	{\mathrm{ISE}(\widehat\mu^{\rm sp})},
	\qquad e\in\{\mathrm{LL},\mathrm{spl},\mathrm{rate}\},
\end{equation}
which is positive when the fully data-driven spacings estimator is more accurate in $\mathrm{ISE}$.

To select $(r,K)$ for our spacings estimator, the subjects are split at random into $\mathcal V=10$ folds,
$\widehat\mu^{\rm sp}_{r,K}{}^{(-v)}$ is fitted on the subjects outside fold
$\mathcal F_v$, and
\begin{equation}\label{eq:cv}
	\mathrm{CV}(r,K)=\frac{1}{N_uM}\sum_{v=1}^{\mathcal V}\sum_{i\in\mathcal F_v}
	\sum_{j=1}^{m_i}\sum_{l=1}^{N_u}
	\bigl\{Y_{ij}(u_l)-\widehat\mu^{\rm sp}_{r,K}{}^{(-v)}(u_l,T_{ij})\bigr\}^{2}
\end{equation}
is minimized over $(r,K)\in\{0,1,2\}\times\{1,\ldots,25\}$. The spacings estimator is refitted on the whole sample at the
selected pair $(r_{\mathrm{CV}},K_{\mathrm{CV}})$. The tuning parameters of the two competitors $\mathrm{LL}$ and $\mathrm{spl}$ are selected by the same type of $10$-fold cross-validation: the bandwidth $h_t$ of $\widehat \mu^{\rm LL}$ over 20 values equispaced in the log-scale between 0.05 and 0.6, and the penalty $\lambda$ of $\widehat \mu^{\rm spl}$ over 20 values equispaced in the log-scale between $10^{-5}$ and 10, the number of interior knots being kept at 20.

\begin{figure}[ht!]
	\centering
	\includegraphics[width=0.48\linewidth, height=4.5cm]{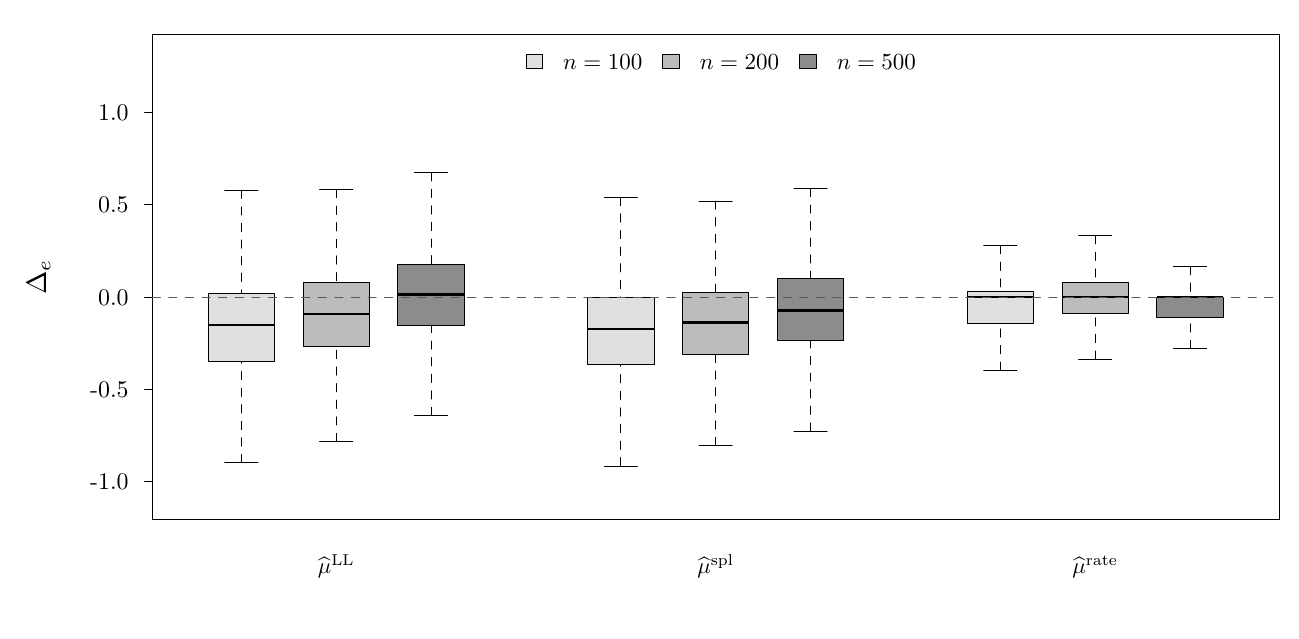}
	\includegraphics[width=0.48\linewidth, height=4.5cm]{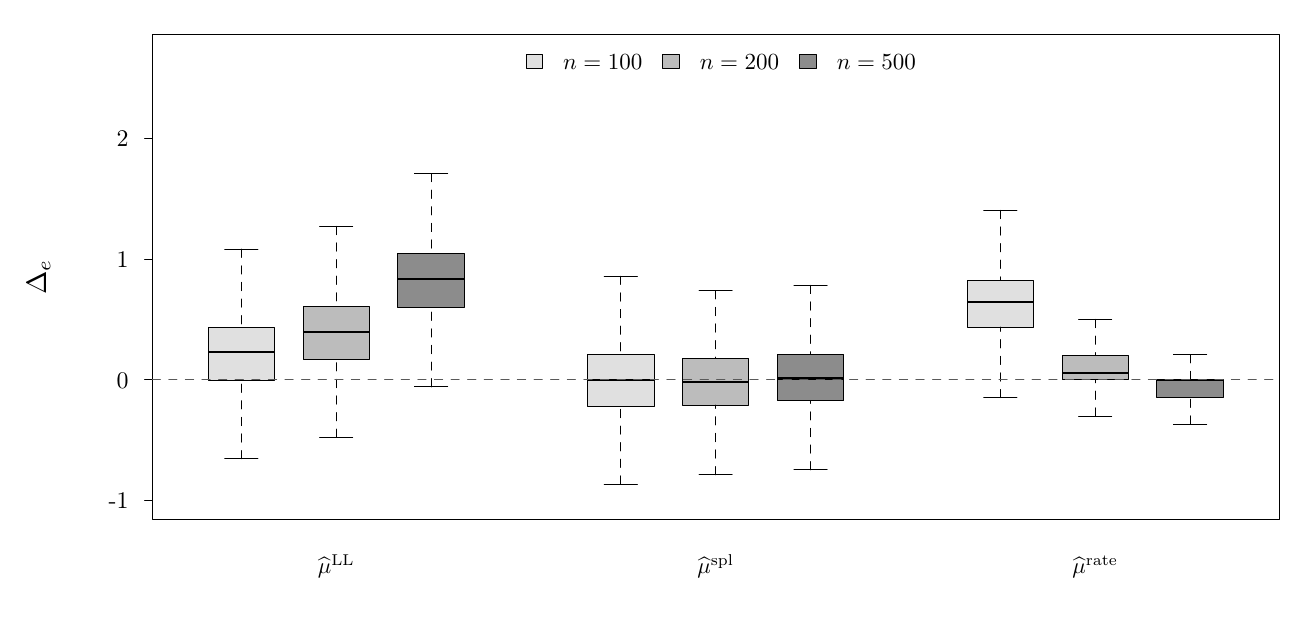}\\\vspace{-.3cm}
	\includegraphics[width=0.48\linewidth, height=4.5cm]{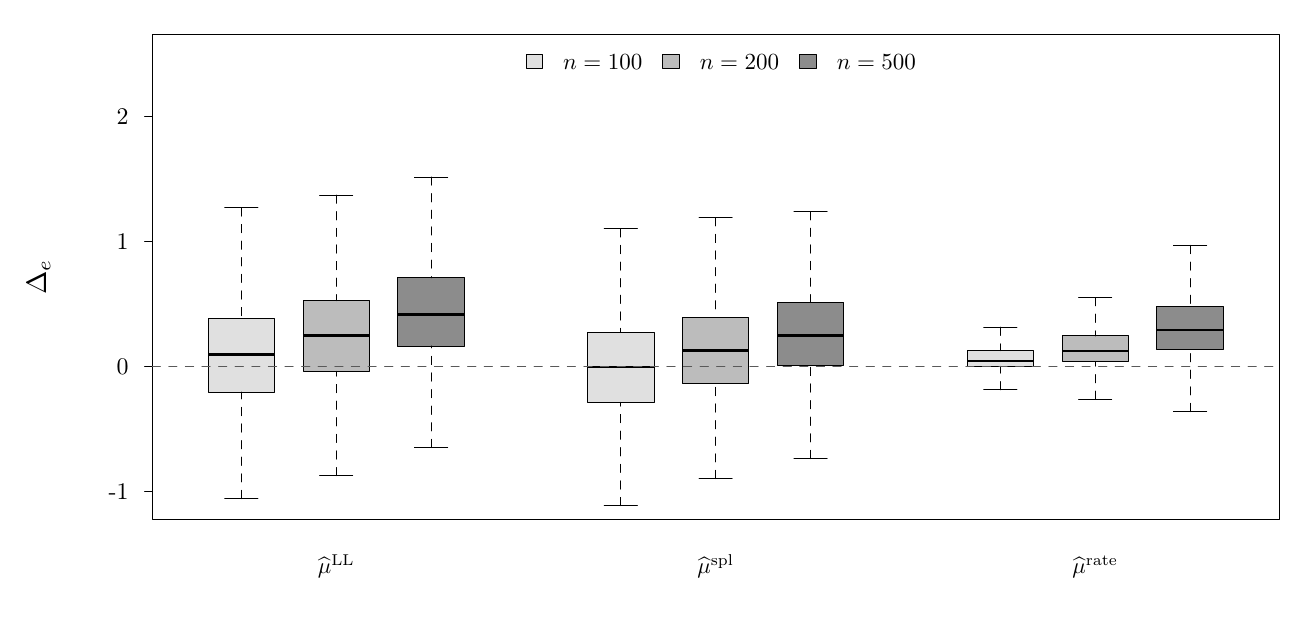}
\vspace{-.5cm}	\caption{\small Distribution of $\Delta_e$ of~\eqref{eq:delta_e} for the data-driven spacings estimator against the competitors, with mean $\mu_1$ (upper row left), $\mu_2$ (upper row right) and $\mu_3$ (lower row). Whiskers extend to $1.5$ times the interquartile range, outliers omitted. Positive values favor the spacings estimator. }
	\label{fig:cvbox_mu1}
\end{figure}

Figure~\ref{fig:cvbox_mu1} reports the distribution of $\Delta_{\rm LL}$, $\Delta_{\rm spl}$ and $\Delta_{\rm rate}$, one plot per mean, one group of boxes per competitor and one box per sample size. The CV systematically improves the performance of the spacings estimator over a $K$ equal to a constant $C$ times the optimal rate. We use an \emph{ad-hoc} fixed choice of the constant $C$ such that $K$ is comparable to the median values of $K_{\rm CV}$ obtained in the experiments (not reported, but available in the output of the codes on GitHub). In general, the $\widehat\mu^{\rm sp}$ estimator outperforms $\widehat\mu^{\rm LL}$ and has comparable performance to $\widehat\mu^{\rm spl}$ across the three mean setups. Details on the choice of the parameters by cross-validation are provided in Table~\ref{tab:cvbox} in Appendix~\ref{append_simus_est}.

\subsection{Simulation study: testing the nullity of the mean}\label{sec:test_simu}

We study the performance of the bootstrap tests based on the rules $\mathbb I\{Q_n>q^{*}_{1-\mathfrak a}\}$ and $\mathbb I\{Q^{\rm sp}_n>q^{\rm sp,*}_{1-\mathfrak a}\}$ as defined in Section~\ref{sec:test} and in equations~\eqref{eq:def_Hn_sp}-\eqref{eq:def_Hn_sp_boot} at the beginning of Section~\ref{sec:numerics}. The data-generating process  is the one of Section~\ref{sec:simulation} for the visit counts $m_i$, the density $g$, the fields $X_i$, $\eta_{ij}$, and the noise scale $\tau$. Under the null the profiles $Y_{ij}$ are generated with $\mu=0$. We consider two  local alternatives $H_1: \mu = \theta \mu_1$ and $H_1: \mu = \theta \mu_3$, with $ \mu_1$ as in Section~\ref{sec:simulation} and $\mu_3$ taken with $a=b=0$. With the amplitude $\theta=0$ we recover the null hypothesis, while values $\theta>0$ scale the deviation from the null hypothesis, and we take them on the equidistant  grid $\{0.025,0.05,\ldots, 0.15\}$. The direction $\mu_3$ reduces to $\mu_3(u,t)=S(u)\cos(c\pi t)=S(u)\phi_{c+1}(t)/\sqrt2$ with $S(u)=\sum_{p\le P}(-1)^p p^{-2}\cos(\pi p u)$. Thus, $\mu_3$ is orthogonal to all elements of the basis $\mathcal B_0$ (half-cosine, no augmentation) but $\phi_{c+1}$.

The basis elements for the test have indices in  $\mathcal K_M = \{1,2,\ldots,\bar K\}$, so that $|\mathcal K_M|=\bar K$. In the case of local alternatives in the direction $\mu_3$, by construction of the alternatives, taking $\bar K\leq c$ misses the signal entirely, $\bar K=c+1$ retains all of it, and $\bar K=25$ adds basis elements that do not contribute to power, yet the test is still expected to reject the null hypothesis. Each configuration is
replicated $R=2000$ times with $N_{\rm b}=999$ bootstrap draws, using the two-point distribution, and the nominal level of the test is $\mathfrak a=0.05$.

Table~\ref{tab:test_level} reports the empirical level, that is, the rejection frequency at $\theta=0$, for different values of $\bar K$. It can be seen that the level is accurate, irrespective of the value of  $\bar K$ we considered, with a slight over-rejection at $n=200$. Meanwhile,  Figure~\ref{fig:test_power_mu1} presents the power curves for different sample sizes and $\bar K = 25$, under the local alternatives in the direction of $\mu_1$. The alternatives are detected even for small amplitudes and moderate sample sizes. Figure~\ref{fig:test_power_mu3} presents the power curves for different values of $\bar K$. As expected, the test has trivial power if the projection of the true mean function onto the span of the basis elements used for testing is null. The power is high as soon as the cosine function used in the definition of $\mu_3$ enters the set of basis functions used for the test. Remarkably, the power remains high even when $\bar K$ is much larger, which in this experiment is useless for detecting the deviation from the null.

\begin{table}[ht!]
\footnotesize	
\centering
	\begin{tabular}{lccc}
		\toprule
		\multirow{2}{*}{$\bar K$} & \multicolumn{3}{c}{$n$} \\
		\cmidrule(lr){2-4}
		& $100$ & $200$ & $500$ \\
		\midrule
	 	$4$ & $0.046\,(0.005)$ & $0.059\,(0.005)$ & $0.053\,(0.005)$ \\
		$5$ & $0.047\,(0.005)$ & $0.060\,(0.005)$ & $0.051\,(0.005)$ \\
		$25$ & $0.049\,(0.005)$ & $0.059\,(0.005)$ & $0.044\,(0.005)$ \\
		\bottomrule
	\end{tabular}
	\caption{\small Empirical level of the test ($R=2000$ replications, Monte Carlo s.e. in parentheses) at nominal level  $\mathfrak a=0.05$, for  the spacings-based statistic ($N_{\rm b}=999$ bootstrap replications).}
	\label{tab:test_level}
\end{table}

\begin{figure}[ht!]
	\centering
	\includegraphics[width=0.3\linewidth]{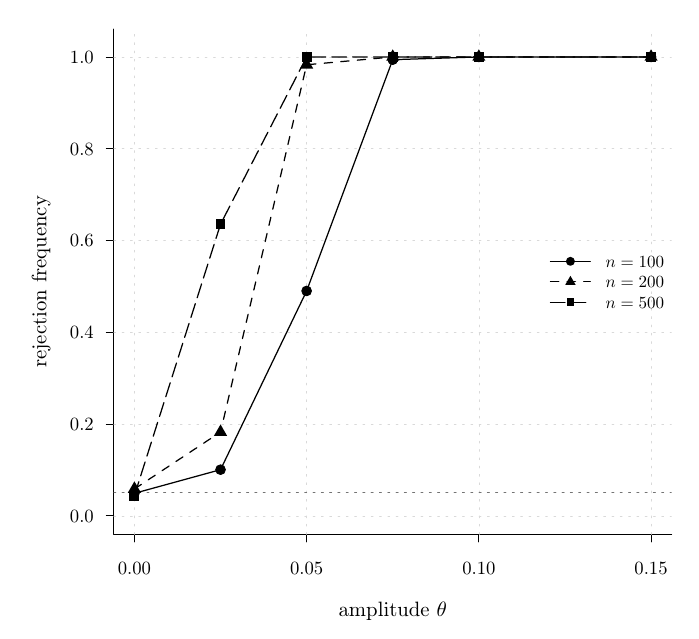}\vspace{-.2cm}
	\caption{\small Rejection frequency against the amplitude $\theta$, direction
		$\mu_1$, $\bar K=25$, one curve per number of subjects $n$. The dotted
		horizontal line is the nominal level $\mathfrak a = 0.05$.}
	\label{fig:test_power_mu1}
\end{figure}

\begin{figure}[ht!]
	\centering
\includegraphics[width=0.3\linewidth]{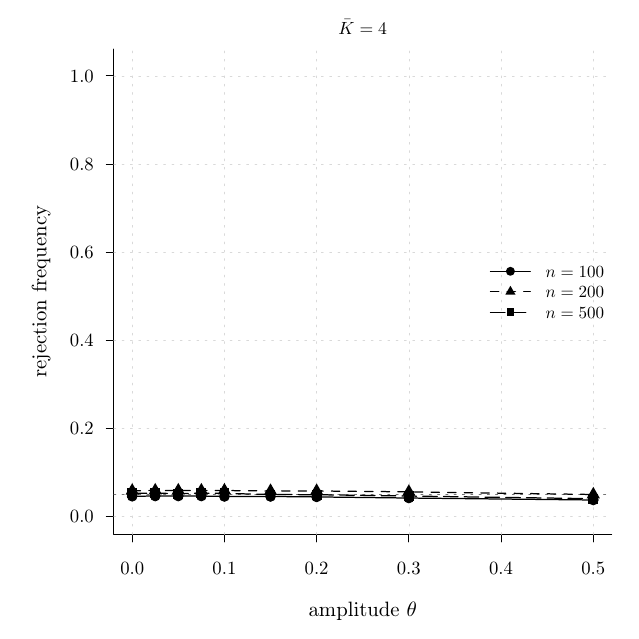}\hspace{-.3cm}
	\includegraphics[width=0.3\linewidth]{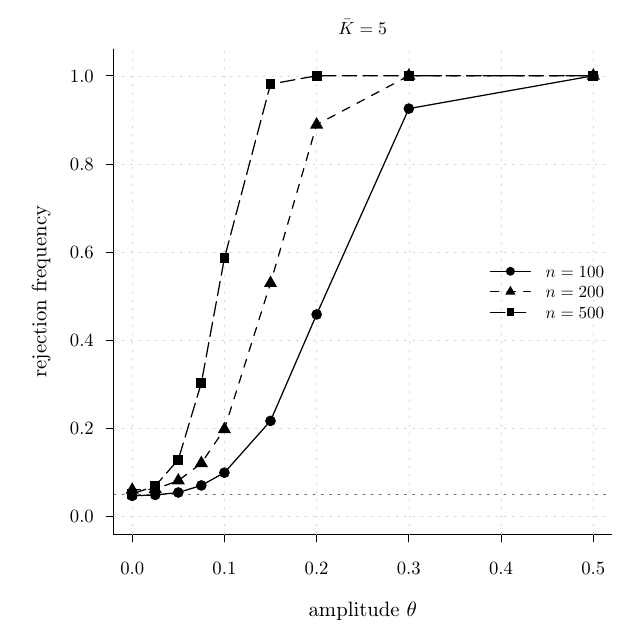}\hspace{-.3cm}
	\includegraphics[width=0.3\linewidth]{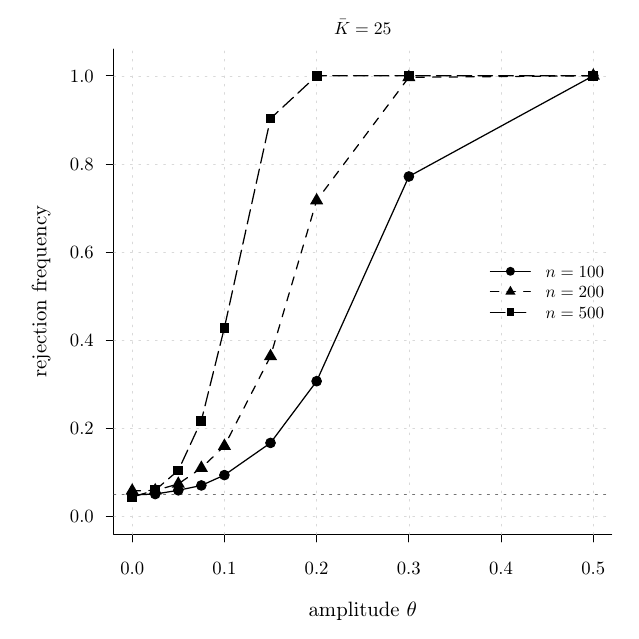}
	\caption{\vspace{-.3cm}\small As Figure~\ref{fig:test_power_mu1}, for the direction $\mu_3$
		with $a=b=0$, at the three truncations $\bar K=c$, $c+1$ and $25$.}
	\label{fig:test_power_mu3}
\end{figure}

\medskip

Finally, we compare the performance of the bootstrap  tests based on the rules $\mathbb I\{Q_n>q^{*}_{1-\mathfrak a}\}$ and $\mathbb I\{Q^{\rm sp}_n>q^{\rm sp,*}_{1-\mathfrak a}\}$, with the true design density plugged in for the test using $Q_n$. Figure~\ref{fig:test_cmp} puts the two power curves side by side, one panel per number of subjects, along the direction $\mu_1$ at the fixed $\bar K = 25$. We notice a better performance of the bootstrap test based on spacings.

\begin{figure}[ht!]
	\centering
	\includegraphics[width=0.3\linewidth]{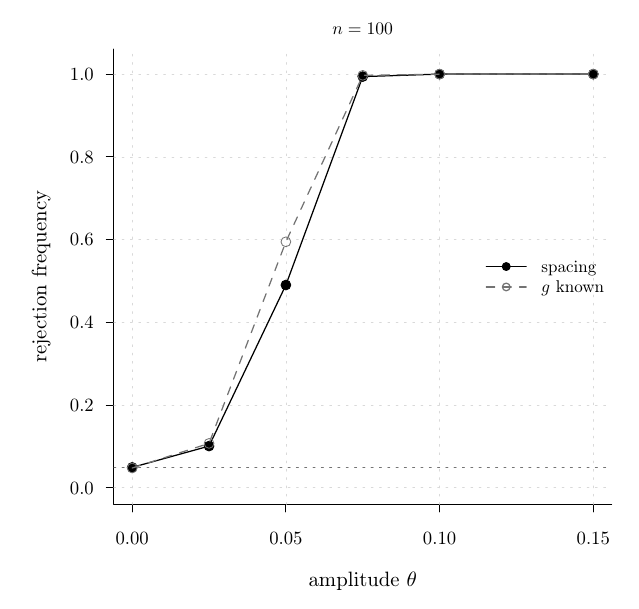}
	\hspace{-.2cm}\includegraphics[width=0.3\linewidth]{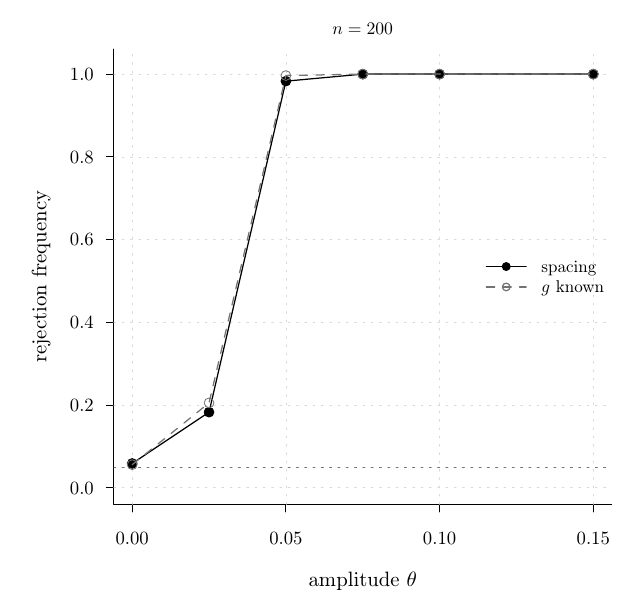}\hspace{-.2cm}
	\includegraphics[width=0.3\linewidth]{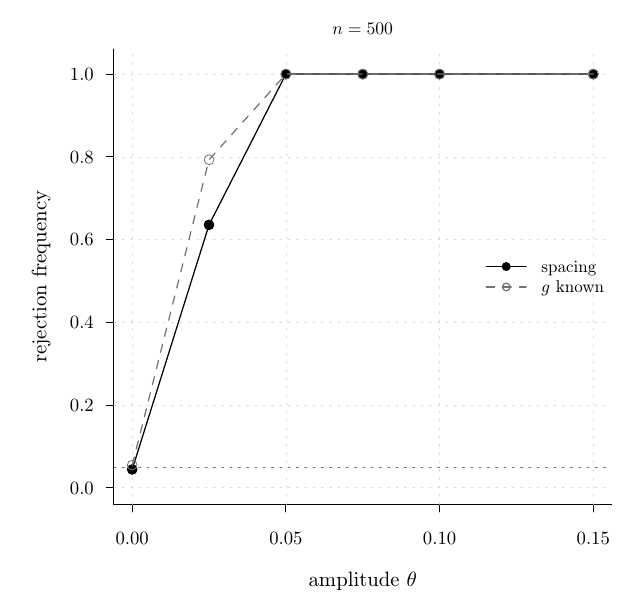}
	\vspace{-.2cm}\caption{\small Rejection frequency against the amplitude $\theta$, direction
		$\mu_1$, $\bar K=25$: the test based on spacings against the test using the
		true $g$. One panel per number of subjects $n$.}
	\label{fig:test_cmp}
\end{figure}

\subsection{Application to diffusion tensor imaging}\label{sec:dti}

\newcommand{\dtiN}{100}
\newcommand{\dtinone}{45}
\newcommand{\dtiM}{235}
\newcommand{\dtiNu}{93}
\newcommand{\dtimmin}{1}
\newcommand{\dtimmax}{6}
\newcommand{\dtidmin}{48}
\newcommand{\dtidmax}{1570}
\newcommand{\dtiV}{10}
\newcommand{\dtiRout}{5}
\newcommand{\dtiblk}{50}
\newcommand{\dtiwinll}{23}
\newcommand{\dtiwinspl}{22}
\newcommand{\dtiwinboth}{21}
\newcommand{\dtipctll}{46}
\newcommand{\dtipctspl}{44}
\newcommand{\dtipctboth}{42}
\newcommand{\dtimedll}{-9.0\times 10^{-3}}
\newcommand{\dtimedspl}{-12.0\times 10^{-3}}
\newcommand{\dtiseed}{2026}
\newcommand{\dtibasis}{\mathcal{B}_0}
\newcommand{\dtiK}{1}
\newcommand{\dtiKbar}{25}
\newcommand{\dtiKmlo}{1}
\newcommand{\dtiKmhi}{25}
\newcommand{\dtiKmlor}{2}
\newcommand{\dtiKmhir}{26}
\newcommand{\dtiNb}{1999}
\newcommand{\dtilevel}{0.95}
\newcommand{\dtipm}{0.850}
\newcommand{\dtip}{1.000}
\newcommand{\dtiQstatm}{-8.0\times 10^{-4}}
\newcommand{\dtiQstat}{-0.212}
\newcommand{\dtibquantilem}{1.5\times 10^{-3}}
\newcommand{\dtibquantile}{0.129} 

For a real-data illustration, we use the diffusion tensor imaging (DTI) study distributed with the \textsf{R} package \texttt{refund} and analyzed by \cite{GCCR2010, PS2015} among others. A multiple
sclerosis patient is a subject $i$ and a visit is a pair $(T_{ij},Y_{ij})$:
the profile $Y_{ij}(\cdot)$ is the fractional anisotropy along the corpus
callosum, recorded at $N_u=\dtiNu$ positions of the tract common to every scan,
and $T_{ij}$ is the number of days elapsed since the reference scan. Thus
$u$ indexes the position along the tract, rescaled to $\mathcal U=[0,1]$. The reference scan, for which $T_{ij}=0$, is left out of every patient's
record, so that each retained visit is a genuine follow-up. We keep the
patients with at least two scans (the baseline and at least one follow-up), that is $n=100$ subjects.
Pooled across these patients, the retained visit times run from $\dtidmin$
to $\dtidmax$ days; we rescale them to $\mathcal T=[0,1]$, and we get $M=235$. 
No other patient is discarded, and the sparsity is longitudinal
only.

The three data-driven estimators from Section~\ref{sec:simulation} are
compared using held-out patients. The patients are split into $\dtiV$ folds.
On the training folds, the criterion~\eqref{eq:cv} selects the pair $(r,K)$ of
$\widehat\mu^{\rm sp}$, and a similar CV criterion selects the bandwidths $h_t$
of $\widehat\mu^{\rm LL}$ and the penalty $\lambda$ of $\widehat\mu^{\rm spl}$
over grids on the same $\dtiV$ folds of patients. The grids for the bandwidth
and the penalty are as in the simulation experiment: 20 values equally spaced
on the log scale in $(0.05, 0.6)$ and $(10^{-5}, 10)$, respectively. The number
of interior knots for the splines is 20. Each estimator is refitted on the training folds at its 
selected tuning and scored on the held-out fold $\mathcal F$, with $M_{\mathcal F}$ data points $(T_{ij}, Y_{ij}(\cdot))$, by
\begin{equation}\label{eq:dti_delta}
	\mathcal E_e=\frac{1}{N_u M_{\mathcal F} }\sum_{i\in\mathcal F}\sum_{j=1}^{m_i}
	\sum_{l=1}^{N_u}\bigl\{Y_{ij}(u_l)-\widehat\mu^{e}(u_l,T_{ij})\bigr\}^{2},
	\qquad
	\Delta_e=\log\bigl(\mathcal E_e/\mathcal E_{\rm sp}\bigr),
	\quad e\in\{\mathrm{LL},\mathrm{spl}\}.
\end{equation}
The values of $\Delta_e$ close to zero indicate practically equivalent fit. The split is repeated
$\dtiRout$ times, which gives $\dtiblk$ paired values, no estimator being
scored at a tuning selected on the patients it is scored on. The splits are reproducible 
using the codes available on \href{https://github.com/mahdy-saidi/Optimal-estimation-and-goodness-of-fit-testing-of-the-mean-for-sparse-longitudinal-functional-data}{GitHub}.
Of these
$\dtiblk$ blocks, $\widehat\mu^{\rm sp}$ has the smaller held-out error on
$\dtiwinll$ against $\widehat\mu^{\rm LL}$, on $\dtiwinspl$ against
$\widehat\mu^{\rm spl}$ and on $\dtiwinboth$ against the two at once, that
is, $\dtipctll\%$, $\dtipctspl\%$ and $\dtipctboth\%$ of the blocks; the
median of $\Delta_e$ is $\dtimedll$ and $\dtimedspl$ against the local-linear smoother and the penalized spline, respectively. The values of $\mathcal E_e$ are presented in Table~\ref{tab:dti} and, for the three estimators, they  are practically indistinguishable. On the whole sample the
criterion selects $(\dtibasis,\,K=\dtiK)$.

\begin{table}[ht!]
\footnotesize 	\centering
	\begin{tabular}{llc}
		\toprule
		estimator & tuning selected & $\mathcal{E}_e$ ($\times 10^{-3}$) \\
		\midrule
		$\widehat\mu^{\rm sp}$ & $\mathcal{B}_0$ (100\%), $K=1$ & $4.23$ \\
		$\widehat\mu^{\rm LL}$ & $\widehat h_t = 0.6$ & $4.21$ \\
		$\widehat\mu^{\rm spl}$ & $\widehat\lambda = 10$ & $4.21$ \\
		\bottomrule
	\end{tabular}\vspace{-.1cm}
	\caption{\small Medians over the $5\times10$ held-out blocks of what the criterion~\eqref{eq:cv} selected on the training patients of each block and of the resulting held-out error $\mathcal{E}_e$ of~\eqref{eq:dti_delta}. The bandwidth $\widehat h_t$ of $\widehat\mu^{\rm LL}$ and the penalty $\widehat\lambda$ of $\widehat\mu^{\rm spl}$, both selected by 10-fold CV, fell at the endpoints of their search grids for the majority of the splits.}
	\label{tab:dti}
\end{table}

In Figure~\ref{fig:DTI_fits}, at $t\in\{0.09, 0.22, 0.71\}$, we plot the spacings
mean-profile estimator and the competing estimators $\mathrm{LL}$ and $\mathrm{spl}$, all
obtained from the profiles of the $n=100$ subjects, together with the $M=235$ observed
profiles. The tuning parameters for $\mathrm{LL}$ and $\mathrm{spl}$ are those reported in
Table~\ref{tab:dti}. The three estimators are practically indistinguishable at the different
values of $t$.

\begin{figure}[ht!]
	\centering
	\includegraphics[width=0.3\linewidth]{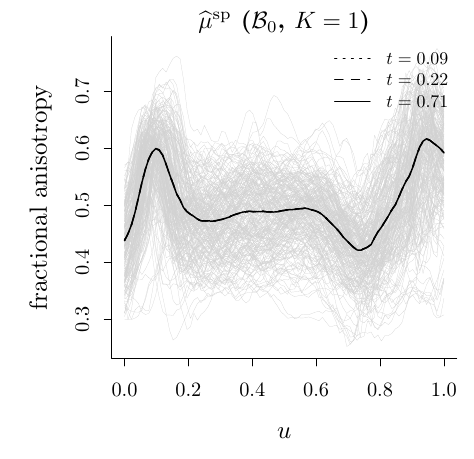}
	\hspace{-.2cm}\includegraphics[width=0.3\linewidth]{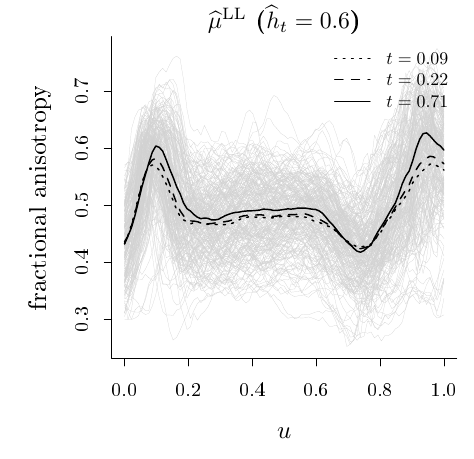}\hspace{-.2cm}
	\includegraphics[width=0.3\linewidth]{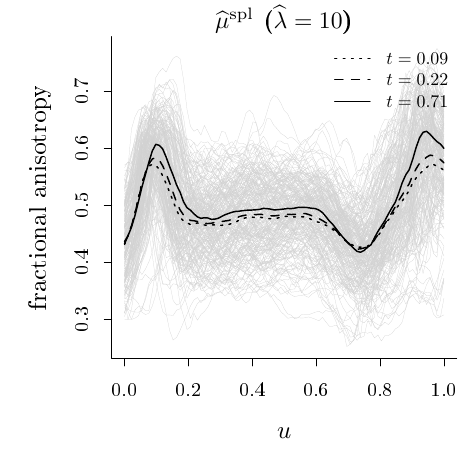}
\vspace{-.2cm}	\caption{\small  Observed profiles from the DTI sample ($n=100$, $M=235$) and the mean-profile estimators: spacings (left), local-linear (middle) and splines (right), at $t\in\{0.09, 0.22, 0.71\}$.}
	\label{fig:DTI_fits}
\end{figure}

The data-driven choice of the tuning parameters in the three methods we used on the real data suggests 
there is no time effect on the mean, or at most, a linear effect. Indeed, for most splits the criterion selects $\mathcal{B}_0$ and $K=1$, while $\widehat h_t$ and $\widehat\lambda$ are selected at the upper end of the grid we used for tuning these parameters. To assess this more formally, we apply our test specifically designed to detect an effect in $t$. We test $H_0:\mu(u,t)=m(u)$, with $m\in L^2(\mathcal U)$ unspecified, using the spacings version of the test proposed in Section~\ref{sec:test}. We proceed in two ways: first, we apply the test with $\mathcal K_M = \{1,2,\ldots, 25\}$ and the profiles $Y_{ij}(\cdot)$  centered by their empirical mean $\widehat m(u) = M^{-1} \sum_{i=1}^{100} \sum_{j=1}^{m_i} Y_{ij}(u)$, $M=235$, as we recommended in Remark~\ref{rem:sp_composite}. Second, we apply it to the raw profiles and  choose $\mathcal K_M = \{2,3,\ldots, 26\}$. We use the bootstrap version of our test, with $N_{\rm b}=\dtiNb$ bootstrap samples of two-point distribution multipliers. The test with the centered profiles gives the $p$-value $\dtipm$, while for that with the raw profiles we get the  $p$-value $\dtip$. The values  $(Q^{\rm sp}_n, q^{\rm sp,*}_{0.95})$ for the two tests are $(\dtiQstatm, \dtibquantilem)$ and $(\dtiQstat, \dtibquantile)$, respectively. The tests agree: the null hypothesis is not rejected. Therefore, based on the real diffusion tensor imaging dataset that we used, \emph{there is no evidence against a mean profile independent of visit time}.

\section{Conclusion}\label{sec:conclusion}

Series expansion is a natural idea in nonparametric estimation and
inference. We carry it over here to the more complex setting of sparse,
randomly designed, longitudinal functional data, where random functions on
a general space are observed repeatedly at a handful of visit times. Such
situations arise across applications with sensor data, in health,
biomechanics, maintenance, and sports, where subjects are monitored but
the follow-up (visit) times are very few.

Our contribution is a single framework that is at once principled and
practical. We define the estimators through weighting schemes, deterministic
or random, required only to satisfy mild high-level conditions, so that the
common weights of the literature fall within our setup. To avoid
knowing or estimating the design density, we introduce a data-driven scheme
based on spacings in the time domain. The building blocks are classical, but
the methodology and the results are not. For estimation we establish explicit, non-asymptotic
risk bounds under sparse random design. For inference we propose a
goodness-of-fit test, for instance, that the mean is constant in time or
equal to a prescribed function, and prove both a non-asymptotic
Gaussian approximation for its statistic and the validity of a bootstrap
calibration that controls the level while remaining powerful against general
alternatives. Extending the scope of classical tools to the setting of repeatedly observed random functions, with non-asymptotic
guarantees throughout, is the contribution of this paper.

The simulations confirm that the methods behave as the theory predicts, and a real-data analysis illustrates their utility. As for the choices the framework involves, basis functions and weights, the short message to practitioners is this: use the orthonormal basis obtained from the cosine system augmented by the monomials $t$ and $t^2$, together with the spacings approach. For estimation, select the polynomial order and the cosine functions by cross-validation. For the goodness-of-fit test, use a fixed, generous truncation of the cosine basis. Augmenting it is usually unnecessary, unless testing whether the mean depends linearly or quadratically on $t$.

Several extensions follow naturally from this work, but are left for future
study. The most immediate is to vector-valued profiles
$Y_{ij}(u)\in\mathbb R^p$, as delivered by most sensors: with $p$ fixed,
each coefficient function simply becomes $\mathbb R^p$-valued, and the
estimators, the risk bounds, and the test carry over with only notational
changes. 
A second direction relaxes the assumption that each profile is observed over all of
$\mathcal U$, replacing it by a grid of evaluation points, common or random. A common
fixed grid needs no separate analysis, as we already pointed out in
Section~\ref{sec:model}: it is the space on which inference is naturally carried out, and
whose resolution the analysis cannot refine. A random grid, drawn independently for each
subject or visit, explores the whole ambient space and is more interesting: there the
inner integration over $u$ carries a genuine discretization error that can be analyzed,
as in the control-neighbor construction of \cite{PW2026}, which would let us quantify how
few $u$-points suffice, as a function of the profiles' regularity, to attain the
completely-observed rate.
The most demanding extension lets the visit domain
$\mathcal T$ be multidimensional, as in \cite{KP2025}. The deterministic and
Monte Carlo weights still apply, and the augmented basis extends through
multivariate polynomials, though the closed forms of
Section~\ref{subsec:aug} no longer have a one-line analogue, and anisotropic
smoothness across the coordinates of $\mathcal T$ would call for additional
theoretical work. Another delicate part would be the spacings construction,
which rests on the order statistics of a one-dimensional sample and has no
direct analogue once $\mathcal T$ has dimension larger than one.\vspace{-.3cm}

\section*{AI Use}
\vspace{-.3cm}
{\small Claude Opus 4.8 and 5.5 (Anthropic), has been used for language improvement and proofreading; for proof and general error/flaw checking and correction; for guiding some of the theoretical derivations; and for prompt-based drafting of certain lemma and theorem statements and of parts of the accompanying proofs. The same tools were also used to assist with the code used in numerical experiments. The authors retain full responsibility for all aspects of this article. We certify that we have verified the details of all AI generated output and have converted any AI generated text to a personal style. To the best of our knowledge, all AI-influenced theoretical derivations, code and references are correct. Where AI assistance guided or contributed to proofs, we have identified and cited all relevant sources, including any not disclosed by the AI output.}\vspace{-.3cm}

\section*{Acknowledgments} \vspace{-.3cm}
{\small V. Patilea gratefully acknowledges the support of the French Agence Nationale de la Recherche (ANR) under reference ANR-24-CE40-2439 (FUNMathStat project), and thanks O. Kassi and N. Bourarach for helpful discussions. Part of this work is based on  M. Saidi's thesis, completed within the \emph{Master for Smart Data Science} program at ENSAI. The authors acknowledge the use of the MRI/DTI data collected at Johns Hopkins University and the Kennedy-Krieger Institute, available through the \textsf{R} package \texttt{refund}.}

\bibliographystyle{abbrv}{\small
	\bibliography{references_V3}}	

@book{AT2007,
  author    = {Adler, Robert J. and Taylor, Jonathan E.},
  title     = {Random Fields and Geometry},
  series    = {Springer Monographs in Mathematics},
  publisher = {Springer},
  address   = {New York},
  year      = {2007}
}

@article{BDT1995,
  author  = {Baszenski, G. and Delvos, F.-J. and Tasche, M.},
  title   = {A united approach to accelerating trigonometric expansions},
  journal = {Computers \& Mathematics with Applications},
  volume  = {30},
  number  = {3--6},
  pages   = {33--49},
  year    = {1995}
}

@article {CD2013,
    AUTHOR = {Comminges, La\"etitia and Dalalyan, Arnak S.},
     TITLE = {Minimax testing of a composite null hypothesis defined via a
              quadratic functional in the model of regression},
   JOURNAL = {Electron. J. Stat.},
  FJOURNAL = {Electronic Journal of Statistics},
    VOLUME = {7},
      YEAR = {2013},
     PAGES = {146--190},
       DOI = {10.1214/13-EJS766},
}

@article {CM2012,
	AUTHOR = {Chen, Kehui and M\"uller, Hans-Georg},
	TITLE = {Modeling repeated functional observations},
	JOURNAL = {J. Amer. Statist. Assoc.},
	FJOURNAL = {Journal of the American Statistical Association},
	VOLUME = {107},
	YEAR = {2012},
	NUMBER = {500},
	PAGES = {1599--1609},
	DOI = {10.1080/01621459.2012.734196},
}

@article{CY2011,
  author  = {Cai, T. Tony and Yuan, Ming},
  title   = {Optimal estimation of the mean function based on discretely
             sampled functional data: Phase transition},
 JOURNAL = {Ann. Statist.},
  FJOURNAL = {The Annals of Statistics},
  volume  = {39},
  number  = {5},
  pages   = {2330--2355},
  year    = {2011},
}

@article {dJ1987,
    AUTHOR = {de Jong, Peter},
     TITLE = {A central limit theorem for generalized quadratic forms},
   JOURNAL = {Probab. Theory Related Fields},
  FJOURNAL = {Probability Theory and Related Fields},
    VOLUME = {75},
      YEAR = {1987},
    NUMBER = {2},
     PAGES = {261--277},
       DOI = {10.1007/BF00354037},
}

@book {DN2003,
	AUTHOR = {David, Herbert Aron and Nagaraja, Haikady Navada},
	TITLE = {Order statistics},
	SERIES = {Wiley Series in Probability and Statistics},
	EDITION = {Third},
	PUBLISHER = {Wiley-Interscience [John Wiley \& Sons], Hoboken, NJ},
	YEAR = {2003},
	PAGES = {xvi+458},
	DOI = {10.1002/0471722162},
}

@book {E1999,
	AUTHOR = {Efromovich, Sam},
	TITLE = {Nonparametric curve estimation},
	SERIES = {Springer Series in Statistics},
	NOTE = {Methods, theory, and applications},
	PUBLISHER = {Springer-Verlag, New York},
	YEAR = {1999},
	PAGES = {xiv+411},
}

@article{EBOS-FW2025,
    doi = {10.1371/journal.pone.0326375},
    author = {Estreich, Harry AND Bullock, Nicola AND Osborne, Mark AND Santos-Fernandez, Edgar AND Wu, Paul Pao-Yen},
    journal = {PLOS ONE},
    publisher = {Public Library of Science},
    title = {An analysis of pacing profiles in sprint kayak racing using functional principal components and hidden Markov models},
    year = {2025},
    month = {07},
    volume = {20},
    url = {https://doi.org/10.1371/journal.pone.0326375},
    pages = {1-15},
    number = {7},
}

@article {GCCR2010,
	AUTHOR = {Greven, Sonja and Crainiceanu, Ciprian and Caffo, Brian and
	Reich, Daniel},
	TITLE = {Longitudinal functional principal component analysis},
	JOURNAL = {Electron. J. Stat.},
	FJOURNAL = {Electronic Journal of Statistics},
	VOLUME = {4},
	YEAR = {2010},
	DOI = {10.1214/10-EJS575},
}

@article {H1984,
    AUTHOR = {Hall, Peter},
     TITLE = {Central limit theorem for integrated square error of
              multivariate nonparametric density estimators},
   JOURNAL = {J. Multivariate Anal.},
  FJOURNAL = {Journal of Multivariate Analysis},
    VOLUME = {14},
      YEAR = {1984},
    NUMBER = {1},
     PAGES = {1--16},
       DOI = {10.1016/0047-259X(84)90044-7},
}

@article {HSSS2021,
	AUTHOR = {Horst, Fabian and Slijepcevic, Djordje and Simak, Marvin and
	Sch\"ollhorn, Wolfgang I.},
	TITLE = {{G}utenberg {G}ait {D}atabase, a ground reaction force
	database of level overground walking in healthy individuals},
	JOURNAL = {Sci. Data},
	FJOURNAL = {Scientific Data},
	VOLUME = {8},
	YEAR = {2021},
	PAGES = {Paper No. 232},
	DOI = {10.1038/s41597-021-01014-6},
}

@article{KP2025,
	title={Optimal inference for the mean of random functions},
	author = {Kassi, Omar and Patilea, Valentin},
	year={2025},
	journal = {arXiv:2504.11025},
	primaryClass={math.ST},
	url={https://arxiv.org/abs/2504.11025},
}

@article {LM1993,
	AUTHOR = {Lu, C. Joseph and Meeker, William Q.},
	TITLE = {Using degradation measures to estimate a time-to-failure
	distribution},
	JOURNAL = {Technometrics},
	FJOURNAL = {Technometrics. A Journal of Statistics for the Physical,
	Chemical and Engineering Sciences},
	VOLUME = {35},
	YEAR = {1993},
	NUMBER = {2},
	PAGES = {161--174},
	DOI = {10.2307/1269661},
}

@article {LPSZ2025,
	AUTHOR = {Leluc, R\'emi and Portier, Fran{\c c}ois and Segers, Johan and
	Zhuman, Aigerim},
	TITLE = {Speeding up {M}onte {C}arlo integration: control neighbors for
	optimal convergence},
	JOURNAL = {Bernoulli},
	FJOURNAL = {Bernoulli. Official Journal of the Bernoulli Society for
	Mathematical Statistics and Probability},
	VOLUME = {31},
	YEAR = {2025},
	NUMBER = {2},
	PAGES = {1160--1180},
	DOI = {10.3150/24-bej1765},
}

@article{LSS2025,
 author = {Liu, Song-Hao and Shao, Qi-Man and Shi, Hao},
 title = {Berry-{Esseen} bounds for degenerate {U}-statistics with application to distance correlation},
 fjournal = {Science China. Mathematics},
 journal = {Sci. China, Math.},
 issn = {1674-7283},
 volume = {68},
 number = {8},
 pages = {1891--1926},
 year = {2025},
 language = {English},
 doi = {10.1007/s11425-024-2416-1},
}

@article {PS2015,
	AUTHOR = {Park, So Young and Staicu, Ana-Maria},
	TITLE = {Longitudinal functional data analysis},
	JOURNAL = {Stat},
	FJOURNAL = {Stat},
	VOLUME = {4},
	YEAR = {2015},
	PAGES = {212--226},
	DOI = {10.1002/sta4.89},
}

@article {PSS2016,
    AUTHOR = {Patilea, Valentin and S\'anchez-Sellero, C\'esar and Saumard,
              Matthieu},
     TITLE = {Testing the predictor effect on a functional response},
   JOURNAL = {J. Amer. Statist. Assoc.},
  FJOURNAL = {Journal of the American Statistical Association},
    VOLUME = {111},
      YEAR = {2016},
    NUMBER = {516},
     PAGES = {1684--1695},
       DOI = {10.1080/01621459.2015.1110031},
}

@article {PW2026,
	AUTHOR = {Patilea, Valentin and Wang, Sunny G. W.},
	TITLE = {Rate accelerated inference for integrals of multivariate
              random functions},
	JOURNAL = {Comput. Statist. Data Anal.},
	FJOURNAL = {Computational Statistics \& Data Analysis},
	VOLUME = {214},
	YEAR = {2026},
	PAGES = {Paper No. 108273, 18},
	ISSN = {0167-9473,1872-7352},
	DOI = {10.1016/j.csda.2025.108273},
	URL = {https://doi.org/10.1016/j.csda.2025.108273},
}

@article{RR1997,
 author = {Rinott, Yosef and Rotar, Vladimir},
 title = {On coupling constructions and rates in the {CLT} for dependent summands with applications to the antivoter model and weighted {{\(U\)}}-statistics},
 fjournal = {The Annals of Applied Probability},
 journal = {Ann. Appl. Probab.},
 issn = {1050-5164},
 volume = {7},
 number = {4},
 pages = {1080--1105},
 year = {1997},
 language = {English},
 doi = {10.1214/aoap/1043862425},
}

@article {S1982,
    AUTHOR = {Stone, Charles J.},
     TITLE = {Optimal global rates of convergence for nonparametric
              regression},
   JOURNAL = {Ann. Statist.},
  FJOURNAL = {The Annals of Statistics},
    VOLUME = {10},
      YEAR = {1982},
    NUMBER = {4},
     PAGES = {1040--1053},
}

@article {SLCR2014,
    AUTHOR = {Staicu, Ana-Maria and Li, Yingxing and Crainiceanu, Ciprian M.
              and Ruppert, David},
     TITLE = {Likelihood ratio tests for dependent data with applications to
              longitudinal and functional data analysis},
   JOURNAL = {Scand. J. Stat.},
  FJOURNAL = {Scandinavian Journal of Statistics. Theory and Applications},
    VOLUME = {41},
      YEAR = {2014},
    NUMBER = {4},
     PAGES = {932--949},
       DOI = {10.1111/sjos.12075},
}

@book {SW2009,
    AUTHOR = {Shorack, Galen R. and Wellner, Jon A.},
     TITLE = {Empirical processes with applications to statistics},
    SERIES = {Classics in Applied Mathematics},
    PUBLISHER = {Society for Industrial and Applied Mathematics,
              Philadelphia, PA},
      YEAR = {2009},
     PAGES = {xli+956},
       DOI = {10.1137/1.9780898719017},
}

@book {T2009,
	AUTHOR = {Tsybakov, Alexandre Borisovich},
	TITLE = {Introduction to nonparametric estimation},
	SERIES = {Springer Series in Statistics},
	PUBLISHER = {Springer, New York},
	YEAR = {2009},
	PAGES = {xii+214},
	DOI = {10.1007/b13794},
}

@article {W2021,
    AUTHOR = {Wang, Qiyao},
     TITLE = {Two-sample inference for sparse functional data},
   JOURNAL = {Electron. J. Stat.},
  FJOURNAL = {Electronic Journal of Statistics},
    VOLUME = {15},
      YEAR = {2021},
    NUMBER = {1},
     PAGES = {1395--1423},
       DOI = {10.1214/21-ejs1802},
}

@article {WBLHRGH2021,
	TITLE = {{PCA} of waveforms and functional {PCA}: A primer for biomechanics},
	FJOURNAL = {Journal of Biomechanics},
JOURNAL = {J.  Biomech.},
	VOLUME = {116},
	PAGES = {110106},
	YEAR = {2021},
	ISSN = {0021-9290},
	DOI = {10.1016/j.jbiomech.2020.110106},
	AUTHOR = {John Warmenhoven and Norma Bargary and Dominik Liebl and Andrew Harrison and Mark A. Robinson and Edward Gunning and Giles Hooker},
}

@article{WCS2021,
  AUTHOR  = {Wang, Shuoyang and Cao, Guanqun and Shang, Zuofeng},
  TITLE   = {Estimation of the mean function of functional data via deep neural networks},
  JOURNAL = {Stat},
  VOLUME  = {10},
  NUMBER  = {1},
  PAGES   = {e393},
  YEAR    = {2021},
  DOI     = {10.1002/sta4.393},
  URL     = {https://doi.org/10.1002/sta4.393},
}

@article {XHSFZC2015,
	AUTHOR = {Xiao, Luo and Huang, Lei and Schrack, Jennifer A. and
	Ferrucci, Luigi and Zipunnikov, Vadim and Crainiceanu, Ciprian
	M.},
	TITLE = {Quantifying the lifetime circadian rhythm of physical
	activity: a covariate-dependent functional approach},
	JOURNAL = {Biostatistics},
	FJOURNAL = {Biostatistics},
	VOLUME = {16},
	YEAR = {2015},
	NUMBER = {2},
	PAGES = {352--367},
	DOI = {10.1093/biostatistics/kxu045},
}

@article {YMW2005,
	AUTHOR = {Yao, Fang and M\"uller, Hans-Georg and Wang, Jane-Ling},
	TITLE = {Functional data analysis for sparse longitudinal data},
	JOURNAL = {J. Amer. Statist. Assoc.},
	FJOURNAL = {Journal of the American Statistical Association},
	VOLUME = {100},
	YEAR = {2005},
	NUMBER = {470},
	PAGES = {577--590},
	DOI = {10.1198/016214504000001745},
}

@incollection{Y1997,
  author    = {Yu, Bin},
  title     = {Assouad, {F}ano, and {Le Cam}},
  booktitle = {Festschrift for Lucien Le Cam: Research Papers in Probability and Statistics},
  editor    = {Pollard, David and Torgersen, Erik and Yang, Grace L.},
  pages     = {423--435},
  year      = {1997},
  publisher = {Springer, New-York},
  doi       = {10.1007/978-1-4612-1880-7_29},
}

@article {Z2022,
    AUTHOR = {Zhang, Zhuo-Song},
     TITLE = {Berry-{E}sseen bounds for generalized {$U$}-statistics},
   JOURNAL = {Electron. J. Probab.},
  FJOURNAL = {Electronic Journal of Probability},
    VOLUME = {27},
      YEAR = {2022},
     PAGES = {Paper No. 134, 36},
       DOI = {10.1214/22-ejp860},
}

@article {ZSG2011,
	AUTHOR = {Zhou, Rensheng R. and Serban, Nicoleta and Gebraeel, Nagi},
	TITLE = {Degradation modeling applied to residual lifetime prediction
	using functional data analysis},
	JOURNAL = {Ann. Appl. Stat.},
	FJOURNAL = {The Annals of Applied Statistics},
	VOLUME = {5},
	YEAR = {2011},
	NUMBER = {2B},
	PAGES = {1586--1610},
	DOI = {10.1214/10-AOAS448},
}

@article {ZSQ2025,
    AUTHOR = {Zhang, Chi and Sang, Peijun and Qin, Yingli},
     TITLE = {Two-sample inference for sparse functional data},
   JOURNAL = {Electron. J. Stat.},
  FJOURNAL = {Electronic Journal of Statistics},
    VOLUME = {19},
      YEAR = {2025},
    NUMBER = {1},
     PAGES = {792--864},
      ISSN = {1935-7524},
       DOI = {10.1214/25-ejs2348},
}

@article{ZW2016,
	AUTHOR = {Zhang, Xiaoke and Wang, Jane-Ling},
	TITLE = {From sparse to dense functional data and beyond},
	JOURNAL = {Ann. Statist.},
	FJOURNAL = {The Annals of Statistics},
	VOLUME = {44},
	YEAR = {2016},
	NUMBER = {5},
	pages = {2281--2321},
}
	\addcontentsline{toc}{section}{References}

\newpage

\appendix 

\section*{Appendix}

\addcontentsline{toc}{section}{Appendix}
\renewcommand{\thesubsection}{\Alph{subsection}}
\setcounter{subsection}{0}    

The Appendix contains the proofs of the theoretical results and some additional information on the numerical experiments. In Section~\ref{app_sec:var_OS} we compare the integrated risks of the observation-level and subject-level weighting schemes. Section~\ref{app_sec:MC} develops the Monte Carlo weights, establishing their admissibility and a second-moment bound, and Section~\ref{app:risk_gknown} gives the proofs of the risk bounds for the known-density estimators. Section~\ref{app_sec:sp} treats the design-density-free spacings estimator, with its risk bound and optimal truncation level. Section~\ref{app_sec:lower_bound} contains the proof of the minimax lower bound attained by the estimators we study. Section~\ref{app_sec:basis0} elaborates on the theoretical and practical aspects related to the basis choice. More precisely, Subsections~\ref{app_sec:basis} and~\ref{app_sec:basis2} provide detailed discussions on the decay rates achievable with the half-cosine basis and its augmented version, the endpoint conditions under which the target regularity is attained, and closed-form expressions for the augmented basis together with uniform bounds on its elements and their derivatives. Section~\ref{app_sec:test} is devoted to the goodness-of-fit test: after establishing a new Berry-Esseen bound for $U$-statistics of order two, including the non-degenerate ones, a result of independent interest proved in Subsection~\ref{app:sec_BE_new}, we prove the Gaussian approximation of our test statistics under both the null and the alternative, and justify the multiplier bootstrap. The supporting technical lemmas used in the proofs of Section~\ref{app_sec:test} are proved in Subsection~\ref{sec:technic_proof}. The proofs for the feasible test statistic based on spacings are given in Section~\ref{append_test_spc}, with the proofs of the supporting lemmas gathered in Subsection~\ref{sec:technic_proof_Spc}. Finally, Section~\ref{append_simus_est} reports additional numerical experiments.

\subsection{Variances of the coefficient estimators with OBS and SUBJ weights}\label{app_sec:var_OS}

Here, we discuss the difference in the variances of the coefficient estimators $\widehat{\beta}_k(u)$, see~\eqref{eq:var_decomp}, induced by the two deterministic weighting schemes, denoted  $\operatorname{Var}_{\rm OBS}$ and $\operatorname{Var}_{\rm SUBJ}$. Recall that $\bar m_A$ (resp. $\bar m_H$) denotes the arithmetic (resp. harmonic) mean of the per-subject visit counts $m_i$. Then,  $\sum_{i,j}\omega_{ij}^{2} = 1/(n\bar m_A)$ for OBS and $1/(n\bar m_H)$ for SUBJ. Let us denote $\sum_{i=1}^{n}\sum_{1\leq j\neq j'\leq m_i}\omega_{ij}\omega_{ij'}$ obtained with OBS and SUBJ weights by $A_{\mathrm{OBS}}/n$ and $A_{\mathrm{SUBJ}}/n$, respectively. Then, 
\[
A_{\mathrm{OBS}}
= \frac{\overline{m^2}}{\bar m_A^{2}}-\frac{1}{\bar m_A} ,
\qquad
A_{\mathrm{SUBJ}}
= 1-\frac{1}{\bar m_H} \qquad \text{ with } 
\quad \overline{m^2}=\frac1n\sum_{i=1}^{n}m_i^2 .
\]
The factors of $D_k(u)$ and $E_k(u)$ pull in opposite directions. Indeed, since $\bar m_H\leq\bar m_A$, the term containing $D_k(u)$ favors the variance  of $\widehat{\beta}_k(u)$ with the OBS weights. On the other hand,
\[
A_{\mathrm{OBS}}-A_{\mathrm{SUBJ}}
=\Bigl(\frac{1}{\bar m_H}-\frac{1}{\bar m_A}\Bigr)+\frac{\sigma_m^{2}}{\bar m_A^{2}}\geq 0,
\qquad \sigma_m^2 = \overline{m^2}-\bar m_A^2 ,
\]
so the term containing $E_k(u)$ favors the variance of $\widehat{\beta}_k(u)$ with the SUBJ weights. In the case where the $m_i$ are bounded, $\operatorname{Var}_{\rm OBS}$ and $\operatorname{Var}_{\rm SUBJ}$ have the same $n^{-1}\asymp M^{-1}$ rate. Their difference is
$$
\operatorname{Var}_{\rm OBS}-\operatorname{Var}_{\rm SUBJ}= \frac 1n \Bigl(\frac{1}{\bar m_H}-\frac{1}{\bar m_A}\Bigr) \{E_k(u)-D_k(u)\} +\frac 1n  \frac{\sigma_m^{2}}{\bar m_A^{2}}E_k(u).
$$
Thus, for a given $k$, the OBS weights are preferable (smaller coefficient variance) if $D_k(u)/E_k(u)$ is large, which occurs, e.g., in the presence of measurement error, or for large $k$ where $E_k(u)$ is expected to be small relative to $D_k(u)$.

\subsection{Random weights by Monte Carlo integration: definitions and properties}\label{app_sec:MC}

Pooling all observations across subjects, we identify each pair $(i,j)$ with a single index running over a set of indices of cardinality $M = \sum_{i=1}^n m_i$. For clarity, below the MC weights are denoted $\omega_{ij}^{\mathrm{MC}}$.

\begin{definition}[Leave-one-out neighbors and Voronoi volumes]\label{def:loo_voronoi} 
	We denote by $\mathcal{S} = \{T_{ij} : i = 1,\ldots,n,\; j = 1,\ldots,m_i\} \subset \mathcal{T}$ the set of all design points. 
	\begin{enumerate}
		\item The leave-one-out nearest neighbor of $t \in \mathcal{T}$ among $\mathcal{S} \setminus \{T_{ij}\}$ is
		\[
		\widehat{N}(t;\,\mathcal{S}\setminus\{T_{ij}\}) = \arg\min_{s \in \mathcal{S}\setminus\{T_{ij}\}}|t - s|,
		\]
		with ties broken by lexicographic order.
		\item For $T_{ij} \in \mathcal{S}$, the \emph{leave-one-out Voronoi}
		cell associated with $T_{ij}$ relative to the leave-one-out set
		$\mathcal{S}\setminus\{T_{i'j'}\}$ is
		\[
		\mathcal{V}_{ij}^{(i'j')}
		:= \{t \in \mathcal{T} : \widehat{N}(t;\,\mathcal{S}\setminus
		\{T_{i'j'}\}) = T_{ij}\},
		\]
		and its \emph{volume} under the design density $g$ is
		\[
		V_{ij}^{(i'j')}
		:= \int_{\mathcal{V}_{ij}^{(i'j')}} g(v)\,\mathrm{d}v.
		\]
		\item The \emph{degree} $\widehat{d}_{ij}$ counts the number of
		indices $(i',j')\neq(i,j)$ for which $T_{ij}$ is the leave-one-out
		nearest neighbor of $T_{i'j'}$, and the cumulative Voronoi
		volume $\widehat{c}_{ij}$ aggregates the corresponding volumes:
		\[
		\widehat{d}_{ij}
		= \sum_{(i',j')\neq(i,j)}
		\mathbf{1}\{T_{i'j'} \in \mathcal{V}_{ij}^{(i'j')}\},
		\quad
		\widehat{c}_{ij}
		= \sum_{(i',j')\neq(i,j)} V_{ij}^{(i'j')}.
		\]
	\end{enumerate}
\end{definition}

The control-neighbors weights are
\[
\omega_{ij}^{\mathrm{MC}} := \frac{1 + \widehat{c}_{ij} - \widehat{d}_{ij}}{M},
\qquad i = 1,\ldots,n,\; j = 1,\ldots,m_i,
\]
and the corresponding estimator obtained by using these weights in \eqref{eq:generic_est} is
\[
\widehat{\beta}_k^{(\mathrm{MC})}(u) = \sum_{i=1}^{n}\sum_{j=1}^{m_i} \omega_{ij}^{\mathrm{MC}}\, Y_{ij}(u)\,\frac{\phi_k(T_{ij})}{g(T_{ij})}.
\]

\medskip


\begin{lemma}\label{prop:cn_admissible}
	Under Assumption~\ref{Ass_gen}, it holds  $\sum_{i,j}(\omega_{ij}^{\mathrm{MC}})^2 \leq 1 + 3/M$.
\end{lemma}

\begin{proof}[Proof of Lemma \ref{prop:cn_admissible}]
	First note that 
	$
	\sum_{i=1}^{n}\sum_{j=1}^{m_i}(1 + \widehat{c}_{ij} - \widehat{d}_{ij})= M,
	$
	and consequently $\sum_{i,j}\omega_{ij}^{\mathrm{MC}} = 1$. See \cite[Proposition~1]{LPSZ2025}.
	Next, note that in dimension 1, we have $\widehat d_{ij}\in\{0,1,2\}$. Hence,
	\[
	|1+\widehat c_{ij}-\widehat d_{ij}|\leq 1+\widehat c_{ij} \quad \text{ and } \quad \sum_{i,j}(\omega_{ij}^{\mathrm{MC}})^2 \leq \frac{1}{M^2}\sum_{i,j}(1+\widehat c_{ij})^2.
	\]
	Expanding and using $\sum_{i,j}\widehat c_{ij} = M$ together with $(\sum_{i,j}\widehat c_{ij})^2 \geq \sum_{i,j}\widehat c_{ij}^2$ (by positivity of $\widehat c_{ij}$),
	\[
	\sum_{i,j}(1+\widehat c_{ij})^2 = M + 2M + \sum_{i,j}\widehat c_{ij}^2 \leq 3M+M^2,
	\]
	which gives $\sum_{i,j}(\omega_{ij}^{\mathrm{MC}})^2 \leq 1 + 3/M$.
\end{proof}

\medskip


\begin{lemma}\label{lem:cn_weighted_second_moment} 
	Under Assumption~\ref{Ass_gen}, for any non-negative integrable function $h(T)$, it holds  
	\begin{equation}\label{eq:bound_hT}
		\mathbb{E}\Bigg[\sum_{i,j}(\omega_{ij}^{\mathrm{MC}})^2\,h(T_{ij})\Bigg]
		\leq \frac{\kappa}{M}\,\int_{\mathcal{T}} h(t)\,g(t)\,\mathrm{d}t,
	\end{equation}
	with $\kappa$ a constant depending only on the lower and upper bounds of the design density $g$. 
\end{lemma}

\begin{proof}[Proof of Lemma \ref{lem:cn_weighted_second_moment}]
	By exchangeability of $(T_{ij},\widehat c_{ij},\widehat d_{ij})$ across $(i,j)$, omitting the subscripts, it holds 
	\begin{equation}\label{eq:exchange}
		\mathbb{E}\Bigg[\sum_{i,j}(\omega_{ij}^{\mathrm{MC}})^2 h(T_{ij})\Bigg] = \frac{1}{M}\,\mathbb{E}\big[(1+\widehat c-\widehat d)^2 h(T)\big].
	\end{equation}
	Since in dimension 1 we have $0\leq \widehat d \leq 2$, it holds
	$
	(1+\widehat c-\widehat d)^2 \leq (1+\widehat c)^2 = 1+2\widehat c+\widehat c^2.
	$
	By \cite[Proposition~2]{KP2025}, we have that
	\(
	\mathbb{E}[\widehat{c} \mid T=t] = \mathbb{E}[\widehat{d} \mid
	T=t].
	\)
	Combined with $\widehat{d} \leq 2$, it gives
	\(
	\mathbb{E}[\widehat{c} \mid T=t] \leq 2.
	\)
	Next, note that in dimension 1,
	$
	\widehat c = (M-1)V + W_L + W_R,\quad 0\leq W_L\leq V_L,\ 0\leq W_R\leq V_R,
	$
	where $V$ is the Voronoi volume of $T$, $V_L,V_R$ are those of its two adjacent design points and $W_L,
	W_R$ are the parts of $V_L, V_R$ absorbed by $T$ when its
	adjacent neighbors are removed. By $(a+b+c)^2\leq 3(a^2+b^2+c^2)$,
	\[
	\widehat c^2 \leq 3(M-1)^2 V^2 + 3V_L^2+3V_R^2.
	\]
	Each Voronoi cell which is not on the boundary is bounded by the midpoints to its two neighbors.
	After the change of variable $U=G(\cdot)$ the $M$ design points become
	i.i.d.\ uniform on $[0,1]$, with adjacent spacings $S_L,S_R$ on either
	side of $U=G(t)$ and each cell volume satisfies (with $g_{\min}=\inf_{\mathcal T} g$, $g_{\max}=\sup_{\mathcal T} g$)
	\begin{equation}\label{eq:neighV_boundary}
		V \le \frac{g_{\max}}{g_{\min}}\,\frac{S_L+S_R}{2}.
	\end{equation}
	Using the inequality $(S_L+S_R)^2\le 2S_L^2+2S_R^2$ and the beta conditional law of adjacent gaps given $T$ a design point, see \citep[Chapter~6]{DN2003}, we get 
	\[
	\mathbb{E}[V^2\mid T=t]
	\le \frac{2\,(g_{\max}/g_{\min})^2}{M^2},
	\]
	and similarly for $V_L,V_R$. Hence, for \(M\geq 2\), since \(((M-1)^2+2)/M^2\leq 1\), we obtain
	\[
	\mathbb{E}[\widehat c^2\mid T=t] \leq 6\,\frac{(M-1)^2+2}{M^2}\,\Bigl(\frac{g_{\max}}{g_{\min}}\Bigr)^2 \leq 6\,\Bigl(\frac{g_{\max}}{g_{\min}}\Bigr)^2.
	\]
	Combining,
	\(
	\mathbb{E}[(1+\widehat c-\widehat d)^2\mid T=t] \leq 5+6(g_{\max}/g_{\min})^2
	\) for every $t \in \mathcal{T}$ and $M \geq 2$.
	Using this bound in \eqref{eq:exchange} with $h\geq 0$, and setting $\kappa := 5 + 6\,(g_{\max}/g_{\min})^2$,  yields \eqref{eq:bound_hT}. For the two extreme design points (near the boundaries of $\T$), the inequality~\eqref{eq:neighV_boundary} has to be replaced by
	\begin{equation} 
		V \le  (g_{\max}/g_{\min}) \, (S_L+S_R).
	\end{equation}
	Using this for all Voronoi cells and following the same steps only changes $\kappa := 5 + 6\,(g_{\max}/g_{\min})^2$ to $\kappa := 5 + 24\,(g_{\max}/g_{\min})^2$. The latter is the value of $\kappa$ which can be retained for the statement. 
\end{proof}

\medskip


\subsection{Proofs for the mean estimators in the known density case}\label{app:risk_gknown}

\begin{proof}[Proof of Theorem~\ref{thm:det_risk_bound}]    
	Under Assumptions~\ref{ass:design} and~\ref{ass:indep}, for deterministic weights and using identity \eqref{eq:var_decomp} which we integrate with respect to $\nu  $, we can decompose the integrated risk of $\widehat \beta_k(\cdot)$ as 
	$
	\mathcal{R}_k = \mathcal{R}_k^{\mathrm{diag}}+\mathcal{R}_k^{\mathrm{off}}+\mathcal{R}_k^{\mathrm{design}},
	$
	where 
	\begin{align*}
		\mathcal{R}_k^{\mathrm{diag}}
		&= \Bigl(\sum_{i,j}\omega_{ij}^2\Bigr) \int_{\mathcal{T}}\frac{\phi_k^2(t)}{g(t)}\bigl\{\Gamma_X(t,t)+\tau^2(t) \nu(\mathcal U)\bigr\}\,\mathrm{d}t,\\
		\mathcal{R}_k^{\mathrm{off}}
		&= \Bigl(\sum_{i}\sum_{j\neq j'}\omega_{ij}\omega_{ij'}\Bigr) \int_{\mathcal U}\mathrm{Var}\!\left(\int_{\mathcal T}\phi_k(t)\,X(u,t)\,\mathrm{d}t\right) {\rm d}\nu(u), 
		\\	\mathcal{R}_k^{\mathrm{design}} &= \Bigl(\sum_{i,j}\omega_{ij}^2\Bigr)  \int_{\mathcal U} \mathrm{Var}\!\left( \mu(u,T)\,\phi_k(T)/g(T) \right) {\rm d}\nu (u) .
	\end{align*}
	Next, we  sum each term over $k = 1,\ldots,K$. Exchanging the finite sums with the integrals, and since $\int_{\mathcal T}\phi_k^2 = 1$, $\forall k\geq 1$, we get
	\begin{multline}
		\Bigl(\sum_{i,j}\omega_{ij}^2\Bigr)^{-1}\sum_{k=1}^K 	\mathcal{R}_k^{\mathrm{diag}} =	\sum_{k=1}^K \int_{\mathcal{T}}
		\frac{\phi^2_k(t)}{g(t)}\,
		\bigl\{\Gamma_X(t,t)+\tau^2(t)\nu(\mathcal U) \bigr\}\,\mathrm{d}t
		\\ = \int_{\mathcal{T}}
		\frac{\Gamma_X(t,t)+\tau^2(t)\nu(\mathcal U) }{g(t)}
		\left(\sum_{k=1}^K\phi^2_k(t)\right)\mathrm{d}t 
		\\ \leq \frac{  \|\Gamma_X(\cdot,\cdot)\|_{\infty} +\| \tau (\cdot)\|_{\infty}^2\nu(\mathcal U) }{g_{\min}}K.
	\end{multline}
	Next, by Fubini's Theorem, and  Parseval's identity together with Assumption~\ref{ass:decay} which gives $$\sum_{k=1}^K\|\beta_k\|^2_{L^2(\mathcal{U})}\geq\|\mu\|^2_{L^2(\mathcal{U}\times\mathcal{T})}-C_sK^{-2s},$$ 
	we get
	\begin{multline}
		\Bigl(\sum_{i,j}\omega_{ij}^2\Bigr)^{-1}\sum_{k=1}^K 	\mathcal{R}_k^{\mathrm{design}} 
		= \sum_{k=1}^K\left(	\int_{\mathcal{T}}\frac{\phi^2_k(t)}{g(t)}\,
		\|\mu(\cdot,t)\|_{L^2(\mathcal{U})}^2\,\mathrm{d}t - \|\beta_k\|_{L^2(\mathcal{U})}^2\right)
		\\ = \int_{\mathcal{T}}
		\frac{\|\mu(\cdot,t)\|_{L^2(\mathcal{U})}^2}{g(t)}
		\left(\sum_{k=1}^K\phi^2_k(t)\right)\mathrm{d}t 
		- \sum_{k=1}^K\|\beta_k\|_{L^2(\mathcal{U})}^2 
		\\ \le  g^{-1}_{\rm min}\sup_{k\geq 1}\|\phi_k\|^2_\infty  \|\mu\|^2_{L^2(\mathcal{U}\times\mathcal{T})} K -\|\mu\|^2_{L^2(\mathcal{U}\times\mathcal{T})}  +C_sK^{-2s}.
	\end{multline}
	This bound can be used with a bounded basis (like the half-cosine basis). If $\sup_{k\geq 1}\|\phi_k\|^2_\infty$ is not finite, we can bound differently: 
	\begin{multline}
		\Bigl(\sum_{i,j}\omega_{ij}^2\Bigr)^{-1}\sum_{k=1}^K 	\mathcal{R}_k^{\mathrm{design}} 
		= \int_{\mathcal{T}}
		\frac{\|\mu(\cdot,t)\|_{L^2(\mathcal{U})}^2}{g(t)}
		\left(\sum_{k=1}^K\phi^2_k(t)\right)\mathrm{d}t 
		- \sum_{k=1}^K\|\beta_k\|_{L^2(\mathcal{U})}^2 \\ \leq g^{-1}_{\rm min} \sup_{t\in\mathcal T}\|\mu(\cdot,t)\|^2_{L^2(\mathcal{U})}  K  +C_sK^{-2s}-\|\mu\|^2_{L^2(\mathcal{U}\times\mathcal{T})}.
	\end{multline}
	The constant $C_{1,\mu}$ was introduced to account for the two cases. 
	
	For the off-diagonal term  $\mathcal{R}_k^{\mathrm{off}}$, let us note that by Fubini's Theorem, 
	\[
	\mathrm{Var}\!\left(\int_{\mathcal T}\phi_k(t)\,X(u,t)\,\mathrm{d}t\right) = \int_{\mathcal{T}^2}\Gamma_X(t,s)\phi_k(t)\phi_k(s)\,\mathrm{d}t\,\mathrm{d}s = \langle \Upsilon_{\Gamma_X}\phi_k,\phi_k\rangle_{L^2(\mathcal{T})}=: 	\alpha_k,
	\]
	where $\Upsilon_{\Gamma_X} : f \mapsto \int_{\mathcal{T}}\Gamma_X(\cdot,s)
	f(s)\mathrm{d}s$ is the integral operator with covariance kernel $\Gamma_X$.
	Mercer's Theorem yields the eigen-decomposition
	$
	\Gamma_X(t,s) = \sum_{\ell=1}^\infty \lambda_\ell\,\psi_\ell(t)\,\psi_\ell(s),
	$
	with eigenvalues $\lambda_1 \geq \lambda_2 \geq \cdots \geq 0$,
	$\{\psi_\ell\}_{\ell \geq 1}$ the corresponding continuous orthonormal
	eigenfunctions, the series converging absolutely and uniformly on
	$\mathcal{T}\times\mathcal{T}$. By the definition of $\alpha_k$, we get
	\[
	\alpha_k
	= \sum_{\ell=1}^\infty \lambda_\ell
	\left(\int_{\mathcal{T}}\psi_\ell(t)\,\phi_k(t)\,\mathrm{d}t\right)
	\left(\int_{\mathcal{T}}\psi_\ell(s)\,\phi_k(s)\,\mathrm{d}s\right)
	= \sum_{\ell=1}^\infty \lambda_\ell\,
	\langle \psi_\ell,\phi_k\rangle_{L^2(\mathcal{T})}^2.
	\]
	Summing over all $k \geq 1$ and applying Parseval's identity
	to each eigenfunction $\psi_\ell$ in the basis $\{\phi_k\}$:
	\[
	0\leq 	\sum_{k=1}^\infty\alpha_k = \sum_\ell\lambda_\ell\sum_k\langle\psi_\ell,\phi_k\rangle^2 = \sum_\ell\lambda_\ell =  \int_\mathcal{T}\Gamma_X(t,t)\,\mathrm{d}t = \mathbb E \bigl[\, \|X\|^2_{L^2(\mathcal{U}\times \mathcal T)\,}\bigr]=: C_2 .
	\]
	We deduce that 
	$$
	\sum_{k=1}^K\mathcal{R}_k^{\mathrm{off}} =\left(\sum_{k=1}^K\alpha_k\right) \Bigg[\sum_i\sum_{j\neq j'}\omega_{ij}\omega_{ij'}\Bigg]\leq  C_2\Bigg[\sum_i\sum_{j\neq j'}\omega_{ij}\omega_{ij'}\Bigg]_+.
	$$
	The result follows by combining the bounds on $\mathcal{R}_k^{\mathrm{diag}}$, $\mathcal{R}_k^{\mathrm{off}}$ and $\mathcal{R}_k^{\mathrm{design}}$. 
\end{proof} 

\medskip


\begin{proof}[Proof of Theorem~\ref{thm:cn_risk_bound}] 
	The MC weights being $\mathcal F^{\rm obs}$-measurable, $B_k(u)$ and $V_k(u)$ are
	uncorrelated due to Assumption~\ref{ass:indep}, so that $\mathrm{Var}(\widehat\beta_k(u))=\mathrm{Var}(B_k(u))
	+\mathbb E[\mathrm{Var}(V_k(u)\mid\mathcal F^{\rm obs})]$, although \eqref{eq:var_decomp}
	itself fails. By \eqref{eq:bias_variance} and Assumption~\ref{Ass_gen}, the $\eta_{ij}$
	being i.i.d.\ with unit variance,
	\begin{multline}
		\mathrm{Var}\bigl(V_k(u)\mid\mathcal F^{\rm obs}\bigr)
		=\sum_{i}\sum_{j,j'}\omega^{\mathrm{MC}}_{ij}\omega^{\mathrm{MC}}_{ij'}
		\frac{\phi_k(T_{ij})\phi_k(T_{ij'})}{g(T_{ij})g(T_{ij'})}\gamma_X(u;T_{ij},T_{ij'})
		\\ +\sum_{i,j}\bigl(\omega^{\mathrm{MC}}_{ij}\bigr)^2
		\frac{\phi^2_k(T_{ij})}{g^2(T_{ij})}\tau^2(T_{ij}).
	\end{multline}
	Integrating with respect to $\nu$, using \eqref{def:GammaX} and the Mercer decomposition $\Gamma_X(t,s)=\sum_\ell\lambda_\ell\psi_\ell(t)\psi_\ell(s)$ used in the proof of Theorem~\ref{thm:det_risk_bound}, the first term becomes
	$\sum_\ell\lambda_\ell\sum_i Z^2_{i,k,\ell}$, where 
	$$
	Z_{i,k,\ell}:=\sum_j
	\omega^{\mathrm{MC}}_{ij}\phi_k(T_{ij})\psi_\ell(T_{ij})/g(T_{ij}).
	$$ 
	By the Cauchy-Schwarz inequality,
	$Z^2_{i,k,\ell}\leq m_i\sum_j(\omega^{\mathrm{MC}}_{ij})^2\phi^2_k(T_{ij})
	\psi^2_\ell(T_{ij})/g^2(T_{ij})$, so that $m_i\leq m_{\max}$ and
	$\sum_\ell\lambda_\ell\psi^2_\ell(t)=\Gamma_X(t,t)$ give
	\[
	\int_{\mathcal U}\mathrm{Var}\bigl(V_k(u)\mid\mathcal F^{\rm obs}\bigr)\mathrm{d}\nu(u)
	\leq\sum_{i,j}\bigl(\omega^{\mathrm{MC}}_{ij}\bigr)^2\frac{\phi^2_k(T_{ij})}{g^2(T_{ij})}
	\bigl\{\Gamma_X(T_{ij},T_{ij})m_{\max}+\tau^2(T_{ij})\nu(\mathcal U)\bigr\}.
	\]
	Taking expectations, Lemma~\ref{lem:cn_weighted_second_moment} applies with a
	nonnegative $h$; summing over $k\leq K$ and using $\int_{\mathcal T}\phi^2_k=1$ bounds
	this contribution by $ C_1^{\mathrm{MC}}K/M$. By Proposition~\ref{lem:cn_error} and
	Assumption~\ref{ass:growth}, $\int_{\mathcal U}\mathrm{Var}(B_k(u))\mathrm{d}\nu(u)\leq
	\bar C_{\mathrm{MC}}(M^{-1-2\alpha}+k^{2q}M^{-3})$, which summed over $k\leq K$ using 
	$\sum_{k=1}^{K}k^{2q}\leq K^{1+2q}$, gives the second term. Finally $\mathcal B^2_K\leq C_sK^{-2s}$ by
	Assumption~\ref{ass:decay}, and \eqref{eq:full_decomp} concludes.
\end{proof}

\medskip


\begin{proof}[Proof of Corollary~\ref{thm:cn_optim}] 
	Optimizing the strictly convex function $K\mapsto \mathcal C_1^{\rm MCM} K/M+C_sK^{-2s}$  gives
	$K^*_{\mathrm{MC}}$ and the first term in the risk bound. Next, substituting $K^*_{\mathrm{MC}}$ into
	$\bar C_{\mathrm{MC}}K^{1+2q}M^{-3}$ term in the bound~\eqref{eq:bound55} gives the second term in the bound for $\mathcal R^*(\widehat \mu)$, and
	\[
	\frac{1+2q}{2s+1}-3=-\frac{2s}{2s+1}-\frac{2(2s+1-q)}{2s+1},
	\]
	an exponent strictly smaller than $-2s/(2s+1)$ if and only if $q<2s+1$.
\end{proof}


\subsection{Proofs for the design-density-free mean estimator}\label{app_sec:sp}

Throughout this section, $G(t)=\int_0^tg$ and $U_l:=G(T_l)$, so that the $U_l$ are
i.i.d.\ uniform on $[0,1]$, and $\Delta_j:=U_{(j)}-U_{(j-1)}$, $j=1,\ldots,M+1$, with
$U_{(0)}=0$ and $U_{(M+1)}=1$; each $\Delta_j$ is $\mathrm{Beta}(1,M)$
\citep[Chapter~6]{DN2003}. Under Assumption~\ref{ass:design}, $G$ is a
$C^1$-diffeomorphism with $(G^{-1})'\leq 1/g_{\min}$, hence
$T_{(l')}-T_{(l)}\leq g_{\min}^{-1}(U_{(l')}-U_{(l)})$ for $l\leq l'$. The symmetric
reflection preserves consecutive gaps, so every window width remains a sum of
$2\mathfrak h$ genuine interior spacings and the computations below hold for all ranks;
see \cite[Section~4.2]{E1999}. Finally, let $R(t):=\#\{l:T_{(l)}\leq t\}$. Then
$\chi_l(t)=\mathbf 1\{T_{(l-\mathfrak h)}\leq t\leq T_{(l+\mathfrak h)}\}$ is equal to 1 if and only if $R(t)-\mathfrak h+1\leq l\leq R(t)+\mathfrak h$, so that at most $2\mathfrak h$ windows cover any $t$, with equality on
$\mathcal T_M:=[T_{(\mathfrak h)},T_{(M-\mathfrak h)}]$.

\begin{proof}[Proof of Lemma~\ref{lem:sp_second_moment}]
	By Cauchy-Schwarz inequality applied to \eqref{eq:sp_est},
	\[
	\mathbb E\Bigl[\sum_l\widehat\Phi^{\,2}_{kl}\Bigr]
	\leq\frac{1}{(2\mathfrak h)^2}\int_{\mathcal T}\phi^2_k(t)\,\mathbb E[\Lambda(t)]\,
	\mathrm{d}t,
	\qquad
	\Lambda(t):=\sum_l\int_{\T}\chi_l(t).
	\]
	For every $l$ with $\chi_l(t)=1$ one has
	$U_{(l-\mathfrak h)}\leq G(t)\leq U_{(l+\mathfrak h)}$, so that
	$T_{(l+\mathfrak h)}-T_{(l-\mathfrak h)}\leq g_{\min}^{-1}
	\{(U_{(l+\mathfrak h)}-G(t))+(G(t)-U_{(l-\mathfrak h)})\}$, where we set $U_{(m)}=0$ for $m<1$ and 
	$U_{(m)}=1$ for $m>M$, so the extended-order-statistic gap upper-bounds the true width and, at boundary ranks, the Beta laws below degenerate to Dirac masses at 0 or 1, respectively. Then, conditionally on
	$\{R(t)=l_0\}$, the block decomposition of uniform order statistics 
	\citep[Chapter~2]{DN2003} gives (with the convention $0/0=0$ when the ratio is indeterminate)
	\[
	\frac{U_{(l+\mathfrak h)}-G(t)}{1-G(t)}\sim
	\mathrm{Beta}(l+\mathfrak h-l_0,\,M-l-\mathfrak h+1),
	\qquad
	\frac{U_{(l-\mathfrak h)}}{G(t)}\sim
	\mathrm{Beta}(l-\mathfrak h,\,l_0-l+\mathfrak h+1),
	\]
	whose means summed over $l\in\{l_0-\mathfrak h+1,\ldots,l_0+\mathfrak h\}$ involve the
	arithmetic sum $1+\cdots+2\mathfrak h=\mathfrak h(2\mathfrak h+1)$ in both cases, so that
	\[
	\mathbb E[\Lambda(t)\mid R(t)=l_0]
	\leq\frac{\mathfrak h(2\mathfrak h+1)}{g_{\min}}
	\Bigl(\frac{G(t)}{l_0+1}+\frac{1-G(t)}{M-l_0+1}\Bigr).
	\]
	As $R(t)\sim\mathrm{Binom}(M,G(t))$, we have
	$\mathbb E[G(t)/(R(t)+1)]=\{1-(1-G(t))^{M+1}\}/(M+1)\leq(M+1)^{-1}$, and symmetrically
	for the second term, whence
	$\mathbb E[\Lambda(t)]\leq 2\mathfrak h(2\mathfrak h+1)/\{(M+1)g_{\min}\}$. Substituting
	and using $\int_{\mathcal T}\phi^2_k=1$ gives the bound, uniformly in $k$.
\end{proof}

\medskip


\begin{proof}[Proof of Theorem~\ref{thm:sp_risk_optim}] 
	Set $\widetilde\mu_M(u,t):=(2\mathfrak h)^{-1}\sum_l\mu(u,T_{(l)})\chi_l(t)$, so that
	$$B^{\rm sp}_k(u)-\beta_k(u)=\langle\phi_k,\widetilde\mu_M(u,\cdot)-\mu(u,\cdot)\rangle,$$
	by \eqref{eq:sp_est}, and let
	$\mathcal A_M:=\int_{\mathcal U}\mathbb E\|\widetilde\mu_M(u,\cdot)
	-\mu(u,\cdot)\|^2_{L^2(\mathcal T)}\,\mathrm{d}\nu(u)$. We bound the three sums in
	\eqref{eq:sp_decomp}.
	
	For the stochastic term, as in the proof of Theorem~\ref{thm:cn_risk_bound},
	with $\omega^{\mathrm{MC}}_{ij} 	\phi_k(T_{ij})/g(T_{ij})$ replaced by $\widetilde\omega_{ij,k}$, Mercer's decomposition, Cauchy-Schwarz inequality and $m_i\leq m_{\max}$ give, back in the pooled index,
	\begin{multline}
		\int_{\mathcal U}\mathrm{Var}\bigl(V^{\rm sp}_k(u)\mid\mathcal F^{\rm obs}\bigr)
		\mathrm{d}\nu(u)
		\leq\sum_l\widehat\Phi^{\,2}_{kl}
		\bigl\{m_{\max}\Gamma_X(T_{(l)},T_{(l)})+\tau^2(T_{(l)})\nu(\mathcal U)\bigr\}
		\\	\leq g_{\min}C^{\rm sp}_1\sum_l\widehat\Phi^{\,2}_{kl},
	\end{multline}
	so that Lemma~\ref{lem:sp_second_moment} and summation over $k\leq K$ bound this
	contribution by $\mathcal S^{\rm sp}C^{\rm sp}_1K$.
	
	For the bias terms, by Bessel's inequality we get
	$$
	\sum_{k\leq K}(B^{\rm sp}_k(u)-\beta_k(u))^2\leq
	\|\widetilde\mu_M(u,\cdot)-\mu(u,\cdot)\|^2_{L^2(\mathcal T)}.
	$$ 
	Since
	$\mathrm{Var}(B^{\rm sp}_k(u))\leq\mathbb E[(B^{\rm sp}_k(u)-\beta_k(u))^2]$ and, by
	Jensen's inequality, $(\mathbb E[B^{\rm sp}_k(u)]-\beta_k(u))^2\leq
	\mathbb E[(B^{\rm sp}_k(u)-\beta_k(u))^2]$, the two terms are each bounded by
	$\mathcal A_M$. On $\mathcal T_M$, where $(2\mathfrak h)^{-1}\sum_l\chi_l(t)=1$,
	Assumption~\ref{ass:holder} and Jensen's inequality give
	$$|\widetilde\mu_M(u,t)-\mu(u,t)|^2\leq L^2_\mu(u)(2\mathfrak h)^{-1}
	\sum_l\chi_l(t)(T_{(l+\mathfrak h)}-T_{(l-\mathfrak h)})^{2\alpha},$$
	so that, integrating
	in $t$ and using $\int_{\mathcal T}\chi_l\leq T_{(l+\mathfrak h)}-T_{(l-\mathfrak h)}$,
	\[
	\int_{\mathcal T_M}\bigl|\widetilde\mu_M(u,t)-\mu(u,t)\bigr|^2\mathrm{d}t
	\leq\frac{L^2_\mu(u)}{2\mathfrak h\,g^{2\alpha+1}_{\min}}
	\sum_l\bigl(U_{(l+\mathfrak h)}-U_{(l-\mathfrak h)}\bigr)^{2\alpha+1}.
	\]
	By convexity, $(U_{(l+\mathfrak h)}-U_{(l-\mathfrak h)})^{2\alpha+1}\leq
	(2\mathfrak h)^{2\alpha}\sum_{j=l-\mathfrak h+1}^{l+\mathfrak h}\Delta^{2\alpha+1}_j$,
	and each $\Delta_j$ belongs to at most $4\mathfrak h$ windows. Moreover, 
	Jensen's inequality with $\mathbb E[\Delta^3_1]=6/\{(M+1)(M+2)(M+3)\}$ gives
	$$
	\mathbb E[\Delta^{2\alpha+1}_1]\leq\{\mathbb E[\Delta^3_1]\}^{(2\alpha+1)/3}
	\leq 6\,(M+1)^{-(2\alpha+1)}, \quad  \forall \alpha\in(0,1].
	$$ 
	Hence this part of $\mathcal A_M$ is
	at most $6\,\|L_\mu\|^2_{L^2(\mathcal U)}g^{-(2\alpha+1)}_{\min}\allowbreak
	\{4\mathfrak h/(M+1)\}^{2\alpha}$. On $\mathcal T\setminus\mathcal T_M$, where
	$\sum_l\chi_l(t)\leq2\mathfrak h$, Jensen's inequality gives $\int_{\mathcal U}\widetilde\mu^2_M(u,t)
	\mathrm{d}\nu(u)\leq\sup_{t}\|\mu(\cdot,t)\|^2_{L^2(\mathcal U)}$, so that the integrand
	is at most equal to  $4\sup_{t}\|\mu(\cdot,t)\|^2_{L^2(\mathcal U)}$, while
	$\mathbb E[\lambda(\mathcal T\setminus\mathcal T_M)]\leq
	g^{-1}_{\min}\{\mathbb E[U_{(\mathfrak h)}]+\mathbb E[1-U_{(M-\mathfrak h)}]\}
	=g^{-1}_{\min}(2\mathfrak h+1)/(M+1)$. (Here, $\lambda$ denotes the Lebesgue measure on $\mathcal T$.) 
	Doubling the sum of the two parts gives the terms
	in $C^{\rm sp}_2$ and $C^{\rm sp}_3$.

	Adding $\mathcal B^2_K\leq C_sK^{-2s}$
	yields the non-asymptotic bound. Its $K$-dependent part, $K\mapsto
	\mathcal S^{\rm sp}C^{\rm sp}_1K+C_sK^{-2s}$, is a strictly convex function, and the expression of $K^*$ follows. Finally
	$\mathcal S^{\rm sp}\asymp M^{-1}$, and $\mathfrak h=o(M^{c})$ for every $c>0$ makes the
	remaining terms, of order $(\mathfrak h/M)^{2\alpha}$ and $\mathfrak h/M$, both
	$o(M^{-2s/(2s+1)})$ under $\alpha > s/(2s+1)$, since $2s/(2s+1)<1$.
\end{proof}

\subsection{Proof of the minimax lower bound}\label{app_sec:lower_bound}

Recall that $M = m_1 +\cdots+ m_n$ is the total number of observed pairs $(T_{ij}, Y_{ij}(\cdot))$, and for the estimation part of our study the visit counts $m_i$ are non-random.

\medskip

\begin{proof}[Proof of Theorem~\ref{thm:lb}]
	A lower bound over any subfamily of $\mathcal P$ is a lower bound over $\mathcal P$. Let $e_\star\equiv\nu(\U)^{-1/2}$ (\emph{i.e.}, constant on $\U$), so that $\|e_\star\|_{L^2(\mathcal U)}=1$, and restrict attention to mean functions
	\begin{equation}\label{eq:hard_model}
		\mu \in \Fclass(s,C_s) \qquad \text{with } \ \mu (u,t) = f(t) e_\star (u),
	\end{equation}
	with $f\in L^2(\mathcal T)$. Let $\mathfrak X$ be a zero-mean Gaussian process on $\T$ with bounded covariance function $\mathfrak K$ (for instance $\mathfrak X\equiv 0$), and take
	$$
	X_i (u,t) = \mathfrak X_i(t)\, e_\star (u), \qquad \eta_{ij}(u)=\mathfrak z_{ij}, \qquad u\in\U,\ t\in\T,
	$$
	where $\{\mathfrak X_i\}$ are independent copies of $\mathfrak X$ and $\{\mathfrak z_{ij}\}$ are i.i.d. $\mathcal N(0,1)$. Let the $T_{ij}$ be i.i.d. with a density $g$ satisfying Assumption~\ref{ass:design} (e.g., the uniform density), let $\tau$ be bounded with $\inf_{\T}\tau\geq\tau_0$, and let $\{\mathfrak X_{i}\}$, $\{T_{ij}\}$ and $\{\mathfrak z_{ij}\}$ be mutually independent. Then Assumption~\ref{Ass_gen} holds, with $\mathfrak K=\Gamma_X$, and the corresponding distributions belong to $\mathcal P$. Moreover,
	$$
	Y_{ij} = Z_{ij}\, e_\star,\quad \text{ with } \ Z_{ij}= f(T_{ij}) + \mathfrak X_i(T_{ij}) + \nu(\U)^{1/2}\tau(T_{ij})\, \mathfrak z_{ij},\quad 1\le i \leq n, \ 1\leq j\leq m_i.
	$$
	Using the lexicographic order for the indices $(i,j)$, define the vectors (column matrices) in $\mathbb R^M$
	$$
	\bT = (T_{11},\ldots, T_{1m_1}, T_{21},\ldots,T_{nm_n})^\top, \qquad 	\bZ = (Z_{11},\ldots, Z_{1m_1}, Z_{21},\ldots,Z_{nm_n})^\top,
	$$
	and $\bfT(f) = (f(T_{11}),\ldots, f(T_{1m_1}), f(T_{21}),\ldots,f(T_{nm_n}))^\top$, and the $M\times M$-matrix $\Sigma_{\bT}$ with entries
	$$
	(\Sigma_{\bT})_{(i,j),(i',j')} = \mathbf 1\{i=i'\}\, \Gamma_X(T_{ij},T_{ij'})+\nu(\U)  \operatorname{diag}(\tau^2(T_{ij})).
	$$ 
	Then, the conditional distribution $\mathbb P_f$ of $\bZ$ given the design $\bT$ is multivariate normal $\mathcal N_{\!M} (\bfT(f),\Sigma_{\bT})$. By construction, the matrix $\Sigma_{\bT}- \kappa\, {\rm Id}_M$, with $\kappa:=\nu(\U)\,\tau_0^2$, is positive semi-definite. Hence, using the expression of the Kullback-Leibler divergence between two multivariate normal distributions with the same covariance matrix, and integrating out with respect to the design, for any $f_1,f_2\in L^2(\mathcal T)$ such that~\eqref{eq:hard_model} holds, we have
	\begin{multline}\label{eq:KL_bound_lb}
		\KL(\mathbb P_{f_1}\,\| \mathbb P_{f_2}) = \frac12 \E \left[ \left(\bfT(f_2)-\bfT(f_1) \right)^\top \Sigma^{-1}_{\bT} \left(\bfT(f_2)-\bfT(f_1) \right) \right]
		\\ \le \frac12 \E \Bigg[\sum_{i,j}\frac{\{f_2 (T_{ij}) - f_1 (T_{ij}) \}^2}{ \kappa}\Bigg]
		=\frac M{2\kappa} \int_{\T}  \{f_2(t)- f_1(t)\}^2 g(t)\,\mathrm dt
		\\	\le \frac{ M\, \overline g }{2\,\kappa}\| f_2-f_1\|_{L^2 (\T)}^2.
	\end{multline}
	Next, let $p\geq 1$ and $\delta>0$ to be chosen below. Let $\omega = (\omega_1,\ldots,\omega_p)\in\{0,1\}^p$, and consider the mean functions
	$\mu_\omega (u,t) = f_\omega(t) e_\star (u)$ with
	$$
	f_\omega(t)=\delta  \sum_{k=1}^p \omega_k \phi_k (t)  ,
	$$
	so that
	$$
	\mu_\omega (u,t) = \sum_{k\ge 1 } \beta_{\omega,k}(u)  \phi_k (t)   , \quad \text{with } \ \beta_{\omega,k}(u) = \begin{cases} \delta \omega_k   e_\star(u) & \text{if } 1\leq k \leq p\\ 	0 & \text{if }   k>  p \end{cases}.
	$$
	Then, it holds
	\begin{equation}\label{eq:proof_lb_aux1}
		\sum_{k\ge 1} k^{2s} \|\beta_{\omega,k}\|^2 _{L^2(\mathcal U)}= \delta^2   \sum_{k\ge 1 } k^{2s}\omega_k^2 \le\delta^2 p^{2s+1}, \qquad \omega\in\{0,1\}^p.
	\end{equation}
	Thus, $\mu_\omega\in \Fclass (s,C_s)$ if $\delta^2 \le C_s p^{-(2s+1)}$. Moreover, for any $\omega,\omega'\in\{0,1\}^p$, it holds
	\begin{multline}\label{eq:proof_lb_aux2}
		\|\mu_\omega - \mu_{\omega'}\|^2_{L^2(\mathcal U\times \T )} =	\|f_\omega - f_{\omega'}\|^2_{L^2(  \T )} \|e_\star \|^2_{L^2(\mathcal U)}  \\
		= \delta^2  \Big\|\sum_{k= 1}^p \{ \omega_k -\omega'_k\}\phi_k\Big \|^2_{L^2(\mathcal T)} = \delta^2 \sum_{k= 1}^p \{ \omega_k -\omega'_k\}^2 = \delta^2 H(\omega,\omega'),
	\end{multline}
	with $H(\cdot,\cdot)$ the Hamming distance on the cube $\{0,1\}^p$. Thus, $\|\mu_\omega - \mu_{\omega'}\|^2_{L^2(\mathcal U\times \T )} = \delta^2$ if $H(\omega,\omega')=1$. On the other hand, applying the $\KL$ bound~\eqref{eq:KL_bound_lb} with $(f_1,f_2)=(f_\omega , f_{\omega'})$ and $H(\omega,\omega')=1$, for any $\delta^2 \le M^{-1} \kappa/\overline g $, we get
	$$
	\KL(\mathbb P_{\omega}\,\| \mathbb P_{\omega'})\leq \frac{ M\, \overline g }{2\,\kappa}\delta^2\leq \frac 12,
	$$
	and by Pinsker's inequality, the total variation distance between $\mathbb P_{\omega}$ and $\mathbb P_{\omega'}$ is also bounded by 1/2. Then, we choose
	$$
	\delta^2 =\frac{\kappa}{\overline g }  M^{-1} \quad \text{and} \quad p= \left\lfloor \left( C_s\frac{\overline g }{\kappa}\right)^{1/(2s+1)} M^{1/(2s+1)}\right\rfloor ,
	$$
	and, by Assouad's Lemma (see \cite[Lemma 2]{Y1997}, \cite[Sec. 2.7.2]{T2009}) we get
	\[
	\inf_{\widehat\mu}\;
	\sup_{\omega}
	\E\,\norm{\widehat\mu-\mu_\omega}_{ L^2(\U\times\T)}^2
	\ge \frac{\delta^2p}{16} \ge c\,M^{-\frac{2s}{2s+1}}
	\ge c\,m_{\rm max}^{-\frac{2s}{2s+1}}n^{-\frac{2s}{2s+1}},
	\]
	with $c$ depending on $s$, $C_s$, $\overline g $ and $\kappa=\nu(\U)\,\tau_0^2$.
\end{proof}

\medskip

\begin{proof}[Proof of Corollaries~\ref{cor:minimax} and~\ref{cor:minimax_sp}] 
	
	\emph{Lower bounds.} Consider the subfamily of distributions used in the proof of Theorem~\ref{thm:lb}, with $\mathfrak X\equiv0$, $\tau\equiv\tau_0$ and $g\equiv1$ on $\T=[0,1]$. Since $\underline g\leq1\leq\overline g$ and $\tau_0\leq\overline\tau$, and since $\Gamma_X\equiv0$ is continuous, this subfamily is contained in $\mathcal P^+$. Hence, the lower bound of Theorem~\ref{thm:lb} holds with $\mathcal P$ replaced by $\mathcal P^+$, which yields the lower bound in Corollary~\ref{cor:minimax}.
	
	For Corollary~\ref{cor:minimax_sp}, $s=2$ and $\{\phi_k\}$ is the augmented basis of Section~\ref{subsec:aug}, for which $L_{\phi,k}=\|\phi_k'\|_\infty\leq C'_\phi\, k$ for some constant $C'_\phi$. Thus, for all $u\in\U$ and $t,t'\in\T$,
	$$
	|\mu_\omega(u,t)-\mu_\omega(u,t')| = \frac {|f_\omega(t)-f_\omega(t')|}{\nu(\U)^{1/2}} \leq \frac{\delta}{\nu(\U)^{1/2}}  \sum_{k=1}^p L_{\phi,k}\,|t-t'| \leq  C'_\phi \frac{\delta p^2}{\nu(\U)^{1/2}}    \,|t-t'| .
	$$
	With the choice of $\delta$ and $p$ in the proof of Theorem~\ref{thm:lb} and $s=2$, we have $\delta p^2\asymp M^{-1/10}\to0$. Thus, Assumption~\ref{ass:holder} holds with $\alpha=1$, $L_{\mu_\omega}\equiv\nu(\U)^{-1/2}C'_\phi\,\delta p^2$ and $\int_{\U}L^2_{\mu_\omega}\mathrm d\nu=C'^{\,2}_\phi\,\delta^2p^4\leq L^2$ for $n$ sufficiently large. Since moreover $L_g=0\leq\overline L_g$, the subfamily is contained in $\mathcal P^+_{L}$, and the lower bound of Theorem~\ref{thm:lb} holds with $\mathcal P$ replaced by $\mathcal P^+_{L}$.
	
	\medskip
	
	\emph{Upper bounds, OBS and SUBJ weights.} Let $P\in\mathcal P^+$ and $c_\phi:=\sup_{k\geq1}\|\phi_k\|_\infty$. Since 
	$$
	\|\mu\|^2_{L^2(\U\times\T)}=\sum_{k\geq1}\|\beta_k\|^2_{L^2(\U)}\leq C_s,
	$$ 
	we have $C_{1,\mu}\leq C_s c_\phi^2$, and thus, in Theorem~\ref{thm:det_risk_bound},
	$$
	C_1\leq \overline C_1:=\bigl(C_s c_\phi^2+\overline\Gamma+\overline\tau^{\,2}\nu(\U)\bigr)/\underline g
	\qquad\text{and}\qquad
	C_2=\int_{\T}\Gamma_X(t,t)\,\mathrm dt\leq\overline\Gamma .
	$$
	Under Condition~\eqref{cd_m_max}, for the OBS weights, $\sum_{i,j}\omega_{ij}^2=M^{-1}$ and 
	$$
	\sum_{i}\sum_{j\neq j'}\omega_{ij}\omega_{ij'}=M^{-2}\sum_i m_i(m_i-1)\leq m_{\max}M^{-1},
	$$ 
	while for the SUBJ weights, 
	$$
	\sum_{i,j}\omega_{ij}^2=n^{-2}\sum_i m_i^{-1}\leq n^{-1}\leq m_{\max}M^{-1}
	$$
	and 
	$$
	\sum_{i}\sum_{j\neq j'}\omega_{ij}\omega_{ij'}=n^{-2}\sum_i (m_i-1)/m_i\leq m_{\max}M^{-1}.
	$$
	Theorem~\ref{thm:det_risk_bound} then yields, uniformly over $\mathcal P^+$,
	$$
	\mathcal R(\widehat\mu_\bullet;K)\leq m_{\max}M^{-1}\bigl(\overline C_1K+C_sK^{-2s}\bigr)+m_{\max}\overline\Gamma M^{-1}+C_sK^{-2s}.
	$$
	Since $\widetilde K^*\leq 2M^{1/(2s+1)}$ and $(\widetilde K^*)^{-2s}\leq M^{-2s/(2s+1)}$, the upper bound in Corollary~\ref{cor:minimax} follows.
	
	\medskip
	
	\emph{Upper bounds, MC weights and spacings.} Let $P\in\mathcal P^+_{L}$. For $u\in\U$ and $t,t'\in\T$, Assumption~\ref{ass:holder} gives $|\mu(u,t)|\leq|\mu(u,t')|+L_\mu(u)$, and integrating with respect to $t'$ over $[0,1]$, $\|\mu(u,\cdot)\|_\infty\leq\|\mu(u,\cdot)\|_{L^2(\T)}+L_\mu(u)$. Since $\int_{\U}\|\mu(u,\cdot)\|^2_{L^2(\T)}\mathrm d\nu(u)=\|\mu\|^2_{L^2(\U\times\T)}\leq C_s$, we get
	\begin{equation}\label{eq:cor_sup_mu}
		\sup_{t\in\T}\|\mu(\cdot,t)\|^2_{L^2(\U)}\leq\int_{\U}\|\mu(u,\cdot)\|^2_\infty\,\mathrm d\nu(u)\leq 2\,(C_s+L^2).
	\end{equation}

	\noindent\emph{-- MC weights.} By Proposition~\ref{lem:cn_error}, Assumption~\ref{ass:growth} holds with $q=1$ and a constant $\bar C_{\mathrm{MC}}$ depending only on $\overline L_g$, $\underline g$, $\overline g$ and the basis. Moreover, $C_1^{\mathrm{MC}}\leq\kappa\,(m_{\max}\overline\Gamma+\overline\tau^{\,2}\nu(\U))/\underline g$, $\int_{\U}L^2_\mu\,\mathrm d\nu\leq L^2$ and, by \eqref{eq:cor_sup_mu}, $\widetilde\mu^{\,2}\leq 2(C_s+L^2)$. With $\alpha=1$, $q=1$ and $K=\widetilde K^*\leq 2M^{1/5}$, we have $\mathfrak c_1(\widetilde K^*)\leq C_\phi^2M^{-2}$ and $\mathfrak c_2(\widetilde K^*)\leq 4M^{-8/5}$. Theorem~\ref{thm:cn_risk_bound} then yields, uniformly over $\mathcal P^+_{L}$,
	\begin{multline}
		\mathcal R(\widehat\mu_\bullet;\widetilde K^*)\leq \bigl(C_s+2C_1^{\mathrm{MC}}\bigr)M^{-4/5} \\ +2\bar C_{\mathrm{MC}}\bigl\{C_\phi^2L^2M^{-2}+8(C_s+L^2)M^{-8/5}\bigr\}M^{-4/5}\leq C\,M^{-4/5}.
	\end{multline}

	\noindent\emph{-- Spacings.} By \eqref{eq:cor_sup_mu}, the constants of Theorem~\ref{thm:sp_risk_optim} satisfy
	$$
	C_1^{\rm sp}\leq\bigl(m_{\max}\overline\Gamma+\overline\tau^{\,2}\nu(\U)\bigr)/\underline g,\qquad
	C_2^{\rm sp}\leq 12L^2/\underline g^{\,3},\qquad
	C_3^{\rm sp}\leq 16\,(C_s+L^2)/\underline g .
	$$
	Since $\mathfrak h\geq1$, we have $\mathcal S^{\rm sp}\leq 2M^{-1}\leq 1$ (w.l.o.g $M\geq 2$), and Theorem~\ref{thm:sp_risk_optim} with $\alpha=1$ and $K=\widetilde K^*$ yields, uniformly over $\mathcal P^+_{L}$,
	$$
	\mathcal R\bigl(\widehat\mu^{\rm sp};\widetilde K^*\bigr)\leq\bigl(C_1^{\rm sp}+C_s\bigr)M^{-4/5}+C_2^{\rm sp}\Bigl(\frac{2\mathfrak h}{M+1}\Bigr)^{2}+C_3^{\rm sp}\,\frac{2\mathfrak h+1}{M+1}.
	$$
	Since $\mathfrak h=O(\log\log M)$, the last two terms are $o(M^{-4/5})$, and the bound is at most $C\,M^{-4/5}$ for $n$ sufficiently large.
	
	\medskip
	
	\emph{Condition~\eqref{cd:mu_twice} and the class $\mathcal P^+_{L}$.} Let $f\in C^2([0,1])$. By the mean value theorem for integrals, there exist $a\in[0,1/3]$ and $b\in[2/3,1]$ such that $f(a)=3\int_0^{1/3}f$ and $f(b)=3\int_{2/3}^1f$, so that $|f(a)|,|f(b)|\leq\sqrt3\,\|f\|_{L^2(\T)}$ by the Cauchy--Schwarz inequality. Since $b-a\geq1/3$, there exists $\xi\in(a,b)$ with $|f'(\xi)|=|f(b)-f(a)|/(b-a)\leq6\sqrt3\,\|f\|_{L^2(\T)}$, and thus
	$$
	\|f'\|_\infty\leq 6\sqrt3\,\|f\|_{L^2(\T)}+\|f''\|_\infty .
	$$
	Applied to $f=\mu(u,\cdot)$, this shows that Condition~\eqref{cd:mu_twice} implies Assumption~\ref{ass:holder} with $\alpha=1$ and $L_\mu(u)=\sup_{t}|\partial_t\mu(u,t)|$, and, 
	$$
	\text{for $\mu\in\Fclass(2,C_s)$, } \ 
	\int_{\U}L^2_\mu\,\mathrm d\nu\leq 216\,C_s+2\int_{\U}\sup_{t}|\partial_t^2\mu(u,t)|^2\,\mathrm d\nu(u).
	$$
\end{proof}

\subsection{Complements on the choice of the  orthonormal basis}\label{app_sec:basis0}

\subsubsection{Basis choice and regularity conditions}\label{app_sec:basis}

In this subsection we elaborate on the decay rates, as defined in Assumption~\ref{ass:decay}, in the case of the half-cosine basis and its augmented version. The focus is on mean functions which are twice continuously differentiable in the variable $t$, though the extension to smoother means is briefly discussed.

Let $\beta_k^{\mathrm c}(u) := \langle \mu(u,\cdot), c_k\rangle$, where $\{c_k\}$ is the half-cosine basis as defined in Section \ref{sec:basis_choice}, on $\mathcal T =[0,1]$. The following result explains the fact that the half-cosine basis cannot accommodate regularities $s\geq 3/2$ as defined in Assumption~\ref{ass:decay}.

\begin{proposition}\label{prop:cosine_cap}
	Suppose $\mu(u,\cdot) \in C^2([0,1])$ for a.e.\ $u$,  with
	$\int_{\mathcal U}\sup_t|\partial_t^{2}\mu(u,t)|^2\,\mathrm{d}\nu(u)<\infty$,
	and that the time boundary derivatives satisfy
	\[
	\|\partial_t\mu(\cdot,0)\pm \partial_t\mu(\cdot,1)\|_{L^2(\mathcal U)}^2>0.
	\]
	Then, writing $\lambda_k=(k-1)\pi$ and, for $k\ge2$,
	\[
	\mathfrak D_k(u):=(-1)^{k-1}\partial_t\mu(u,1)-\partial_t\mu(u,0),\qquad
	R_k(u):=-\int_0^1\partial_t^2\mu(u,t)\cos(\lambda_k t)\,\mathrm{d}t,
	\]
	the cosine coefficients admit the decomposition
	\[
	\beta_k^{\mathrm c}(u)=\frac{\sqrt2}{\lambda_k^{2}}\big\{\mathfrak D_k(u)+R_k(u)\big\},
	\qquad \|R_k\|_{L^2(\mathcal U)}\to0 .
	\]
	Hence $\|\beta_k^{\mathrm c}\|_{L^2(\mathcal U)}\asymp k^{-2}$, so $C^2$-regularity with respect to $t$
	guarantees Assumption~\ref{ass:decay} for $s<3/2$, and $s^\star=3/2$. Meanwhile,
	Assumption~\ref{ass:decay} fails for every $s\ge3/2$, irrespective of higher
	smoothness of the mean (thus $s^\star=3/2$ is not attained in this case).
\end{proposition} 

\begin{proof}[Proof of Proposition~\ref{prop:cosine_cap}] 
	By definition, we have
	$\beta_k^{\mathrm c}(u)=\sqrt2\int_0^1\mu(u,t)\cos(\lambda_k t)\,\mathrm{d}t$.
	Set here $G(u):=\sup_{t\in[0,1]}|\partial_t^2\mu(u,t)|$, so that
	$\|G\|_{L^2(\mathcal U)}^2=\int_{\mathcal U}G^2(u)\,\mathrm{d}\nu(u)<\infty$ by hypothesis, and write
	$c_-:=\|\partial_t\mu(\cdot,0)-\partial_t\mu(\cdot,1)\|_{L^2(\mathcal U)}>0$,
	$c_+:=\|\partial_t\mu(\cdot,0)+\partial_t\mu(\cdot,1)\|_{L^2(\mathcal U)}>0$.
	
	Fix $u$ with $\mu(u,\cdot)\in C^2([0,1])$. Integrating by parts we get 
	\[
	\int_0^1\mu(u,t)\cos(\lambda_k t)\,\mathrm{d}t
	=-\frac{1}{\lambda_k}\int_0^1\partial_t\mu(u,t)\sin(\lambda_k t)\,\mathrm{d}t .
	\]
	Integrating by parts a second time and using $\cos(\lambda_k)=(-1)^{k-1}$, $\cos(0)=1$,
	\begin{multline}
		\int_0^1\partial_t\mu(u,t)\sin(\lambda_k t)\,\mathrm{d}t
		=\frac{1}{\lambda_k}\bigl\{\partial_t\mu(u,0)-(-1)^{k-1}\partial_t\mu(u,1)\bigr\}
		\\ +\frac{1}{\lambda_k}\int_0^1\partial_t^2\mu(u,t)\cos(\lambda_k t)\,\mathrm{d}t .	
	\end{multline}
	
	Combining the two displays,
	\begin{multline}
		\int_0^1\mu(u,t)\cos(\lambda_k t)\,\mathrm{d}t
		\\ =\frac{1}{\lambda_k^{2}}\bigl\{(-1)^{k-1}\partial_t\mu(u,1)-\partial_t\mu(u,0)\bigr\}
		-\frac{1}{\lambda_k^{2}}\int_0^1\partial_t^2\mu(u,t)\cos(\lambda_k t)\,\mathrm{d}t
		\\ =\frac{1}{\lambda_k^{2}}\bigl\{\mathfrak D_k(u)+R_k(u)\bigr\},
	\end{multline}
	and multiplying by $\sqrt2$ yields
	$\beta_k^{\mathrm c}(u)=\sqrt2 \lambda_k^{-2}\{\mathfrak D_k(u)+R_k(u)\}$.
	
	Next, we show that the remainder is negligible. For a.e. $u$, $\partial_t^2\mu(u,\cdot)\in C([0,1])\subset L^2([0,1])$, so necessarily 
	$R_k(u)\to0$ as $k\to\infty$. Moreover
	$|R_k(u)|\le\int_0^1|\partial_t^2\mu(u,t)|\,\mathrm{d}t\le G(u)$, and
	dominated convergence then gives
	$
	\|R_k\|_{L^2(\mathcal U)}^2\rightarrow0 .
	$
	In particular $\|R_k\|_{L^2(\mathcal U)}\le\|G\|_{L^2(\mathcal U)}=:C_G$ for all $k$.
	Since $\partial_t\mu(u,1)-\partial_t\mu(u,0)=\int_0^1\partial_t^2\mu(u,t)\,\mathrm{d}t$ is bounded by
	$G(u)$, the difference lies in $L^2(\mathcal U)$. Together with the finiteness of $c_+$
	this yields $\partial_t\mu(\cdot,0),\partial_t\mu(\cdot,1)\in L^2(\mathcal U)$, so all norms below are finite.
	Next, by definition of $\mathfrak D_k$ and $(-1)^{k-1}$,
	\[
	\|\mathfrak D_k\|_{L^2(\mathcal U)}=
	\begin{cases}
		\|\partial_t\mu(\cdot,0)-\partial_t\mu(\cdot,1)\|_{L^2(\mathcal U)}=c_-, & k\ \text{odd},\\[2pt]
		\|\partial_t\mu(\cdot,0)+\partial_t\mu(\cdot,1)\|_{L^2(\mathcal U)}=c_+, & k\ \text{even}.
	\end{cases}
	\]
	Set $\underline c:=\min(c_-,c_+)>0$ and $\overline C:=\max(c_-,c_+)<\infty$. Choose $K_0$ with
	$\|R_k\|_{L^2(\mathcal U)}\le \underline c/2$ for $k\ge K_0$ (possible from above). The triangle inequality gives,
	for $k\ge K_0$,
	\begin{equation}\label{lower_bb}	
		\frac{\underline c}{2}\le\|\mathfrak D_k\|_{L^2(\mathcal U)}-\|R_k\|_{L^2(\mathcal U)}
		\le\|\mathfrak D_k+R_k\|_{L^2(\mathcal U)}
		\le\|\mathfrak D_k\|_{L^2(\mathcal U)}+\|R_k\|_{L^2(\mathcal U)}\le \overline C+C_G .
	\end{equation}
	
	Hence, since $\lambda_k=(k-1)\pi\asymp k$,
	\[
	\|\beta_k^{\mathrm c}\|_{L^2(\mathcal U)}
	=\frac{\sqrt2}{\lambda_k^{2}}\,\|\mathfrak D_k+R_k\|_{L^2(\mathcal U)}\asymp k^{-2}.
	\]
	
	Finally, Assumption~\ref{ass:decay} amounts to $\sum_{k\ge2}\lambda_k^{2s}\|\beta_k^{\mathrm c}\|_{L^2(\mathcal U)}^2<\infty$.
	From~\eqref{lower_bb}, this series is comparable to $\sum_{k\ge2}k^{2s}k^{-4}=\sum_{k\ge2}k^{2s-4}$, which converges if and only if $s<3/2$. Thus $C^2$-regularity guarantees Assumption~\ref{ass:decay} for every $s<3/2$. Moreover, because the lower bound in~\eqref{lower_bb} holds for all large $k$ regardless of any additional smoothness of $\mu$, the series diverges for every $s\ge3/2$, so Assumption~\ref{ass:decay} fails there irrespective of higher smoothness of the mean with respect to $t$.
\end{proof}

\begin{remark}\label{rem:augmented_cap}
	A similar computation, which we omit, shows what the augmentation achieves.
	Augmenting the half-cosine basis by $t$ and $t^2$ and orthonormalizing the
	enlarged family absorbs the two endpoint first derivatives, so the $k^{-2}$ term
	in $\mathfrak D_k$ vanishes and the leading obstruction moves one integration-by-parts pair
	further, to the third endpoint derivatives. Consequently, if
	$\mu(u,\cdot)\in C^4([0,1])$ with
	$\int_{\mathcal U}\sup_t|\partial_t^4\mu(u,t)|^2\,\mathrm{d}\nu(u)<\infty$, the
	augmented-basis coefficients satisfy $\|\beta_k\|_{L^2(\mathcal U)}\lesssim k^{-4}$,
	and Assumption~\ref{ass:decay} holds for every $s<7/2$ --- against $s<3/2$ for the
	simple basis. More generally, for any even positive integer $r$, augmenting by $t,\dots,t^{r}$ (with
	$\partial_t^{r+2}\mu$ square-integrable over $\mathcal U$) raises the threshold to
	$s<r+3/2$, provided $\mu(u,\cdot)\in C^{r+2}([0,1])$.
\end{remark}

\medskip

Now, focusing on $\mu(u,\cdot)\in C^{2}([0,1])$, we show how simply augmenting the cosine basis allows one to raise the threshold of $s$ and to achieve optimal rates for twice continuously differentiable mean functions. Before that, let us introduce some notation. 
For $f\in L^2([0,1])$ and $k\ge2$, write $\gamma_k(f)=\int_0^1 f(t)c_k(t) \,\mathrm dt$
with $\lambda_k=(k-1)\pi$, and $\gamma_1(f)=\int_0^1 f$. Note that, using integration by parts, for $k\ge2$, we get the following $\gamma_k(f)$ with $f=1, t$ or $t^2$:
\begin{equation}\label{eq:mono_coef}
	\gamma_k(1)=0,\qquad
	\gamma_k(t)=\sqrt2\,\frac{(-1)^{k-1}-1}{\lambda_k^{2}},\qquad
	\gamma_k(t^2)=2\sqrt2\,\frac{(-1)^{k-1}}{\lambda_k^{2}} .
\end{equation}
Thus $\gamma_k(t)=0$ for odd $k$, whereas $\gamma_k(t^2)\ne0$ for all $k\ge2$.

For $r\in\{0,1,2\}$ let $\Pi_r$
denote the polynomials of degree $\le r$, and let $\{\phi_k^{(r)}\}_{k\ge1}$ be the
orthonormal basis of $L^2([0,1])$ obtained by Gram-Schmidt from the ordered system
$\bigl(1,t,\dots,t^{r},\ c_2,\ c_3,\dots\bigr)$, so that $\phi_k^{(0)}=c_k$;
set $\beta_k^{(r)}(u)=\langle\mu(u,\cdot),\phi_k^{(r)}\rangle$ and
$\mu_0'=\partial_t\mu(\cdot,0)$, $\mu_1'=\partial_t\mu(\cdot,1)$.

\begin{proposition}[Endpoint conditions for $s=2$ under augmentation]\label{prop:endpoint_landscape} 
	Suppose $\mu(u,\cdot)\in C^2([0,1])$ for a.e.\ $u$ and
	$\int_{\mathcal U}\sup_t|\partial_t^{2}\mu(u,t)|^2\,\mathrm d\nu(u)<\infty$
	(so that $\mu_0',\mu_1'\in L^2(\mathcal U)$). For each $r\in\{0,1,2\}$,
	Assumption~\ref{ass:decay} holds with $s=2$ (and $s^\star\ge 2$), i.e.
	\[
	\sum_{k\ge1} k^{4}\,\|\beta_k^{(r)}\|_{L^2(\mathcal U)}^2<\infty,
	\]
	if and only if the corresponding endpoint condition holds:
	\begin{enumerate}
		\item[\textup{(i)}] $r=0$ \textup{(no augmentation):} $\mu_0'=\mu_1'=0$ in $L^2(\mathcal U)$;
		\item[\textup{(ii)}] $r=1$ \textup{(augmentation by $t$):} $\mu_0'=\mu_1'$ in $L^2(\mathcal U)$;
		\item[\textup{(iii)}] $r=2$ \textup{(augmentation by $t,t^2$):} no condition.
	\end{enumerate}
	Moreover, whenever the relevant condition fails,
	$\sum_{k}k^{2s}\|\beta_k^{(r)}\|_{L^2(\mathcal U)}^2=\infty$ for every $s\ge3/2$,
	irrespective of any higher smoothness of $\mu$. In particular, $r=0$ recovers
	Proposition~\ref{prop:cosine_cap}.
\end{proposition}

\begin{proof}
	Fix $r$ and simplify $\phi_k=\phi_k^{(r)}$, $\beta_k=\beta_k^{(r)}$, and let
	$V_N=\operatorname{span}\{\phi_1,\dots,\phi_N\}$. We write $C^2=C^2([0,1])$. Since Gram-Schmidt preserves partial
	spans, $V_N=\Pi_r\oplus\operatorname{span}\{c_2,\dots,c_{N-r}\}$
	for $N\ge r+1$ (recall $c_1\in\Pi_0$).

	As $\{\phi_k\}$ is orthonormal, $\sum_{k>N}|\beta_k(u)|^2=\operatorname{dist}_{L^2([0,1])}(\mu(u,\cdot),V_N)^2$
	for a.e.\ $u$. Put $a_k=\|\beta_k\|_{L^2(\mathcal U)}^2$ and
	$\mathcal D_N=\sum_{k>N}a_k=\int_{\mathcal U}\operatorname{dist}(\mu(u,\cdot),V_N)^2\,\mathrm d\nu(u)$. Then, for every $ \zeta>0$, we can write
	\begin{equation}\label{eq:au_step1a}
		\sum_{N\ge1}N^{2 \zeta-1}\mathcal D_N=\sum_{k\ge2}\Bigl(\textstyle\sum_{N=1}^{k-1}N^{2 \zeta-1}\Bigr)a_k,
		\qquad \sum_{N=1}^{k-1}N^{2 \zeta-1}\asymp k^{2 \zeta},
	\end{equation}
	and hence 
	\begin{equation}\label{eq:au_step1b}
		\sum_k k^{2 \zeta}a_k<\infty\iff\sum_N N^{2 \zeta-1}\mathcal D_N<\infty.
	\end{equation}
	Next, expanding in the complete cosine basis and minimizing over the $\cos$-block and over
	$q\in\Pi_r$ (the constant of $q$ annihilates $\gamma_1$, leaving the tail), we get
	\begin{equation}\label{eq:dist_inf}
		\operatorname{dist}(\mu(u,\cdot),V_N)^2=\inf_{q\in\Pi_r}\sum_{k>N-r}\bigl|\gamma_k(\mu(u,\cdot)-q)\bigr|^2 .
	\end{equation}
	
	\emph{Sufficiency.} Suppose there is $q^*\in\Pi_r$, with coefficient functions in $L^2(\mathcal U)$, such that
	$g^*:=\mu-q^*$ satisfies $\partial_t g^*(u,0)=\partial_t g^*(u,1)=0$ a.e.\ Taking $q=q^*$ in \eqref{eq:dist_inf}, $\mathcal D_N\le\sum_{k>N-r}\|\gamma_k(g^*)\|_{L^2(\mathcal U)}^2$.
	Since $g^*(u,\cdot)\in C^2$ has vanishing endpoint derivatives, two integrations by parts give
	$\gamma_k(g^*)=-\lambda_k^{-2}\langle\partial_t^2 g^*(u,\cdot),c_k\rangle$, so by Bessel's inequality
	\[
	\sum_{k}\lambda_k^{4}\,\|\gamma_k(g^*)\|_{L^2(\mathcal U)}^2
	=\sum_k\bigl\|\langle\partial_t^2 g^*,c_k\rangle\bigr\|_{L^2(\mathcal U)}^2
	\le \|\partial_t^2 g^*\|_{L^2(\mathcal T\times\mathcal U)}^2<\infty .
	\]
	With $\lambda_k\asymp k$ and~\eqref{eq:au_step1a} applied to $\{\gamma_k(g^*)\}$ (with $ \zeta=2$),
	$\sum_N N^3\mathcal D_N\lesssim\sum_k k^4\|\gamma_k(g^*)\|^2<\infty$; by~\eqref{eq:au_step1b},
	$\sum_k k^4 a_k<\infty$.
	
	\emph{Necessity.} Suppose no $q\in\Pi_r$ makes both endpoint derivatives of $\mu-q$ vanish. From above, there is   $\mathcal P\subseteq\{k\ge2\}$ on which $\gamma_k(q)=0$ for every $q\in\Pi_r$ and on 	which the obstruction persists:
	for $r=0$, $\gamma_k(q)=0$ for all $k\ge2$, and $\mathcal P$ is a parity with
	$\|\mathfrak D_k\|_{L^2(\mathcal U)}>0$ (even carries $\|\mu_0'+\mu_1'\|$, odd carries
	$\|\mu_1'-\mu_0'\|$, not both zero); for $r=1$, $\gamma_k(1)=\gamma_k(t)=0$ on odd $k$, so
	$\mathcal P=\{k\ \text{odd}\}$ and $\|\mathfrak D_k\|=\|\mu_1'-\mu_0'\|>0$ there. On $\mathcal P$,
	$\gamma_k(\mu-q)=\gamma_k(\mu)$ for all $q\in\Pi_r$, so \eqref{eq:dist_inf} yields
	\[
	\mathcal D_N\ \ge\ \sum_{k>N-r,\ k\in \mathcal P}\|\gamma_k(\mu)\|_{L^2(\mathcal U)}^2 .
	\]
	By Proposition~\ref{prop:cosine_cap}, $\gamma_k(\mu)=\sqrt2\,\lambda_k^{-2}(\mathfrak D_k+R_k)$ with
	$\|R_k\|\to0$ and $\|\mathfrak D_k\|_{L^2(\mathcal U)}=c_{\mathcal P}>0$ on $\mathcal P$; hence
	$\|\gamma_k(\mu)\|\ge  (c_{\mathcal P}/\sqrt2)\,\lambda_k^{-2}$ for large $k\in \mathcal P$, so 	$\mathcal D_N\gtrsim\sum_{k>N,\,k\in {\mathcal P} }\lambda_k^{-4}\asymp N^{-3}$. By~\eqref{eq:au_step1b}, for every
	$s\ge3/2$, $\sum_k k^{2s}a_k\asymp\sum_N N^{2s-1}\mathcal D_N\gtrsim\sum_N N^{2s-4}=\infty$.
	(For $r=2$ this case is vacuous: $q^*$ always exists, see below.)
	
	It remains to translate the findings in endpoint conditions. A suitable $q^*$ exists iff $(\mu_0',\mu_1')\in\mathcal R_r:=\{(\partial_t q(0),\partial_t q(1)):q\in\Pi_r\}$.
	Since $\partial_t(\mathfrak c_0)=0$, $\partial_t(\mathfrak c_0+\mathfrak c_1t)\equiv \mathfrak c_1$, and
	$\partial_t(\mathfrak c_0+ \mathfrak c_1t+ \mathfrak c_2t^2)=\mathfrak c_1+2\mathfrak c_2t$, one has $\mathcal R_0=\{(0,0)\}$,
	$\mathcal R_1=\{(\mathfrak c,\mathfrak c):\mathfrak c\in\mathbb R\}$, $\mathcal R_2=\mathbb R^2$; in $L^2(\mathcal U)$ this
	reads $\mu_0'=\mu_1'=0$ for $r=0$, $\mu_0'=\mu_1'$ for $r=1$, and no condition for $r=2$
	(take $\mathfrak c_1=\mu_0'$, $\mathfrak c_2=(\mu_1'-\mu_0')/2\in L^2(\mathcal U)$). Combining this with the \emph{Sufficiency} and \emph{Necessity} steps proves the equivalence and the $s\ge3/2$ sharpness.
\end{proof}

\subsubsection{Closed-form expressions for augmented cosine basis}\label{app_sec:basis2}

Before providing the closed-form expressions of the augmented cosine basis, we display basis elements' shapes in Figure~\ref{fig:augbasis}, as well as their uniform norms as $k$ grows.

\begin{figure}[htbp]
	\centering
	\includegraphics[width=0.45\textwidth]{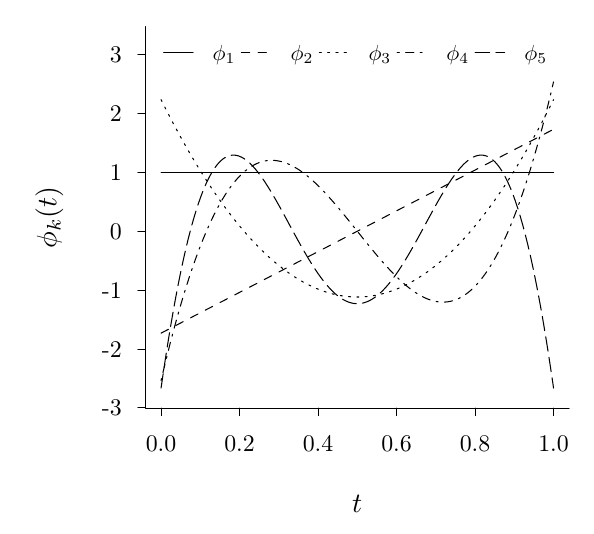}	\includegraphics[width=0.45\textwidth]{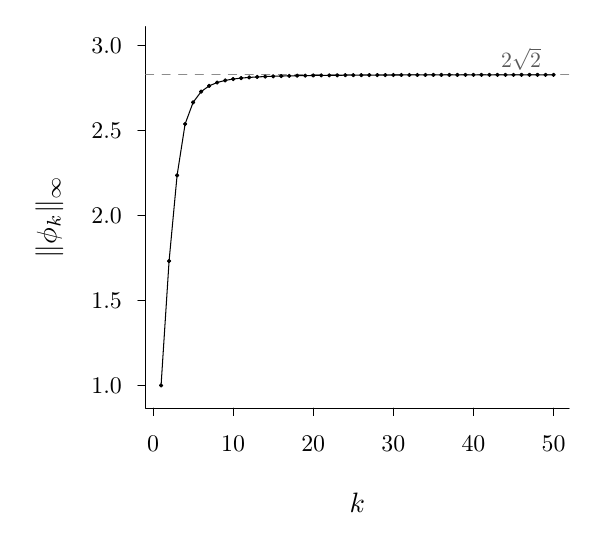}
	\caption{\small Augmented basis functions $\phi_k$ and their sup-norms $\|\phi_k\|_\infty$}
	\label{fig:augbasis}
\end{figure}

Let $\{\phi_k\}=\mathcal B_{2,K}$ be the augmented basis of Section~\ref{subsec:aug} and, for
$k\geq4$, let $d=d(k)=\lfloor(k-2)/2\rfloor$ and $\varpi_\ell$, $1\leq\ell\leq d$, be as in Section~\ref{sec:basis_choice}, that is $\varpi_\ell=(2\ell-1)\pi$ for even $k$ and
$\varpi_\ell=2\ell\pi$ for odd $k$. The Gram-Schmidt procedure puts $\phi_k$ in the form of a polynomial
$P_k\in\Pi_2$ plus $d$ same-parity cosines,
\[
\phi_k(t) = P_k(t) + \sqrt2\sum_{\ell=1}^{d}\theta_{k,\ell}\cos(\varpi_\ell t) .
\]
A rank-one recursion resolves the coefficients. With the projection coefficients and their
partial sums, for $i\geq0$ (note the slight abuse of notation: in this section $i$ is no longer used to represent a subject in the longitudinal data sample), that are
\begin{equation}\label{eq:rec_ac}
	a_\ell := \langle c_{2\ell},\phi_2\rangle = -\frac{4\sqrt6}{(2\ell-1)^2\pi^2},
	\quad
	A_i := 1-\sum_{\ell=1}^{i}a_\ell^2,
\end{equation}
and
\begin{equation}\label{eq:rec_bc}
	b_\ell := \langle c_{2\ell+1},\phi_3\rangle = \frac{3\sqrt{10}}{\ell^2\pi^2},
	\quad
	B_i := 1-\sum_{\ell=1}^{i}b_\ell^2,
\end{equation}
(so $A_0=B_0=1$), one has for $k\geq4$ that the polynomial part is a scalar multiple of $\phi_2$
(even $k$) or $\phi_3$ (odd $k$): for even $k$,
\[
P_k = -\frac{a_d}{\sqrt{A_{d-1}A_d}}\,\phi_2,
\qquad
\theta_{k,d}=\sqrt{\frac{A_{d-1}}{A_d}},
\qquad
\theta_{k,\ell}=\frac{a_d\,a_\ell}{\sqrt{A_{d-1}A_d}}\ \ (1\le \ell\le d-1),
\]
and for odd $k$ the same with $\phi_3,\,b,\,B$ in place of $\phi_2,\,a,\,A$. We prove these
formulae by induction on $d$.

\medskip


\begin{proof}[Proof of the closed form of $\phi_k$]
	Let $k$ be even, so that $k=2d+2$ and $\phi_k$ orthonormalizes $c_{k-2}=c_{2d}$; the odd case
	is identical with $(\phi_3,b_\ell,B_i,c_{2\ell+1})$ in place of $(\phi_2,a_\ell,A_i,c_{2\ell})$.
	Since $c_k(t)=\sqrt2\cos((k-1)\pi t)$ is symmetric about $t=1/2$ for odd $k$ and antisymmetric
	for even $k$, and since $\phi_1,\phi_3$ are symmetric while $\phi_2$ is antisymmetric, the
	function $c_{2d}$ is orthogonal to $\phi_1$, to $\phi_3$ and to every previously built
	$\phi_{k'}$ with $k'$ odd; the Gram-Schmidt step of Section~\ref{subsec:aug} therefore reduces
	to
	\begin{equation}\label{eq:gs_even}
		\psi_{k}=c_{2d}-\langle c_{2d},\phi_2\rangle\,\phi_2
		-\sum_{i=1}^{d-1}\langle c_{2d},\phi_{2i+2}\rangle\,\phi_{2i+2},
	\end{equation}
	the functions $\phi_{2i+2}$, $1\leq i\leq d-1$, being the antisymmetric ones already
	constructed. All the quantities $A_i$ are positive: $A_0=1$, and since $\psi_k\neq0$ by the
	linear independence of the dictionary \eqref{eq:orth_system}, the identity
	$\|\psi_{k}\|^2=A_d/A_{d-1}$ obtained at each step shows inductively that $A_d>0$ whenever
	$A_{d-1}>0$.
	
	\noindent\textbf{Initialization ($d=1$).} The sum in \eqref{eq:gs_even} is empty and
	$\langle c_2,\phi_2\rangle=a_1$, so $\psi_4=c_2-a_1\phi_2$ and, as $\|\phi_2\|=1$,
	$\|\psi_4\|^2=1-a_1^2=A_1$. Hence $\phi_4=(c_2-a_1\phi_2)/\sqrt{A_1}$, which is the announced
	form with $A_0=1$, $\theta_{4,1}=\sqrt{A_0/A_1}$ and $P_4=-a_1\phi_2/\sqrt{A_0A_1}$.
	
	\noindent\textbf{Induction step.} Assume the announced form for $\phi_{2i+2}$,
	$1\leq i\leq d-1$. For such an $i$ one has $\langle c_{2d},c_{2\ell}\rangle=0$ for every
	$\ell\leq i<d$, so only the $\phi_2$-part of $\phi_{2i+2}$ contributes and
	\[
	\langle c_{2d},\phi_{2i+2}\rangle
	=-\frac{a_i}{\sqrt{A_{i-1}A_i}}\,\langle c_{2d},\phi_2\rangle
	=-\frac{a_i\,a_d}{\sqrt{A_{i-1}A_i}} .
	\]
	Substituting this and $\langle c_{2d},\phi_2\rangle=a_d$ into \eqref{eq:gs_even}, and using
	again the announced form of each $\phi_{2i+2}$, the coefficient of $\phi_2$ in $\psi_{k}$
	equals $-a_d\bigl(1+\sum_{i<d}a_i^2/(A_{i-1}A_i)\bigr)$, while that of $c_{2\ell}$,
	$1\leq\ell\leq d-1$, equals
	$a_d\,a_\ell\bigl(A_\ell^{-1}+\sum_{\ell<i<d}a_i^2/(A_{i-1}A_i)\bigr)$. Both collapse through
	the telescoping identity $a_i^2/(A_{i-1}A_i)=A_i^{-1}-A_{i-1}^{-1}$, which follows from
	$A_{i-1}-A_i=a_i^2$, and give respectively $-a_d/A_{d-1}$ and $a_d\,a_\ell/A_{d-1}$, that is,
	\begin{equation}\label{eq:psi_d}
		\psi_{k}=c_{2d}-\frac{a_d}{A_{d-1}}\,\mathfrak q_d,
		\qquad
		\mathfrak q_d:=\phi_2-\sum_{\ell<d}a_\ell\,c_{2\ell} .
	\end{equation}
	Since $\|\mathfrak q_d\|^2=1-\sum_{\ell<d}a_\ell^2=A_{d-1}$ and
	$\langle c_{2d},\mathfrak q_d\rangle=a_d$, we get $\|\psi_{k}\|^2=1-a_d^2/A_{d-1}=A_d/A_{d-1}$,
	and dividing \eqref{eq:psi_d} by this norm yields
	\[
	\phi_{k}
	=-\frac{a_d}{\sqrt{A_{d-1}A_d}}\,\phi_2
	+\sqrt{\frac{A_{d-1}}{A_d}}\;c_{2d}
	+\frac{a_d}{\sqrt{A_{d-1}A_d}}\sum_{\ell<d}a_\ell\,c_{2\ell},
	\]
	which is the announced form, since $\sqrt2\cos(\varpi_\ell t)=c_{2\ell}$ for even $k$.
\end{proof}

\medskip

Being antisymmetric and orthogonal to $c_1$, $\phi_2$ expands as
$\phi_2=\sum_{\ell\geq1}a_\ell c_{2\ell}$ in $L^2(\mathcal T)$, the series converging absolutely
and uniformly on $\mathcal T$ since $|a_\ell|=O(\ell^{-2})$ and $\|c_{2\ell}\|_\infty=\sqrt2$; by
Parseval, $\sum_\ell a_\ell^2=1$ and $A_i=\sum_{\ell>i}a_\ell^2$. Hence
$\mathfrak q_d=\sum_{\ell\geq d}a_\ell c_{2\ell}$ in \eqref{eq:psi_d}, and the closed form admits
the equivalent tail form, for even $k$,
\begin{equation}\label{eq:phi_tail}
	\phi_{k}=\sqrt{\frac{A_d}{A_{d-1}}}\;c_{2d}
	-\frac{a_d}{\sqrt{A_{d-1}A_d}}\sum_{\ell>d}a_\ell\,c_{2\ell},
\end{equation}
and likewise for odd $k$, with $\phi_3=\sum_{\ell\geq1}b_\ell c_{2\ell+1}$ and
$B_i=\sum_{\ell>i}b_\ell^2$. The finite form is the one used for computations, the tail form
\eqref{eq:phi_tail} the one used in the proof of Lemma~\ref{lem:uniform_bound} below.

\medskip

Consider now the augmentation by the single monomial $t$, \emph{i.e.,} 
$\mathcal D_{1,K}\!=(c_1,t,c_2,\ldots,c_K)$, whose antisymmetric part is again
$(\phi_2,c_2,c_4,\ldots)$ while its symmetric part $(c_1,c_3,c_5,\ldots)$ is already orthonormal.
Hence, with the notation $\{\phi^{(r)}_k\}$ of Supplement~\ref{app_sec:basis} (so that
$\phi^{(2)}_k=\phi_k$), one has $\phi^{(1)}_1=c_1$, $\phi^{(1)}_2=\phi_2$ and, for $\ell\geq1$,
\begin{equation}\label{eq:aug_t} 
	\phi^{(1)}_{2\ell+1}=\phi_{2\ell+2},
	\qquad
	\phi^{(1)}_{2\ell+2}=c_{2\ell+1}=\sqrt2\cos(2\ell\pi t) :
\end{equation}
the single monomial $t$ modifies only the parity subsequence it can reach, the even-frequency
cosines being left untouched. This basis is the relevant one when the endpoint slopes agree
($\Delta=0$ in Table~\ref{table:clear}); it then also achieves $s^\star\ge 2$, with one basis
function fewer at every truncation, and \eqref{eq:aug_t} shows that it requires no separate
implementation.

Finally, we will show that the augmented basis preserves the rates of the uniform norms for the basis elements and their derivatives.

\medskip


\begin{lemma}[Uniform boundedness of the augmented bases]\label{lem:uniform_bound}
	For $r\in\{1,2\}$ and every $k\geq1$, it holds
	\[
	\|\phi^{(r)}_k\|_\infty<4\sqrt2
	\qquad\text{and}\qquad
	\|(\phi^{(r)}_k)^\prime\|_\infty\lesssim k .
	\]
\end{lemma}

\begin{proof}[Proof of Lemma~\ref{lem:uniform_bound}]
	By \eqref{eq:aug_t}, each $\phi^{(1)}_k$ is either $c_1$, or a cosine $c_{k-1}$, for which
	$\|c_{k-1}\|_\infty=\sqrt2$ and $\|c_{k-1}^\prime\|_\infty\leq\sqrt2\pi k$, or one of the
	$\phi_k$; only $r=2$ has thus to be treated. Moreover $\|\phi_1\|_\infty=1$,
	$\|\phi_2\|_\infty=\sqrt3$, $\|\phi_3\|_\infty=\sqrt5$, with bounded derivatives, so we take
	$k\geq4$ and set $z:=(k-1)/2\geq3/2$ and $\mathfrak s_p(x):=\sum_{n\geq0}(x+n)^{-p}$,
	$p\in\{2,4\}$. We write the proof for even $k$; the odd case is identical with
	$(b_\ell,B_i)$ in place of $(a_\ell,A_i)$.

	Concerning the uniform norm, all the $a_\ell$ share the same sign, so \eqref{eq:phi_tail} reads
	$\phi_k=\theta_{k,d}^{-1}c_{2d}-\sum_{\ell>d}\varsigma_\ell c_{2\ell}$ with
	$\varsigma_\ell:=|a_d||a_\ell|/\sqrt{A_{d-1}A_d}>0$, whence, as $\|c_j\|_\infty=\sqrt2$ for
	$j\geq2$,
	\[
	\|\phi_k\|_\infty\leq\sqrt2\,\bigl(\theta_{k,d}^{-1}+\Theta\bigr),
	\qquad \Theta:=\sum_{\ell>d}\varsigma_\ell .
	\]
	In both parities $|a_\ell|=\mathfrak a/(\pi^2x_\ell^2)$, with $x_\ell=(2\ell-1)/2$ and
	$\mathfrak a=\sqrt6$ (resp.\ $x_\ell=\ell$ and $\mathfrak a=3\sqrt{10}$), and in both cases
	$x_d=z-1$ and $x_{d+1+n}=z+n$ for $n\geq0$. Since $A_i=\sum_{\ell>i}a_\ell^2$,
	\[
	A_d=\frac{\mathfrak a^2}{\pi^4}\mathfrak s_4(z),
	\quad
	A_{d-1}=\frac{\mathfrak a^2}{\pi^4}\mathfrak s_4(z-1),
	\quad
	|a_d|=\frac{\mathfrak a}{\pi^2(z-1)^{2}},
	\quad
	\sum_{\ell>d}|a_\ell|=\frac{\mathfrak a}{\pi^2}\mathfrak s_2(z),
	\]
	so that $\mathfrak a$ and $\pi$ cancel and both parities give
	\[
	\theta_{k,d}^{-1}=\sqrt{\frac{\mathfrak s_4(z)}{\mathfrak s_4(z-1)}},
	\qquad
	\Theta=\frac{\mathfrak s_2(z)}{(z-1)^{2}\sqrt{\mathfrak s_4(z)\,\mathfrak s_4(z-1)}} .
	\]
	As $\mathfrak s_4(z-1)=\mathfrak s_4(z)+(z-1)^{-4}$, we get $\theta_{k,d}^{-1}<1$. By
	convexity of $x\mapsto x^{-p}$, the midpoint and trapezoidal comparisons give
	$\mathfrak s_2(x)\leq(x-1/2)^{-1}$ and $\mathfrak s_4(x)\geq(2x+3)/(6x^4)$; applied at $x=z$
	and $x=z-1$, and with $2z=k-1$,
	\[
	\Theta\leq\frac{12z^{2}}{(2z-1)\sqrt{(2z+1)(2z+3)}}
	=\frac{3(k-1)^{2}}{(k-2)\sqrt{k(k+2)}}<3 ,
	\]
	the last inequality because $(k-2)^2k(k+2)-(k-1)^4=2k^3-10k^2+12k-1>0$ for every $k\geq4$
	(substituting $k=4+v$, $v\geq0$, gives $2v^3+14v^2+28v+15$). Hence
	$\|\phi_k\|_\infty<\sqrt2(1+3)=4\sqrt2$.
	
	\noindent\textbf{Derivative.} The series in \eqref{eq:phi_tail} converges pointwise but not
	uniformly, and must not be differentiated term by term (its termwise derivative vanishes at
	$t=0$, whereas $\phi_k^\prime(0)\neq0$ in general); we therefore differentiate the finite
	closed form,
	\[
	\phi_k^\prime=-\frac{a_d}{\sqrt{A_{d-1}A_d}}\,\phi_2^\prime
	+\theta_{k,d}\,c_{2d}^\prime+\sum_{\ell<d}\theta_{k,\ell}\,c_{2\ell}^\prime .
	\]
	First, $|a_d|/\sqrt{A_{d-1}A_d}=\pi^2\Theta/\{\mathfrak a\,\mathfrak s_2(z)\}
	\leq3\pi^2z/\sqrt6$, by $\Theta<3$, $\mathfrak s_2(z)\geq z^{-1}$ and
	$\mathfrak a\geq\sqrt6$, while $\phi_2^\prime\equiv2\sqrt3$ (resp.\
	$\|\phi_3^\prime\|_\infty=6\sqrt5$): the first term is $O(k)$. Second, by
	$\mathfrak s_4(z)\geq(3z^3)^{-1}$,
	$\theta_{k,d}^{2}=1+(z-1)^{-4}/\mathfrak s_4(z)\leq1+3z^{3}/(z-1)^{4}\leq163$ for
	$z\geq3/2$, the map $z\mapsto z^3/(z-1)^4$ being decreasing on $(1,\infty)$; with
	$\|c_{2d}^\prime\|_\infty\leq\sqrt2\pi k$, the second term is $O(k)$. Third,
	$\theta_{k,\ell}=\varsigma_\ell$ and $|a_\ell|(2\ell-1)=4\mathfrak a/\{\pi^2(2\ell-1)\}$, so
	that
	\[
	\Bigl|\sum_{\ell<d}\theta_{k,\ell}\,c_{2\ell}^\prime(t)\Bigr|
	=\sqrt2\pi\,\frac{4\Theta}{\mathfrak s_2(z)}
	\Bigl|\sum_{\ell<d}\frac{\sin\{(2\ell-1)\pi t\}}{2\ell-1}\Bigr|
	\leq12\sqrt2\pi\,z\Bigl(1+\frac{3\pi}{4}\Bigr),
	\]
	using $\Theta<3$, $\mathfrak s_2(z)\geq z^{-1}$ and the uniform bound
	$\bigl|\sum_{\ell\leq N}(2\ell-1)^{-1}\sin\{(2\ell-1)\vartheta\}\bigr|\leq1+3\pi/4$, valid
	for every $N\geq1$ and $\vartheta\in\mathbb R$: by the symmetries of the sum one may take
	$\vartheta\in(0,\pi/2]$, the first $N_0:=\min(N,\lceil1/\vartheta\rceil)$ terms are bounded
	by $N_0\vartheta\leq1+\pi/2$ through $|\sin x|\leq x$, and the remaining ones by $\pi/4$
	through Abel summation, the telescoping identity
	$2\sin\vartheta\,\sin\{(2\ell-1)\vartheta\}=\cos\{(2\ell-2)\vartheta\}-\cos(2\ell\vartheta)$,
	$2N_0+1\geq2/\vartheta$ and Jordan's inequality $\sin\vartheta\geq2\vartheta/\pi$. (For odd
	$k$, $|b_\ell|\,2\ell=2\mathfrak a/(\pi^2\ell)$ and the same argument, run with
	$2\sin(\vartheta/2)\sin(\ell\vartheta)
	=\cos\{(\ell-\frac12)\vartheta\}-\cos\{(\ell+\frac12)\vartheta\}$, gives
	$\bigl|\sum_{\ell\leq N}\ell^{-1}\sin(\ell\vartheta)\bigr|\leq1+2\pi$ and the bound
	$6\sqrt2\pi z(1+2\pi)$.) The third term is thus $O(z)$, and adding the three bounds gives
	$\|\phi_k^\prime\|_\infty\lesssim k$.
\end{proof}

\medskip

Finally, the closed forms above make the spacings coefficients of \eqref{eq:sp_est} explicit. With
$\delta_{l,p}$ and $\Sigma_l(\varpi)$ as in Section~\ref{sec:explicit_spc}, the window primitives
are $\int t^{p}=\delta_{l,p+1}/(p+1)$ and $\int\cos(\varpi t)=\Sigma_l(\varpi)/\varpi$, so that,
integrating $\phi_k$ over $[T_{(l-\mathfrak h)},T_{(l+\mathfrak h)}]$ and using that $P_k$ is a
scalar multiple of $\phi_2$ (even $k$) or of $\phi_3$ (odd $k$), one gets for $k\geq4$
\begin{align*}
	\widehat\Phi_{kl}
	&=-\frac{a_d}{\sqrt{A_{d-1}A_d}}\,\widehat\Phi_{2l}
	+\frac{\sqrt2}{2\mathfrak h}\sum_{\ell=1}^{d}\theta_{k,\ell}\,
	\frac{\Sigma_l(\varpi_\ell)}{\varpi_\ell}
	\qquad (k\ \text{even}),\\
	\widehat\Phi_{kl}
	&=-\frac{b_d}{\sqrt{B_{d-1}B_d}}\,\widehat\Phi_{3l}
	+\frac{\sqrt2}{2\mathfrak h}\sum_{\ell=1}^{d}\theta_{k,\ell}\,
	\frac{\Sigma_l(\varpi_\ell)}{\varpi_\ell}
	\qquad (k\ \text{odd}),
\end{align*}
with $\widehat\Phi_{2l}$ and $\widehat\Phi_{3l}$ as in Section~\ref{sec:basis_choice}, and $A$ and $B$ terms are defined in \eqref{eq:rec_ac} and \eqref{eq:rec_bc}, respectively. For the
basis augmented by single monomial $t$, and writing
$\widehat\Phi^{(1)}_{kl}:=(2\mathfrak h)^{-1}\int_{\mathcal T}\phi^{(1)}_k(t)\chi_l(t)\,\mathrm dt$,
identity \eqref{eq:aug_t} gives at once, for $\ell\geq1$,
\[
\widehat\Phi^{(1)}_{1l}=\frac{\delta_{l,1}}{2\mathfrak h},
\qquad
\widehat\Phi^{(1)}_{2l}=\widehat\Phi_{2l},
\qquad
\widehat\Phi^{(1)}_{2\ell+1,l}=\widehat\Phi_{2\ell+2,l},
\qquad
\widehat\Phi^{(1)}_{2\ell+2,l}=\frac{\sqrt2}{2\mathfrak h}\,\frac{\Sigma_l(2\ell\pi)}{2\ell\pi} .
\]
In either basis, no numerical quadrature is involved, and increasing the truncation by one unit
costs a single additional coefficient per window. 

Note that the user friendly calculations of $\widehat\Phi_{kl}$ we propose above involve the quantities $\Sigma_l(\varpi)$, and a very few of them are likely to require the extension of $\phi_k$ outside $[0,1]$ if $T_{(l-\hh)}<0$ or $T_{(l+\hh)}>1$. This extension is not covered by the formal definition of the $\widehat\Phi_{kl}$, but the difference with the theoretical $\widehat\Phi_{kl}$ is clearly negligible.


\subsection{Proofs for testing the mean}\label{app_sec:test}

Herein $\mathcal T =[0,1]$ and, if not stated differently, $\mathcal K_M$ is a set of consecutive indices growing with the sample size. For testing purposes $\min \mathcal K_M$ is fixed, and therefore $|\mathcal K_M|+ \min \mathcal K_M = 1+  \max \mathcal K_M$. Let $\eta(u)$ denote a random field distributed as the standardized noise profiles $\eta_{ij}(u)$ assumed i.i.d. by Assumption~\ref{ass:noise}. In this section $\phi_k$ are elements in an orthonormal system given by the half-cosine functions on $[0,1]$ augmented by the monomials $t,\ldots,t^r$ as defined in Section~\ref{subsec:aug}. Here, the focus is on the cases $0\leq r\leq 2$.

\begin{proof}[Proof of Lemma~\ref{lem:deg_Ustat2}] 
	First, by Assumption~\ref{Ass_gen},
	\[
	\mathbb E\left[\frac{\phi_k(T_{ij})}{g(T_{ij})}\,Y_{ij}(u)\right]
	=\int_{\mathcal T}\mu(u,t)\,\phi_k(t)\,\mathrm{d}t=\beta_k(u).
	\]
	Second, by Assumption~\ref{ass:random_mi} the counts $ m_i$ are independent of
	$\{T_{ij}\},\{X_i\},\{\eta_{ij}\}$, so  we get
	\begin{equation}\label{eq:cond_mean_Lambda}
		\mathbb E[\Lambda_{ik}(\cdot)\mid  m_i]=m_i\,\beta_k(\cdot),
		\quad
		\mathbb E[\Lambda_{ik}(\cdot)]=\mathbb E[\mathfrak m]\,\beta_k (\cdot), \quad\text{where}\ \ 
		\Lambda_{ik}(u)=\sum_{j=1}^{m_i} \frac{\phi_k(T_{ij})}{g(T_{ij})}\,Y_{ij}(u).
	\end{equation}
	Next, since by definition
	\begin{equation}\label{eq:def_eq_HM}
		H_n(\Xi_i,\Xi_{i'})= \sum_{k\in\mathcal K_M} \int_{\mathcal U}\Lambda_{ik}(u) \Lambda_{i'k}(u){\rm d}\nu(u) = \sum_{k\in\mathcal K_M}   \big\langle \Lambda_{ik}(u)  , \Lambda_{i'k}(u) \big\rangle_{L^2(\mathcal U)},
	\end{equation}
	by the independence assumptions, for $i\neq i'$, it holds
	\begin{multline}
		\mathbb E[H_n(\Xi_i,\Xi_{i'})\mid m_1, \ldots, m_n]
		=\sum_{k\in\mathcal K_M}\bigl\langle\mathbb E[\Lambda_{ik}\mid  m_i],\mathbb E[\Lambda_{i'k}\mid  m_{i'}]\bigr\rangle_{L^2(\mathcal U)}
		\\=m_im_{i'}\sum_{k\in\mathcal K_M}\|\beta_k\|_{L^2(\mathcal U)}^2
		=m_im_{i'}\,\mathcal S(\mathcal K_M;\mu).	
	\end{multline}
	The counts being i.i.d., taking expectation yields $\mathbb E[Q_n] = \mathbb E[H_n(\Xi_i,\Xi_{i'})] = \mathbb E [\mathfrak m]^2\mathcal S(\mathcal K_M;\mu)$. 
	
	To check the degenerate $U$-statistic properties, first note that $H_n(\Xi_i,\Xi_{i'})=H_n(\Xi_{i'},\Xi_{i})$, and
	\begin{equation}\label{eq:app_aux5} 
		\mathbb E[H_n(\Xi_i,\Xi_{i'})\mid \Xi_i]= \sum_{k\in\mathcal K_M}\bigl\langle\Lambda_{ik},\mathbb E[\Lambda_{i'k}]\bigr\rangle_{L^2(\mathcal U)}
		=\mathbb E[\mathfrak m]\sum_{k\in\mathcal K_M}\bigl\langle\Lambda_{ik},\beta_k\bigr\rangle_{L^2(\mathcal U)} .
	\end{equation}
	If $\mathcal S(\KM;\mu)=0$, then $\beta_k=0$ for all $k\in\mathcal K_M$, so  $\mathbb E[H_n(\Xi_i,\Xi_{i'})\!\mid\! \Xi_i]=0$ and thus $Q_n$ is degenerate.
\end{proof}

\subsubsection{A new Berry-Esseen bound for $U$-statistics of order two}\label{app:sec_BE_new}

Let $\Xi_1,\dots,\Xi_n$ ($n\ge2$) be our i.i.d. sample. To prove our Gaussian approximation bound, an extension of \cite[Th~2.2]{LSS2025} which has its own interest, we use the following version of \cite[Lemma~4.1]{LSS2025}, \emph{i.e.}, \cite[Corollary~4.3]{Z2022}, with no remainder and contraction rate $1/n$. Write $\bXi=(\Xi_1,\dots,\Xi_n)$ for the sample and $\Dn=\sigma(\Xi_1,\dots,\Xi_n)$.

\begin{lemma}\label{lem:zhang}
	Let $(\bXi,\bXi')$ be an exchangeable pair, and let $W=W(\bXi)$ satisfy $\E W=0$ and $\operatorname{Var} W=1$. Let $\mathcal Z=F(\bXi,\bXi')$ be antisymmetric with $\E\{\mathcal Z\mid\Dn\}=W/n$, and let $\mathcal Z^*$ be a symmetric function of $(\bXi,\bXi')$ with $\mathcal Z^*\ge|\mathcal Z|$. Writing $W'=W(\bXi')$,
	\[
	\sup_z|\mathbb P(W\le z)-\Phi(z)|\le\E\Bigl|1-  (n/2)\,\E\{\mathcal Z(W-W')\mid\Dn\}\Bigr|+n\,\E\bigl|\E\{\mathcal Z^*(W-W')\mid\Dn\}\bigr| .
	\]
	
\end{lemma}

The key point in Lemma~\ref{lem:zhang} is that $\mathcal Z$ need not equal $W-W'$.
Next, let $\Xi$, $\tilde\Xi$ be independent copies of $\Xi_1$, let $g_1$ be a real function with $\E g_1(\Xi)=0$ and $\E g_1^4(\Xi)<\infty$, and $g_2$ be a symmetric kernel with
\[
\E [g_2(\Xi,\xi)]=0\ \text{ for every }\xi,\qquad \E [g_2^4(\Xi,\tilde\Xi)]<\infty .
\]
For constants $c_1,c_2>0$, define
\begin{gather*}
	L=c_1\sum_{i=1}^ng_1(\Xi_i),\qquad Q=c_2\sum_{1\le i<j\le n}g_2(\Xi_i,\Xi_j),\\
	\vv_1=\E [g_1^2(\Xi)],\qquad \vv_2=\E [g_2^2(\Xi,\tilde\Xi)],\qquad
	\vv=c_1^2\,n\,\vv_1+ c_2^2\,\frac{n(n-1)}{2 } \vv_2,\qquad W=\frac{L+Q}{\sqrt{\vv}} .
\end{gather*}
We assume $\vv>0$. Since $\E[g_1(\Xi)g_2(\Xi,\tilde\Xi)]=\E[g_1(\Xi)\E\{g_2(\Xi,\tilde\Xi)\mid\Xi\}]=0$, the terms $L$ and $Q$ are uncorrelated, so $\vv=\operatorname{Var}(L+Q)$, $\E W=0$ and $\operatorname{Var} W=1$. Finally, let
\begin{equation}\label{eq:def_G_tau}
	\mathfrak f(\xi):=\E[g_1(\Xi)g_2(\Xi,\xi)],\qquad G_2(\xi,\xi'):=\E[g_2(\Xi,\xi)g_2(\Xi,\xi')].
\end{equation}

Now, we can state our general Gaussian approximation result. 

\begin{theorem}\label{thm:abstract}
	There exists an absolute constant $C$ such that
	\begin{multline*}
		\sup_{z\in\mathbb R}\bigl|\mathbb P(W\le z)-\Phi(z)\bigr|
		\le C\Bigl\{\frac{1}{\sqrt n}
		+\frac{c_1^2\sqrt{n\,\E g_1^4(\Xi)}}{\vv}
		+\frac{c_1c_2\,n\sqrt{\E[g_1^2(\Xi)g_2^2(\Xi,\tilde\Xi)]}}{\vv}\\
		+\frac{c_1c_2\,n^{3/2}\sqrt{\E\mathfrak f^2(\Xi)}}{\vv}
		+\frac{c_2^2\,n^{2}}{\vv}\sqrt{n^{-1}\E g_2^4(\Xi,\tilde\Xi)+\E G_2^2(\Xi,\tilde\Xi)}\Bigr\}.
	\end{multline*}
\end{theorem}

\smallskip

\begin{proof}[Proof of Theorem~\ref{thm:abstract}]
	\textbf{Step 1: choice of $\mathcal Z$.} Let $I$ be uniform on $\{1,\dots,n\}$ and let $\Xi_I'$ be an independent copy of $\Xi_1$. Let $\bXi'$ be $\bXi$ with $\Xi_I$ replaced by $\Xi_I'$; then $(\bXi,\bXi')$ is exchangeable. Writing $\Xi'$ for $\Xi_I'$, set
	\[
	a_i=c_1\{g_1(\Xi_i)-g_1(\Xi')\},\qquad b_i=c_2\sum_{j\ne i}\{g_2(\Xi_i,\Xi_j)-g_2(\Xi',\Xi_j)\}.
	\]
	Then $L-L'=a_I$, $Q-Q'=b_I$ and $W-W'=(a_I+b_I)/\sqrt\vv$. Averaging over $I$ and $\Xi'$, and using $\E g_1=0$ and $\E g_2(\Xi',\xi)=0$, gives
	\[
	\E\{L-L'\mid\Dn\}=\frac Ln,\qquad \E\{Q-Q'\mid\Dn\}=\frac{2Q}{n}.
	\]
	The two parts contract at different rates, which explains the following definitions: let 
	$$
	U:= \frac{L+Q/2}{\sqrt{\vv}}\qquad  \text{ and } \qquad \mathcal Z:=U(\bXi) - U(\bXi') = \frac{a_I+ b_I/2}{\sqrt\vv},
	$$
	so that $\mathcal Z$ is an antisymmetric function of $(\bXi,\bXi')$ and
	\[
	\E\{\mathcal Z\mid\Dn\}=\frac Wn .
	\]
	Next, let $\mathcal Z^*:=|\mathcal Z|$, which is symmetric.
	
	
	\textbf{Step 2: the term $A:=\E|1- (n/2)\E\{\mathcal Z(W-W')\mid\Dn\}|$.} For a function $\mathfrak F$ of $(a_I,b_I)$,
	\[
	\E\{\mathfrak F(a_I,b_I)\mid\Dn\}=\frac1n\sum_{i=1}^n\E' \mathfrak F(a_i,b_i),
	\]
	where $\E'$ integrates over $\Xi'$ only. Since 
	$$
	\mathcal Z(W-W')=(a+b/2)(a+b)/\vv \quad \text{ and } \quad (a+b/2)(a+b)=a^2+\frac32ab+\frac12b^2,
	$$
	it holds 
	\[
	\frac n2\E\{\mathcal Z(W-W')\mid\Dn\}
	=\underbrace{\frac{1}{2\vv}\sum_i\E'a_i^2}_{(\mathrm{I})}
	+\underbrace{\frac{3}{4\vv}\sum_i\E'[a_ib_i]}_{(\mathrm{II})}
	+\underbrace{\frac{1}{4\vv}\sum_i\E'b_i^2}_{(\mathrm{III})} .
	\]
	Note that the whole expression has mean $1$.   Indeed, by exchangeability and antisymmetry, it holds $\E[\mathcal Z W']=-\E[\mathcal Z W]$, so $\E[\mathcal Z(W-W')]=2\E[W\,\E\{\mathcal Z\mid\Dn\}]=2\E W^2/n=2/n$. Separately, (I) has mean $p_L:=c_1^2n\vv_1/\vv$, (III) has mean $p_Q:=c_2^2\binom n2\vv_2/\vv$, and (II) has mean $0$; note that $p_L+p_Q=1$. We bound the three deviations from these means.
	\begin{enumerate}[label=(\roman*),leftmargin=*]
		\item \emph{Linear part.} $\E'a_i^2=c_1^2\{g_1^2(\Xi_i)+\vv_1\}$, so 
		$$
		(\mathrm I)-p_L=\frac{c_1^2}{2\vv}\sum_i\{g_1^2(\Xi_i)-\vv_1\}.
		$$
		Then, 
		$$
		\E|(\mathrm{I})-p_L|\le c_1^2(2\vv)^{-1} \sqrt{n\E g_1^4} \,.
		$$

		\item \emph{Cross part.} Using $\E' g_1(\Xi')=0$ and $\E' g_2(\Xi',\xi)=0$,
		\[
		\sum_i\E'[a_ib_i]=c_1c_2\Bigl\{\sum_{i\ne j}g_1(\Xi_i)g_2(\Xi_i,\Xi_j)+(n-1)\sum_{j}\mathfrak f(\Xi_j)\Bigr\}.
		\]
		The double sum is a $U$-statistic with kernel $\{g_1(\xi)+g_1(\xi')\}g_2(\xi,\xi')$. This kernel has mean zero, first projection $\mathfrak f$, and second moment at most $4\E[g_1^2(\Xi)g_2^2(\Xi,\tilde\Xi)]$. Its variance is therefore at most
		\[
		2n^2\E[g_1^2g_2^2]+n^3\E\mathfrak f^2 .
		\]
		By Minkowski's inequality,
		\[
		\E|(\mathrm{II})|\le\frac{3c_1c_2}{4\vv}\bigl\{n\sqrt{2\E[g_1^2g_2^2]}+2n^{3/2}\sqrt{\E\mathfrak f^2}\bigr\}.
		\]

		\item \emph{Degenerate part.} Since $(\mathrm{III})/p_Q$ is invariant under rescaling of $Q$,
		it equals the quantity $(2\lambda)^{-1}\E[D_n^2\mid\mathcal K]$ of \cite{LSS2025}, see their
		(4.5)--(4.7), with $m=r=2$ and $\lambda=2/n$. Their Lemma~4.2, with $A_1=2\E G_2^2$ (see
		their proof of Theorem~2.2), bounds its second moment. Then, by the Cauchy-Schwarz inequality,
		\[
		\E|(\mathrm{III})-p_Q|\le p_Q\,\frac{C}{\vv_2}\sqrt{n^{-1}\E g_2^4+\E G_2^2}
		\le\frac{C\,c_2^2n^2}{\vv}\sqrt{n^{-1}\E g_2^4+\E G_2^2}.
		\]
	\end{enumerate}
	
	
	\textbf{Step 3: the term $B:=n\,\E|\E\{|\mathcal Z|(W-W')\mid\Dn\}|$.} Let $\mathfrak q(a,b):=|a+ b/2|(a+b)$, which is an odd function of $(a,b)$, and set
	\[
	\mathfrak t_i:=\E'\mathfrak q(a_i,b_i).
	\]
	Then $B=\vv^{-1}\E|\sum_i\mathfrak t_i|$. The argument rests on two remarks.
	\begin{itemize}[leftmargin=*]
		\item \emph{Leave-one-out antisymmetry.} Condition on $\bXi_{-i}:=(\Xi_j)_{j\ne i}$. Exchanging the i.i.d.\ pair $(\Xi_i,\Xi')$ changes $(a_i,b_i)$ into $(-a_i,-b_i)$. Since $\mathfrak q$ is odd,
		\[
		\E[\mathfrak t_i\mid\bXi_{-i}]=0 .
		\]
		The same holds for $\mathfrak t_i^{(-k)}:=\E'\mathfrak q(a_i,b_i^{(-k)})$, where $b_i^{(-k)}$ omits the index $k$. Moreover, $\mathfrak t_i^{(-k)}$ does not depend on $\Xi_k$.
		\item \emph{Exact covariance identity.} For $i\ne k$, expanding the product and using the previous point (each of the three cross terms has a factor with zero conditional mean) gives
		\[
		\E[\mathfrak t_i\mathfrak t_k]=\E\bigl[(\mathfrak t_i-\mathfrak t_i^{(-k)})(\mathfrak t_k-\mathfrak t_k^{(-i)})\bigr].
		\]
	\end{itemize}
	Next we bound the variance of $\sum_i\mathfrak t_i$. Since $\E\mathfrak t_i=0$,
	\[
	\operatorname{Var}\Bigl(\sum_i\mathfrak t_i\Bigr)=\sum_i\E\mathfrak t_i^2+\sum_{i\ne k}\E[\mathfrak t_i\mathfrak t_k].
	\]
	\emph{Diagonal terms.} By Jensen's inequality for $\E'$ and $|\mathfrak q(a,b)|\le(|a|+|b|)^2$,
	\[
	\E\mathfrak t_i^2\le8(\E a_i^4+\E b_i^4)\le C\big\{c_1^4\E g_1^4+c_2^4(n\E g_2^4+n^2\E\vartheta^2)\big\},
	\]
	where $\vartheta(\xi):=\E g_2^2(\Xi,\xi)$; for $\E b_i^4$ we used that, conditionally on
	$(\Xi_i,\Xi')$, $b_i$ is a sum of $n-1$ independent centered terms.
	
	\emph{Off-diagonal terms.} Write $\mathsf e_{ik}:=c_2\{g_2(\Xi_i,\Xi_k)-g_2(\Xi',\Xi_k)\}=b_i-b_i^{(-k)}$.
	For all real $a,b,\tilde b$ with $b-\tilde b=\mathsf e_{ik}$,
	\[
	|\mathfrak q(a,b)-\mathfrak q(a,\tilde b)|\le\frac32|\mathsf e_{ik}|\,(|a|+|\tilde b|+|\mathsf e_{ik}|).
	\]
	Since $\mathfrak t_i-\mathfrak t_i^{(-k)}=\E'\bigl[\mathfrak q(a_i,b_i)-\mathfrak q(a_i,b_i^{(-k)})\bigr]$,
	Jensen's inequality for $\E'$, followed by the inequality above with
	$(a,b,\tilde b)=(a_i,b_i,b_i^{(-k)})$ (and using $(x+y+z)^2 \le 3(x^2+y^2+z^2)$), gives
	\[
	\E\bigl(\mathfrak t_i-\mathfrak t_i^{(-k)}\bigr)^2
	\le\E\bigl[\bigl\{\mathfrak q(a_i,b_i)-\mathfrak q(a_i,b_i^{(-k)})\bigr\}^2\bigr]
	\le\frac{27}{4}\,\E\bigl[\mathsf e_{ik}^2\bigl(a_i^2+(b_i^{(-k)})^2+\mathsf e_{ik}^2\bigr)\bigr].
	\]
	Conditionally on $(\Xi_i,\Xi',\Xi_k)$, the variable $b_i^{(-k)}$ is a sum of $n-2$ independent
	centered terms. This gives
	\begin{gather*}
		\E[\mathsf e_{ik}^2(b_i^{(-k)})^2]\le16c_2^4n\,\E\vartheta^2,\qquad
		\E[\mathsf e_{ik}^2a_i^2]\le8c_1^2c_2^2\{\E[g_1^2g_2^2]+\vv_1\vv_2\},\qquad
		\E\mathsf e_{ik}^4\le16c_2^4\E g_2^4 .
	\end{gather*}
	By the covariance identity and the Cauchy--Schwarz inequality,
	\[
	\sum_{i\ne k}\bigl|\E[\mathfrak t_i\mathfrak t_k]\bigr|
	\le n^2\max_{i\ne k}\E\bigl(\mathfrak t_i-\mathfrak t_i^{(-k)}\bigr)^2
	\le C\Bigl\{c_2^4\bigl(n^3\E\vartheta^2+n^2\E g_2^4\bigr)+c_1^2c_2^2n^2\bigl(\E[g_1^2g_2^2]+\vv_1\vv_2\bigr)\Bigr\}.
	\]
	Combining the diagonal and off-diagonal terms,
	\[
	\operatorname{Var}\Bigl(\sum_i\mathfrak t_i\Bigr)\le C\Bigl\{c_1^4n\E g_1^4+c_2^4\bigl(n^2\E g_2^4+n^3\E\vartheta^2\bigr)+c_1^2c_2^2n^2\bigl(\E[g_1^2g_2^2]+\vv_1\vv_2\bigr)\Bigr\},
	\]
	and, since $\E\sum_i\mathfrak t_i=0$, $B\le\vv^{-1}\{\operatorname{Var}(\sum_i\mathfrak t_i)\}^{1/2}$.

	
	\textbf{Step 4: collecting the bounds.} Two inequalities let us simplify. On the one hand, 
	by Jensen's inequality, $\E\vartheta^2\le\E g_2^4$. Hence $c_2^2\{n^{3/2}\sqrt{\E\vartheta^2}+n\sqrt{\E g_2^4}\}\le2c_2^2n^2\sqrt{n^{-1}\E g_2^4}$. On the other hand, 
	by the arithmetic-geometric mean inequality, $\vv\ge c_1c_2\{2n^2(n-1)\vv_1\vv_2\}^{1/2}$. Hence $$c_1c_2n\sqrt{\vv_1\vv_2}/\vv\le\{2(n-1)\}^{-1/2}\le n^{-1/2}.$$
	Now, combining Steps~2 and~3 with these inequalities proves the theorem. 
\end{proof} 

\smallskip

\begin{remark}	
	Note that for a rank-one kernel $g_2(\xi,\xi')=\kappa\, g_1(\xi)g_1(\xi')$, one has
	$\E\mathfrak f^2=\vv_1\vv_2$, and $W$ is asymptotically of the form
	$aZ+b(Z^2-1)$ with $Z\sim\mathcal N(0,1)$, hence not Gaussian when $L$ and $Q$
	have comparable variances. Accordingly, the fourth term in the bound of
	Theorem~\ref{thm:abstract} is then of order one.	
\end{remark}

\subsubsection{Preliminary results}

The proof of Theorem~\ref{prop:dist_tstat} is based on Theorem~\ref{thm:abstract}. To check the conditions of our abstract Berry-Esseen bound, we will use several lemmas that we present in the following. For the sake of readability, their proofs are postponed to Supplement~\ref{sec:technic_proof}. First, recall the notation
\begin{equation}\label{eq:not_aux1}
	\Psi^2(t,s;\mathcal K_M) =\sum_{k,k'\in\mathcal K_M}\phi_k(t)\phi_{k'}(t)\phi_k(s)\phi_{k'}(s)=
	\Bigg[\sum_{k\in\mathcal K_M}\phi_k(t)\phi_k(s) \Bigg]^2.
\end{equation}
Note that $\Psi(\cdot,\cdot;\mathcal K_M)$ is the reproducing kernel of the projector
onto $\operatorname{span}\{\phi_k:k\in\mathcal K_M\}$.  In particular, it satisfies
\begin{equation}\label{eq:id_basis_P}
	\int_{\mathcal T}\Psi^2(t,s;\mathcal K_M)\,{\rm d}s= \sum_{k\in\mathcal K_M}\phi^2_k(t).
\end{equation}


\smallskip

\begin{lemma}\label{lem:avg-cosine-L2-full}
	Let $c_1\equiv1$, $c_k(t)=\sqrt2\cos((k-1)\pi t)$, $k\ge2$, be the half-cosine basis of
	Section~\ref{subsec:aug}, let $\mathcal B_{0,K}=\{c_1,\ldots,c_K\}$ and, for $r\in\{1,2\}$,
	let $\mathcal B_{r,K}=\{\phi_1,\ldots,\phi_{K+r}\}$ be the augmented basis of
	Section~\ref{subsec:aug} (with $\phi_k=c_k$ when $r=0$). Let
	$\mathcal K_M\subset\{1,\ldots,K+r\}$ be any nonempty subset.
	
	(i) If $r=0$, then, with $n_1=|\mathcal K_M\cap\{2,3,\ldots\}|$,
	\[
	\Bigl\|\frac 1{|\mathcal K_M|} \!\int_{\mathcal T}\Psi^2(t,s;\mathcal K_M)\,{\rm d}s-1\Bigr\|_{L^2(\mathcal T)}
	=\Bigl\|\frac 1{|\mathcal K_M|}\!\sum_{k\in\mathcal K_M}\phi_k^2-1\Bigr\|_{L^2(\mathcal T)}
	=\frac{\sqrt{n_1/2}}{|\mathcal K_M|}\le\frac{1}{\sqrt{2\,|\mathcal K_M|}},
	\]
	with equality if and only if $1\notin\mathcal K_M$.
	
	(ii) If $r\in\{1,2\}$, then
	\[
	\Bigl\|\,|\mathcal K_M|^{-1}\!\int_{\mathcal T}\Psi^2(t,s;\mathcal K_M)\,{\rm d}s-1\Bigr\|_{L^2(\mathcal T)}
	\le\frac{C}{\sqrt{|\mathcal K_M|}},
	\qquad C=\frac{1}{\sqrt2}+12+20\sqrt6 .
	\]
	
	(iii) If $r\in\{0,1,2\}$, then
	\begin{equation}\label{eq:Psi4a}
		\iint_{\mathcal T\times\mathcal T}\Psi^4(t,s;\mathcal K_M)\,\mathrm dt\,\mathrm ds
		\le|\mathcal K_M|^3\sup_{k\ge1}\|\phi_k\|_\infty^4\le2^{10}|\mathcal K_M|^3 .
	\end{equation}
\end{lemma}

\smallskip

The next result provides a representation of integrals with respect to $\Psi^2(t,s;\mathcal K_M)$. Assuming that $\mathcal K_M$ is a set of \emph{consecutive} indices simplifies the representation. 


\smallskip

\begin{lemma}\label{lem:order-Gh2}
	Let $h_1,h_2\in C({\mathcal T})$ and assume $h_1$ is Hölder continuous with exponent
	$\gamma\in(0,1]$ and Hölder constant $C_1>0$. For a finite set $\mathcal K_M$ and an
	orthonormal system $\{\phi_k:k\in\mathcal K_M\}$, let
	$\mathfrak G(t)=\int_{\mathcal T}\Psi^2(t,s; \mathcal K_M)h_1(s)\,{\rm d}s$ and define
	$\mathcal R_{\mathfrak G}$ by
	\[
	|\mathcal K_M|^{-1}	\int_{\mathcal T} \mathfrak G(t)\,h_2(t)\,{\rm d}t
	=\int_{\mathcal T} h_1(t)h_2(t)\,{\rm d} t+ \mathcal R_{\mathfrak G}.
	\]
	(i) If $\phi_k(t)=\sqrt2\cos(f_k\pi t)$, where $F=\{f_k:k\in\mathcal K_M\}$ is a set of
	$|\mathcal K_M|$ consecutive positive integers, then, for any $\delta\in(0,1]$,
	\[
	|\mathcal R_{\mathfrak G}| \le  \frac{\|h_1h_2\|_{L^2(\mathcal T)}}{\sqrt{2|\mathcal K_M|}}
	+\frac{\|h_2\|_\infty}{|\mathcal K_M|}\Bigg(2C_1|\mathcal K_M|\,\delta^{\gamma}
	+\frac{32\|h_1\|_\infty}{\pi^2\delta}
	+4\|h_1\|_\infty\big(1+2\log|\mathcal K_M|\big)\Bigg).
	\]
	(ii) If $\{\phi_k\}=\mathcal B_{r,K}$, $r\in\{0,1,2\}$, $q=r+1$, and $\mathcal K_M$ is a set of
	consecutive indices with $|\mathcal K_M|\geq 2q$, then, for any $\delta\in(0,1]$,
	\begin{multline*}
		|\mathcal R_{\mathfrak G}| \le  \frac{\|h_1h_2\|_{L^2(\mathcal T)}}{\sqrt{|\mathcal K_M|}}
		+\frac{2\|h_2\|_\infty}{|\mathcal K_M|}\Bigg(C_1|\mathcal K_M|\,\delta^{\gamma}
		+\frac{32\|h_1\|_\infty}{\pi^2\delta}
		+4\|h_1\|_\infty\big(1+2\log|\mathcal K_M|\big)\Bigg)\\
		+\|h_1\|_\infty\|h_2\|_\infty\Bigg(\frac{4\sqrt q}{\sqrt{|\mathcal K_M|}}+\frac{7q}{|\mathcal K_M|}\Bigg).
	\end{multline*}
	In particular, in both cases,
	$$
	|\mathcal R_{\mathfrak G}|\lesssim \|h_2\|_\infty (C_1+\|h_1\|_\infty) |\mathcal K_M|^{-\gamma/(1+\gamma)}.
	$$
\end{lemma}

\medskip 

Before proving the next results, let us note that by definition and the independence assumptions,
\begin{multline}
	\mathbb E [\Lambda_{ik}(u)]  \mathbb E [\Lambda_{i'k}(u)]=\mathbb E \Bigg[ \sum_{j=1}^{m_i} \frac{\phi_k(T_{ij})}{g(T_{ij})}\,Y_{ij}(u)\Bigg]
	\mathbb E \Bigg[ \sum_{j'=1}^{m_{i'}} \frac{\phi_k(T_{i'j'})}{g(T_{i'j'})}\,Y_{i'j'}(u)\Bigg] \\ =\mathbb E [\Lambda_{ik}(u)\Lambda_{i'k}(u)],
\end{multline}
provided $i\neq i'$, and thus, by Fubini's theorem,
\begin{equation}\label{eq:useful2}
	\big\langle \mathbb E [\Lambda_{ik}(u)]  ,\ \mathbb E [\Lambda_{i'k}(u)] \big\rangle_{L^2(\mathcal U)}= \mathbb E  \big\langle \Lambda_{ik}(u)  , \Lambda_{i'k}(u) \big\rangle_{L^2(\mathcal U)}.
\end{equation}
Moreover, 
\begin{multline}\label{eq:useful3}
	\mathbb E [\Lambda_{i'k}(u)]\sum_{j=1}^{m_i} \frac{\phi_k(T_{ij})}{g(T_{ij})}\,Y_{ij}(u)=  \mathbb E \Bigg[\Lambda_{i'k}(u)\sum_{j=1}^{m_i} \frac{\phi_k(T_{ij})}{g(T_{ij})}\,Y_{ij}(u)\mid \Xi_i \Bigg] \\ = \mathbb E \big[\Lambda_{ik}(u) \Lambda_{i'k}(u) \mid \Xi_i \big].
\end{multline}

Recall that in Section~\ref{sec:test}, we introduced the Hilbert space
$\mathcal H_M=\bigoplus_{k\in\mathcal K_M}L^2(\mathcal U)$, with inner product
$\langle a,b\rangle_{\mathcal H_M}=\sum_{k\in\mathcal K_M}\int_{\mathcal U}a_k b_k\,{\rm d}\nu$. Next, let us introduce some more notation:  
$$
\mathfrak e_1 = \mathbb E[\mathfrak m], \quad  \mathfrak e_2 = \mathbb E[\mathfrak m(\mathfrak m-1)],
$$
and the vector of functions of $u$
\begin{equation}
	\widetilde \Lambda_{i} (\cdot)= \big(\widetilde \Lambda_{ik}(\cdot)\big)_{k\in\mathcal K_M}\quad \text{with} \quad \widetilde \Lambda_{ik}(u)=\sum_{j=1}^{m_i} \frac{\phi_k(T_{ij})}{g(T_{ij})}\,Y_{ij}(u) - \underbrace{ \mathbb E [\Lambda_{ik}(u)]}_{\mathfrak e_1 \beta_k(u)},	
\end{equation}
so that $\mathbb E \big[\, \widetilde \Lambda_{i}(\cdot)\big]=0\in\mathcal H_M$, such that
\begin{equation}\label{eq:def_Htilde}
	\widetilde H_n(\Xi_i,\Xi_{i'}) = \big\langle \,\widetilde \Lambda_{i}, \widetilde \Lambda_{i'} \big\rangle_{\mathcal H_M} = \sum_{k\in \mathcal K_M}\big\langle \,\widetilde \Lambda_{ik}, \widetilde \Lambda_{i'k} \big\rangle_{L^2(\mathcal U)}.
\end{equation}
Indeed, we can write
\begin{multline}
	\big\langle \,\widetilde \Lambda_{ik}, \widetilde \Lambda_{i'k} \big\rangle_{L^2(\mathcal U)} = \!\int_{\mathcal U} \Bigg\{ \sum_{j=1}^{m_i} \frac{\phi_k(T_{ij})}{g(T_{ij})}\,Y_{ij}(u) - \mathbb E [\Lambda_{ik}(u)]\Bigg\}\\ \hspace{6cm} \times  \Bigg\{ \sum_{j=1}^{m_{i'}} \frac{\phi_{k}(T_{i'j})}{g(T_{i'j})}\,Y_{i'j}(u) - \mathbb E [\Lambda_{i'k}(u)]\Bigg\} {\rm d}\nu(u)\\
	=  \int_{\mathcal U}\Lambda_{ik}(u) \Lambda_{i'k}(u){\rm d}\nu(u)
	- \int_{\mathcal U}\mathbb E [\Lambda_{ik}(u)]\Lambda_{i'k}(u)  {\rm d}\nu(u) \\ - \int_{\mathcal U} \Lambda_{ik}(u)\mathbb E [\Lambda_{i'k}(u)] 	{\rm d}\nu(u)
	+ \int_{\mathcal U} \mathbb E [\Lambda_{ik}(u)]   \mathbb E [\Lambda_{i'k}(u)] 	{\rm d}\nu(u),
\end{multline}
and summing over $k\in\mathcal K_M$ and combining with~\eqref{eq:def_eq_HM}, \eqref{eq:useful2} and \eqref{eq:useful3}, we deduce~\eqref{eq:def_Htilde}. Let us also recall that $\mathcal C_M:\mathcal H_M\to\mathcal H_M$ stands for the covariance operator of
$\tilde\Lambda_i$, \emph{i.e.}, the operator with kernel
\begin{equation}\label{eq:op_CM}
	\mathcal C_M[(k,u),(k',u')]=\mathbb E \big[\, \widetilde\Lambda_{ik}(u)\, \widetilde\Lambda_{ik'}(u')\big].
\end{equation}
Its eigenvalues are denoted $\lambda_{1}\ge\lambda_2\ge\cdots\ge 0$, so that 
\[
\operatorname{Trace}(\mathcal C_M^2)=\sum_{\ell\ge1}\lambda_\ell^2.
\]


For the sake of readability, we recall some more notation from Section~\ref{sec:test}: 
\begin{equation}
	V =	2 \int_{\mathcal U\times \mathcal U}\!\int_{\mathcal T}
	\frac{\overline \rho^2(u,u';t)}{g^2(t)}\,{\rm d}t\,{\rm d}\nu(u)\,{\rm d}\nu(u'),
\end{equation}
with
$$
\overline \rho(u,u';t)=\E\bigl[Y_{ij}(u)Y_{ij}(u')\mid T_{ij}=t\bigr]=\rho(u,u';t)+\mu(u,t)\mu(u',t),
$$
$\rho(u,u';t)=c_X(u,u';t,t)+\tau^2(t)\,\varrho_\eta(u,u')$, 
\begin{equation}
	c_X(u,u';t,t')=\mathbb E[X(u,t)X(u',t')] \quad \text{ and } \quad \varrho_\eta(u,u')=\mathbb E[\eta_{ij}(u)\eta_{ij}(u')].
\end{equation}
Moreover, $\widetilde Y_{ij}(u)=X_i(u,T_{ij})+\tau(T_{ij})\eta_{ij}(u)$ denotes the centered single-visit profiles, as introduced in Assumption~\ref{ass:profile_kurtosis}. Note that $\rho(\cdot,\cdot;t)$ is the conditional spatial covariance function of the centered random field (single-visit profile) $\widetilde Y_{ij}(\cdot) = X_{i}(\cdot,t)+\tau(t)\eta_{ij}(\cdot)$ given $T_{ij}=t$. Finally, let us introduce the conditional expectation operator $\mathbb E_i [\cdot] = \mathbb E [\cdot\mid m_i, T_{i,1},\ldots, T_{i,m_i}]$.

\smallskip

\begin{lemma}\label{app_lem:CM}
	Under Assumptions~\ref{Ass_gen} and~\ref{ass:random_mi}, the covariance operator of the
	within-subject fluctuation $\mathrm F_i=(\mathrm F_{ik})_{k\in\mathcal K_M}$, with
	$\mathrm F_{ik}(u)=\sum_{j=1}^{m_i}\{\phi_k(T_{ij})/g(T_{ij})\}\widetilde Y_{ij}(u)$,
	satisfies $\mathcal C_M^{\mathrm F}=\mathfrak e_1\,\mathcal C^{(1)}+\mathfrak e_2\,\mathcal C^{(2)}$,
	where $\mathcal C^{(1)}$ and $\mathcal C^{(2)}$ have kernels
	\[
	\mathsf A_{kk'}(u,u')=\int_{\mathcal T}\frac{\phi_k(t)\phi_{k'}(t)}{g(t)}\,\rho(u,u';t)\,{\rm d}t,
	\]
	\[
	\mathsf B_{kk'}(u,u')=\int_{\mathcal T}\!\int_{\mathcal T}\phi_k(t)\phi_{k'}(s)\,c_X(u,u';t,s)\,{\rm d}t\,{\rm d}s ,
	\]
	respectively. Moreover, $\mathcal C_M=\mathcal C_M^{\mathrm F}+\mathcal C_M^{\mathrm D}$, where
	$\mathcal C_M^{\mathrm D}=\operatorname{Cov}(\mathrm D_i)$ is the covariance of the design
	fluctuation $\mathrm D_i=\mathbb E_i[\Lambda_i]-\mathbb E[\Lambda_i]$, with kernel
	\[
	\mathfrak e_1\int_{\mathcal T}\frac{\phi_k(t)\phi_{k'}(t)}{g(t)}\,\mu(u,t)\mu(u',t)\,{\rm d}t
	+(\mathfrak e_2-\mathfrak e_1^2)\,\beta_k(u)\beta_{k'}(u').
	\]
	Both $\mathcal C_M^{\mathrm F}$ and $\mathcal C_M^{\mathrm D}$ are positive semi-definite, and
	$\mathcal C_M^{\mathrm F}$ does not depend on $\mu$. Consequently,
	\[
	\mathcal C_M=\mathfrak e_1\,\overline{\mathcal C}^{(1)}+\mathfrak e_2\,\mathcal C^{(2)}
	+(\mathfrak e_2-\mathfrak e_1^2)\,\mathfrak b\otimes\mathfrak b,
	\qquad \mathfrak b=(\beta_k)_{k\in\mathcal K_M}\in\mathcal H_M,
	\]
	where $\overline{\mathcal C}^{(1)}$ has kernel $\overline{\mathsf A}_{kk'}$, defined as
	$\mathsf A_{kk'}$ with $\rho$ replaced by $\overline\rho$. Under $H_0$ in~\eqref{eq:H0_full},
	$\mathfrak b=0$, so that $\mathcal C_M=\mathfrak e_1\,\overline{\mathcal C}^{(1)}+\mathfrak e_2\,\mathcal C^{(2)}$;
	if moreover $\mu\equiv0$, then $\mathcal C_M^{\mathrm D}=0$.
\end{lemma}

\subsubsection{Proof of Theorem~\ref{prop:dist_tstat} and Corollary~\ref{corr:test_level}}

Our Theorem~\ref{prop:dist_tstat} follows by applying Theorem~\ref{thm:abstract} with
\begin{equation}\label{eq:not_abstract_to_BE}
	g_1(\Xi)=\mathfrak e_1\langle \widetilde \Lambda (\Xi),\mathfrak b\rangle_{\HM}, \quad
	g_2 = \widetilde H_n,\quad c_1=2/n \quad  \text{and} \quad  c_2=2/\{n(n-1)\},
\end{equation}
where $\mathfrak b=(\beta_k)_{k\in\mathcal K_M}\in\HM$.

\smallskip

\begin{proof}[Proof of Theorem~\ref{prop:dist_tstat}]
	\textbf{\!Step 1: \!reduction to Theorem~\ref{thm:abstract}.}
	By~\eqref{eq:app_aux5}, $\mathbb E[H_n(\Xi_i,\Xi_{i'})\mid\Xi_i]
	=\mathfrak e_1\langle\Lambda_i,\mathfrak b\rangle_{\HM}$ and
	$\mathbb E[H_n(\Xi_i,\Xi_{i'})]=\mathfrak e_1^2\mathcal S(\mathcal K_M;\mu)
	=\mathfrak e_1\langle\mathbb E\Lambda_i,\mathfrak b\rangle_{\HM}$, so that the function
	$g_1$ of Section~\ref{sec:test} is indeed given by~\eqref{eq:not_abstract_to_BE}. Moreover,
	$\widetilde H_n$ is symmetric and, by~\eqref{eq:def_Htilde} and
	$\mathbb E[\widetilde\Lambda_i]=0$, $\mathbb E[\widetilde H_n(\Xi,\xi)]=0$ for every $\xi$.
	By~\eqref{eq:dec_Hoeff_Qn}, $Q_n-\mathbb E[Q_n]=L+Q$, with $L$ and $Q$ defined as in
	Section~\ref{app:sec_BE_new} with the choices~\eqref{eq:not_abstract_to_BE}, and
	\[
	c_1^2\,n\,\vv_1+c_2^2\,\frac{n(n-1)}{2}\,\vv_2=\frac4n\,\vv_1+\frac{2\,\vv_2}{n(n-1)}
	=\operatorname{Var}(Q_n)=\vv .
	\]
	Let $h(\xi)=\mathbb E[H_n(\xi,\Xi)]$, so that $g_1=h-\mathbb E H_n$. By Jensen's inequality,
	$\mathbb E h^4\le\mathbb E H_n^4$, hence $\mathbb E g_1^4\le16\,\mathbb E H_n^4<\infty$, and by
	Minkowski's inequality,
	$\|\widetilde H_n\|_{L^4}\le\|H_n\|_{L^4}+2\|h\|_{L^4}+|\mathbb E H_n|\le4\|H_n\|_{L^4}$, hence
	$\mathbb E\widetilde H_n^4\le 4^4\,\mathbb E H_n^4<\infty$.
	
	\textbf{Step 2: $V>0$ and $\vv_2>0$.} Let $\mathcal R_t$ be the conditional covariance
	operator of $\widetilde Y_{ij}$ given $T_{ij}=t$, with kernel $\rho(\cdot,\cdot;t)$, so that
	$\operatorname{Tr}\mathcal R_t=\Gamma_X(t)+\tau^2(t)\nu(\mathcal U)$. By
	Condition~\eqref{eq:non_def_Y}, $\operatorname{Tr}\mathcal R_t>0$ on a subset of $\mathcal T$
	with positive Lebesgue measure. Since
	$\overline\rho^{\,2}=\rho^2+2\rho\,\mu\otimes\mu+(\mu\otimes\mu)^2$ and
	$\iint\rho(u,u';t)\mu(u,t)\mu(u',t)\,{\rm d}\nu(u)\,{\rm d}\nu(u')
	=\langle\mu(\cdot,t),\mathcal R_t\,\mu(\cdot,t)\rangle_{L^2(\mathcal U)}\ge0$,
	\[
	V\ge2\int_{\mathcal T}\frac{\|\mathcal R_t\|_{\rm HS}^2}{g^2(t)}\,{\rm d}t>0 ,
	\]
	because $\mathcal R_t\neq0$ whenever $\operatorname{Tr}\mathcal R_t>0$. Next, since
	$\widetilde\Lambda_i$ and $\widetilde\Lambda_{i'}$ are independent with covariance operator
	$\mathcal C_M$,
	\begin{equation}\label{eq:v2_trace}
		\vv_2=\mathbb E\langle\widetilde\Lambda_i,\widetilde\Lambda_{i'}\rangle_{\HM}^2
		=\mathbb E\langle\widetilde\Lambda_i,\mathcal C_M\widetilde\Lambda_i\rangle_{\HM}
		=\operatorname{Tr}(\mathcal C_M^2)=\sum_{\ell\ge1}\lambda_\ell^2 .
	\end{equation}
	By Lemma~\ref{app_lem:CM}, $\mathcal C_M=\mathcal C_M^{\mathrm F}+\mathcal C_M^{\mathrm D}$
	and both operators are positive semi-definite. We deduce that
	$\operatorname{Tr}(\mathcal C_M^{\mathrm F}\mathcal C_M^{\mathrm D})\ge0$ and
	$\vv_2\ge\operatorname{Tr}\{(\mathcal C_M^{\mathrm F})^2\}$. Moreover,
	$\mathcal C_M^{\mathrm F}=\mathfrak e_1\mathcal C^{(1)}+\mathfrak e_2\mathcal C^{(2)}$, where
	$\mathcal C^{(2)}$ is positive semi-definite, being the covariance operator of
	$(\int_{\mathcal T}\phi_k(t)X(\cdot,t)\,{\rm d}t)_{k\in\mathcal K_M}$, and
	\[
	\operatorname{Tr}\mathcal C^{(1)}=\sum_{k\in\mathcal K_M}\int_{\mathcal T}
	\frac{\phi_k^2(t)}{g(t)}\operatorname{Tr}\mathcal R_t\,{\rm d}t>0,
	\]
	since each $\phi_k$ is a nonzero analytic function on $[0,1]$ and thus vanishes only on a
	finite set. Hence $\mathcal C_M^{\mathrm F}\neq0$ and $\vv_2>0$. In particular $\vv>0$, and
	the standardized statistic $W=(Q_n-\mathbb E Q_n)/\sqrt{\vv}$ of
	Theorem~\ref{thm:abstract} is well defined.
	
	\textbf{Step 3: the Kolmogorov distance.} By definition of $T_n$, $\delta_n$ and
	$\varsigma_n(\mu)$,
	\[
	T_n-\delta_n\mathcal S(\mathcal K_M;\mu)
	=\frac{n\{Q_n-\mathbb E Q_n\}}{\mathfrak e_1V^{1/2}|\mathcal K_M|^{1/2}}
	=\varsigma_n(\mu)\,W,
	\qquad
	\varsigma_n(\mu)=\frac{n\,\vv^{1/2}}{\mathfrak e_1V^{1/2}|\mathcal K_M|^{1/2}}>0 .
	\]
	Hence
	$\sup_z|\mathbb P(T_n\le z)-\Phi((z-\delta_n\mathcal S(\mathcal K_M;\mu))/\varsigma_n(\mu))|
	=\sup_x|\mathbb P(W\le x)-\Phi(x)|$, and it remains to bound the five terms of
	Theorem~\ref{thm:abstract}. We use $\vv\ge4\vv_1/n$, $\vv\ge2\vv_2/\{n(n-1)\}$ and, by the
	arithmetic-geometric mean inequality (see Step~4 of the proof of
	Theorem~\ref{thm:abstract}), $c_1c_2\,n\sqrt{\vv_1\vv_2}/\vv\le\{2(n-1)\}^{-1/2}$.
	
	\emph{Second term.} If $\vv_1=0$, then $g_1=0$ a.s.\ and the term vanishes. Otherwise,
	$c_1^2\sqrt{n\,\mathbb E g_1^4}/\vv\le(4/n^2)\sqrt n\,\Theta_1\vv_1\cdot n/(4\vv_1)
	=\Theta_1/\sqrt n$.
	
	\emph{Third term.} By the definition of $\Theta_{12}$, it holds
	$$
	c_1c_2\,n\sqrt{\mathbb E[g_1^2\widetilde H_n^2]}/\vv
	=\Theta_{12}\,c_1c_2\,n\sqrt{\vv_1\vv_2}/\vv\le\Theta_{12}/\sqrt{2(n-1)}\le\Theta_{12}/\sqrt n.
	$$
	
	\emph{Fourth term.} By~\eqref{eq:def_Htilde},
	$\mathfrak f(\xi)=\mathbb E[g_1(\Xi)\widetilde H_n(\Xi,\xi)]
	=\mathfrak e_1\langle\mathcal C_M\mathfrak b,\widetilde\Lambda(\xi)\rangle_{\HM}$, so that
	\[
	\mathbb E\mathfrak f^2=\mathfrak e_1^2\langle\mathcal C_M\mathfrak b,
	\mathcal C_M^2\mathfrak b\rangle_{\HM}
	=\mathfrak e_1^2\|\mathcal C_M^{3/2}\mathfrak b\|_{\HM}^2
	\le\lambda_1^2\,\mathfrak e_1^2\langle\mathfrak b,\mathcal C_M\mathfrak b\rangle_{\HM}
	=\lambda_1^2\,\vv_1 .
	\]
	Hence, we get
	$$
	c_1c_2\,n^{3/2}\sqrt{\mathbb E\mathfrak f^2}/\vv
	\le\lambda_1\sqrt{n/\vv_2}\;c_1c_2\,n\sqrt{\vv_1\vv_2}/\vv
	\le\lambda_1\sqrt{n/\{2(n-1)\}}\,\vv_2^{-1/2}\le\lambda_1/\sqrt{\vv_2}.
	$$
	
	\emph{Fifth term.} Since $$c_2^2\,n^2/\vv\le c_2^2\,n^2\big/\{c_2^2\,n(n-1)\vv_2/2\}\le4/\vv_2,$$
	and since $G_2(\xi,\xi')=\langle\widetilde\Lambda(\xi),\mathcal C_M\widetilde\Lambda(\xi')
	\rangle_{\HM}$ gives
	$\mathbb E G_2^2=\operatorname{Tr}(\mathcal C_M^4)=\sum_\ell\lambda_\ell^4\le\lambda_1^2\vv_2$
	by~\eqref{eq:v2_trace}, we get
	\[
	\frac{c_2^2n^2}{\vv}\sqrt{n^{-1}\mathbb E\widetilde H_n^4+\mathbb E G_2^2}
	\le4\Bigl(\frac{\mathbb E\widetilde H_n^4}{n\,\vv_2^2}\Bigr)^{1/2}
	+\frac{4\lambda_1}{\sqrt{\vv_2}} .
	\]
	Collecting the five terms proves the first part of the theorem.
	
	\textbf{Step 4: the second part.} Let $S_\phi=\sup_k\|\phi_k\|_\infty\le4\sqrt2$ (see
	Lemma~\ref{lem:uniform_bound}). We first show $\mathbb E H_n^4<\infty$. Let $(e_{\mathrm r})_{\mathrm r\ge1}$
	be an orthonormal basis of $L^2(\mathcal U)$. By the Cauchy-Schwarz inequality and
	Assumption~\ref{ass:profile_kurtosis},
	\begin{multline}
		\mathbb E\bigl[\|\widetilde Y_{ij}\|^4\mid T_{ij}=t\bigr]
		=\sum_{\mathrm r,\mathrm r'}\mathbb E\bigl[\langle e_{\mathrm r},\widetilde Y_{ij}\rangle^2
		\langle e_{\mathrm r'},\widetilde Y_{ij}\rangle^2\mid T_{ij}=t\bigr]
		\\ \le\kappa_4\bigl(\operatorname{Tr}\mathcal R_t\bigr)^2
		\le\kappa_4\,\nu(\mathcal U)^2\bigl(\|c_X\|_\infty+\|\tau\|_\infty^2\bigr)^2 .	
	\end{multline}
	See also~\eqref{eq:ineq_traceR}. Under Assumption~\ref{ass:holder}, integrating
	$\|\mu(\cdot,t)\|\le\|\mu(\cdot,s)\|+\|L_\mu\|_{L^2(\nu)}|t-s|^\alpha$ over $s\in[0,1]$ gives
	$\bar\mu\le\|\mu\|_{L^2(\mathcal U\times\mathcal T)}+\|L_\mu\|_{L^2(\nu)}$. Hence
	$\mathfrak M_4:=\mathbb E\|Y_{ij}\|^4\le8\{\bar\mu^4+\kappa_4\nu(\mathcal U)^2
	(\|c_X\|_\infty+\|\tau\|_\infty^2)^2\}<\infty$. Next,
	$$
	\|\Lambda_i\|_{\HM}^2\le m_i\sum_{j}\Psi(T_{ij},T_{ij};\mathcal K_M)\|Y_{ij}\|^2/g^2(T_{ij})
	\le m_{\max}\,g_{\min}^{-2}\,S_\phi^2\,|\mathcal K_M|\sum_j\|Y_{ij}\|^2, 
	$$
	so that
	\begin{equation}\label{eq:Lambda4}
		\mathbb E\|\widetilde\Lambda_i\|_{\HM}^4\le16\,\mathbb E\|\Lambda_i\|_{\HM}^4
		\le16\,m_{\max}^4\,g_{\min}^{-4}\,S_\phi^4\,\mathfrak M_4\,|\mathcal K_M|^2 .
	\end{equation}
	Since $|H_n(\Xi_i,\Xi_{i'})|\le\|\Lambda_i\|_{\HM}\|\Lambda_{i'}\|_{\HM}$ by~\eqref{eq:def_eq_HM},
	$\mathbb E H_n^4\le(\mathbb E\|\Lambda_1\|_{\HM}^4)^2<\infty$, and the first part applies.
	
	Next, we need to bound $\lambda_1=\|\mathcal C_M\|_{\rm op}$. By the arguments used to derive~\eqref{eq:lambda1_fourth} in Proposition~\ref{lem:fourth} we get
	\begin{equation}\label{eq:lambda1_bd}
		\lambda_1\le C_\lambda(1+\bar\mu^2),
	\end{equation}
	with $C_\lambda$ depending only on $m_{\max}$, $g_{\min}$, $\nu(\mathcal U)$, $\|c_X\|_\infty$,
	$\|\tau\|_\infty$ and $\mathbb E\|X\|^2_{L^2(\mathcal U\times\mathcal T)}$. Moreover, by Proposition~\ref{lem:secondmoment} and Proposition~\ref{lem:fourth}, there exist constants $c_0,C_4>0$ and an integer $K_0$, not depending on $n$, such that, for
	$|\mathcal K_M|\ge K_0$,
	\begin{equation}\label{eq:v2_H4}
		\vv_2\ge c_0\,|\mathcal K_M|
		\qquad\text{and}\qquad
		\mathbb E\widetilde H_n^4\le C_4\,|\mathcal K_M|^3 .
	\end{equation}
	Then $\lambda_1/\sqrt{\vv_2}\le C_\lambda(1+\bar\mu^2)\,c_0^{-1/2}|\mathcal K_M|^{-1/2}$ and
	$\{\mathbb E\widetilde H_n^4/(n\vv_2^2)\}^{1/2}\le C_4^{1/2}c_0^{-1}(|\mathcal K_M|/n)^{1/2}$.
	
	Finally, we bound $\Theta_{12}$. Since
	$\mathbb E[\widetilde H_n^2(\xi,\Xi)]
	=\langle\widetilde\Lambda(\xi),\mathcal C_M\widetilde\Lambda(\xi)\rangle_{\HM}
	\le\lambda_1\|\widetilde\Lambda(\xi)\|_{\HM}^2$, the Cauchy--Schwarz inequality gives
	\[
	\Theta_{12}^2\le\frac{\lambda_1\,\mathbb E[g_1^2(\Xi)\|\widetilde\Lambda(\Xi)\|_{\HM}^2]}
	{\vv_1\vv_2}
	\le\lambda_1\,\Theta_1\,\frac{(\mathbb E\|\widetilde\Lambda\|_{\HM}^4)^{1/2}}{\vv_2}
	\le C_\Theta\,\Theta_1,
	\]
	by~\eqref{eq:Lambda4}, \eqref{eq:lambda1_bd} and~\eqref{eq:v2_H4}. Hence
	$\Theta_{12}\le C_\Theta^{1/2}\Theta_1^{1/2}\le C_\Theta^{1/2}(1+\Theta_1)$. Collecting the
	bounds gives the second part for $|\mathcal K_M|\ge K_0$. For $|\mathcal K_M|<K_0$, the bound
	is trivial with $C'\ge K_0^{1/2}$, since a Kolmogorov distance is at most one.
\end{proof}


\smallskip 

\begin{proposition}[Second moment of $\widetilde H_n$  and the variance $V$]\label{lem:secondmoment}
	Suppose that Assumptions~\ref{Ass_gen} and~\ref{ass:random_mi} and
	Condition~\eqref{eq:non_def_Y} hold, that $\|c_X\|_\infty<\infty$, that $\tau$ and
	$t\mapsto c_X(u,u';t,t)$ satisfy the Hölder condition of Assumption~\ref{ass:CLT_mild_i},
	and that $\mathcal K_M$ is a set of consecutive indices. Then, for $\widetilde H_n$ defined
	in~\eqref{eq:def_Htilde},
	$$
	\vv_2=\mathbb E\bigl[\widetilde H_n^2(\Xi_i,\Xi_{i'})\bigr]=\operatorname{Trace}(\mathcal C_M^2)
	\geq \mathfrak e_1^2\operatorname{Trace}\bigl((\mathcal C^{(1)})^2\bigr).
	$$
	Moreover, $\|\rho/g\|>0$ and
	\begin{equation}\label{eq:trC1_rho}
		\left||\mathcal K_M|^{-1}\operatorname{Trace}\bigl((\mathcal C^{(1)})^2\bigr) - \|\rho/g \|^2\right|
		\leq C_1''  |\mathcal K_M|^{-\varkappa/(1+\varkappa)},
	\end{equation}
	with $\varkappa\in (0,1]$ the Hölder exponent from Assumption~\ref{ass:CLT_mild_i} and $C_1''$
	a constant depending on $g_{\rm min}$ and the Lipschitz constant $L_g$ of $g$, as well as on
	$\|c_X\|_{\infty}$, $\|\tau\|_\infty$, $\nu(\mathcal U)$ and the Hölder constant $L_{\tau,c}$
	from Assumption~\ref{ass:CLT_mild_i}. In particular, $\vv_2\geq \mathfrak e_1^2\|\rho/g\|^2
	|\mathcal K_M|/2$ as soon as $|\mathcal K_M|^{\varkappa/(1+\varkappa)}\geq 2C_1''/\|\rho/g\|^2$.
	
	If, in addition, Assumption~\ref{ass:holder} holds, then, with
	$\bar\varkappa=\min\{\varkappa,\alpha\}$,
	\begin{equation}\label{eq:trCM_rhobar}
		\left||\mathcal K_M|^{-1}\operatorname{Trace}(\mathcal C_M^2) - \mathfrak e_1^2\|\overline \rho/g \|^2\right|
		\leq C_2''  |\mathcal K_M|^{-\bar\varkappa/(1+\bar\varkappa)},
	\end{equation}
	where $C_2''$ depends on the same quantities as $C_1''$, on $m_{\max}$, and on $\mu$ only through
	$\|\mu\|_{L^2(\mathcal U\times\mathcal T)}$ and $\|L_\mu\|_{L^2(\nu)}$. Finally,
	$V=2\,\|\overline\rho/g\|^2\geq 2\,\|\rho/g\|^2>0$.
\end{proposition}

\begin{proof}[Proof of Proposition~\ref{lem:secondmoment}]
	Here $\|\cdot \|=\|\cdot \|_{L^2(\mathcal U \times \mathcal U \times \mathcal T)}$.
	As $\Xi_i$ and $\Xi_{i'}$ are independent, $\mathbb E[\widetilde\Lambda_{i'}]=0$, and
	$\widetilde H_n = \langle\widetilde\Lambda_i,\widetilde\Lambda_{i'}\rangle_{\mathcal H_M}$, by
	the definition of $\mathcal C_M$ we get
	$\mathbb E[\widetilde H_n^2\mid\Xi_i]
	=\langle\widetilde\Lambda_i,\mathcal C_M\widetilde\Lambda_i\rangle_{\mathcal H_M}$, and thus
	\[
	\mathbb E[\widetilde H_n^2]
	=\mathbb E\bigl[\langle\widetilde\Lambda_i,\mathcal C_M\widetilde\Lambda_i\rangle_{\mathcal H_M}\bigr]
	=\operatorname{Trace}(\mathcal C_M^2).
	\]
	By Lemma~\ref{app_lem:CM}, $\mathcal C_M=\mathcal C_M^{\mathrm F}+\mathcal C_M^{\mathrm D}$ and
	$\mathcal C_M^{\mathrm F}=\mathfrak e_1\mathcal C^{(1)}+\mathfrak e_2\mathcal C^{(2)}$, with
	$\mathcal C_M^{\mathrm D}$, $\mathcal C^{(1)}$ and $\mathcal C^{(2)}$ positive semi-definite.
	Since the trace of the product of two PSD operators is nonnegative,
	$\operatorname{Trace}(\mathcal C_M^2)\geq\operatorname{Trace}((\mathcal C_M^{\mathrm F})^2)
	\geq\mathfrak e_1^2\operatorname{Trace}((\mathcal C^{(1)})^2)$.
	
	\emph{A general approximation.} For a symmetric kernel $r(u,u';t)$, continuous in $t$, let
	$\mathcal C[r]$ be the operator on $\mathcal H_M$ with kernel
	$\int_{\mathcal T}\phi_k(t)\phi_{k'}(t)r(u,u';t)g^{-1}(t)\,{\rm d}t$, so that
	$\mathcal C^{(1)}=\mathcal C[\rho]$, and let $\overline{\mathcal C}^{(1)}=\mathcal C[\overline\rho]$.
	Using the notation in~\eqref{eq:not_aux1}, we can write
	\begin{multline}\label{eq:trCD2}
		\operatorname{Trace}(\mathcal C[r]^2)
		=\int_{\mathcal U\times \mathcal U}\!\int_{\mathcal T \times \mathcal T}
		\frac{\Psi^2(t,s;\mathcal K_M)}{g(t)g(s)}\,
		r(u,u';t)\,r(u,u';s)\,{\rm d}t\,{\rm d}s\,{\rm d}\nu(u)\,{\rm d}\nu(u')\\
		=\int_{\mathcal U\times \mathcal U}\!\int_{\mathcal T}\frac{r(u,u';t)}{g(t)}\,
		\mathfrak G(u,u';t)\, {\rm d} t\,{\rm d}\nu(u)\,{\rm d}\nu(u'),
		\\ \text{with} \quad 
		\mathfrak G (u,u';t) = \int_{\mathcal T}
		\Psi^2(t,s;\mathcal K_M) \frac{r(u,u';s)}{g(s)}{\rm d} s.
	\end{multline}
	Fix $u,u'$ and let $h(t)=h_{u,u'}(t)=r(u,u';t)/g(t)$. Assume that
	$|r(u,u';t)-r(u,u';s)|\leq L_r(u,u')|t-s|^{\gamma}$ for some $\gamma\in(0,1]$. Since
	$|1/g(t)-1/g(s)|\leq L_g g_{\min}^{-2}|t-s|\leq L_g g_{\min}^{-2}|t-s|^{\gamma}$ on
	$\mathcal T=[0,1]$, $h$ is Hölder continuous with exponent $\gamma$ and constant
	$L_h(u,u')=g_{\min}^{-1}L_r(u,u')+g_{\min}^{-2}L_g\|r(u,u';\cdot)\|_\infty$.
	By Lemma~\ref{lem:order-Gh2} applied with $h_1=h_2=h$, for a universal constant $C$,
	\[
	\left| |\mathcal K_M|^{-1} \int_{\mathcal T}\mathfrak G(u,u';t)\, h(t)\, {\rm d} t
	- \int_{\mathcal T} h^2(t){\rm d} t \right|
	\le C\,\|h\|_\infty\bigl\{L_h(u,u')+\|h\|_\infty\bigr\}|\mathcal K_M|^{-\gamma/(1+\gamma)}.
	\]
	Integrating out $u,u'$, we get
	\begin{multline}\label{eq:trCr_gen}
		\left| |\mathcal K_M|^{-1}\operatorname{Trace}(\mathcal C[r]^2)- \|r/g \|^2 \right|
		\\	\le C\,|\mathcal K_M|^{-\gamma/(1+\gamma)}
		\int_{\mathcal U\times\mathcal U}\|h_{u,u'}\|_\infty\bigl\{L_h(u,u')+\|h_{u,u'}\|_\infty\bigr\}
		{\rm d}\nu(u)\,{\rm d}\nu(u'),
	\end{multline}
	provided the last integral is finite.
	
	\emph{The kernel $\rho$.} Let $c_\rho=\|c_X\|_\infty+\|\tau\|_\infty^2$. Since
	$\operatorname{Var}(\eta_{ij}(u))=1$, $|\varrho_\eta(u,u')|\leq1$, so that
	$|\rho(u,u';t)|\leq c_\rho$ and $\|h_{u,u'}\|_\infty\leq g_{\min}^{-1}c_\rho$. Moreover,
	$|\tau^2(t)-\tau^2(s)|\leq2\|\tau\|_\infty L_{\tau,c}|t-s|^\varkappa$, hence
	$L_\rho(u,u')\leq L_{\tau,c}(1+2\|\tau\|_\infty)$, with $\gamma=\varkappa$, uniformly in $u,u'$.
	Then~\eqref{eq:trCr_gen} gives~\eqref{eq:trC1_rho}, with $C_1''$ depending only on
	$g_{\min}$, $L_g$, $c_\rho$, $L_{\tau,c}$, $\|\tau\|_\infty$ and $\nu(\mathcal U)$.
	
	\emph{The kernel $\overline\rho$.} Let $\mathfrak M(u)=\sup_{t\in\mathcal T}|\mu(u,t)|$. By
	Assumption~\ref{ass:holder}, integrating
	$|\mu(u,t)|\leq|\mu(u,s)|+L_\mu(u)$ over $s\in[0,1]$ gives
	$\mathfrak M(u)\leq\int_{\mathcal T}|\mu(u,s)|{\rm d}s+L_\mu(u)$, so that
	$\|\mathfrak M\|_{L^2(\nu)}\leq m_\mu:=\|\mu\|_{L^2(\mathcal U\times\mathcal T)}+\|L_\mu\|_{L^2(\nu)}$.
	Next, $|\overline\rho(u,u';t)|\leq c_\rho+\mathfrak M(u)\mathfrak M(u')$ and, since
	$|t-s|\leq1$,
	\[
	|\mu(u,t)\mu(u',t)-\mu(u,s)\mu(u',s)|
	\leq\bigl\{L_\mu(u)\mathfrak M(u')+\mathfrak M(u)L_\mu(u')\bigr\}|t-s|^{\bar\varkappa},
	\]
	so that $\overline\rho$ satisfies the Hölder condition with $\gamma=\bar\varkappa$ and
	$L_{\overline\rho}(u,u')\leq L_{\tau,c}(1+2\|\tau\|_\infty)
	+L_\mu(u)\mathfrak M(u')+\mathfrak M(u)L_\mu(u')$. Each term of the product
	$\|h_{u,u'}\|_\infty\{L_h(u,u')+\|h_{u,u'}\|_\infty\}$ is then bounded by a constant times a
	product of factors $1$, $\mathfrak M$, $L_\mu$, $\mathfrak M^2$ or $\mathfrak M L_\mu$ in $u$
	and in $u'$, which are integrable with respect to $\nu$ by the Cauchy-Schwarz inequality,
	since $\nu(\mathcal U)<\infty$ and $\mathfrak M,L_\mu\in L^2(\nu)$. Thus, by~\eqref{eq:trCr_gen},
	\begin{equation}\label{eq:trC1bar}
		\left| |\mathcal K_M|^{-1}\operatorname{Trace}\bigl((\overline{\mathcal C}^{(1)})^2\bigr)
		- \|\overline\rho/g \|^2 \right| \le C\,|\mathcal K_M|^{-\bar\varkappa/(1+\bar\varkappa)},
	\end{equation}
	with $C$ depending on $C_1''$ and $m_\mu$.
	
	\emph{The remaining terms.} By Lemma~\ref{app_lem:CM},
	$\mathcal C_M=\mathfrak e_1\overline{\mathcal C}^{(1)}+\mathfrak e_2\mathcal C^{(2)}
	+\varpi\,\mathfrak b\otimes\mathfrak b$ with $\varpi=\mathfrak e_2-\mathfrak e_1^2$, so that
	\begin{multline*}
		\operatorname{Trace}(\mathcal C_M^2)-\mathfrak e_1^2\operatorname{Trace}((\overline{\mathcal C}^{(1)})^2)
		=2\mathfrak e_1\mathfrak e_2\operatorname{Trace}(\overline{\mathcal C}^{(1)}\mathcal C^{(2)})
		+\mathfrak e_2^2\operatorname{Trace}((\mathcal C^{(2)})^2)\\
		+2\varpi\bigl\{\mathfrak e_1\langle\mathfrak b,\overline{\mathcal C}^{(1)}\mathfrak b\rangle_{\mathcal H_M}
		+\mathfrak e_2\langle\mathfrak b,\mathcal C^{(2)}\mathfrak b\rangle_{\mathcal H_M}\bigr\}
		+\varpi^2\|\mathfrak b\|_{\mathcal H_M}^4 .
	\end{multline*}
	By Bessel's inequality in $L^2(\mathcal T\times \mathcal T)$ applied to the orthonormal system
	$\{\phi_k\otimes\phi_{k'}\}$,
	\begin{multline}
		\operatorname{Trace}((\mathcal C^{(2)})^2)
		=\sum_{k,k'\in\mathcal K_M}\|\mathsf B_{kk'}\|_{L^2(\mathcal U\times \mathcal U)}^2
		\\	\le \int_{\mathcal U\times \mathcal U}\bigl\|c_X(u,u';\cdot,\cdot)\bigr\|_{L^2(\mathcal T\times \mathcal T)}^2
		\,{\rm d}\nu(u)\,{\rm d}\nu(u')\leq \|c_X\|^2_{\infty}\nu^2(\mathcal U).
	\end{multline}
	Next, by the Cauchy-Schwarz inequality and~\eqref{eq:trC1bar},
	\[
	\operatorname{Trace}(\overline{\mathcal C}^{(1)}\mathcal C^{(2)})\leq
	\sqrt{\operatorname{Trace}((\mathcal C^{(2)})^2)}\sqrt{\operatorname{Trace}((\overline{\mathcal C}^{(1)})^2)}
	\leq C|\mathcal K_M|^{1/2}.
	\]
	Finally, $\|\mathfrak b\|_{\mathcal H_M}^2=\mathcal S(\mathcal K_M;\mu)
	\leq\|\mu\|^2_{L^2(\mathcal U\times\mathcal T)}$ by Bessel's inequality,
	$\|\mathcal C^{(2)}\|_{\rm op}\leq\|c_X\|_\infty\nu(\mathcal U)$, and, writing
	$a(t)=\sum_{k\in\mathcal K_M}\phi_k(t)a_k$ for $a\in\mathcal H_M$,
	\begin{multline}
		\langle a,\overline{\mathcal C}^{(1)}a\rangle_{\mathcal H_M}
		=\int_{\mathcal T}\!\int_{\mathcal U\times\mathcal U}\frac{a(t)(u)\,\overline\rho(u,u';t)\,a(t)(u')}{g(t)}
		\,{\rm d}\nu(u)\,{\rm d}\nu(u')\,{\rm d}t
		\\ \leq g_{\min}^{-1}\bigl\{c_\rho\nu(\mathcal U)+\bar\mu^2\bigr\}\|a\|_{\mathcal H_M}^2,
	\end{multline}
	because $\int_{\mathcal T}\|a(t)\|^2_{L^2(\mathcal U)}{\rm d}t=\|a\|^2_{\mathcal H_M}$ and
	$\bar\mu\leq m_\mu$ (see the proof of Theorem~\ref{prop:dist_tstat}). Since
	$|\varpi|\leq m_{\max}^2$, we deduce that
	$$
	|\mathcal K_M|^{-1}\left|\operatorname{Trace}(\mathcal C_M^2)
	- \mathfrak e_1^2\operatorname{Trace}((\overline{\mathcal C}^{(1)})^2)\right|
	\leq C'|\mathcal K_M|^{-1/2},
	$$
	for some constant $C'$ depending on $\|c_X\|_\infty$, $\nu(\mathcal U)$, $m_{\max}$, $g_{\min}$,
	$\|\tau\|_\infty$ and $m_\mu$. Combined with~\eqref{eq:trC1bar} and
	$\bar\varkappa/(1+\bar\varkappa)\leq1/2$, this gives~\eqref{eq:trCM_rhobar}. Under $H_0$
	in~\eqref{eq:H0_full}, $\mathfrak b=0$ and the terms involving $\varpi$ vanish.
	
	\emph{Positivity.} It remains to prove that  $\|\rho/g \|>0$ under the
	condition~\eqref{eq:non_def_Y}. Since $g\leq g_{\rm max}$, it suffices to show that
	$\|\rho \|>0$, with
	$\rho(u,u';t)= \mathbb E[X(u,t)X(u',t)]+\tau^2(t)\,\mathbb E[\eta_{ij}(u)\eta_{ij}(u')]$
	defined above. Note that $\rho(u,u';t)$ is a symmetric kernel which defines the integral
	operator $\mathcal R_t=\mathcal A_t+\tau^2(t)\mathcal B$ with $\mathcal A_t$ and  $\mathcal B$
	the covariance operators of $X(\cdot,t)$ and $\eta_{ij}$, respectively. By definition,
	$\mathcal R_t$ is a positive semi-definite (PSD) operator, and it holds
	$\|\rho (\cdot,\cdot;t)\|_{L^2 (\mathcal U\times \mathcal U)}^2 = \|\mathcal R_t \|^2_{\rm HS}
	=\operatorname{Trace}(\mathcal R^2_t)$, where $\|\cdot \|_{\rm HS}$ denotes the Hilbert-Schmidt
	norm. Note also that since $\mathcal A_t$ and  $\mathcal B$ are also PSD operators, then
	$\operatorname{Trace}(\mathcal A_t\mathcal B)\geq 0$. As a consequence,
	\begin{multline}
		\operatorname{Trace}(\mathcal R^2_t) = \operatorname{Trace}(\mathcal A^2_t)
		+2\tau^2 (t) \operatorname{Trace}(\mathcal A_t\mathcal B)+\tau^4(t)\operatorname{Trace}(\mathcal B^2)
		\\ 	\geq \operatorname{Trace}(\mathcal A^2_t)+\tau^4(t)\operatorname{Trace}(\mathcal B^2).
	\end{multline}
	We deduce
	\[
	\| \rho \|^2 = \int_{\mathcal T}\operatorname{Trace}(\mathcal R_t^2)\,\mathrm dt
	\ge \int_{\mathcal T}\!\operatorname{Trace}(\mathcal A_t^2)\,\mathrm dt
	+\operatorname{Trace}(\mathcal B^2)\!\int_{\mathcal T}\tau^4(t)\,\mathrm dt .
	\]
	Now, $\operatorname{Trace}(\mathcal B)=\int_{\mathcal U}\operatorname{Var}(\eta_{ij}(u))\,{\rm d}\nu(u)
	=\nu(\mathcal U)>0$, so $\mathcal B\neq0$ and $\operatorname{Trace}(\mathcal B^2)$ is also
	positive. Moreover, $\operatorname{Trace}(\mathcal A_t)=\Gamma_X(t)$, so that
	$\mathcal A_t\neq0$, and thus $\operatorname{Trace}(\mathcal A^2_t)>0$, whenever
	$\Gamma_X(t)>0$. Combining these facts we get the following conclusion under the
	condition~\eqref{eq:non_def_Y}: if either $\|\tau\|_{L^2(\mathcal T)}>0$ or
	$\|\Gamma_X\|_{L^2(\mathcal T)}>0$, then $\| \rho \|$ and thus $\| \rho/g \|$ are positive.
	Finally, since $\overline\rho^{\,2}=\rho^2+2\rho\,\mu\otimes\mu+(\mu\otimes\mu)^2$ and
	$\iint\rho(u,u';t)\mu(u,t)\mu(u',t)\,{\rm d}\nu(u)\,{\rm d}\nu(u')
	=\langle\mu(\cdot,t),\mathcal R_t\,\mu(\cdot,t)\rangle_{L^2(\mathcal U)}\geq0$, we get
	$\|\overline\rho/g\|\geq\|\rho/g\|$. Thus $V= 2\| \overline\rho/g \|^2\geq2\|\rho/g\|^2>0$.
	The lower bound on $\vv_2$ follows from
	$\vv_2\geq\mathfrak e_1^2\operatorname{Trace}((\mathcal C^{(1)})^2)$ and~\eqref{eq:trC1_rho}.
	Now the proof is complete.
\end{proof}

\smallskip

\begin{proposition}[Fourth-order moments]\label{lem:fourth}
	Suppose that Assumptions~\ref{Ass_gen}, \ref{ass:random_mi} and~\ref{ass:CLT_mild} hold, that
	$\bar\mu=\sup_{t\in\mathcal T}\|\mu(\cdot,t)\|_{L^2(\mathcal U)}<\infty$, and that $\{\phi_k\}$
	is one of the bases $\mathcal B_{r,K}$, $r\in\{0,1,2\}$. Then, for any
	$\mathcal K_M\subset\{1,\ldots,K+r\}$, with $G_2$ defined in Section~\ref{app:sec_BE_new} for
	$g_2=\widetilde H_n$,
	\[
	\mathbb E[G_2^2(\Xi_i,\Xi_{i'})]=\operatorname{Trace}(\mathcal C_M^4)
	\le \lambda_1 ^2 \,\operatorname{Trace}(\mathcal C_M^2),
	\qquad
	\mathbb E\big[\widetilde H_n^4(\Xi_i,\Xi_{i'})\big]\le C\,|\mathcal K_M|^3,
	\]
	where $\lambda_1$ denotes the largest eigenvalue of the operator $\mathcal C_M$ defined
	by~\eqref{eq:op_CM}, which satisfies
	\begin{equation}\label{eq:lambda1_fourth}
		\lambda_1\le \frac{\mathfrak e_1\,\nu(\mathcal U)\,c_\rho}{g_{\min}}
		+\mathfrak e_2\,\nu(\mathcal U)\|c_X\|_\infty
		+\frac{m_{\max}\,\mathfrak e_1\,\bar\mu^2}{g_{\min}},
		\qquad c_\rho=\|c_X\|_\infty+\|\tau\|_\infty^2,
	\end{equation}
	and
	$$
	C = C_0\, m_{\rm max}^6\mathfrak e_1^2 + C_1\, m_{\rm max}^7\mathfrak e_1 \max\{\bar\mu^8,\bar\mu^4\},
	$$
	with $C_0, C_1$ constants depending only on $g_{\rm min}$, $\kappa_4$, $\nu(\mathcal U)$,
	$\|c_X\|_\infty$ and $\|\tau\|_\infty$. If, moreover, the conditions of
	Proposition~\ref{lem:secondmoment} hold, then there exist $c>0$ and $K_0$ such that, for
	$|\mathcal K_M|\geq K_0$,
	$\mathbb E[\widetilde H_n^4]\le C c^{-2}|\mathcal K_M|\,\{\mathbb E[\widetilde H_n^2]\}^2$.
\end{proposition}

\begin{proof}[Proof of Proposition~\ref{lem:fourth}]
	Recall that
	$$
	\widetilde H_n(\Xi_i,\Xi_{i'}) = \big\langle \,\widetilde \Lambda_{i}, \widetilde \Lambda_{i'} \big\rangle_{\mathcal H_M} = \sum_{k\in \mathcal K_M}\big\langle \,\widetilde \Lambda_{ik}, \widetilde \Lambda_{i'k} \big\rangle_{L^2(\mathcal U)},
	$$
	with $\widetilde \Lambda_{i}=(\widetilde \Lambda_{ik})_{k\in\mathcal K_M}$ and
	$\widetilde \Lambda_{ik}(u)=\Lambda_{ik}(u) -  \mathbb E [\Lambda_{ik}(u)]$.
	Let $\widetilde\lambda(\xi)$ denote the realization of $\widetilde \Lambda_{i}$ when $\Xi_i = \xi$.
	Throughout, we use that, by Lemma~\ref{lem:uniform_bound},
	$\Psi(t,t;\mathcal K_M)=\sum_{k\in\mathcal K_M}\phi_k^2(t)\le 32|\mathcal K_M|$, and that
	$\|\mathcal R_t\|_{\rm op}\le\operatorname{Trace}(\mathcal R_t)\le\nu(\mathcal U)c_\rho$ for every
	$t\in\mathcal T$.
	
	To bound the expectation of  $G_2^2$, note that
	$$
	G_2(\xi,\xi')=\mathbb E\bigl[\widetilde H_n(\Xi_1,\xi)\,\widetilde H_n(\Xi_1,\xi')\bigr]
	=\bigl\langle\widetilde\lambda(\xi),\mathcal C_M\widetilde\lambda(\xi')\bigr\rangle_{\mathcal H_M}.
	$$
	For independent $\Xi_i,\Xi_{i'}$ we compute $\mathbb E[G_2^2]$ by integrating out one
	copy at a time. Conditioning on $\widetilde\Lambda_i$ and using that $\widetilde\Lambda_{i'}$
	is an independent copy with $\mathbb E[\widetilde\Lambda_{i'}\otimes\widetilde\Lambda_{i'}]
	=\mathcal C_M$, we get
	\[
	\mathbb E\bigl[\langle\widetilde\Lambda_i,\mathcal C_M\widetilde\Lambda_{i'}\rangle_{\mathcal H_M}^2
	\mid\widetilde\Lambda_i\bigr]
	=\bigl\langle \mathcal C_M\widetilde\Lambda_i,\ \mathcal C_M\,(\mathcal C_M\widetilde\Lambda_i)\bigr\rangle_{\mathcal H_M}
	=\bigl\langle\widetilde\Lambda_i,\mathcal C_M^{3}\widetilde\Lambda_i\bigr\rangle_{\mathcal H_M}.
	\]
	Taking expectation over $\widetilde\Lambda_i$ gives
	\[
	\mathbb E\bigl[G_2^2(\Xi_i,\Xi_{i'})\bigr]
	=\mathbb E\bigl[\langle\widetilde\Lambda_i,\mathcal C_M^{3}\widetilde\Lambda_i\rangle_{\mathcal H_M}\bigr]
	=\operatorname{Trace}(\mathcal C_M^{4}).
	\]
	Since $\mathcal C_M$ is a covariance operator with eigenvalues $\lambda_1\geq \lambda_2\geq\cdots \geq 0$,  it holds
	$$
	\operatorname{Trace}(\mathcal C_M^{4})=\sum_{\ell\ge1}\lambda_\ell^{4}\leq \lambda^2_1 \sum_{\ell\ge1}\lambda_\ell^{2} = \lambda^2_1 \operatorname{Trace}(\mathcal C_M^{2}),
	$$
	which gives the first assertion.
	
	Next, we recall the notation
	$\mathbb E_i [\cdot] = \mathbb E [\cdot\mid m_i, T_{i,1},\ldots, T_{i,m_i}]$, and the decomposition
	$\widetilde\Lambda_i=\mathrm F_i+\mathrm D_i$ from Lemma~\ref{app_lem:CM}, with
	$$
	\mathrm F_{ik}(u)=\sum_{j=1}^{m_i}\frac{\phi_k(T_{ij})}{g(T_{ij})}\widetilde Y_{ij}(u) \quad \text{ and} \quad
	\mathrm D_{ik}(u)=\mathbb E_i[\Lambda_{ik}(u)]-\mathbb E[\Lambda_{ik}(u)]
	$$
	the within-subject fluctuation and the design fluctuation, respectively, where
	$\widetilde Y_{ij}(u) = X_{i}(u,T_{ij})+\tau(T_{ij})\eta_{ij}(u)$. Since
	$\mathbb E_i[\Lambda_{ik}]=\sum_{j=1}^{m_i} \{\phi_k(T_{ij})/g(T_{ij})\}\mu(\cdot,T_{ij})$
	and $\mathbb E[\Lambda_{ik}]=\mathfrak e_1\beta_k$, we have
	$\mathrm D_i=\sum_{j=1}^{m_i}\mathrm b_{ij}-\mathfrak e_1\mathfrak b$, where
	$\mathrm b_{ij}\in \mathcal H_M$ and has the components 
	$$\{\phi_k(T_{ij})/g(T_{ij})\}\mu(\cdot,T_{ij}),\quad k\in\mathcal K_M,$$ 
	and $\mathfrak b = (\beta_k)_{k\in\mathcal K_M}$, so that
	$\|\mathfrak b\|_{\mathcal H_M}^2=\mathcal S(\mathcal K_M;\mu)\le\|\mu\|_{L^2(\mathcal U\times\mathcal T)}^2
	\le\bar\mu^2$. Note that, under a composite null hypothesis, $\mathfrak b=0$ but
	$\mathrm D_i\neq0$ unless $\mu\equiv0$.
	
	\medskip\noindent\emph{The operator norm of $\mathcal C_M$.} By Lemma~\ref{app_lem:CM},
	$\lambda_1=\|\mathcal C_M\|_{\mathrm{op}}\leq \mathfrak e_1\|\mathcal C^{(1)}\|_{\rm op}
	+\mathfrak e_2\|\mathcal C^{(2)}\|_{\rm op}+\|\mathcal C_M^{\mathrm D}\|_{\rm op}$. Writing
	$a(t)=\sum_{k\in\mathcal K_M}\phi_k(t)a_k$ for $a\in\mathcal H_M$, so that
	$\int_{\mathcal T}\|a(t)\|^2_{L^2(\mathcal U)}{\rm d}t=\|a\|^2_{\mathcal H_M}$, we have
	$$
	\langle a,\mathcal C^{(1)}a\rangle_{\mathcal H_M}=\int_{\mathcal T}\langle a(t),\mathcal R_t a(t)\rangle
	g^{-1}(t)\,{\rm d}t\le g_{\min}^{-1}\sup_t\|\mathcal R_t\|_{\rm op}\|a\|^2_{\mathcal H_M},
	$$ 
	while
	$\|\mathcal C^{(2)}\|_{\rm op}\le\{\operatorname{Trace}((\mathcal C^{(2)})^2)\}^{1/2}
	\le\nu(\mathcal U)\|c_X\|_{\infty}$ by the proof of Proposition~\ref{lem:secondmoment}. The bound
	on $\|\mathcal C_M^{\mathrm D}\|_{\rm op}$ is given in~\eqref{eq:D-op} below. This
	proves~\eqref{eq:lambda1_fourth}.
	
	\medskip\noindent\emph{The decomposition of $\widetilde H_n$.} We can write
	$$
	\widetilde H_n(\Xi_i,\Xi_{i'}) = \widetilde H_{0,n}(\Xi_i,\Xi_{i'}) + \widetilde H_{1,n}(\Xi_i,\Xi_{i'}) + \widetilde H_{1,n}(\Xi_{i'},\Xi_{i}) + \widetilde H_{2,n}(\Xi_i,\Xi_{i'}),
	$$
	where
	\[
	\widetilde H_{0,n}(\Xi_i,\Xi_{i'})=\langle\mathrm F_i,\mathrm F_{i'}\rangle_{\mathcal H_M}
	=\sum_{j=1}^{m_i}\sum_{j'=1}^{m_{i'}} h_{jj'},
	\]
	\[
	h_{jj'}=h_{ii',jj'}:=\frac{\Psi(T_{ij},T_{i'j'};\mathcal K_M)}{g(T_{ij})g(T_{i'j'})}\,
	\langle\widetilde Y_{ij},\widetilde Y_{i'j'}\rangle_{L^2(\mathcal U)} ,
	\]
	\[
	\widetilde H_{1,n}(\Xi_i,\Xi_{i'})=\langle\mathrm  F_i,\mathrm D_{i'}\rangle_{\mathcal H_M},
	\qquad
	\widetilde H_{2,n}(\Xi_i,\Xi_{i'})=\langle \mathrm D_i,\mathrm D_{i'}\rangle_{\mathcal H_M}.
	\]
	
	\medskip\noindent\emph{Bound on $\widetilde H_{0,n}$.}
	The double  sum in the definition of $\widetilde H_{0,n}$ has $m_i m_{i'}\le m_{\max}^2$ terms, so by Jensen's inequality,
	\[
	\widetilde H_{0,n}^4(\Xi_i,\Xi_{i'}) \le(m_i m_{i'})^3\sum_{j,j'}h_{jj'}^4\le m_{\max}^6\sum_{j\le m_i,\,j'\le m_{i'}}h_{jj'}^4 .
	\]
	By the independence assumptions and the i.i.d. assumption on the $\Xi_i$'s, we get
	\begin{equation}\label{eq:dec_hjj}
		\mathbb E\Biggl[\sum_{j\le m_i,j'\le m_{i'}}h_{jj'}^4\Biggr]=\mathbb E[m_i]\,\mathbb E[m_{i'}]\,\mathbb E[h_{11}^4]
		=\mathfrak e_1^2\,\mathbb E[h_{11}^4],
	\end{equation}
	and hence
	\begin{equation}\label{eq:H4-to-h4}
		\mathbb E \big[\widetilde H_{0,n}^4(\Xi_i,\Xi_{i'})\big]\le m_{\max}^6\,\mathfrak e_1^2\,\mathbb E[h_{11}^4].
	\end{equation}
	To bound $\mathbb E[h_{11}^4]$ we study the fourth-order moment of
	$\Psi(T_{ij},T_{i'j'};\mathcal K_M)/\{g(T_{ij})g(T_{i'j'})\}$ and the conditional moment of order 4
	of $\langle\widetilde Y_{ij},\widetilde Y_{i'j'}\rangle_{L^2(\mathcal U)}$ given the visit times.
	The control of the former follows from the lower bound of $g$ and
	Lemma~\ref{lem:avg-cosine-L2-full}-(iii). For the latter, conditioning on
	$(T_{ij},\widetilde Y_{ij},T_{i'j'})=(t,\mathfrak y, t')$ and applying
	Assumption~\ref{ass:profile_kurtosis} to $\widetilde Y_{i'j'}$ with direction $\mathfrak y$
	and $T_{i'j'}=t'$, by our independence assumptions we get
	\begin{multline*}
		\mathbb E\bigl[\langle\widetilde Y_{ij},\widetilde Y_{i'j'}\rangle_{L^2(\mathcal U)}^4\mid T_{ij}=t,T_{i'j'}=t',\widetilde Y_{ij}=\mathfrak y \bigr]
		=\mathbb E\bigl[\langle\mathfrak y,\widetilde Y_{i'j'}\rangle_{L^2(\mathcal U)}^4\mid T_{i'j'}=t'\bigr]
		\\ \le\kappa_4\,\langle\mathfrak y,\mathcal R_{t'}\mathfrak y \rangle^2
		\le\kappa_4\,\sup_{s\in\mathcal T } \|\mathcal R_s\|^2_{\rm op }\, \|\mathfrak y \|_{L^2(\mathcal U)}^4 .
	\end{multline*}
	On the other hand, Assumption~\ref{ass:profile_kurtosis}, being stated for arbitrary directions
	$\mathfrak y$, also controls the total fourth moment of the profile: for every $t\in\mathcal T$,
	\begin{equation}\label{eq:ineq_traceR}
		\mathbb E\bigl[\|\widetilde Y_{ij}\|_{L^2(\mathcal U)}^4 \mid T_{ij}=t\bigr]
		\le\kappa_4\operatorname{Trace}(\mathcal R_t)^2
		\le\kappa_4\nu^2(\mathcal U)\,c_\rho^2<\infty,
	\end{equation}
	the trace of $\mathcal R_t$ replacing the directional variance
	$\langle \mathfrak y,\mathcal R_t \mathfrak y\rangle$. Indeed, let $\{e_\ell\}_{\ell\ge1}$ be an
	orthonormal basis of $L^2(\mathcal U)$ given by the eigenfunctions (which depend on $t$) of the
	conditional covariance operator $\mathcal R_t$, with eigenvalues $\mathfrak r_\ell\ge0$.
	Write $Y_{ij,\ell}:=\langle e_\ell,\widetilde Y_{ij}\rangle_{L^2(\mathcal U)}$. Then
	$\mathbb E[Y_{ij,\ell}^2\mid T_{ij} =t]=\langle e_\ell,\mathcal R_t e_\ell\rangle=\mathfrak r_\ell$, and
	Parseval's identity gives $\|\widetilde Y_{ij}\|_{L^2(\mathcal U)}^2=\sum_\ell Y_{ij,\ell}^2$ and
	$\operatorname{Trace}(\mathcal R_t)=\sum_\ell\mathfrak r_\ell$.
	Taking $\mathfrak y=e_\ell$ in Assumption~\ref{ass:profile_kurtosis} yields the coordinatewise bound
	$\mathbb E[Y_{ij,\ell}^4\mid T_{ij} = t]\le\kappa_4\,\mathfrak r_\ell^2$, $\ell\ge1$.
	Expanding the square and bounding each cross-moment by the Cauchy-Schwarz inequality,
	\begin{multline}
		\mathbb E\bigl[\|\widetilde Y_{ij}\|_{L^2(\mathcal U)}^4\mid T_{ij} =t\bigr]
		=\sum_{\ell,\ell'}\mathbb E[Y_{ij,\ell}^2 Y_{ij,\ell'}^2\mid  T_{ij}=t]
		\\ \le\sum_{\ell,\ell'}\sqrt{\mathbb E[Y_{ij,\ell}^4\mid  T_{ij}=t]\,\mathbb E[Y_{ij,\ell'}^4\mid  T_{ij}=t]}
		\le\kappa_4\Bigl(\sum_\ell\mathfrak r_\ell\Bigr)^2 ,
	\end{multline}
	which is the first inequality in~\eqref{eq:ineq_traceR}. Combining these facts with Lemma~\ref{lem:avg-cosine-L2-full}-(iii), we finally get
	\begin{multline*}
		\mathbb E [h^4_{11}] = \mathbb E \left[\frac{\Psi^4(T_{i1},T_{i'1};\mathcal K_M)}{g^4(T_{i1})g^4(T_{i'1})}\,
		\mathbb E \left(\langle\widetilde Y_{i1},\widetilde Y_{i'1}\rangle^4_{L^2(\mathcal U)}\mid T_{i1},T_{i'1} \right)\right]\\
		\le \kappa_4^2\,\nu^4(\mathcal U)\,c_\rho^4 \iint_{\mathcal T\times \mathcal T}
		\frac{\Psi^4(t,s;\mathcal K_M)}{g^3(t) g^3(s)}\,\mathrm dt\,\mathrm ds
		\leq 2^{10}\,\kappa_4^2\,\nu^4(\mathcal U)\,c_\rho^4\, g^{-6}_{\rm min}\, |\mathcal K_M|^3.
	\end{multline*}
	Using this in~\eqref{eq:H4-to-h4} gives the rate $|\mathcal K_M|^3$ and the constant
	$C_0m_{\max}^6\mathfrak e_1^2$.
	
	\medskip\noindent\emph{Two bounds on the design fluctuation.} Since
	$$
	\|\mathrm b_{ij}\|_{\mathcal H_M}^2
	=\frac{\|\mu(\cdot,T_{ij})\|_{L^2(\mathcal U)}^2}{g(T_{ij})^2}\,\Psi(T_{ij},T_{ij};\mathcal K_M)
	\le \frac{\bar\mu^2}{g_{\min}^2}\,32|\mathcal K_M|
	$$
	and $m_i\le m_{\max}$, we get
	\begin{equation}\label{eq:D-as}
		\|\mathrm D_i\|_{\mathcal H_M}
		\le\sum_{j=1}^{m_i}\|\mathrm b_{ij}\|_{\mathcal H_M}+\mathfrak e_1\|\mathfrak b\|_{\mathcal H_M}
		\le C_{\mathrm D}\,|\mathcal K_M|^{1/2},
		\qquad C_{\mathrm D}=m_{\max}\bigl(\sqrt{32}/g_{\min}+1\bigr)\bar\mu .
	\end{equation}
	Moreover, for any $\mathfrak y \in\mathcal H_M$ with $\|\mathfrak y\|_{\mathcal H_M}=1$ and
	components $\mathfrak y_k$, writing $A(u,t)=\sum_{k\in\mathcal K_M}\mathfrak y_k(u)\phi_k(t)$ and
	using $\int_{\mathcal T}\|A(\cdot,t)\|_{L^2(\mathcal U)}^2\,\mathrm dt=\|\mathfrak y\|_{\mathcal H_M}^2=1$,
	we get
	\[
	\mathbb E\big[\langle \mathfrak y,\mathrm b_{ij}\rangle_{\mathcal H_M}^2\big]
	=\mathbb E\Bigl[\frac{1}{g^2(T_{ij})}\langle A(\cdot,T_{ij}),\mu(\cdot,T_{ij})\rangle_{L^2(\mathcal U)}^2\Bigr]
	\le\frac{\bar\mu^2}{g_{\min}}\int_{\mathcal T}\|A(\cdot,t)\|_{L^2(\mathcal U)}^2\,\mathrm dt
	=\frac{\bar\mu^2}{g_{\min}}.
	\]
	Since the $\mathrm b_{ij}$, $j\le m_i$, are i.i.d.\ given $m_i$,
	$$
	\operatorname{Var}(\langle \mathfrak y,\mathrm D_i\rangle_{\mathcal H_M})
	\le\mathbb E[m_i^2]\,\mathbb E[\langle \mathfrak y,\mathrm b_{i1}\rangle_{\mathcal H_M}^2]
	\le m_{\max}\mathfrak e_1\,\mathbb E[\langle \mathfrak y,\mathrm b_{i1}\rangle_{\mathcal H_M}^2],
	$$ 
	whence
	\begin{equation}\label{eq:D-op}
		\|\mathcal C_M^{\mathrm D}\|_{\mathrm{op}}
		=\sup_{\|\mathfrak y\|=1}\operatorname{Var}\big(\langle \mathfrak y,\mathrm D_i\rangle_{\mathcal H_M}\big)
		\le m_{\max}\,\mathfrak e_1\,\frac{\bar\mu^2}{g_{\min}}=:C_{\mathrm D}',
		\qquad \text{where } \ \mathcal C_M^{\mathrm D}=\operatorname{Cov}(\mathrm D_i).
	\end{equation}
	Also, note that $\operatorname{Trace}(\mathcal C_M^{\mathrm D})=\mathbb E\|\mathrm D_i\|_{\mathcal H_M}^2\le C_{\mathrm D}^2|\mathcal K_M|$
	by \eqref{eq:D-as}.
	
	\medskip\noindent\emph{Bound on $\widetilde H_{2,n}$.} With $\Xi_i,\Xi_{i'}$ independent and
	\eqref{eq:D-as}, we have
	$$
	\langle\mathrm D_i,\mathrm D_{i'}\rangle_{\mathcal H_M}^2\le\|\mathrm D_i\|_{\mathcal H_M}^2\|\mathrm D_{i'}\|_{\mathcal H_M}^2
	\le C_{\mathrm D}^4|\mathcal K_M|^2 ,
	$$
	so, since $\mathrm D_i$ is centered,
	\[
	\mathbb E\big[\widetilde H_{2,n}^4(\Xi_i,\Xi_{i'})\big]
	\le C_{\mathrm D}^4|\mathcal K_M|^2\,\mathbb E[\langle\mathrm D_i,\mathrm D_{i'}\rangle_{\mathcal H_M}^2]
	=C_{\mathrm D}^4|\mathcal K_M|^2\,\operatorname{Trace}\bigl((\mathcal C_M^{\mathrm D})^2\bigr).
	\]
	By \eqref{eq:D-op}, $\operatorname{Trace}((\mathcal C_M^{\mathrm D})^2)
	\le\|\mathcal C_M^{\mathrm D}\|_{\mathrm{op}}\operatorname{Trace}(\mathcal C_M^{\mathrm D})
	\le C_{\mathrm D}'C_{\mathrm D}^2|\mathcal K_M|$, hence
	\begin{equation}\label{eq:H2-bound}
		\mathbb E\big[\widetilde H_{2,n}^4(\Xi_i,\Xi_{i'})\big]\le C_{\mathrm D}'C_{\mathrm D}^6\,|\mathcal K_M|^3.
	\end{equation}
	
	\medskip\noindent\emph{Bound on $\widetilde H_{1,n}$.} Put
	$\mathrm D_{i'}(t):=\sum_{k\in\mathcal K_M}\phi_k(t)\mathrm D_{i'k}\in L^2(\mathcal U)$. Then
	$$
	\widetilde H_{1,n}(\Xi_i,\Xi_{i'})=\sum_{j=1}^{m_i}v_{ij}, \quad \text{ with } \ v_{ij}=\frac{1}{g(T_{ij})}\big\langle\widetilde Y_{ij},\mathrm D_{i'}(T_{ij})\big\rangle_{L^2(\mathcal U)}.
	$$
	So by Jensen's inequality and exchangeability,
	$
	\mathbb E\big[\widetilde H_{1,n}^4(\Xi_i,\Xi_{i'})\big]\le m_{\max}^3\mathfrak e_1\,\mathbb E[v_{11}^4].
	$
	Conditioning on $(T_{i1},\mathrm D_{i'})$, where $\mathrm D_{i'}$ is independent of
	$(T_{i1},\widetilde Y_{i1})$, and applying Assumption~\ref{ass:profile_kurtosis} to
	$\widetilde Y_{i1}$ with direction $\mathrm D_{i'}(T_{i1})$, we get
	\[
	\mathbb E\big[v_{11}^4\mid T_{i1}=t,\mathrm D_{i'}\big]
	\le\frac{\kappa_4}{g^4(t)}\big\langle\mathrm D_{i'}(t),\mathcal R_t\mathrm D_{i'}(t)\big\rangle_{L^2(\mathcal U)}^2
	\le\frac{\kappa_4}{g^4(t)}\,\sup_{s\in\mathcal T } \|\mathcal R_s\|^2_{\rm op } \,\|\mathrm D_{i'}(t)\|_{L^2(\mathcal U)}^4 .
	\]
	Taking expectations, we deduce
	$$
	\mathbb E[v_{11}^4]\le\kappa_4\sup_{s\in\mathcal T } \|\mathcal R_s\|^2_{\rm op } \, g_{\min}^{-3}
	\int_{\mathcal T}\mathbb E\|\mathrm D_{i'}(t)\|^4\,\mathrm dt.
	$$
	Now, by orthonormality
	$\int_{\mathcal T}\|\mathrm D_{i'}(t)\|^2\,\mathrm dt=\|\mathrm D_{i'}\|_{\mathcal H_M}^2$, and by
	the Cauchy-Schwarz inequality in $k$,
	$\|\mathrm D_{i'}(t)\|^2\le\Psi(t,t;\mathcal K_M)\|\mathrm D_{i'}\|_{\mathcal H_M}^2
	\le 32|\mathcal K_M|\,\|\mathrm D_{i'}\|_{\mathcal H_M}^2$.
	Hence,
	\[
	\int_{\mathcal T}\|\mathrm D_{i'}(t)\|^4\,\mathrm dt
	\le 32|\mathcal K_M|\,\|\mathrm D_{i'}\|_{\mathcal H_M}^4
	\le 32\,C_{\mathrm D}^4\,|\mathcal K_M|^3
	\]
	by \eqref{eq:D-as}. Therefore
	\begin{equation}\label{eq:H1-bound}
		\mathbb E[\widetilde H_{1,n}^4]\le  32\,\mathfrak e_1\,\kappa_4\, m_{\max}^3\, g_{\min}^{-3}\,\nu^2(\mathcal U)\,c_\rho^2
		\,C_{\mathrm D}^4\,|\mathcal K_M|^3 .
	\end{equation}
	
	\medskip
	Combining these facts, by the elementary inequality
	\[
	\widetilde H_n^4(\Xi_i,\Xi_{i'})\le 4^3\Bigl(\widetilde H_{0,n}^4(\Xi_i,\Xi_{i'})+\widetilde H^4_{1,n}(\Xi_i,\Xi_{i'})
	+\widetilde H^4_{1,n}(\Xi_{i'},\Xi_i)+\widetilde H_{2,n}^4(\Xi_i,\Xi_{i'})\Bigr),
	\]
	the bounds in \eqref{eq:H4-to-h4}, \eqref{eq:H2-bound} and \eqref{eq:H1-bound} give
	$\mathbb E\big[\widetilde H_n^4(\Xi_i,\Xi_{i'})\big]\le C|\mathcal K_M|^3$, with $C$ as stated,
	since $C_{\mathrm D}^4\lesssim m_{\max}^4\bar\mu^4$ and
	$C_{\mathrm D}'C_{\mathrm D}^6\lesssim m_{\max}^7\mathfrak e_1\bar\mu^8$.
	
	Finally, under the conditions of Proposition~\ref{lem:secondmoment}, both
	$\mathcal C_M^{\mathrm F}=\operatorname{Cov}(\mathrm F_i)$ and $\mathcal C_M^{\mathrm D}$ being
	positive semi-definite,
	$$
	\operatorname{Trace}(\mathcal C_M^2)\ge\operatorname{Trace}((\mathcal C_M^{\mathrm F})^2)
	\ge\mathfrak e_1^2\operatorname{Trace}((\mathcal C^{(1)})^2)\ge c\,|\mathcal K_M|,
	$$
	with $c=\mathfrak e_1^2\|\rho/g\|^2/2>0$ for $|\mathcal K_M|\geq K_0$, by
	Proposition~\ref{lem:secondmoment}. Hence
	$$|\mathcal K_M|^3\le c^{-2}|\mathcal K_M| \big\{\operatorname{Trace}(\mathcal C_M^2) \big\}^2$$ and
	\[
	\mathbb E\big[\widetilde H_n^4(\Xi_i,\Xi_{i'})\big]\le C c^{-2}\,|\mathcal K_M|\,\big\{\operatorname{Trace}(\mathcal C_M^2) \big\}^2
	=C c^{-2}\,|\mathcal K_M|\,  \big\{\mathbb E\big[\widetilde H_n^2(\Xi_i,\Xi_{i'})\big]\big\}^2,
	\]
	which completes the proof.
\end{proof}

\smallskip

\begin{proof}[Proof of Corollary~\ref{corr:test_level}]
	Recall from Step~3 of the proof of Theorem~\ref{prop:dist_tstat} that
	\begin{equation}\label{eq:Tn_W}
		T_n=\delta_n\mathcal S(\mathcal K_M;\mu)+\varsigma_n(\mu)\,W,
		\qquad W=\frac{Q_n-\mathbb E Q_n}{\sqrt{\vv}},\quad \mathbb E W=0,\quad \operatorname{Var}W=1,
	\end{equation}
	and, by~\eqref{eq:v2_trace},
	\begin{equation}\label{eq:varsigma_expr}
		\varsigma_n^2(\mu)=\frac{4n\,\vv_1+2n\,\vv_2/(n-1)}{\mathfrak e_1^2\,V\,|\mathcal K_M|},
		\qquad \vv_2=\operatorname{Trace}(\mathcal C_M^2),\quad V=2\|\overline\rho/g\|^2 .
	\end{equation}
	Here $C$ denotes a constant which may change from line to line, and depends on the mean only
	through upper bounds on $\bar\mu$ and $\|L_\mu\|_{L^2(\nu)}$.
	
	(i) \emph{Level.} Under $H_0$ in~\eqref{eq:H0_full}, $\beta_k=0$ for all $k\in\mathcal K_M$, so
	that $g_1\equiv0$, $\vv_1=0$ and $\Theta_1=0$. Then, by~\eqref{eq:varsigma_expr} and
	Proposition~\ref{lem:secondmoment},
	\[
	\varsigma_n^2(\mu)=\frac{n}{n-1}\cdot
	\frac{|\mathcal K_M|^{-1}\operatorname{Trace}(\mathcal C_M^2)}{\mathfrak e_1^2\|\overline\rho/g\|^2}
	=\frac{n}{n-1}\bigl\{1+\epsilon_n\bigr\},
	\quad
	|\epsilon_n|\le\frac{C_2''}{\mathfrak e_1^2\|\rho/g\|^2}\,|\mathcal K_M|^{-\bar\varkappa/(1+\bar\varkappa)},
	\]
	where we used $\|\overline\rho/g\|\geq\|\rho/g\|>0$. Since $|\sqrt x-1|\le|x-1|$ for $x\geq0$, we
	get $\varsigma_n(\mu)=1+O(|\mathcal K_M|^{-\bar\varkappa/(1+\bar\varkappa)}+n^{-1})$.
	Next, for $\sigma>0$, by the mean value theorem, with $\xi$ between $z$ and $z/\sigma$, so that
	$|\xi|\geq\min\{1,\sigma^{-1}\}|z|$,
	\[
	\sup_{z\in\mathbb R}|\Phi(z/\sigma)-\Phi(z)|
	=\sup_{z\in\mathbb R}|z|\,\phi(\xi)\,|\sigma^{-1}-1|
	\le\frac{\max\{1,\sigma\}}{\sqrt{2\pi e}}\,\frac{|\sigma-1|}{\sigma}
	\le|\sigma-1| \quad \text{ for } \sigma\geq1/2,
	\]
	with $\phi=\Phi'$ and using $\sup_x|x|\phi(x)=(2\pi e)^{-1/2}$. Hence, by the triangle inequality
	and the second part of Theorem~\ref{prop:dist_tstat} with $\Theta_1=0$,
	\begin{multline}
		\sup_z|\mathbb P(T_n\le z)-\Phi(z)|
		\le C'\Bigl\{|\mathcal K_M|^{-1/2}+\Bigl(\frac{|\mathcal K_M|}{n}\Bigr)^{1/2}+\frac{1}{\sqrt n}\Bigr\}
		+|\varsigma_n(\mu)-1|
		\\ \le C''\bigl\{\{|\mathcal K_M|/n\}^{1/2}+|\mathcal K_M|^{-\bar\varkappa/(1+\bar\varkappa)}\bigr\},
	\end{multline}
	since $n^{-1}\le n^{-1/2}\le(|\mathcal K_M|/n)^{1/2}$ and $\Kc^{-1/2}\le\Kc^{-\bar\varkappa/(1+\bar\varkappa)}$ (as $\bar\varkappa\le 1$). When $\varsigma_n(\mu)<1/2$, the last display holds trivially provided $C''\geq2^{(1+\bar\varkappa)/\bar\varkappa}\,C_2''/(\mathfrak e_1^2\|\rho/g\|^2)$,
	because then $|\mathcal K_M|^{-\bar\varkappa/(1+\bar\varkappa)}$ is bounded below. As
	$|\mathcal K_M|\to\infty$ and $|\mathcal K_M|/n\to0$, 	$\mathbb P(T_n>z_{1-\mathfrak a})\to\mathfrak a$ follows from the last display.
	
	(ii) \emph{Consistency.} Let $\mu=\mu_n=\mu_0+r_n\Upsilon$ and write
	$\mathcal S_n=\mathcal S(\mathcal K_M;\mu_n)=r_n^2\mathcal S(\mathcal K_M;\Upsilon)$,
	$V_n=V(\mu_n)$ and $\varsigma_n=\varsigma_n(\mu_n)$. Since $\mu_0$ and $\Upsilon$ satisfy
	Assumption~\ref{ass:holder} and $\{r_n\}$ is bounded, $\bar\mu_n:=\sup_t\|\mu_n(\cdot,t)\|
	\le\bar\mu_0+\sup_n r_n\,\bar\Upsilon$ is bounded uniformly in $n$ (see Step~4 of the proof of
	Theorem~\ref{prop:dist_tstat}), and $\mu_n$ satisfies Assumption~\ref{ass:holder} with
	$\|L_{\mu_n}\|_{L^2(\nu)}\le\|L_{\mu_0}\|_{L^2(\nu)}+\sup_n r_n\,\|L_\Upsilon\|_{L^2(\nu)}$.
	
	\emph{Bounds on $V_n$.} By Proposition~\ref{lem:secondmoment}, $V_n\geq2\|\rho/g\|^2>0$, where
	the lower bound does not depend on the mean. Moreover, since
	$$
	\iint\overline\rho^{\,2}(u,u';t)\,{\rm d}\nu(u)\,{\rm d}\nu(u')
	\le2\|\mathcal R_t\|_{\rm HS}^2+2\|\mu_n(\cdot,t)\|^4_{L^2(\mathcal U)},
	$$
	it holds $V_n\le4g_{\min}^{-2}\{\nu^2(\mathcal U)c_\rho^2+\bar\mu_n^4\}\le C$.
	
	\emph{Bounds on $\vv_1$ and $\vv_2$.} By~\eqref{eq:not_abstract_to_BE},
	$\vv_1=\mathbb E g_1^2=\mathfrak e_1^2\langle\mathfrak b,\mathcal C_M\mathfrak b\rangle_{\HM}
	\le\mathfrak e_1^2\lambda_1\|\mathfrak b\|_{\HM}^2=\mathfrak e_1^2\lambda_1\mathcal S_n$. Next,
	$\vv_2=\operatorname{Trace}(\mathcal C_M^2)\le\lambda_1\operatorname{Trace}(\mathcal C_M)
	=\lambda_1\mathbb E\|\widetilde\Lambda_i\|^2_{\HM}\le\lambda_1\mathbb E\|\Lambda_i\|^2_{\HM}$, and,
	as in Step~4 of the proof of Theorem~\ref{prop:dist_tstat},
	$\mathbb E\|\Lambda_i\|^2_{\HM}\le32\,m_{\max}^2g_{\min}^{-2}\,|\mathcal K_M|\,
	\{\bar\mu_n^2+\nu(\mathcal U)c_\rho\}$. Finally, $\lambda_1\le C$ by~\eqref{eq:lambda1_fourth}.
	Plugging these bounds into~\eqref{eq:varsigma_expr} and using $2n/(n-1)\le4$, we get
	\[
	\varsigma_n^2\le C\Bigl\{\frac{n\,\mathcal S_n}{|\mathcal K_M|}+1\Bigr\}.
	\]
	
	\emph{Conclusion.} Let $a_n=n|\mathcal K_M|^{-1/2}\mathcal S_n
	=n|\mathcal K_M|^{-1/2}r_n^2\mathcal S(\mathcal K_M;\Upsilon)\to\infty$. By the bounds on $V_n$,
	$\delta_n\mathcal S_n=\mathfrak e_1V_n^{-1/2}a_n\geq c\,a_n\to\infty$ for some $c>0$, and
	\[
	\frac{\varsigma_n^2}{(\delta_n\mathcal S_n)^2}
	\le C\,\frac{n\mathcal S_n/|\mathcal K_M|+1}{a_n^2}
	= C\Bigl\{\frac{1}{a_n\,|\mathcal K_M|^{1/2}}+\frac{1}{a_n^2}\Bigr\}\longrightarrow0 .
	\]
	Hence, by~\eqref{eq:Tn_W} and Chebyshev's inequality, for $n$ large enough that
	$\delta_n\mathcal S_n>z_{1-\mathfrak a}$,
	\[
	\mathbb P\bigl(T_n\le z_{1-\mathfrak a}\bigr)
	=\mathbb P\bigl(\varsigma_nW\le z_{1-\mathfrak a}-\delta_n\mathcal S_n\bigr)
	\le\frac{\varsigma_n^2}{(\delta_n\mathcal S_n-z_{1-\mathfrak a})^2}\longrightarrow0,
	\]
	that is, the rejection probability tends to one.
\end{proof}


\subsubsection{Validity of the multiplier bootstrap}\label{app:bootstrap}

Write $\mathcal D_n=\{\Xi_1,\dots,\Xi_n\}$ for the data and let
$\mathbb P^*,\mathbb E^*,\operatorname{Var}^*$ denote probability, expectation and
variance conditional on $\mathcal D_n$ (\emph{i.e.}, over the multipliers
$\zeta_1,\dots,\zeta_n$ only). Let $L_{ii'}=H_n(\Xi_i,\Xi_{i'})$, and define the
symmetric $n\times n$ matrix $\mathbb H=(L_{ii'})_{i,i'}$ with $L_{ii}:=0$. Set
\[
\widehat M_n=\frac{1}{n(n-1)}\sum_{1\leq i\neq i'\leq n}L_{ii'}^2
=\frac{1}{n(n-1)}\operatorname{Trace}(\mathbb H^2)
\]
and\[
R_i=\sum_{i'\neq i}L_{ii'}^2,
\qquad
\Sigma_M^2:=\operatorname{Trace}(\mathcal C_M^2)=\vv_2 .
\]
By Proposition~\ref{lem:secondmoment},
$\Sigma_M^2=\mathfrak e_1^2\|\overline\rho/g\|^2|\mathcal K_M|\{1+o(1)\}$ and
$\Sigma_M^2\geq c\,|\mathcal K_M|$ for $|\mathcal K_M|\geq K_0$, while $\lambda_1\le C$ by
Proposition~\ref{lem:fourth}. Moreover, under $H_0$ in~\eqref{eq:H0_full},
$\operatorname{Var}(Q_n)=\vv=2\Sigma_M^2/\{n(n-1)\}$.

Here, we state and prove a more detailed version of Theorem~\ref{th:bootstrap}.

\begin{theorem}[Bootstrap validity]\label{th:bootstrap_app}
	Let $\zeta_1,\dots,\zeta_n$ be i.i.d., independent of $\mathcal D_n$, with
	$\mathbb E\zeta_1=0$, $\mathbb E\zeta_1^2=1$ and $\mathbb E\zeta_1^4<\infty$, and let
	$\widehat s_n^2=\operatorname{Var}^*(Q_n^*)$.
	Assume the conditions of Corollary~\ref{corr:test_level} hold. Under $H_0$
	in~\eqref{eq:H0_full}:
	\begin{enumerate}
		\item[(i)] $\displaystyle \widehat M_n/\Sigma_M^2\xrightarrow{\ \mathbb P\ }1$ and
		$\widehat s_n^2=2\widehat M_n/\{n(n-1)\}=\operatorname{Var}(Q_n)\{1+o_{\mathbb P}(1)\}$;
		\item[(ii)]
		$\displaystyle
		\sup_{z\in\mathbb R}\bigl|\mathbb P^*(Q_n^*/\widehat s_n\le z)-\Phi(z)\bigr|
		\xrightarrow{\ \mathbb P\ }0$;
		\item[(iii)] the test $\mathbb I\{Q_n>q^*_{1-\mathfrak a}\}$, with $q^*_{1-\mathfrak a}$
		the conditional $(1-\mathfrak a)$-quantile of $Q_n^*$, has asymptotic level $\mathfrak a$.
	\end{enumerate}
	Moreover, under the local alternatives $H_{1,n}$ of
	Corollary~\ref{corr:test_level}-(ii), if
	$n|\mathcal K_M|^{-1/2}r_n^2\mathcal S(\mathcal K_M;\Upsilon)\to\infty$, then
	$\mathbb P(Q_n>q^*_{1-\mathfrak a})\to1$.
\end{theorem}

\begin{proof}[Proof of Theorem~\ref{th:bootstrap_app}]
	For simplicity, we write $\langle \cdot,\cdot\rangle$ instead of
	$\langle \cdot,\cdot\rangle_{\mathcal H_M}$. Throughout, $C$ denotes a constant depending
	only on the constants of Propositions~\ref{lem:secondmoment} and~\ref{lem:fourth}, and on the
	mean only through an upper bound on $\bar\mu$, changing from line to line; we also assume
	$|\mathcal K_M|\geq K_0$.
	
	\medskip\noindent\emph{Preliminary moment bounds.} Writing
	$\Lambda_i=\widetilde\Lambda_i+\mathfrak e_1\mathfrak b$, with
	$\|\mathfrak b\|^2=\mathcal S:=\mathcal S(\mathcal K_M;\mu)$,
	\begin{equation}\label{eq:Hn_decomp}
		L_{ii'}=\langle\Lambda_i,\Lambda_{i'}\rangle
		=\widetilde H_n(\Xi_i,\Xi_{i'})+g_1(\Xi_i)+g_1(\Xi_{i'})+\mathfrak e_1^2\,\mathcal S,
		\qquad g_1(\Xi_i)=\mathfrak e_1\langle\widetilde\Lambda_i,\mathfrak b\rangle .
	\end{equation}
	Since $\widetilde H_n$ is degenerate and $g_1$ is centered, the cross terms have zero mean
	and
	\begin{equation}\label{eq:Mn_mean}
		\mathbb E[L_{12}^2]=\Sigma_M^2+2\vv_1+\mathfrak e_1^4\,\mathcal S^2,
		\qquad \vv_1=\mathfrak e_1^2\langle\mathfrak b,\mathcal C_M\mathfrak b\rangle\le\mathfrak e_1^2\lambda_1\mathcal S .
	\end{equation}
	Next, $\mathbb E[L_{12}^2\mid\Xi_1]=\langle\Lambda_1,\mathbb E[\Lambda_2\otimes\Lambda_2]\Lambda_1\rangle
	=\langle\Lambda_1,(\mathcal C_M+\mathfrak e_1^2\mathfrak b\otimes\mathfrak b)\Lambda_1\rangle
	\le(\lambda_1+\mathfrak e_1^2\mathcal S)\|\Lambda_1\|^2$, so that, by~\eqref{eq:Lambda4},
	\begin{equation}\label{eq:cond_L2}
		\mathbb E\bigl[(\mathbb E[L_{12}^2\mid\Xi_1])^2\bigr]
		\le(\lambda_1+\mathfrak e_1^2\mathcal S)^2\,\mathbb E\|\Lambda_1\|^4\le C|\mathcal K_M|^2\le C\Sigma_M^4 .
	\end{equation}
	Under $H_0$, $\mathfrak b=0$, hence $L_{ii'}=\widetilde H_n(\Xi_i,\Xi_{i'})$, and
	Proposition~\ref{lem:fourth} gives
	\begin{equation}\label{eq:L4}
		\mathbb E[L_{12}^4]=\mathbb E[\widetilde H_n^4]\le C|\mathcal K_M|\,\Sigma_M^4 .
	\end{equation}
	
	\medskip\noindent\emph{(i) Consistency of the bootstrap variance.}
	Under $H_0$, $\widehat M_n$ is a $U$-statistic of order $2$ with symmetric kernel $L_{12}^2\ge0$
	and mean $\mathbb E[L_{12}^2]=\Sigma_M^2$ by~\eqref{eq:Mn_mean}. By elementary calculations,
	\[
	\operatorname{Var}(\widehat M_n)
	=\frac{4(n-2)}{n(n-1)}\,\mathcal V_1+\frac{2}{n(n-1)}\,\mathcal V_2,
	\qquad
	\mathcal V_1=\operatorname{Var}\bigl(\mathbb E[L_{12}^2\mid\Xi_1]\bigr),\quad
	\mathcal V_2=\operatorname{Var}(L_{12}^2).
	\]
	By~\eqref{eq:cond_L2}, $\mathcal V_1\le C\Sigma_M^4$, and by~\eqref{eq:L4},
	$\mathcal V_2\le\mathbb E[L_{12}^4]\le C|\mathcal K_M|\Sigma_M^4$, whence
	$\operatorname{Var}(\widehat M_n)/\Sigma_M^4\le C\{n^{-1}+|\mathcal K_M|n^{-2}\}\to0$. By
	Chebyshev's inequality, $\widehat M_n/\Sigma_M^2\xrightarrow{\ \mathbb P\ }1$. Finally, given
	$\mathcal D_n$, using $L_{ii'}=L_{i'i}$ and $L_{ii}=0$,
	\[
	Q_n^*=\sum_{1\le i<i'\le n}a_{ii'}\,\zeta_i\zeta_{i'},
	\qquad a_{ii'}=\frac{2L_{ii'}}{n(n-1)},
	\]
	is a zero-diagonal quadratic form in the i.i.d.\ multipliers, with
	$\mathbb E^*[Q_n^*]=0$ and
	\begin{multline}
		\widehat s_n^2=\operatorname{Var}^*(Q_n^*)=\sum_{i<i'}a_{ii'}^2
		=\frac{4}{n^2(n-1)^2}\times \frac12\operatorname{Trace}(\mathbb H^2)
		\\ =\frac{2\widehat M_n}{n(n-1)}
		=\frac{2\Sigma_M^2}{n(n-1)}\{1+o_{\mathbb P}(1)\}
		=\operatorname{Var}(Q_n)\{1+o_{\mathbb P}(1)\}.
	\end{multline}
	
	\medskip\noindent\emph{(ii) Conditional Central Limit Theorem.}
	By \cite[Theorem~2.1]{dJ1987}, $Q_n^*/\widehat s_n\xrightarrow{d}\mathcal N(0,1)$ whenever the
	array $(a_{ii'})$ satisfies the conditions of that result. For a zero-diagonal quadratic
	form (`clean' in the terminology of \cite{dJ1987}), these conditions are scale invariant and
	are implied by
	\begin{equation}\label{eq:dJ_cond}
		\mathsf m_n:=\frac{\max_{1\le i\le n}R_i}{\operatorname{Trace}(\mathbb H^2)}\longrightarrow0,
		\qquad
		\mathsf t_n:=\frac{\operatorname{Trace}(\mathbb H^4)}{[\operatorname{Trace}(\mathbb
			H^2)]^2}\longrightarrow0 .
	\end{equation}
	Indeed, the first condition is the negligibility of each row of the array, and the
	standardized fourth cumulant of $Q_n^*$ equals $12\,\mathsf t_n$ plus terms involving the
	law of $\zeta_1$, which are bounded by $C(\mathbb E\zeta_1^4)^2\,\mathsf m_n$. Hence checking
	\eqref{eq:dJ_cond} suffices for every multiplier law with the stated conditions.
	If~\eqref{eq:dJ_cond} holds in $\mathbb P$-probability, then (ii) follows by the
	subsequence principle. Along any subsequence, extract a further one on which
	\eqref{eq:dJ_cond} holds almost surely. Fix a realization $\omega$ of $\mathcal D_n$ in the
	corresponding probability-one event; then the triangular array
	$\{a_{ii'}\}=\{2L_{ii'}(\omega)/[n(n-1)]\}$ is deterministic and, as a sequence in $n$,
	satisfies \eqref{eq:dJ_cond} along the subsequence. Hence \cite[Theorem~2.1]{dJ1987}
	applies to the conditional law of $Q_n^*/\widehat s_n$ given $\mathcal D_n=\omega$, yielding
	$Q_n^*/\widehat s_n\xrightarrow{d}\mathcal N(0,1)$ conditionally along the subsequence. Since
	$\Phi$ is continuous, it follows that
	$\sup_{z}|\mathbb P^*(Q_n^*/\widehat s_n\le z)-\Phi(z)|\rightarrow 0$ for almost every such
	$\omega$, \emph{i.e.}, almost surely along the subsequence. By the subsequence
	characterization of convergence in probability, the convergence holds in
	$\mathbb P$-probability along the full sequence.
	
	For checking~\eqref{eq:dJ_cond}, first note that by (i),
	$\operatorname{Trace}(\mathbb H^2)=n(n-1)\widehat M_n=n(n-1)\Sigma_M^2\{1+o_{\mathbb P}(1)\}$.
	
	\medskip

	\emph{The term $\mathsf t_n$.} Expand
	$\operatorname{Trace}(\mathbb H^4)=\sum_{i_1,i_2,i_3,i_4}L_{i_1i_2}L_{i_2i_3}L_{i_3i_4}L_{i_4i_1}$,
	where consecutive indices are distinct since $L_{ii}=0$. For four distinct indices,
	integrating out $\widetilde\Lambda_{i_2}$ and $\widetilde\Lambda_{i_4}$ using
	$\mathbb E[\widetilde\Lambda\otimes\widetilde\Lambda]=\mathcal C_M$,
	\[
	\mathbb E\bigl[L_{i_1i_2}L_{i_2i_3}L_{i_3i_4}L_{i_4i_1}\bigr]
	=\mathbb E\bigl[\langle\widetilde\Lambda_{i_1},\mathcal C_M\widetilde\Lambda_{i_3}\rangle^2\bigr]
	=\operatorname{Trace}(\mathcal C_M^4)\le\lambda_1^2\Sigma_M^2,
	\]
	by Proposition~\ref{lem:fourth}. The terms with $i_1=i_3$ or $i_2=i_4$, but not both, are at
	most $2n^3$ and have expectation bounded by
	$\mathbb E[L_{12}^2L_{13}^2]=\mathbb E[(\mathbb E[L_{12}^2\mid\Xi_1])^2]\le C\Sigma_M^4$
	by~\eqref{eq:cond_L2}; the terms with $i_1=i_3$ and $i_2=i_4$ are at most $n^2$ and equal
	$\mathbb E[L_{12}^4]\le C|\mathcal K_M|\Sigma_M^4$ by~\eqref{eq:L4}. Hence
	\[
	\frac{\mathbb E[\operatorname{Trace}(\mathbb H^4)]}{n^2(n-1)^2\Sigma_M^4}
	\le\frac{n^4\lambda_1^2\Sigma_M^2+Cn^3\Sigma_M^4+Cn^2|\mathcal K_M|\Sigma_M^4}{n^2(n-1)^2\Sigma_M^4}
	\le C\Bigl\{\frac{1}{|\mathcal K_M|}+\frac1n+\frac{|\mathcal K_M|}{n^2}\Bigr\}\longrightarrow0 .
	\]
	As $\operatorname{Trace}(\mathbb H^4)\ge0$, Markov's inequality and the concentration of
	$\operatorname{Trace}(\mathbb H^2)$ yield $\mathsf t_n\xrightarrow{\ \mathbb P\ }0$.
	
	\medskip
	
	\emph{The term $\mathsf m_n$.} Conditionally on $\Xi_i$, the $n-1$ summands of $R_i$ are
	i.i.d., so, by~\eqref{eq:cond_L2} and~\eqref{eq:L4},
	$$
	\mathbb E[R_i^2]\le(n-1)\mathbb E[L_{12}^4]+(n-1)^2\mathbb E[(\mathbb
	E[L_{12}^2\mid\Xi_1])^2]\le C\{n|\mathcal K_M|+n^2\}\Sigma_M^4 .
	$$
	Therefore, using $\max_iR_i\le(\sum_iR_i^2)^{1/2}$ and Jensen's inequality,
	$$
	\mathbb E[\max_iR_i]\le\Bigl(\sum_i\mathbb E[R_i^2]\Bigr)^{1/2}\le C\{n^2|\mathcal K_M|+n^3\}^{1/2}\Sigma_M^2,
	$$
	while $\operatorname{Trace}(\mathbb H^2)\ge n(n-1)\Sigma_M^2/2$ with probability tending
	to one by (i). Hence
	$$
	\mathsf m_n=O_{\mathbb P}\bigl(|\mathcal K_M|^{1/2}n^{-1}+n^{-1/2}\bigr)=o_{\mathbb P}(1).
	$$
	This establishes \eqref{eq:dJ_cond} in probability and proves (ii).
	
	\medskip\noindent\emph{(iii) Level.} Under $H_0$, $\mathbb E Q_n=0$, and by
	Theorem~\ref{prop:dist_tstat} (second part, with $\Theta_1=0$) and
	identity~\eqref{eq:Tn_W}, $W=Q_n/\sqrt{\vv}$ satisfies
	$\sup_x|\mathbb P(W\le x)-\Phi(x)|\to0$. Since $\widehat s_n$ is $\mathcal D_n$-measurable,
	hence constant under $\mathbb P^*$, the conditional quantile of $Q_n^*$ factorizes as
	$q^*_{1-\mathfrak a}=\widehat s_n\,\widetilde q_{1-\mathfrak a}$, where
	$\widetilde q_{1-\mathfrak a}$ is the conditional $(1-\mathfrak a)$-quantile of
	$Q_n^*/\widehat s_n$. By (ii) and the continuity and strict monotonicity of $\Phi$,
	$\widetilde q_{1-\mathfrak a}\xrightarrow{\ \mathbb P\ }z_{1-\mathfrak a}$, and by (i),
	$\widehat s_n/\sqrt{\vv}\xrightarrow{\ \mathbb P\ }1$. Hence
	$X_n:=q^*_{1-\mathfrak a}/\sqrt{\vv}\xrightarrow{\ \mathbb P\ }z_{1-\mathfrak a}$. For any
	$\epsilon>0$,
	\[
	\bigl|\mathbb P(Q_n>q^*_{1-\mathfrak a})-\mathbb P(W>z_{1-\mathfrak a})\bigr|
	\le\mathbb P(|X_n-z_{1-\mathfrak a}|>\epsilon)
	+\mathbb P(z_{1-\mathfrak a}-\epsilon<W\le z_{1-\mathfrak a}+\epsilon),
	\]
	and the right-hand side tends to $\Phi(z_{1-\mathfrak a}+\epsilon)-\Phi(z_{1-\mathfrak a}-\epsilon)
	\le2\epsilon$. Since $\epsilon$ is arbitrary and $\mathbb P(W>z_{1-\mathfrak a})\to\mathfrak a$,
	we get $\mathbb P(Q_n>q^*_{1-\mathfrak a})\to\mathfrak a$. Note that this argument does not
	require the dependence between $q^*_{1-\mathfrak a}$ and $Q_n$ to be controlled.
	
	\medskip\noindent\emph{Consistency.} Let $\mu=\mu_n$ and
	$\mathcal S_n=\mathcal S(\mathcal K_M;\mu_n)$, and let
	$a_n=n|\mathcal K_M|^{-1/2}\mathcal S_n\to\infty$. As in the proof of
	Corollary~\ref{corr:test_level}, $\bar\mu_n$ is bounded, $\lambda_1\le C$ and
	$\Sigma_M^2\le C|\mathcal K_M|$. By the conditional Chebyshev inequality,
	$\mathbb P^*(Q_n^*>\widehat s_n/\sqrt{\mathfrak a})\le\mathfrak a$, so that
	$q^*_{1-\mathfrak a}\le\widehat s_n/\sqrt{\mathfrak a}$. By~\eqref{eq:Mn_mean},
	\[
	\mathbb E[\widehat s_n^2]=\frac{2\,\mathbb E[L_{12}^2]}{n(n-1)}
	\le\frac{C}{n^2}\bigl\{|\mathcal K_M|+\mathcal S_n+\mathcal S_n^2\bigr\},
	\]
	and, by Markov's inequality,
	$\widehat s_n=O_{\mathbb P}\bigl(n^{-1}\{|\mathcal K_M|^{1/2}+\mathcal S_n^{1/2}+\mathcal S_n\}\bigr)$.
	Hence
	\[
	\frac{q^*_{1-\mathfrak a}}{\mathfrak e_1^2\,\mathcal S_n}
	=O_{\mathbb P}\Bigl(\frac{|\mathcal K_M|^{1/2}}{n\mathcal S_n}+\frac{1}{(n\mathcal S_n)^{1/2}}+\frac1n\Bigr)
	=O_{\mathbb P}\Bigl(\frac1{a_n}+\frac{1}{(a_n|\mathcal K_M|^{1/2})^{1/2}}+\frac1n\Bigr)
	=o_{\mathbb P}(1).
	\]
	On the other hand, $\mathbb E Q_n=\mathfrak e_1^2\mathcal S_n$ by Lemma~\ref{lem:deg_Ustat2},
	and, as in the proof of Corollary~\ref{corr:test_level},
	$$
	\operatorname{Var}(Q_n)=4\vv_1/n+2\vv_2/\{n(n-1)\}\le C\{\mathcal S_n/n+|\mathcal K_M|/n^2\},
	$$
	so that
	\[
	\frac{\operatorname{Var}(Q_n)}{(\mathbb E Q_n)^2}
	\le C\Bigl\{\frac{1}{a_n|\mathcal K_M|^{1/2}}+\frac{1}{a_n^2}\Bigr\}\longrightarrow0,
	\]
	and $Q_n/(\mathfrak e_1^2\mathcal S_n)\xrightarrow{\ \mathbb P\ }1$ by Chebyshev's inequality.
	Hence
	$$
	\mathbb P(Q_n>q^*_{1-\mathfrak a})\geq
	\mathbb P\bigl(Q_n>\mathfrak e_1^2\mathcal S_n/2,\ q^*_{1-\mathfrak a}<\mathfrak e_1^2\mathcal S_n/2\bigr)\to1,
	$$
	and the proof is complete.
\end{proof}


\subsubsection{Proofs of the technical lemmas}\label{sec:technic_proof}

Here, we provide the proofs of Lemmas~\ref{lem:avg-cosine-L2-full}, \ref{lem:order-Gh2} and  \ref{app_lem:CM}.

\smallskip

\begin{proof}[Proof of Lemma~\ref{lem:avg-cosine-L2-full}]
	(i) The first equality follows from~\eqref{eq:id_basis_P}. Next, $c_1^2=1$ and, for
	$k\ge2$, $c_k^2=1+\psi_k$ with $\psi_k(t)=\cos(2(k-1)\pi t)$. Hence
	$|\mathcal K_M|^{-1}\sum_{k\in\mathcal K_M}\phi_k^2-1
	=|\mathcal K_M|^{-1}\sum_{k\in\mathcal K_M,\,k\ge2}\psi_k$. The $\psi_k$, $k\ge2$, are pairwise
	orthogonal in $L^2(\mathcal T)$ with $\|\psi_k\|^2_{L^2(\mathcal T)}=1/2$, so the squared norm
	equals $n_1/(2|\mathcal K_M|^2)$, and $n_1\le|\mathcal K_M|$, with equality if and only if
	$1\notin\mathcal K_M$.
	
	(ii) By~\eqref{eq:id_basis_P}, it suffices to bound the $L^2(\mathcal T)$-norm of
	$|\mathcal K_M|^{-1}\sum_{k\in\mathcal K_M}\phi_k^2-1$. In $\mathcal B_{r,K}$, the elements
	$\phi_k$, $k\le r+1$, are polynomials of degree at most $r$, while for $k\ge r+2$, $\phi_k$ is
	obtained by orthogonalizing $c_{k-r}$. We decompose
	\begin{multline*}
		|\mathcal K_M|^{-1}\!\sum_{k\in\mathcal K_M}\phi_k^2-1
		=|\mathcal K_M|^{-1}\!\!\sum_{k\in\mathcal K_M,\,k\le r+1}\!\!(\phi_k^2-1)
		+|\mathcal K_M|^{-1}\!\!\sum_{k\in\mathcal K_M,\,k\ge r+2}\!\!(c_{k-r}^2-1)\\
		+|\mathcal K_M|^{-1}\!\!\sum_{k\in\mathcal K_M,\,k\ge r+2}\!\!(\phi_k^2-c_{k-r}^2)
		=:\mathsf R_1+\mathsf R_2+\mathsf R_3.
	\end{multline*}
	Since $\|\phi_k\|_\infty^2\le5$ for $k\le3$, $\|\mathsf R_1\|_\infty\le4(r+1)/|\mathcal K_M|
	\le12\,|\mathcal K_M|^{-1/2}$. As in (i), the frequencies $k-r-1\ge1$ in $\mathsf R_2$ are distinct,
	so $\|\mathsf R_2\|_{L^2(\mathcal T)}\le\{2|\mathcal K_M|\}^{-1/2}$.
	
	For $\mathsf R_3$, let first $r=2$ and $k\ge4$, with $d$, $z=(k-1)/2$ and $\mathfrak s_4$ as in the
	proof of Lemma~\ref{lem:uniform_bound}. By the tail form~\eqref{eq:phi_tail}, the coefficient
	of $c_{k-2}$ in $\phi_k$ is $\sqrt{A_d/A_{d-1}}$ (resp.\ $\sqrt{B_d/B_{d-1}}$ for odd $k$), so
	that, using $1-\sqrt x\le1-x$ for $x\in[0,1]$ and
	$\mathfrak s_4(x)\ge(2x+3)/(6x^4)$,
	\[
	\|\phi_k-c_{k-2}\|^2_{L^2(\mathcal T)}=2\Bigl(1-\sqrt{A_d/A_{d-1}}\Bigr)
	\le\frac{2a_d^2}{A_{d-1}}=\frac{2(z-1)^{-4}}{\mathfrak s_4(z-1)}\le\frac{12}{2z+1}=\frac{12}{k}.
	\]
	Hence, by Lemma~\ref{lem:uniform_bound},
	$$\|\phi_k^2-c_{k-2}^2\|_{L^2(\mathcal T)}
	\le\|\phi_k-c_{k-2}\|_{L^2(\mathcal T)}\,(\|\phi_k\|_\infty+\|c_{k-2}\|_\infty)
	\le10\sqrt6\,k^{-1/2}.$$ For $r=1$, by~\eqref{eq:aug_t}, each $\phi_k^{(1)}$, $k\ge3$, either
	equals $c_{k-1}$ or equals $\phi_{k+1}$, which is obtained from $c_{k-1}$; the same bound
	holds. Since the indices in $\mathcal K_M$ are distinct, it holds 
	$$
	\sum_{k\in\mathcal K_M,\,k\ge r+2}k^{-1/2}\le\sum_{j=1}^{|\mathcal K_M|}j^{-1/2}
	\le2|\mathcal K_M|^{1/2},
	$$ 
	and thus
	$\|\mathsf R_3\|_{L^2(\mathcal T)}\le20\sqrt6\,|\mathcal K_M|^{-1/2}$. Collecting the three bounds
	proves (ii).
	
	(iii) By the Cauchy--Schwarz inequality,
	$$
	\Psi^2(t,s;\mathcal K_M)\le\Psi(t,t;\mathcal K_M)\Psi(s,s;\mathcal K_M)
	\le\{|\mathcal K_M|\sup_k\|\phi_k\|_\infty^2\}^2.
	$$ 
	Hence
	\[
	\iint_{\mathcal T\times\mathcal T}\Psi^4
	\le|\mathcal K_M|^2\sup_k\|\phi_k\|_\infty^4\iint_{\mathcal T\times\mathcal T}\Psi^2
	=|\mathcal K_M|^3\sup_k\|\phi_k\|_\infty^4,
	\]
	since $\iint\Psi^2(t,s;\mathcal K_M)\,{\rm d}s\,{\rm d}t
	=\int_{\mathcal T}\sum_{k\in\mathcal K_M}\phi_k^2(t)\,{\rm d}t=|\mathcal K_M|$
	by~\eqref{eq:id_basis_P}. Finally, $\sup_k\|\phi_k\|_\infty\le4\sqrt2$ for $r\in\{1,2\}$
	by Lemma~\ref{lem:uniform_bound}, and $\sup_k\|c_k\|_\infty=\sqrt2$, so that
	$\sup_k\|\phi_k\|_\infty^4\le(4\sqrt2)^4=2^{10}$.
\end{proof}

\medskip

\begin{proof}[Proof of Lemma~\ref{lem:order-Gh2}]
	(i) We write
	$$
	\mathfrak G(t)=\int_{\mathcal T}\Psi^2(t,s; \mathcal K_M)h_1(s)\,{\rm d}s = h_1(t) \int_{\mathcal T}\Psi^2(t,s; \mathcal K_M)\,{\rm d}s + \mathfrak D (t),
	$$
	with $ \mathfrak D (t) = \int_{\mathcal T}\Psi^2(t,s; \mathcal K_M)\{h_1(s)-h_1(t)\}\,{\rm d}s$.
	By~\eqref{eq:id_basis_P}, the Cauchy-Schwarz inequality and
	Lemma~\ref{lem:avg-cosine-L2-full}-(i), we get
	\begin{multline}
		\left||\mathcal K_M|^{-1}	\int_{\mathcal T} \{\mathfrak G(t) - \mathfrak D(t)\}\,h_2(t)\,{\rm d}t -\int_{\mathcal T} h_1(t)h_2(t)\,{\rm d} t\right|
		\\ \leq \|h_1h_2\|_{L^2(\mathcal T)}\Big\|\,|\mathcal K_M|^{-1}\!\sum_{k\in\mathcal K_M}\phi_k^2-1\Big\|_{L^2(\mathcal T)}
		= \frac{\|h_1h_2\|_{L^2(\mathcal T)}}{\sqrt{2|\mathcal K_M|}}.
	\end{multline}
	Next, we can write
	$$
	\left|\int_{\mathcal T}\! \mathfrak D(t)\,h_2(t)\,{\rm d}t\right|\le\|h_2\|_\infty\! \int_{\mathcal T}|\mathfrak D(t)|\,{\rm d}t
	\le\|h_2\|_\infty\iint_{\mathcal T \times \mathcal T}\!\! \Psi^2(t,s;\mathcal K_M)\,|h_1(s)-h_1(t)|\,{\rm d}s\,{\rm d}t.
	$$
	Since $2\cos(f\pi t)\cos(f\pi s)=\cos(f\pi(t-s))+\cos(f\pi(t+s))$, we can write
	\[
	\sum_{k\in\mathcal K_M}\phi_k(t)\phi_k(s)=C_{F}(t-s)+C_{F}(t+s),\qquad
	C_{F}(x):=\sum_{f\in F}\cos(f\pi x),
	\]
	so $\Psi^2(t,s; \mathcal K_M) \le 2C_{F}(t-s)^2+2C_{F}(t+s)^2$.
	Note that $C_{F}$ is even, $2$-periodic (\emph{i.e.}, $C_{F}(x)=C_{F}(x+2)$) and, for $x\notin2\mathbb Z$,
	\begin{equation}\label{eq:aux_CM2}
		|C_{F} (x)|\le \min \big\{ | {\mathcal K_M} | ,|\sin(\pi x/2)|^{-1} \big\}.
	\end{equation}
	The first bound is trivial. For the second, note that, with $\iota^2 = -1$ it holds
	$|\exp (\iota \theta)-1|=2|\sin(\theta/2)|$ and, $F$ being a set of $|\mathcal K_M|$
	consecutive integers, using the expression of a sum of a geometric series we get
	\[
	|C_{F}(x)|\le\Bigg| \sum_{f\in F}\exp (\iota f\pi x)\Bigg|
	=\Bigg|\frac{\exp (\iota| {\mathcal K_M} |\pi x) -1}{\exp (\iota\pi x)-1}\Bigg|
	=\frac{|\sin(| {\mathcal K_M} |\pi x/2)|}{|\sin(\pi x/2)|}\le\frac{1}{|\sin(\pi x/2)|}.
	\]
	Then, since $t-s$ has density $1-|x|\leq1$ on $[-1,1]$ and
	$\int_{-1}^1\cos(f\pi x)\cos(f'\pi x)\,{\rm d}x=\delta_{ff'}$ for $f,f'\geq1$, it holds
	$$
	\iint_{\mathcal T\times \mathcal T }C^2_{F}(t-s)\,{\rm d}s\,{\rm d}t\le\int_{-1}^1 C^2_{F}(x)\,{\rm d}x=|{\mathcal K_M}|.
	$$
	Splitting the inner $s$-integral at $|t-s|=\delta$, using the Hölder condition on
	$\{|t-s|\leq\delta\}$, and $|h_1(s)-h_1(t)|\leq2\|h_1\|_\infty$ together with~\eqref{eq:aux_CM2}
	and
	$$
	\int_{\delta<|x|<1}\sin^{-2}(\pi x/2)\,{\rm d}x=\frac4\pi\cot(\pi\delta/2)
	\le\frac{8}{\pi^2\delta}
	$$
	on $\{|t-s|>\delta\}$, we get
	\[
	\iint_{\mathcal T\times \mathcal T}C^2_{F}(t-s)|h_1(s)-h_1(t)|\,{\rm d}s\,{\rm d}t
	\le C_1 \delta^{\gamma}\,|{\mathcal K_M}|+\frac{16\|h_1\|_\infty}{\pi^2\delta}.
	\]
	Next, bounding $|h_1(s)-h_1(t)|\le2\|h_1\|_\infty$ and using that $C_{F}$ is even,
	$2$-periodic and symmetric about $1$, and that $t+s$ has density $w(y)=\min(y,2-y)$ on $[0,2]$,
	\[
	\iint_{\mathcal T\times \mathcal T}C^2_{F} (t+s)\,{\rm d}s\,{\rm d}t
	=\int_0^2 C^2_{F}(y)\,w(y)\,{\rm d}y
	=2\int_0^1 C^2_{F}(y)\,y\,{\rm d}y.
	\]
	Using~\eqref{eq:aux_CM2} and the inequality $\sin(\pi y/2)\ge y$ on $[0,1]$, splitting the
	integral at $y=1/|{\mathcal K_M}|$ we get
	\[
	2\int_0^1 C^2_{F} (y)\,y\,{\rm d}y
	\le 2\int_0^{1/|{\mathcal K_M}|}|{\mathcal K_M}|^{2} y\,{\rm d}y+2\int_{1/|{\mathcal K_M}|}^1\frac{{\rm d}y}{y}
	=1+2\log |{\mathcal K_M}|,
	\]
	hence $\iint_{\mathcal T\times \mathcal T}C_{F}(t+s)^2|h_1(s)-h_1(t)|\,{\rm d}s\,{\rm d}t\le2\|h_1\|_\infty(1+2\log |{\mathcal K_M}|)$.
	Combining these facts we get
	\begin{multline}
		\iint_{\mathcal T^2}\Psi^2(t,s; \mathcal K_M) |h_1(s)-h_1(t)|\,{\rm d}s\,{\rm d}t
		\le 2C_1 \delta^{\gamma}\,|{\mathcal K_M}| + \frac{32\|h_1\|_\infty}{\pi^2\delta}
		\\ + 4\|h_1\|_\infty(1+2\log |{\mathcal K_M}|),
	\end{multline}	
	from which (i) follows.
	
	(ii) Let $\mathcal D_{r,K}=(d_1,\ldots,d_{K+r})=(c_1,p_1,\ldots,p_r,c_2,\ldots,c_K)$ be the
	dictionary~\eqref{eq:orth_system}, with $c_1\equiv1$ and $c_k(t)=\sqrt2\cos((k-1)\pi t)$,
	$k\geq2$. Gram-Schmidt orthogonalization preserves nested spans:
	$\mathcal S_j:=\operatorname{span}\{\phi_1,\ldots,\phi_j\}=\operatorname{span}\{d_1,\ldots,d_j\}$.
	Let $\Pi_j$ denote the orthogonal projector onto $\mathcal S_j$ ($\Pi_0=0$). For
	$\mathcal K_M=\{a+1,\ldots,b\}$, the projector onto $\operatorname{span}\{\phi_k:k\in\mathcal K_M\}$
	is $\Pi_b-\Pi_a$, whose kernel is $\Psi(\cdot,\cdot;\mathcal K_M)$. Let $m(j)=\max(j-r,1)$ and
	$\mathcal S_j^{\rm c}=\operatorname{span}\{c_2,\ldots,c_{m(j)}\}$ (empty if $m(j)=1$). Since
	$\mathcal S_j^{\rm c}\subset\mathcal S_j$ and the dictionary is linearly independent,
	$\mathcal S_j=\mathcal S_j^{\rm c}\oplus\mathcal E_j$ with
	$\dim\mathcal E_j=j-m(j)+1\leq q$. Hence $\Pi_j=\Pi_j^{\rm c}+Q_j$, with $\Pi_j^{\rm c}$ and
	$Q_j$ the orthogonal projectors onto $\mathcal S_j^{\rm c}$ and $\mathcal E_j$, and
	\[
	\Psi(t,s;\mathcal K_M)=\Psi^{\rm c}(t,s)+P(t,s),
	\]
	where $\Psi^{\rm c}$ is the kernel of $\Pi_b^{\rm c}-\Pi_a^{\rm c}$, that is, of the projector
	onto $\operatorname{span}\{c_k:k\in\mathcal K'\}$ with $\mathcal K'=\{m(a)+1,\ldots,m(b)\}$, and
	$P$ is the kernel of $Q_b-Q_a$. The frequencies $\{k-1:k\in\mathcal K'\}$ are consecutive
	positive integers, $|\mathcal K_M|-q\leq|\mathcal K'|\leq|\mathcal K_M|$, so that
	$|\mathcal K'|\geq|\mathcal K_M|/2$, and
	$\|P\|_{L^2(\mathcal T\times\mathcal T)}=\|Q_b-Q_a\|_{\rm HS}\leq\|Q_b\|_{\rm HS}+\|Q_a\|_{\rm HS}\leq2\sqrt q$.
	
	Let $\mathcal R'$ be the remainder of part (i) for the cosine system indexed by $\mathcal K'$.
	Since $\iint(\Psi^{\rm c})^2=|\mathcal K'|$, we have $|\mathcal R'|\leq2\|h_1\|_\infty\|h_2\|_\infty$,
	and
	\[
	|\mathcal K_M|^{-1}\!\iint(\Psi^{\rm c})^2(t,s)h_1(s)h_2(t)\,{\rm d}s\,{\rm d}t-\int_{\mathcal T}h_1h_2
	=\mathcal R'+\Big(\frac{|\mathcal K'|}{|\mathcal K_M|}-1\Big)\Big(\int_{\mathcal T}h_1h_2+\mathcal R'\Big),
	\]
	where the last term is bounded by $3q\|h_1\|_\infty\|h_2\|_\infty/|\mathcal K_M|$. Moreover,
	$\Psi^2-(\Psi^{\rm c})^2=2\Psi^{\rm c}P+P^2$ and, by the Cauchy-Schwarz inequality,
	\begin{multline}
		|\mathcal K_M|^{-1}\iint|\Psi^2-(\Psi^{\rm c})^2|\,|h_1(s)h_2(t)|\,{\rm d}s\,{\rm d}t
		\leq\frac{\|h_1\|_\infty\|h_2\|_\infty}{|\mathcal K_M|}\Big(2|\mathcal K'|^{1/2}\,2\sqrt q+4q\Big)
		\\ 	\leq\|h_1\|_\infty\|h_2\|_\infty\Big(\frac{4\sqrt q}{\sqrt{|\mathcal K_M|}}+\frac{4q}{|\mathcal K_M|}\Big).
	\end{multline}
	Finally, part (i) applied to $\mathcal K'$, with $|\mathcal K_M|/2\leq|\mathcal K'|\leq|\mathcal K_M|$,
	bounds $|\mathcal R'|$ by the first two terms in (ii). Collecting the bounds proves (ii).
	The last statement follows by taking $\delta=|\mathcal K_M|^{-1/(1+\gamma)}$, since
	$\gamma/(1+\gamma)\leq1/2$, $\|h_1h_2\|_{L^2(\mathcal T)}\leq\|h_1\|_\infty\|h_2\|_\infty$ and
	$(1+2\log|\mathcal K_M|)/|\mathcal K_M|\lesssim|\mathcal K_M|^{-1/2}$.
\end{proof}

\smallskip

\begin{proof}[Proof of Lemma~\ref{app_lem:CM}]
	Write $Y_{ij}(u)=\mu(u,T_{ij})+\widetilde Y_{ij}(u)$ and, for $k\in\mathcal K_M$,
	$\psi_k(u,t)=\phi_k(t)\mu(u,t)/g(t)$, so that $\mathbb E[\psi_k(u,T_{ij})]=\beta_k(u)$.
	Then $\widetilde\Lambda_i=\mathrm F_i+\mathrm D_i$, where
	\[
	\mathrm F_{ik}(u)=\sum_{j=1}^{m_i}\frac{\phi_k(T_{ij})}{g(T_{ij})}\,\widetilde Y_{ij}(u),
	\ \ 
	\mathrm D_{ik}(u)=\mathbb E_i[\Lambda_{ik}(u)]-\mathbb E[\Lambda_{ik}(u)]
	=\sum_{j=1}^{m_i}\psi_k(u,T_{ij})-\mathfrak e_1\beta_k(u).
	\]
	By Assumptions~\ref{Ass_gen} and~\ref{ass:random_mi}, $X_i$ and the $\eta_{ij}$ are
	centered and independent of $(m_i,T_{i1},\ldots,T_{im_i})$, so that
	$\mathbb E_i[\mathrm F_i]=0$, while $\mathrm D_i$ is
	$(m_i,T_{i1},\ldots,T_{im_i})$-measurable. Hence
	$\mathbb E[\mathrm F_i\otimes\mathrm D_i]
	=\mathbb E[\mathbb E_i[\mathrm F_i]\otimes\mathrm D_i]=0$, and
	$\mathcal C_M=\mathcal C_M^{\mathrm F}+\mathcal C_M^{\mathrm D}$ with
	$\mathcal C_M^{\mathrm F}=\operatorname{Cov}(\mathrm F_i)$ and
	$\mathcal C_M^{\mathrm D}=\operatorname{Cov}(\mathrm D_i)$. Both are covariance
	operators, hence positive semi-definite: for $a\in\mathcal H_M$,
	$\langle a,\mathcal C_M^{\mathrm F}a\rangle_{\mathcal H_M}
	=\mathbb E\langle a,\mathrm F_i\rangle_{\mathcal H_M}^2\ge0$, and similarly for
	$\mathcal C_M^{\mathrm D}$. Since $\mathrm F_i$ does not involve $\mu$, neither does
	$\mathcal C_M^{\mathrm F}$. All the expectations below are finite because $m_i\le m_{\max}$,
	$g\ge g_{\min}>0$, the $\phi_k$ are bounded, $\tau$ is bounded,
	$\mathbb E\|X\|^2_{L^2(\mathcal U\times\mathcal T)}<\infty$ and
	$\mu\in L^2(\mathcal U\times\mathcal T)$.
	
	\emph{The operator $\mathcal C_M^{\mathrm F}$.} Fix $k,k'\in\mathcal K_M$ and
	$u,u'\in\mathcal U$. Condition on $m_i$ and split the double sum defining
	$\mathbb E[\mathrm F_{ik}(u)\mathrm F_{ik'}(u')\mid m_i]$ into diagonal ($j=j'$) and
	off-diagonal ($j\neq j'$) pairs. For $j=j'$, conditionally on $T_{ij}=t$, we have
	\[
	\mathbb E\big[\widetilde Y_{ij}(u)\widetilde Y_{ij}(u')\mid T_{ij}=t\big]
	=c_X(u,u';t,t)+\tau^2(t)\varrho_\eta(u,u')=\rho(u,u';t),
	\]
	because $X_i$ and $\eta_{ij}$ are independent and centered. Integrating with respect
	to the density $g$ of $T_{ij}$ gives
	\[
	\mathbb E\Bigl[\frac{\phi_k(T_{ij})\phi_{k'}(T_{ij})}{g^2(T_{ij})}\,
	\rho(u,u';T_{ij})\Bigr]=\mathsf A_{kk'}(u,u').
	\]
	For $j\neq j'$, the visit times $T_{ij},T_{ij'}$ are independent, and the noises
	$\eta_{ij},\eta_{ij'}$ are independent, centered and independent of $X_i$, so only
	the common field $X_i$ contributes:
	$$
	\mathbb E[\widetilde Y_{ij}(u)\widetilde Y_{ij'}(u')\mid T_{ij}=t,T_{ij'}=s]
	=c_X(u,u';t,s).
	$$ 
	Integrating with respect to $g(t)g(s)$, the factors
	$\{g(t)g(s)\}^{-1}$ cancel and we obtain $\mathsf B_{kk'}(u,u')$. There are $m_i$ diagonal
	and $m_i(m_i-1)$ off-diagonal pairs. Thus, taking the expectation with respect to $m_i$,
	which is independent of the other variables, gives the coefficients $\mathfrak e_1$
	and $\mathfrak e_2$, that is,
	$\mathcal C_M^{\mathrm F}=\mathfrak e_1\mathcal C^{(1)}+\mathfrak e_2\mathcal C^{(2)}$.
	
	\emph{The operator $\mathcal C_M^{\mathrm D}$.} Since $\mathbb E[\mathrm D_{ik}(u)]=0$, it holds
	$$
	\mathbb E[\mathrm D_{ik}(u)\mathrm D_{ik'}(u')]
	=\mathbb E\Bigl[\sum_{j,j'=1}^{m_i}\psi_k(u,T_{ij})\psi_{k'}(u',T_{ij'})\Bigr]
	-\mathfrak e_1^2\beta_k(u)\beta_{k'}(u').
	$$
	The same split into diagonal and off-diagonal pairs gives, for $j=j'$,
	\[
	\mathbb E[\psi_k(u,T_{ij})\psi_{k'}(u',T_{ij})]
	=\int_{\mathcal T}\frac{\phi_k(t)\phi_{k'}(t)}{g(t)}\,\mu(u,t)\mu(u',t)\,{\rm d}t,
	\]
	and, for $j\neq j'$, by independence of the visit times,
	$\mathbb E[\psi_k(u,T_{ij})\psi_{k'}(u',T_{ij'})]=\beta_k(u)\beta_{k'}(u')$.
	Taking the expectation with respect to $m_i$, the kernel of
	$\mathcal C_M^{\mathrm D}$ is
	\[
	\mathfrak e_1\int_{\mathcal T}\frac{\phi_k(t)\phi_{k'}(t)}{g(t)}\,\mu(u,t)\mu(u',t)\,{\rm d}t
	+(\mathfrak e_2-\mathfrak e_1^2)\,\beta_k(u)\beta_{k'}(u').
	\]
	
	\emph{Conclusion.} Adding the kernels and using
	$\overline\rho(u,u';t)=\rho(u,u';t)+\mu(u,t)\mu(u',t)$, the two diagonal
	contributions combine into $\mathfrak e_1\overline{\mathsf A}_{kk'}(u,u')$, whence
	\[
	\mathcal C_M=\mathfrak e_1\,\overline{\mathcal C}^{(1)}+\mathfrak e_2\,\mathcal C^{(2)}
	+(\mathfrak e_2-\mathfrak e_1^2)\,\mathfrak b\otimes\mathfrak b .
	\]
	Under $H_0$ in~\eqref{eq:H0_full}, $\mathcal S(\mathcal K_M;\mu)\le\mathcal S(\mathcal K;\mu)=0$,
	so $\beta_k=0$ for all $k\in\mathcal K_M$, that is $\mathfrak b=0$, and
	$\mathcal C_M=\mathfrak e_1\,\overline{\mathcal C}^{(1)}+\mathfrak e_2\,\mathcal C^{(2)}$.
	If moreover $\mu\equiv0$, then $\mathrm D_i\equiv0$ and $\mathcal C_M^{\mathrm D}=0$.
\end{proof}


\subsection{Proofs for testing the mean using spacings statistic}\label{append_test_spc}

Recall that in the spacings approach, the ordered sample $T_{(1)}\leq\cdots\leq T_{(M)}$ is extended by symmetric reflection, $T_{(l)}:=2T_{(1)}-T_{(2-l)}$ for $l\leq 0$ and $T_{(l)}:=2T_{(M)}-T_{(2M-l)}$ for $l>M$. Each $l$ corresponds to a single pair $(i,j)$. We define the set of indices 
$$
\mathcal L_n = \{l\in\mathbb N^*: 1\leq l-\mathfrak h < l+\mathfrak h \leq M\}.
$$
Moreover, with $l=\varrho(i,j)$, $l'=\varrho(i',j')$ the ranks of $T_{ij}$ and $T_{i'j'}$ among the order statistics of the visit times, let $\varrho_1^{-1}(l)$ be the subject index $i$ of the pair $(i,j)$ with rank $l$, and recall the definition
$$
\mathcal I_{n}  = \{(l,l'): 1\le l,l'\le M , \varrho_1^{-1}(l)\neq\varrho_1^{-1}(l')\}
$$ 
for the set of pairs of pairs $(i,j), (i',j')$ with different subject indices $i,i'$. 
Moreover, with $l,l'$ corresponding to the pairs $(i,j)$, $(i',j')$, respectively, let
$$
\mathcal Y_{l,l'}(u) = Y_{i,j}(u)Y_{i',j'} (u) \quad \text{ and } \quad   \mathcal Y_{l,l'} = \int_{\mathcal U} Y_{i,j}(u)Y_{i',j'} (u){\rm d}\nu(u).
$$
Since the order statistics $T_{(l)}$ are $\mathcal{F}^{\rm obs}$-measurable, if Assumption~\ref{ass:indep} holds, for $(l, l')\in \mathcal I_{n} $ we get 
\begin{equation}\label{eq:orth_Yllprime}
	\mathbb E [ \mathcal Y_{l,l'}(u) \!\mid \!\mathcal F^{\rm obs}] =  \mu(u,T_{(l)}) \mu (u,T_{(l')})  
	\  \text{ and } \
	\mathbb E [ \mathcal Y_{l,l'} \!\mid \! \mathcal F^{\rm obs}] =\! \int_{\mathcal U} \!\mu(u,T_{(l)}) \mu (u,T_{(l')}) {\rm d}\nu(u).
\end{equation}
Finally, recall the notation
$$
\mathcal S(\KM;\mu)=\sum_{k\in\KM}\|\beta_k\|^2_{L^2(\mathcal U)}, \quad \beta_k(u)=\int_{\mathcal T}\mu(u,t)\phi_k(t)\,{\rm d}t, \quad 
\bar \mu=\sup_t\|\mu(\cdot,t)\|_{L^2(\mathcal U)}.
$$

\subsubsection{Preliminaries}\label{sec:app_prelim_spacings} 

Recall that for technical reasons we assume that the number of visit times $m_i$ for the subject $i$ is a random variable. Thus, $M=\sum_{i=1}^n m_i$ is a random variable. We assume $m_1,\ldots,m_n$ i.i.d. distributed as $\mathfrak m$ and the $m_i$ are independent of all other variables.  Recall that for the results with the proofs presented in the following, the $\sigma$-field $\mathcal F^{\rm obs}$ introduced in Section~\ref{sec:model} is redefined as the $\sigma$-field generated by the $T_{ij}$ and the $m_i$.

Let $G$ denote the distribution function corresponding to $g$. Let $U_{(l)}=G(T_{(l)})$ the order statistics of the uniform sample obtained from sample of design points. Set $Q=G^{-1}$ so that  $Q'(u)=1/g(Q(u))\le 1/g_{\min}$, and $Q'$ is $L'$-Lipschitz where $L'$ depends on the Lipschitz constant $L_g$ of $g$ and $g_{\rm min}$. Moreover, let 
\[
\xi_l:=(M+1)\Bigl(\frac{U_{(l+\hh)}-U_{(l-\hh)}}{2\hh}-\frac1{M+1}\Bigr),
\qquad 1\le l\le M . 
\]

Fix $1\leq l\leq M$ and extend the sample of the uniform order statistics by symmetric reflection.
Then we can write,
\begin{multline}
	\frac{1}{2\mathfrak h }\bigl(T_{(l+\mathfrak h )}-T_{(l-\mathfrak h )}\bigr)-\frac{1}{(M+1)g(T_{(l)})}
	= \frac{1}{2\mathfrak h }\int_{U_{(l-\mathfrak h )}}^{U_{(l+\mathfrak h )}}\bigl\{Q'(u)-Q'(U_{(l)})\bigr\} {\rm d}u
	\\
	+\frac{1}{g(T_{(l)})}\Bigl(\frac{U_{(l+\mathfrak h )}-U_{(l-\mathfrak h )}}{2\mathfrak h }-\frac{1}{M+1} \Bigr).
\end{multline}
Note that
\begin{equation}\label{eq:simplify1}
	\left| \frac{1}{2\mathfrak h }\int_{U_{(l-\mathfrak h )}}^{U_{(l+\mathfrak h )}}\bigl\{Q'(u)-Q'(U_{(l)})\bigr\} {\rm d}u\right|	\leq L'\frac{\{U_{(l+\mathfrak h)} - U_{(l-\mathfrak h)}\}^2}{2\mathfrak h} .
\end{equation}

Recall that the spacings-based test statistic is 
\begin{equation}\label{eq:defQspn}
	Q^{\rm sp}_n= \frac{(M+1)^2}{n(n-1)}  \sum_{(l,l')\in\mathcal I_{n}}\sum_{k\in\mathcal K_M}\widehat\Phi_{kl}\widehat\Phi_{kl'}
	\mathcal Y_{l,l'} ,
\end{equation}
with $\widehat\Phi_{kl}$, $k\in\mathcal K_M$, defined in~\eqref{eq:sp_est},  $1\le l\le M$. We can write
\begin{equation}\label{eq:equiv_sp_g1}
	\widehat\Phi_{kl} = \frac{\phi_k (T_{(l)})}{(M +1)g(T_{(l)})} + \widehat D_{kl}
\end{equation}
and, with $\chi_l(t)=\mathbf 1\{T_{(l-\mathfrak h)}\leq t\leq T_{(l+\mathfrak h)}\}$, 
\begin{multline}\label{eq:def_frakdl}
	\widehat D_{kl} =	 \frac 1{2\mathfrak h}\int_{\mathcal T} \chi_l(t)\left\{ \phi_k (t) - \phi_k (T_{(l)}) \right\} {\rm d} t  +\mathfrak d_l  \phi_k (T_{(l)})= : \widehat D_{kl,1}+ \widehat D_{kl, 2}, \\ \text{where } \ \mathfrak d_l :=   \left\{ \frac{\int_{\mathcal T} \chi_l(t) {\rm d} t }{2\mathfrak h} - \frac 1{(M +1)g(T_{(l)})}\right\}.
\end{multline}
Note that 
$$
\mathfrak d_l :=   \left\{ \frac{T_{(l+\mathfrak h)} - T_{(l-\mathfrak h)}}{2\mathfrak h} - \frac 1{(M +1)g(T_{(l)})}\right\},\qquad \forall l\in\Ln,
$$
and
\begin{equation}\label{eq:neg_D1l}
	\left| \widehat D_{kl,1} \right|	\leq \| \phi_k^\prime \|_\infty\frac{\big\{\min\{T_{(l+\mathfrak h)},1\} - \max\{T_{(l-\mathfrak h)},0\}\big\}^2}{2\mathfrak h}, \qquad \forall 1\le l \le M .
\end{equation}

Next, we can decompose
\begin{multline}\label{eq:bound_diff_spQ}
	Q^{\rm sp}_n
	= Q_n+  2\frac{M+1}{n(n-1)}  \sum_{(l,l')\in\mathcal I_{n}}\sum_{k\in\mathcal K_M}\frac{\phi_k (T_{(l)})}{g(T_{(l)})} \widehat D_{kl'} 
	\mathcal Y_{l,l'}
	\\\qquad \qquad +\frac{(M+1)^2}{n(n-1)} \sum_{(l,l')\in\mathcal I_{n}}\sum_{k\in\mathcal K_M}\widehat D_{kl}\widehat D_{kl'}
	\mathcal Y_{l,l'} 
	=:Q_n + \frac 2{n(n-1)} \mathcal Q_{1,n} + \frac 1{n(n-1)}\mathcal Q_{2,n}.
\end{multline}
By construction and~\eqref{eq:orth_Yllprime},  if $\mu\equiv 0$ we get  $\mathbb E[ Q_{1,n}]= \mathbb E[  Q_{2,n}]=0$, and $\mathbb E[\mathcal Q^{\rm sp}_n]= \mathbb E[\mathcal Q_{n}]$.

\subsubsection{The statement under the null hypothesis}

We want to prove that, under the conditions of Theorem~\ref{prop:dist_tstat}, if $\mu\equiv 0$, it holds 
\begin{equation}\label{eq:purpose_sp}
	\sup_{z\in\mathbb R} \left| \PP (	Q^{\rm sp}_n\le z) - \PP (	Q_n\le z)\right|\leq C_0 v_{0,n} ,
\end{equation}
where $v_{0,n}$ is a sequence decreasing to zero with a rate to be determined, and $C_0$ is a constant depending on the data generating process, but not on $\mathcal K_M$ and $\hh$. (Here, the probability is unconditional, the randomness of the $m_i$'s is integrated out). Equivalently, the aim is to prove
\begin{equation}\label{eq:purpose_s2p}
	\sup_{z\in\mathbb R} \left| \PP (	T^{\rm sp}_n\le z) - \PP (	T_n\le z)\right|\leq C_0 v_{0,n} ,
\end{equation}
with $T_n$ defined in~\eqref{eq:def_t_testS} and 
\begin{equation}\label{eq:def_t_test_sp}
	T^{\rm sp}_n=\frac{n\,Q^{\rm sp}_n}{ \mathbb E[\mathfrak m]V^{1/2}|\mathcal K_M|^{1/2}}.
\end{equation}

For $0< d_{0,n}\rightarrow 0$, we define the events 
$$
\mathcal E_1(z) = \{	T^{\rm sp}_n\leq z, 	|T^{\rm sp}_n- 	T_n| \leq d_{0,n}\}\qquad \mathcal E_2(z) =\{	T^{\rm sp}_n\leq z, |T^{\rm sp}_n- 	T_n|> d_{0,n} \},
$$
so that $\{	T^{\rm sp}_n\le z \} = \mathcal E_1(z) \cup \mathcal E_2(z)$ and $\mathcal E_1(z) \cap \mathcal E_2(z)=\emptyset$. Now, for any $z$, 
we have $\mathcal E_2(z) \subset\{|T^{\rm sp}_n - T_n|> d_{0,n} \}$ and thus, by Markov inequality, it holds
$$
\sup_{z\in\mathbb R} \PP \big(\mathcal E_2(z)\big) \leq \PP (|T^{\rm sp}_n - T_n|> d_{0,n}) \leq d_{0,n}^{-1} \E |T^{\rm sp}_n - T_n|.
$$
On the other hand, for any $z$, we can write 
\begin{multline}
	\{	T_n \leq z - d_{0,n}\} \setminus \{|T^{\rm sp}_n- 	T_n| > d_{0,n} \}\subset \mathcal E_1(z) \\ \subset \{	T_n \leq z+ d_{0,n} , 	|T^{\rm sp}_n- 	T_n| \leq d_{0,n}\}\subset \{	T_n \leq z+ d_{0,n} \}, 
\end{multline}
We deduce
\begin{equation}
	\sup_{z\in\mathbb R} \left| \PP \big(\mathcal E_1(z)\big) - \PP (	T_n\le z)\right|
	\leq \frac{d_{0,n}\|\Phi'\|_\infty}{\varsigma_n(0)} + 2B_{\rm BE,n}+ d_{0,n}^{-1}\,\E|T^{\rm sp}_n-T_n|,
\end{equation}
where
$$
B_{\rm BE,n} = C \left[ |\mathcal K_M|^{-1/2} + \{|\mathcal K_M|/n\}^{1/2} \right]
$$
is the bound from Theorem~\ref{prop:dist_tstat} (second part, with $\Theta_1=0$ since $\mu\equiv0$),
for the approximation of the law of $T_n$ by $\Phi(\cdot/\varsigma_n(0))$. Since
$\varsigma_n(0)\to1$ by Corollary~\ref{corr:test_level}(i), we have $\varsigma_n(0)\geq1/2$ for $n$
large enough, so that the first term is at most $2d_{0,n}\|\Phi'\|_\infty$. Combining these facts, we get 
\begin{equation}\label{eq:purpose_sp3}
	\sup_{z\in\mathbb R} \left| \PP (	T^{\rm sp}_n\le z) - \PP (	T_n\le z)\right|\lesssim d_{0,n} + d_{0,n}^{-1} \E |T^{\rm sp}_n - T_n| + B_{\rm BE,n}.
\end{equation}
It remains to bound  $\E |T^{\rm sp}_n - T_n|$ (unconditional expectation), under the null $H_0$ where it holds $\mathbb E[T^{\rm sp}_n]= \mathbb E[ T_{n}]=0$. We have
$$
\E |T^{\rm sp}_n - T_n|  = \frac{1}{ \mathbb E[\mathfrak m]V^{1/2}} n|\mathcal K_M|^{-1/2} \E |Q^{\rm sp}_n - Q_n|.
$$
Thus,
\begin{equation}\label{eq:purpose_sp39}
	\sup_{z\in\mathbb R} \left| \PP (	T^{\rm sp}_n\le z) - \PP (	T_n\le z)\right|\lesssim d_{0,n} + \frac{d_{0,n}^{-1}n|\mathcal K_M|^{-1/2}}{ \mathbb E[\mathfrak m]V^{1/2}}   \E |Q^{\rm sp}_n - Q_n| + B_{\rm BE,n},
\end{equation}
and this will serve to derive~\eqref{eq:purpose_s2p}. We will show that, under some mild additional restrictions on $|\mathcal K_M|$ (\emph{i.e.}, $\hh|\KM|/n^{4/5}\to0$ instead of  $|\KM|/n\to0$), the right-hand bound tends to zero as soon as $d_{0,n} + \{\mathfrak h^{1/2}  d_{0,n}\}^{-1} \rightarrow 0$.

\subsubsection*{Formal arguments for the statement under the null hypothesis}

The calculations above show that in order to guarantee~\eqref{eq:purpose_s2p} we have to suitably bound  $\E |Q^{\rm sp}_n - Q_n| $ under the null hypothesis. For this we bound $\operatorname{Var}(Q^{\rm sp}_n-Q_n) $ and use Cauchy-Schwarz inequality. To bound the variance of the difference $Q^{\rm sp}_n-Q_n$,  we use the decomposition~\eqref{eq:bound_diff_spQ}. Under the null hypothesis, $\E [\mathcal Q_{1,n}] =\E [\mathcal Q_{2,n}]=0$, we bound the variance for each of these terms. 

This section contains technical lemmas providing variance bounds that will be combined in Section~\ref{sec:formal_conc_null} to prove that the bound in~\eqref{eq:purpose_sp39} tends to zero if $d_{0,n} + \{\mathfrak h^{1/2}  d_{0,n}\}^{-1} \rightarrow 0$.

\subsubsection*{The bound for  $\operatorname{Var}(\mathcal Q_{1,n})$} 

\begin{lemma}[The variance of $\mathcal Q_{1,n}$ under $H_0$]\label{lem:var_Q1}
	Under Assumption~\ref{Ass_gen}, Condition~\eqref{cd_m_max}, and if $\mu\equiv0$, with $\hh\to\infty$ and $\hh^2|\KM|/n\to0$, there is a
	constant $C$ depending only on
	$g_{\min}$, $g_{\max}$, $L_g$, $m_{\max}$, $\nu(\mathcal U)$, $\|c_X\|_\infty$, $\|\tau\|_\infty$ such that
	\[
	\operatorname{Var} \biggl(\frac{1}{n(n-1)}\mathcal Q_{1,n}\biggr)\ \le\ C\,\frac{|\KM|}{n^2\,\hh}.
	\]
	Consequently,
	$\operatorname{Var}\bigl(n|\KM|^{-1/2}\{n(n-1)\}^{-1}\mathcal Q_{1,n}\bigr)\le C'/\hh$,
	with $C'$ independent of $|\KM|$.
\end{lemma}

\subsubsection*{The bound for $\operatorname{Var}(\mathcal Q_{2,n})$}\label{sec:subsecQQ}
To control the variance of $\mathcal Q_{2,n}$, using the decomposition $\widehat D_{kl} =\widehat D_{kl,1}+ \widehat D_{kl, 2}$ as in~\eqref{eq:def_frakdl}, we get
\begin{multline}\label{eq:dec_Q2n_aux} 
	\mathcal Q_{2,n}
	= \sum_{(l,l')\in\mathcal I_{n}}\sum_{k\in\mathcal K_M}\{(M+1)\widehat D_{kl,1}\}\{(M+1)\widehat D_{kl',1}\}	\mathcal Y_{l,l'} 
	\\ +2\sum_{(l,l')\in\mathcal I_{n}}\sum_{k\in\mathcal K_M}\!\{(M+1)\widehat D_{kl,1}\}\{(M+1)\widehat D_{kl',2}\}	\mathcal Y_{l,l'} \\
	+ \sum_{(l,l')\in\mathcal I_{n}}\sum_{k\in\mathcal K_M}\{(M+1)\widehat D_{kl,2}\}\{(M+1)\widehat D_{kl',2}\}	\mathcal Y_{l,l'} \\ 
	=: \mathcal Q_{21,n}+2\mathcal Q_{22,n}+\mathcal Q_{23,n}.
\end{multline}
We first investigate $\mathcal Q_{21,n}$ and $\mathcal Q_{23,n}$,   the cross-product term $\mathcal Q_{22,n}$ will follow.

\begin{lemma}[The variance of $\mathcal Q_{21,n}$ under $H_0$]\label{lem:var_Q21}
	Under Assumption~\ref{Ass_gen}, Condition~\eqref{cd_m_max}, and if $\mu\equiv0$, with
	$\hh\to\infty$ and $\hh^2|\KM|/n\to0$, and with $\mathcal Q_{21,n}$ as in
	\eqref{eq:bound_diff_spQ}, there is a constant $C$, depending only on
	$g_{\min},L_g,m_{\max},\nu(\mathcal U),\|c_X\|_\infty,\|\tau\|_\infty$, such that
	\[
	\operatorname{Var} \biggl(\frac{1}{n(n-1)}\mathcal Q_{21,n}\biggr)   \le C\,\frac{|\KM|^6\,\hh^4}{n^6}.
	\]
	Consequently,
	$$
	\operatorname{Var}\bigl(n|\KM|^{-1/2}\{n(n-1)\}^{-1}\mathcal Q_{21,n}\bigr)\le C\,\frac{|\KM|^5\,\hh^5}{n^4} \hh^{-1} ,
	$$
	and $ \operatorname{Var}\bigl(n|\KM|^{-1/2}\{n(n-1)\}^{-1}\mathcal Q_{21,n}\bigr)\le C'/\hh$, for some constant $C'$, if $n^{4/5}\{\hh |\KM|\}^{-1}\rightarrow \infty$.
\end{lemma}

\medskip

\begin{remark}\label{rem:frust1} 
	It is reasonable to expect that the condition $\hh\,|\KM|=o(n^{4/5})$ imposed in Lemma~\ref{lem:var_Q21} can be relaxed to  $\hh^2\,|\KM|=o(n)$ if, instead using the crude bound for $|\widehat D_{kl,1}|\le\|\phi_k'\|_\infty S_l^2/(2\hh)$, we proceed to a more refined control of the terms involving it. See also the comment at the end of the proof. This would require the investigation of the rate of $U$-statistics of order 4 involving the spacings (which are not independent variables). Such long calculations are left for future work.
\end{remark}

\medskip

To control the moments of $ \mathcal Q_{23,n} $, we  will need several technical lemmas. First, we introduce some notation and recall some results.  Recall that $G$ denote the distribution function such that $G'=g$, and $Q=G^{-1}$ so that $0\le Q'\le g_{\min}^{-1}$. Assuming $g$ is differentiable with $|g'|\le L_g$, we also have  $|Q''|\le L':=L_g g_{\min}^{-3}$. Let $U_{(l)} = G(T_{(l)})$, $1\le l\le M$, the order statistics obtained by probability integral transform of the ordered design variables. These are uniform order statistics (conditionally given $M$). Below, we also consider $U_{(l)}$ with $-\hh +1\leq l \leq 0$ and $M+1\leq l \le M+\hh$ by transforming the reflected design points at the boundaries. Let $E_j\stackrel{\text{iid}}\sim\mathrm{Exp}(1)$, $j\geq 1$, and define
$$
\Gamma_k=\sum_{1\le j\le k}E_j.
$$
R\'enyi representation (see \cite{DN2003,SW2009}) states that, conditionally given $M$,
\begin{equation}\label{eq:full_renyi0}
	\left(U_{(1)},\ldots, U_{(M)}\right)\stackrel{d}{=}\left(\frac{\Gamma_1}{\Gamma_{M+1}},\ldots,\frac{\Gamma_M}{\Gamma_{M+1}}\right),
\end{equation}
Next, for $\hh\in\mathbb N^*$, let $l\in\Ln=\{l:\hh+1\le l\le M-\hh\}$, and define also
\begin{equation}\label{eq:renyi_defs_use}
	W_l=\sum_{l-\mathfrak h+1}^{l+\mathfrak h}E_j, \qquad \delta_l=\frac{W_l}{2\mathfrak h}-1,\quad \text{with } l\in\Ln, 
	\qquad \epsilon=\frac{\Gamma_{M+1}}{M+1}-1,
\end{equation}
and 
$$
A_l=\sum_{l-\mathfrak h+1}^{l}E_j,
\quad B_l=\sum_{l+1}^{l+\mathfrak h}E_j,\qquad l\in\Ln.
$$
By R\'enyi's representation, one has the exact identity in distribution, conditionally given $M$,  
\begin{equation}\label{eq:dl-exact}
	(M+1)\mathfrak d_l\stackrel{d}{=}\frac{1}{g(T_{(l)})}\frac{\delta_l-\epsilon}{1+\epsilon}+(M+1)\mathfrak s_l,
	\quad |(M+1)\mathfrak s_l|\le\frac{L'(A_l^2+B_l^2)}{4\mathfrak h(M+1)(1+\epsilon)^2},\quad l\in\Ln.
\end{equation}
This allows to compute any moment involving $\mathfrak d_l$ in terms of the $E_j$'s and their functionals. 
Namely, it holds (using also the identity $\operatorname{Cov}(W_l,\Gamma_{M+1})=2\mathfrak h$ and the fact that $W_l\sim \Gamma(2\mathfrak h,1)$): for any $ l\in\Ln$,
\begin{equation}\label{eq:moments}
	\mathbb E [\delta_l^2]=\frac1{2\mathfrak h},\ \ \operatorname{Var}(\delta_l^2)=\frac{2}{(2\mathfrak h)^2}+O(\mathfrak h^{-3}),\ \
	\mathbb E  [\epsilon^2\mid M]=\frac1{M+1},\ \ \mathbb E [\delta_l\epsilon\mid M]=\frac1{M+1}.
\end{equation}
Let $\mathcal G_0=\{|\epsilon|\le 1/2\}$, so that $\mathbb P (\mathcal G_0^c\mid M)\le e^{-cM}$. Hence, with 
integrands that are polynomially dominated in $M$, the contribution of the set $\mathcal G_0^c$ to the conditional moments given $M$ decreases in absolute value as fast as $e^{-cM}$. Thus, we can drop the contribution coming from $\mathcal G_0^c$. Note that on $\mathcal G_0$ it holds $|(1+\epsilon)^{-2}-1 | \leq 10|\epsilon|$. 

All spacing identities in distributions above (e.g., in~\eqref{eq:dl-exact}) and below are for $l\in\Ln$. The boundary ranks, which are at most $2\hh$, are treated separately and shown to have negligible contribution.

\medskip

\begin{lemma}[The variance of $\mathcal Q_{23,n}$ under $H_0$]\label{lem:var_Q23}
	Under Assumption~\ref{Ass_gen}, Condition~\eqref{cd_m_max}, and if $\mu\equiv0$, with
	$\hh\to\infty$, $\hh^2|\KM|/n\to0$, and $\mathcal Q_{23,n}$ as in
	\eqref{eq:bound_diff_spQ}, there is a constant $C$ depending only on
	$g_{\min},L_g,m_{\max},\nu(\mathcal U),\|c_X\|_\infty,\|\tau\|_\infty$ such that
	\begin{equation}\label{eq:var_Q23H0}
		\operatorname{Var}\Bigl(\frac{1}{n(n-1)}\mathcal Q_{23,n}\mid M\Bigr) \le C\,\frac{|\KM|}{n^2\,\hh^2}.
	\end{equation}
	Consequently, under $H_0$,
	$\operatorname{Var}\bigl(n|\mathcal K_M|^{-1/2}\{n(n-1)\}^{-1}\mathcal Q_{23,n}\bigr)\le C'/\hh^2$,
	with $C'$ independent of $|\KM|$.
\end{lemma}

For the mixed term, using the Cauchy-Schwarz inequality, it can be shown that 
\begin{equation}\label{eq:bound_Q22H0}
	\operatorname{Var}\Bigl(\frac{1}{n(n-1)}\mathcal Q_{22,n}\Bigr)\le C''\{|\KM|^6\hh^4 n^{-6}\}^{1/2}\{|\KM| \hh^{-2}n^{-2}\}^{1/2}
	= C''  |\KM|^{7/2}\hh n^{-4},
\end{equation}
for some constant $C''$ independent of $|\KM|$.
Consequently, under $H_0$, if $n^{4/5}\{\hh |\KM|\}^{-1}\rightarrow \infty$, 
\begin{equation}\label{eq:bdgza}
	\operatorname{Var}\bigl(n|\mathcal K_M|^{-1/2} \{n(n-1)\}^{-1}\mathcal Q_{22,n}\bigr) \le  C'''\hh^{-3/2},
\end{equation}
with $C'''$ independent of $|\KM|$.


\subsubsection{Conclusion under the null hypothesis}\label{sec:formal_conc_null}

\begin{proof}[Proof of Theorem~\ref{thm:spacings_H0H1}-(I)]
	We can now justify \eqref{eq:purpose_s2p}. Recall from \eqref{eq:bound_diff_spQ} that
	\[
	n(n-1)\bigl(Q^{\rm sp}_n-Q_n\bigr)=2\mathcal Q_{1,n}+\mathcal Q_{21,n}+2\mathcal Q_{22,n}+\mathcal Q_{23,n},
	\]
	so, under $H_0$ where each summand is centered, Minkowski's inequality gives
	\begin{multline}\label{eq:mink_null}
		n(n-1)\,\sqrt{\operatorname{Var}(Q^{\rm sp}_n-Q_n)}
		\le 2 \sqrt{\operatorname{Var}(\mathcal Q_{1,n})}
		+  \sqrt{\operatorname{Var}(\mathcal Q_{21,n})}
		\\ +2\sqrt{\operatorname{Var}(\mathcal Q_{22,n})}
		+\sqrt{\operatorname{Var}(\mathcal Q_{23,n})}  .
	\end{multline}
	Lemmas~\ref{lem:var_Q1}, \ref{lem:var_Q21} and~\ref{lem:var_Q23} bound three of these four variances, and the fourth follows from them, see~\eqref{eq:bdgza}. Consequently,
	\begin{equation}\label{eq:var_diff_null}
		\frac{n|\KM|^{-1/2}}{\E[\mathfrak m]V^{1/2}}\E |Q^{\rm sp}_n-Q_n|\le \frac{C}{\hh^{1/2}},
	\end{equation}
	with $C$ independent of $|\KM|$ and $\hh$. Insert this bound into \eqref{eq:purpose_sp39} and deduce
	\begin{equation}\label{eq:v0n} 
		\sup_{z\in\mathbb R}\bigl|\PP(T^{\rm sp}_n\le z)-\PP(T_n\le z)\bigr|
		\le C_0 v_{0,n},
		\qquad
		v_{0,n}=\hh^{-1/4}+|\KM|^{-1/2}+\{|\KM|/n\}^{1/2},
	\end{equation}
	for some constant $C_0$, provided $\hh\rightarrow \infty$, $\hh^2|\KM|/n\to0$  and $\hh|\KM|/n^{4/5}\to0$.  
	This establishes \eqref{eq:purpose_s2p}. In particular, with the choice
	$\hh \asymp \log\log n$ from Section~\ref{sec:spacings} and any $|\KM|\to\infty$ with
	$\hh|\KM|/n^{4/5}\to0$, we have $v_{0,n}\to0$. The spacings statistic
	$T^{\rm sp}_n$ inherits, under $H_0$, the Gaussian approximation of $T_n$ established in
	Theorem~\ref{prop:dist_tstat}, and hence the asymptotic level $\mathfrak a$ of
	Corollary~\ref{corr:test_level}. 
\end{proof}


\subsubsection{The statement under the alternative hypothesis}

We want to prove that, under the sequence of local alternatives $H_{1,n}: \mu_n = r_n\Upsilon$ with $r_n>0$, $\Upsilon\in L^2(\mathcal U \times \mathcal T)$ such that $\mathcal S( {\mathcal K}_M;\Upsilon)>0$ for sufficiently large $n$, it holds 
$$
\forall \mathfrak a\in (0,1),\quad  \PP (	T^{\rm sp}_n > z_{1-\mathfrak a})\rightarrow 1,
$$
which means that the test based on the spacings test statistics $T^{\rm sp}_n $ is consistent against local alternatives. For this, note that by Markov's inequality, for every fixed $\eta>0$,
\begin{equation}\label{eq:markov_H1}
	\PP\bigl(|Q^{\rm sp}_n-Q_n|>\eta\,\mathcal S(\KM;\mu)\bigr)
	\le\frac{\E\bigl|Q^{\rm sp}_n-Q_n\bigr|}{\eta\,\mathcal S(\KM;\mu)}
	\le\frac{\bigl|\E[Q^{\rm sp}_n-Q_n]\bigr|
	}{\eta\,\mathcal S(\KM;\mu)} 
	+ \frac{\sqrt{\operatorname{Var}\bigl(Q^{\rm sp}_n-Q_n\bigr)}}{\eta\,\mathcal S(\KM;\mu)}.
\end{equation}
Thus, it suffices that bias and standard-deviation ratios tend to zero. This will follow by Corollary~\ref{prop:bias_H1} and Lemma~\ref{prop:fluct_H1} below.

\subsubsection*{Formal arguments for the statement under the alternative hypothesis}\label{sec:H1}

This section contains technical lemmas providing bounds that will be combined in Section~\ref{sec:formal_conc_H1n} to prove that the bound in~\eqref{eq:markov_H1} tends to zero. For simplicity, we work under the uniform design $g\equiv1$ on $\mathcal T=[0,1]$. The case of a smooth, bounded from below density $g$ adds only negligible terms. With $\{r_n\}$ a bounded sequence of positive numbers, we consider the sequence of
local alternatives
\begin{equation}\label{eq:app_local_defH1}
	H_{1,n}: \mu=\mu_n=r_n\Upsilon,\qquad 
	\Upsilon\in L^2(\mathcal U\times\mathcal T),\quad
	\bar\Upsilon:=\sup_{t\in\mathcal T}\|\Upsilon(\cdot,t)\|_{L^2(\mathcal U)}<\infty.
\end{equation}
Recall the notation $\mathcal S (\mathcal K_M;\Upsilon)= \sum_{k\in\mathcal K_M} \|\beta_{\Upsilon,k}\|^2_{L^2(\mathcal U)}$, with $\beta_{\Upsilon,k} (u) = \int_{\mathcal T}\Upsilon(u,t)\phi_k(t)\,{\rm d}t$. Under $H_{1,n}$, it holds $\mathcal S (\mathcal K_M;\mu)=r_n^2 \mathcal S (\mathcal K_M;\Upsilon)$. The results below are derived under the condition that $M$ is sufficiently large such that  $\mathcal K_M$ satisfies $\mathcal S(\KM;\Upsilon)\geq c_{\mathcal S}>0$ for some constant $c_{\mathcal S}>0$,
and the detection condition
\begin{equation}\label{eq:detect}
	\frac{n\,r_n^2\,\mathcal S(\KM;\Upsilon)}{|\KM|^{1/2}}\longrightarrow\infty, \qquad \text{as } \ M\rightarrow \infty.
\end{equation}
Finally, note that under $H_{1,n}$, it holds 
$$
\bar\mu=\sup_{t\in\mathcal T}\|\mu(\cdot,t)\|_{L^2(\mathcal U)}=r_n\bar\Upsilon.
$$

For consistency and first-order local power it suffices to show that the effect of the substitution
of $Q_n$ by $Q^{\rm sp}_n$ is negligible relative to the signal, \emph{i.e.},
\begin{equation}\label{eq:H1-target}
	\bigl|Q^{\rm sp}_n-Q_n\bigr|=o_{\PP}\bigl(\mathcal S(\KM;\mu)\bigr),
	\qquad\text{equivalently}\qquad
	T^{\rm sp}_n-T_n=o_{\PP}\bigl(\delta_n\,\mathcal S(\KM;\mu)\bigr),
\end{equation}
after which the local power of $T^{\rm sp}_n$ coincides with that of $T_n$ given by
Corollary~\ref{corr:test_level}. We split $|Q^{\rm sp}_n-Q_n|$ into a \emph{bias} part and a
\emph{fluctuation} part:
\[
\bigl|Q^{\rm sp}_n-Q_n\bigr|\le
\underbrace{\bigl|\, \E[Q^{\rm sp}_n-Q_n]\, \bigr|}_{\text{bias}}
+\underbrace{\bigl|(Q^{\rm sp}_n-Q_n)-\E[\cdot]\,\bigr|}_{\text{fluctuation}} .
\]
By construction and the independence assumptions, the bias term was absent under $H_0$.

Next, note that under Assumption~\ref{ass:coeff-decay_H1n}, by Cauchy-Schwarz inequality, it holds
\begin{equation}\label{eq:sums_bds1}
	\sum_{k\in\KM}k^{2}\|\beta_k\|_{L^2(\mathcal U)}
	\le \Bigl(\sum_{k\ge 1} k^{4}\|\beta_{\Upsilon,k}\|_{L^2(\mathcal U)}^2\Bigr)^{1/2}r_n|\mathcal K_M|^{1/2}
	\le C_\Upsilon^{1/2}\,r_n |\KM|^{1/2},	
\end{equation}
and
\begin{equation}\label{eq:sums_bds2}
	\sum_{k\in\KM}k\,\|\beta_k\|_{L^2(\mathcal U)}
	\le\Bigl(\sum_{k\geq 1}k^{-2}\Bigr)^{1/2} C_\Upsilon^{1/2}r_n\le Cr_n .
\end{equation}
In particular, the series $\sum_{k\ge1}\beta_k(u)\phi_k(t)$ and
$\sum_{k\ge1}\beta_k(u)\phi_k'(t)$ converge in $L^2(\mathcal U)$-norm, absolutely and uniformly in $t$, by~\eqref{eq:sums_bds2}. Moreover, for any $u$, the map $t\mapsto\mu(\cdot,t)$ is continuously differentiable with derivative $\partial_t\mu(u,t)$ and it holds $\partial_t\mu(u,t)=\sum_{k\ge1}\beta_k(u)\phi_k'(t)$.

Next, we set
\[
\pi_n:=\frac{\sum_{i=1}^n m_i(m_i-1)}{M(M-1)}\le\frac{m_{\max}-1}{M-1},
\quad\text{so that }\ (1-\pi_n)M(M-1)=\sum_{1\le i\neq i'\le n}m_im_{i'} .
\]
Since $m_i$ are i.i.d. distributed like $\mathfrak m$, it holds 
$$
\E[(1-\pi_n)M(M-1)]=\sum_{i\neq i'}\E[m_i]\E[m_{i'}]=n(n-1)\E[\mathfrak m]^2.
$$
In particular, since the $m_i$'s are independent of all the other variables and by the proof of Lemma~\ref{lem:deg_Ustat2}, we have
\begin{equation}\label{eq:mean_Qn_fixed_m}
	\E[Q_n\mid m_1,\ldots,m_n ]=\frac{(1-\pi_n)M(M-1)}{n(n-1)}\,\mathcal S(\KM;\mu).
\end{equation}

\subsubsection*{Bounding the bias term of $Q^{\rm sp}_n-Q_n$ under $H_{1,n}$}

By the decomposition~\eqref{eq:bound_diff_spQ}, we have to investigate $\E[2\mathcal Q_{1,n}+\mathcal Q_{2,n}]$.

\begin{lemma}\label{lem:mean_Qsp_H1}
	Assumptions~\ref{Ass_gen} and~\ref{ass:random_mi} hold with
	$g\equiv1$ on $\mathcal T=[0,1]$, $\mathcal K_M$ is a set of consecutive indices, and $n\ge4\hh$. Then, under the local alternatives $H_{1,n}$ defined in~\eqref{eq:app_local_defH1},  
	\begin{multline}\label{eq:mean_Qsp_H1}
		n(n-1)\,\E\bigl[Q^{\rm sp}_n-Q_n\mid m_1,\ldots,m_n\bigr]=\E\bigl[2\mathcal Q_{1,n}+\mathcal Q_{2,n}\mid m_1,\ldots,m_n\bigr]
		\\=(1-\pi_n)\Bigl\{-\Bigl(1+\frac1{2\hh}\Bigr)M\int_{\mathcal T}\Psi(t,t;\KM)\,
		\|\mu(\cdot,t)\|^2\,{\rm d}t+\mathcal R_M\Bigr\}.
	\end{multline}
	If, moreover, $\hh|\KM|\le n$, then the remainder is bounded as $|\mathcal R_M|\le C\bigl\{\hh M+\hh|\KM|^2\bigr\}r_n^2 $, with $C$ depending only on $C_\Upsilon$. In particular,
	if 	$\hh^2|\KM|/n\to0$ and $\hh/|\KM|\to0$, then
	\begin{equation}\label{eq:replace_corr1}
		\E\bigl[Q^{\rm sp}_n - Q_n\mid M \bigr]
		=-\{1+o(1)\}\,\frac{(1-\pi_n)\,M}{n(n-1)}\,\Bigl(1+\frac1{2\hh}\Bigr)\,|\KM|\,
		\|\mu\|^2_{L^2(\mathcal U\times\mathcal T)} ,	
	\end{equation}
	and
	\begin{equation}\label{eq:replace_corr2}
		\left|\E\bigl[Q^{\rm sp}_n - Q_n\bigr]\right|
		\le \{1+o(1)\}\,\frac{\E[\mathfrak m]}{n}\,\Bigl(1+\frac1{2\hh}\Bigr)\,|\KM|\,
		\|\mu\|^2_{L^2(\mathcal U\times\mathcal T)} .
	\end{equation}
\end{lemma}

\medskip 

\begin{corollary}\label{prop:bias_H1}
	Assumptions~\ref{Ass_gen} and~\ref{ass:random_mi} hold with
	$g\equiv1$ on $\mathcal T=[0,1]$, $\mathcal K_M$ is a set of consecutive indices, and $n\ge4\hh$.
	Moreover, $\hh^2|\KM|/n\to0$ and $\hh/|\KM|\to0$. Then, under the local alternatives $H_{1,n}$, it holds 	
	\[
	\frac{\bigl|\E[Q^{\rm sp}_n-Q_n]\bigr|}{\mathcal S(\KM;\mu)}
	\longrightarrow\ 0, \qquad \text{ as } n\rightarrow \infty.
	\]
\end{corollary}

\subsubsection*{Bounding the fluctuation term of $Q^{\rm sp}_n-Q_n$ under $H_{1,n}$}

To bound the fluctuation term $\bigl|(Q^{\rm sp}_n-Q_n)-\E[\cdot]\,\bigr|$ we bound the standard deviation of $(Q^{\rm sp}_n-Q_n)$. Note that, under the simplifying assumption of uniform design ($g\equiv 1$), by R\'enyi's representation and the notation in~\eqref{eq:dl-exact}, one has the exact identity in distribution, conditionally given $M$,  
\begin{equation}\label{eq:dl-exactH1}
	b_l:=(M+1)\mathfrak d_l\stackrel{d}{=}\frac{\delta_l-\epsilon}{1+\epsilon},\quad l\in\Ln.
\end{equation}
Let us note by elementary calculations that we omit, it holds 
\begin{equation}\label{eq:app_m:mean} 
	\E[b_l]=0, \quad \E[b_l^2]=(2\hh)^{-1}\{1+o(1)\} \quad  \text{and} \quad \E[b_l^4]\le C\hh^{-2},
\end{equation}
and for disjoint windows $|l-l'|\ge2\hh$,
\begin{multline}\label{eq:app_m:cov} 
	\Cov(b_l,b_{l'}\mid M)=-(M+2)^{-1}\{1+o(1)\} \\ \text{ and } \quad \bigl|\E[b_l^{a}b_{l'}^{b}\mid M]-\E[b_l^{a}\mid M]\,\E[b_{l'}^{b}\mid M]\bigr|\le C_{a,b}/M.
\end{multline}
for integers $a,b\ge1$ with $a+b\le4$. Moreover, for every bounded Lipschitz $\Phi:\mathcal T^q\to\mathbb R$ and ranks whose windows are pairwise disjoint, it holds 
\begin{equation}\label{eq:app_m:center} 
	\E\bigl[b_{l_1}^2\cdots b_{l_s}^2\,\Phi(U_{(l_1)},\dots)\bigr]
	=(2\hh)^{-s}\E\bigl[\Phi(U_{(l_1)},\dots)\bigr]\{1+o(1)\}
\end{equation}
uniformly in the ranks.

\smallskip

\begin{lemma}\label{prop:fluct_H1}
	Assume that the conditions of Corollary~\ref{prop:bias_H1} hold. Moreover, $\hh|\KM|^{3/2}/n \rightarrow 0$. Then, under the local alternatives $H_{1,n}$, with $r_n$ satisfying~\eqref{eq:detect}, it holds 	
	\[
	\frac{\sqrt{\operatorname{Var}\bigl(Q^{\rm sp}_n-Q_n\bigr)}}{\mathcal S(\KM;\mu)}\longrightarrow0 , \qquad \text{ as } n\rightarrow \infty.
	\]
\end{lemma}

\medskip

\begin{remark}\label{rem:frust2}
	Using  crude bounds for controlling the squared terms in the cases $\widetilde{\mathcal Q}_n=\mathcal Q_{21,n}$, we used more restrictive condition   $\hh|\KM|^{3/2}/n \rightarrow 0$. As noted in Remark~\ref{rem:frust1}, it is reasonable to expect that this  condition can be relaxed to  $\hh^2\,|\KM|=o(n)$ by studying the rates of $U$-statistics of order 4 involving the spacings (which are not independent variables). Such refinements are left for future work.
\end{remark}

\subsubsection{Conclusion under the local alternative hypotheses}\label{sec:formal_conc_H1n}

\begin{proof}[Proof of Theorem~\ref{thm:spacings_H0H1}-(II)]
	Combining Corollary~\ref{prop:bias_H1} and Lemma~\ref{prop:fluct_H1} and using~\eqref{eq:markov_H1}
	yields $|Q^{\rm sp}_n-Q_n|=o_{\PP}(\mathcal S(\KM;\mu))$, that is,
	$$T^{\rm sp}_n-T_n=o_{\PP}(\delta_n\mathcal S(\KM;\mu)).$$ Moreover, by~\eqref{eq:Tn_W} and the proof of
	Corollary~\ref{corr:test_level(ii)}, $\varsigma_n/\{\delta_n\mathcal S(\KM;\mu)\}\to0$, so that, by
	Chebyshev's inequality, $T_n/\{\delta_n\mathcal S(\KM;\mu)\}\to1$ in probability. Hence
	$T^{\rm sp}_n/\{\delta_n\mathcal S(\KM;\mu)\}\to1$ in probability and, since~\eqref{eq:detect} implies
	$\delta_n\mathcal S(\KM;\mu)\to\infty$, $\PP_{H_{1,n}}(T^{\rm sp}_n>z_{1-\mathfrak a})\to1$ for any
	$\mathfrak a\in(0,1)$.
\end{proof}


\subsubsection{Proofs of the technical lemmas for spacings test statistic}\label{sec:technic_proof_Spc}

We start by stating and proving two auxiliary technical results in Lemmas~\ref{lem:spacing_toolbox} and~\ref{lem:energy_sp}. They will be cited in the other proofs. Next, this section of proofs is organized as follows. First, we prove Lemma~\ref{lem:var_Q21} which provides a bound on the variance of $\mathcal Q_{21,n}$ under $H_0$.  We use the auxiliary lemma for the proof of Lemma~\ref{lem:var_Q23} which provides a bound on the variance of $\mathcal Q_{23,n}$ under $H_0$. The last proof under $H_0$ is that of Lemma~\ref{lem:var_Q1} which provides a bound on the variance of $\mathcal Q_{1,n}$. Next, we present the proofs for the local alternatives $H_{1,n}$: first Lemma~\ref{lem:mean_Qsp_H1} (with the technical complement in Lemma~\ref{cor:mean_Qsp_H1}) on the bias of the difference $Q_n^{\rm sp} - Q_n$, and finally Lemma~\ref{prop:fluct_H1} on the fluctuations of the difference.

First, we present an auxiliary technical result on the moments on uniform statistics spacings. For this, we use the R\'enyi representation: conditionally given $M$,
\begin{equation}\label{eq:full_renyi}
	\left(U_{(1)},\ldots, U_{(M)}\right)\stackrel{d}{=}\left(\frac{\Gamma_1}{\Gamma_{M+1}},\ldots,\frac{\Gamma_M}{\Gamma_{M+1}}\right),
\end{equation}
where $\Gamma_j=E_1+\cdots+E_j$, $1\le j\le M+1$, and $E_1,\ldots,E_{M+1}$ are i.i.d.\
$\mathrm{Exp}(1)$ variables, defined on some probability space.
All the quantities considered below (the reflected points $U_{(l)}$, $l\le0$ or $l>M$, the
spacings, $L_l$, $\xi_l$, $V_l$, $\Lambda_r$, $|\mathcal T\setminus\mathcal I_M|$) are of the
form $\Phi(U_{(1)},\ldots,U_{(M)})$, for a measurable map $\Phi:\mathbb R^M\to\mathbb R^q$
depending only on $M$ and $\hh$. By~\eqref{eq:full_renyi}, the conditional law of
$\Phi(U_{(1)},\ldots,U_{(M)})$ given $M$ coincides with the law of
$\Phi(\Gamma_1/\Gamma_{M+1},\ldots,\Gamma_M/\Gamma_{M+1})$. Hence any property of the
\emph{joint law} of finitely many such quantities (distribution, moments, independence) can be
checked on the right-hand side of~\eqref{eq:full_renyi}.

\begin{lemma}[Spacing toolbox]\label{lem:spacing_toolbox}
	Let Assumptions~\ref{Ass_gen} and~\ref{ass:random_mi} hold with $g\equiv1$ on
	$\mathcal T=[0,1]$, and let $n\ge4\hh$. For $1\le l\le M$, let
	\[
	L_l:=U_{(l+\hh)}-U_{(l-\hh)}=\frac{2\hh\,(1+\xi_l)}{M+1},
	\]
	with the reflected points when $l-\hh\le0$ or $l+\hh>M$.
	Conditionally given $M$, the following hold.
	\begin{enumerate}[label=(\alph*),ref=\ref{lem:spacing_toolbox}(\alph*)]
		\item\label{tb:pou} \emph{(Partition of unity.)} Let $1-\hh\le r\le M+\hh-1$ and
		$t\in(U_{(r)},U_{(r+1)})$. Then $\{l\in\mathbb Z:\chi_l(t)=1\}=\{r-\hh+1,\ldots,r+\hh\}$.
		Consequently $\sum_{l=1}^M\chi_l(t)\le2\hh$ for almost every $t\in\mathcal T$, with
		equality for almost every $t\in\mathcal I_M:=[U_{(\hh)},U_{(M-\hh+1)})$.

		\item\label{tb:mom} \emph{(Moments.)} For $l\in\Ln$, the conditional law of
		$L_l$ given $M$ is $\mathrm{Beta}(2\hh,M+1-2\hh)$, and
		\[
		\E[\xi_l\mid M]=0,\quad \operatorname{Var}(\xi_l\mid M)=\frac{M+1-2\hh}{2\hh(M+2)}
		=\frac1{2\hh}-\frac{2\hh+1}{2\hh(M+2)},\quad \E[\xi_l^4\mid M]\le C\hh^{-2}.
		\]
		For all $1\le l\le M$ and all integers $p\ge1$, $\E[(1+\xi_l)^p\mid M]\le(2+p)^p$.
		
		\item\label{tb:dec} \emph{(Decoupling.)} For $l\in\Ln$, let
		\[
		V_l:=\frac{U_{(l-\hh)}}{1-L_l}.
		\]
		Then $V_l\in[0,1]$ and
		\begin{equation}\label{eq:dec_pointwise}
			\bigl|U_{(l)}-V_l\bigr|\le L_l
		\end{equation}
		pointwise, and $V_l$ and $\xi_l$ are conditionally independent given $M$. In particular,
		$\E[\varphi(|U_{(l)}-V_l|)\mid M]\le\E[\varphi(L_l)\mid M]$ for every non-negative
		non-decreasing function $\varphi$.
		
		\item\label{tb:cells} \emph{(Cells.)} For $\hh\le r\le M-\hh$, let
		$\Delta_{r+1}:=U_{(r+1)}-U_{(r)}$ and $\Lambda_r:=U_{(r+\hh)}-U_{(r-\hh+1)}$. Then, for
		$p\in\{1,2\}$,
		\[
		\sum_{r=\hh}^{M-\hh}\E\bigl[\Delta_{r+1}\Lambda_r^p \!\mid\! M\bigr]\le C\Bigl(\frac{\hh}{M}\Bigr)^p ,
		\ \
		\E\bigl[|\mathcal T\setminus\mathcal I_M| \!\mid\! M\bigr]=\frac{2\hh}{M+1},\ \
		\E\bigl[|\mathcal T\setminus\mathcal I_M|^2 \!\mid\! M\bigr]\le\frac{8\hh^2}{M^2},
		\]
		where $|\cdot|$ denotes the Lebesgue measure (length).
	\end{enumerate}
	The constants $C$ are absolute.
\end{lemma}

\begin{proof}[Proof of Lemma~\ref{lem:spacing_toolbox}]
	For simplicity, we omit the
	conditioning on $M$ in the moments and write, for instance, $\E[\cdot]$ instead of
	$\E[\cdot\mid M]$. Since the $U_{(j)}$ are almost surely distinct and lie in $(0,1)$, we 	may argue on the event $\{0<U_{(1)}<\cdots<U_{(M)}<1\}$.
	
	\emph{(a)}  The extended sequence
	$(U_{(l)})_{1-2\hh\le l\le M+2\hh}$ is strictly increasing, because the reflections preserve
	the order. For $t\in(U_{(r)},U_{(r+1)})$, $\chi_l(t)=1$ if and only if $l-\hh\le r$ and
	$l+\hh\ge r+1$, that is, $r-\hh+1\le l\le r+\hh$. This set has $2\hh$ elements, and it is
	contained in $\{1,\ldots,M\}$ if and only if $\hh\le r\le M-\hh$, that is,
	$t\in\mathcal I_M$. For $t\notin(U_{(1-\hh)},U_{(M+\hh)})$, no window
	$[U_{(l-\hh)},U_{(l+\hh)}]$ with $1\le l\le M$ contains $t$.
	
	\smallskip
	
	\emph{(b)} Write $\Delta_j:=U_{(j)}-U_{(j-1)}$ for the increments of the extended sequence.
	By the definition of the reflection, $\Delta_j=\Delta_{3-j}$ for
	$j\le1$ and $\Delta_j=\Delta_{2M+1-j}$ for $j\ge M+1$. Since $M\ge n\ge4\hh$, this defines a
	map $\rho$ from $\{2-\hh,\ldots,M+\hh\}$ to $\{2,\ldots,M\}$, at most two-to-one, such that
	$\Delta_j=\Delta_{\rho(j)}$. Hence, for $1\le l\le M$,
	\begin{equation}\label{eq:Ll_gaps}
		L_l=\sum_{j=l-\hh+1}^{l+\hh}\Delta_j=\sum_{i\in J_l}c_i\Delta_i,
		\qquad c_i\in\{1,2\},\quad J_l\subset\{2,\ldots,M\},\quad |J_l|\le2\hh .
	\end{equation}
	For $l\in\Ln$, $\rho(j)=j$ on the window, $J_l=\{l-\hh+1,\ldots,l+\hh\}$ and all $c_i=1$.

	Next, by~\eqref{eq:full_renyi}, the vector
	$(\Delta_2,\ldots,\Delta_M)$ has the same law as
	$(E_2,\ldots,E_M)/\Gamma_{M+1}$, which is a sub-vector of a
	$\mathrm{Dirichlet}(1,\ldots,1)$ vector of dimension $M+1$. Consequently, for any set
	$J\subset\{2,\ldots,M\}$ of distinct indices,
	$\sum_{i\in J}\Delta_i\sim\mathrm{Beta}(|J|,M+1-|J|)$. For $l\in\Ln$, this gives
	$L_l\sim\mathrm{Beta}(2\hh,M+1-2\hh)$, and the expressions for $\E[\xi_l]$ and
	$\operatorname{Var}(\xi_l)$ follow from the Beta mean and variance, since
	$\xi_l=(M+1)L_l/(2\hh)-1$. The fourth central moment of a $\mathrm{Beta}(a,b)$ variable with
	$1\le a\le b$ is bounded by $Ca^2/(a+b)^4$, whence
	$\E[\xi_l^4]=\{(M+1)/(2\hh)\}^4\E[(L_l-\E L_l)^4]\le C\hh^{-2}$. For general
	$1\le l\le M$, \eqref{eq:Ll_gaps} gives the inequality
	$1+\xi_l=(M+1)L_l/(2\hh)\le(M+1)S_l/\hh$, where $S_l:=\sum_{i\in J_l}\Delta_i\sim
	\mathrm{Beta}(|J_l|,M+1-|J_l|)$. Hence
	\[
	\E[(1+\xi_l)^p]\le\Bigl(\frac{M+1}{\hh}\Bigr)^p\prod_{i=0}^{p-1}\frac{|J_l|+i}{M+1+i}
	\le\Bigl(\frac{2\hh+p}{\hh}\Bigr)^p\le(2+p)^p .
	\]
	
	\smallskip
	
	\emph{(c)} Let $l\in\Ln$.   Since $L_l\le U_{(M)}-U_{(1)}<1$, $V_l$ is
	well defined and non-negative, and $V_l\le1$ because $U_{(l-\hh)}+L_l=U_{(l+\hh)}\le1$. Let
	$a_l:=U_{(l)}-U_{(l-\hh)}\in[0,L_l]$. Then
	\[
	U_{(l)}-V_l=U_{(l-\hh)}+a_l-\frac{U_{(l-\hh)}}{1-L_l}
	=a_l-V_l\,L_l .
	\]
	Both $a_l$ and $V_lL_l$ lie in $[0,L_l]$, which yields~\eqref{eq:dec_pointwise}. The last
	claim of (c) follows, because $\varphi$ is non-decreasing.
	
	Concerning the independence, the pair $(\xi_l,V_l)$ is a measurable map of
	$(U_{(1)},\ldots,U_{(M)})$, namely of $(U_{(l-\hh)},U_{(l+\hh)})$. Evaluating the same
	function at $(\Gamma_j/\Gamma_{M+1})_{1\le j\le M}$, with $J_l=\{l-\hh+1,\ldots,l+\hh\}$ and
	\[
	A:=\sum_{i\in J_l}E_i,\qquad B:=\Gamma_{M+1}-A=\sum_{i\notin J_l}E_i,\qquad
	B_-:=\sum_{i=1}^{l-\hh}E_i ,
	\]
	gives the pair
	\[
	\Bigl(\frac{(M+1)A}{2\hh(A+B)}-1\, , \ \frac{B_-/\Gamma_{M+1}}{1-A/\Gamma_{M+1}}\Bigr)
	=\Bigl(\frac{(M+1)A}{2\hh(A+B)}-1\, \ \frac{B_-}{B}\Bigr).
	\]
	The first component is a function of $(A,B)$, the second is a function of
	$(E_i/B)_{i\notin J_l}$. The variable $A$ is independent of $(E_i)_{i\notin J_l}$, and $B$
	is independent of $(E_i/B)_{i\notin J_l}$ (a standard property of i.i.d.\ Gamma variables).
	Hence $(A,B)$ is independent of $(E_i/B)_{i\notin J_l}$, and the two components are
	independent. By~\eqref{eq:full_renyi}, $(\xi_l,V_l)$ has the same conditional law given $M$
	as this pair, so $\xi_l$ and $V_l$ are conditionally independent given $M$.
	
	\emph{(d)} For $\hh\le r\le M-\hh$, we have pointwise
	$\Lambda_r=\sum_{i=r-\hh+2}^{r+\hh}\Delta_i$, a sum of $2\hh-1$ increments with distinct
	indices in $\{2,\ldots,M\}$, one of which is $\Delta_{r+1}$. By~\eqref{eq:full_renyi}, the
	vector $(\Delta_2,\ldots,\Delta_M)$ is exchangeable, so
	$\E[\Delta_i\Lambda_r^p]$ does not depend on $i\in\{r-\hh+2,\ldots,r+\hh\}$. Summing over
	$i$ gives $\E[\Delta_{r+1}\Lambda_r^p]=\E[\Lambda_r^{p+1}]/(2\hh-1)$. Moreover,
	$\Lambda_r\sim\mathrm{Beta}(2\hh-1,M+2-2\hh)$, so that
	\[
	\E[\Delta_{r+1}\Lambda_r^p]\le\frac1{2\hh-1}\Bigl(\frac{2\hh+p}{M+1}\Bigr)^{p+1}
	\le C\frac{\hh^p}{M^{p+1}}.
	\]
	Summing over at most $M$ values of $r$ gives the first bound. Next, pointwise,
	$|\mathcal T\setminus\mathcal I_M|=U_{(\hh)}+\{1-U_{(M-\hh+1)}\}$, and by~\eqref{eq:full_renyi}
	each of the two terms is $\mathrm{Beta}(\hh,M+1-\hh)$ distributed. Hence
	$\E|\mathcal T\setminus\mathcal I_M|=2\hh/(M+1)$ and
	\[
	\E|\mathcal T\setminus\mathcal I_M|^2\le\frac{4\hh(\hh+1)}{(M+1)(M+2)}\le\frac{8\hh^2}{M^2}.
	\]
\end{proof}

\medskip

Next, we propose a second auxiliary result.

\begin{lemma}[Auxiliary lemma]\label{lem:energy_sp} 
	The conditions of Assumption~\ref{ass:design} hold, $\hh\in\mathbb N^*$, $\hh\to\infty$ and  $n^{-1}\hh^2\to 0$ (conditions satisfied with $\hh \asymp  \log\log n$). Then, for  $1\leq l\leq M$,
	\[
	\mathfrak d_l=\frac{T_{(l+\hh)}-T_{(l-\hh)}}{2\hh}-\frac{1}{(M+1)g(T_{(l)})},
	\qquad
	\mathcal E_M(c)=\sum_{l\in\Ln}c(T_{(l)})\bigl[(M+1)\mathfrak d_l\bigr]^2,
	\]
	with $c:\mathcal T\to\mathbb R$ a bounded and Lipschitz continuous function. Then
	\begin{equation}\label{eq:mean_energy_sp}
		\mathbb E[\mathcal E_M(c)\mid M ]=\{1+o(1)\}\frac{M+1}{2\hh}\int_{\mathcal T}c(t)g^{-1} (t) {\rm d}t,\qquad 
		\operatorname{Var}(\mathcal E_M(c)\mid M  )\leq C M/\hh,
	\end{equation}
	with the $o(1)$ depending on   $g_{\min}$, the Lipschitz
	constant $L_g$ of $g$, and  $\|c\|_\infty$ and Lipschitz constant of $c$.
	
\end{lemma}

\smallskip

\begin{proof}[Proof of Lemma~\ref{lem:energy_sp}]
	For simplicity, in the following we write $\E[\cdots]$, 
	$\operatorname{Var}[\cdots]$ and $\operatorname{Cov}[\cdots]$ instead of 
	$\E[\cdots\mid M]$, $\operatorname{Var}[\cdots\mid M]$ and $\operatorname{Cov}[\cdots\mid M]$, respectively. 
	Recall \eqref{eq:dl-exact}, \eqref{eq:moments} from Section~\ref{sec:subsecQQ}.	We need to derive a workable expansion for the covariance of functionals of any pair of order statistics $T_{(l)}, T_{(l')}$. First recall the exact relationship
	\begin{equation}\label{eq:exact_cov_order}
		\operatorname{Cov}(U_{(l)},U_{(l')})=\frac{l(M+1-l')}{(M+1)^2(M+2)},\quad \text{ for } \ 1\le l\le l'\le M.
	\end{equation}
	Then, for any differentiable function $\varphi,\varrho:\mathcal T \to\mathbb R$ with bounded derivatives,
	writing $\bar\varphi=\varphi\circ Q$ on $[0,1]$ and using~\eqref{eq:exact_cov_order},
	two-term Taylor expansion of $\bar\varphi,\bar\varrho$ about $\mathbb E  U_{(l)},\mathbb E  U_{(l')}$
	with the $O(\operatorname{Var} U_{(\cdot)})=O(M^{-1})$ remainder gives, for $l\le l'$,
	\begin{equation}\label{eq:cov-os}
		\operatorname{Cov}\bigl(\varphi(T_{(l)}),\varrho(T_{(l')})\bigr)
		=\bar\varphi'(\mathbb E  U_{(l)})\bar\varrho'(\mathbb E  U_{(l')})\frac{l(M+1-l')}{(M+1)^2(M+2)}+O(M^{-3/2})
	\end{equation}
	with the constant in the remainder $O(M^{-3/2})$ depending on the uniform norms of the derivatives of $\varphi,\varrho$ and $g_{\rm min}$. Consequently, since 
	$$
	\sum_{1\le l\le l'\le M}\frac{l(M+1-l')}{(M+1)^2(M+2)}
	=\frac{1}{(M+1)^2(M+2)}\cdot\frac{M^4}{24}(1+o(1))=\frac{M}{24}(1+o(1)),
	$$
	(the constants in $o(1)$ are universal), it holds 
	\begin{equation}\label{eq:cov-sum}
		\sum_{l\ne l'}\bigl|\operatorname{Cov}(\varphi(T_{(l)}),\varrho(T_{(l')}))\bigr|
		\le \|\bar\varphi'\|_\infty\|\bar\varrho'\|_\infty\,\frac{M}{12}(1+o(1))+O(M^{1/2})\leq C' M.
	\end{equation}
	
	Let us now compute the conditional mean given $M$.  Expand $[(M+1)\mathfrak d_l]^2=Z^{(1)}_l+2Z^{(2)}_l+Z^{(3)}_l$ as
	the square of \eqref{eq:dl-exact}. For the leading term, with
	$\bar\psi(U_{(l)}):=g(T_{(l)})^{-2}$ (so that $\|\bar\psi'\|_\infty\le 2L_g g_{\min}^{-4}$),
	on $\mathcal G_0$ we get
	\[
	\mathbb E [c(T_{(l)})Z^{(1)}_l]
	=\mathbb E \bigl[(\delta_l-\epsilon)^2\,(c\bar\psi)(U_{(l)})\bigr]\{1+o(1)\}.
	\]
	By \eqref{eq:moments}, $\mathbb E [(\delta_l-\epsilon)^2]=1/(2\mathfrak h)-1/(M+1)$, and
	by Cauchy--Schwarz with $\operatorname{Var}((\delta_l-\epsilon)^2)\lesssim \mathfrak h^{-2}$,
	$\operatorname{Var}(U_{(l)})\lesssim M^{-1}$, we get 
	$$
	\operatorname{Cov}((\delta_l-\epsilon)^2,(c\bar\psi)(U_{(l)}))=O(\mathfrak h^{-1}M^{-1/2}).
	$$ 
	Hence
	$\mathbb E [c(T_{(l)})Z^{(1)}_l]= (2\mathfrak h)^{-1}\mathbb E [(c/g^2)(T_{(l)})]+O(\mathfrak h^{-1}M^{-1/2})$, with the constant in the $O(\mathfrak h^{-1}M^{-1/2})$ depending on $c$ and $g$. 
	Summing over $\mathcal L_n$ and using the sample-sum identity
	$$\sum_{l=1}^M(c/g^2)(T_{(l)})=\sum_{i=1}^n \sum_{j=1}^{m_i}(c/g^2)(T_{ij}),$$ with
	$\mathbb E [(c/g^2)(T_{ij})]=\int_{\mathcal T} c/g$, we get
	\[
	\sum_{l\in\mathcal L_n}\mathbb E [c(T_{(l)})Z^{(1)}_l]
	=\{1+o(1)\} \frac{M+1}{2\mathfrak h}\int_{\mathcal T}cg^{-1}(t) {\rm d}t.
	\]
	The $o(1)$ is due to the boundary ranks not included in $\mathcal L_n$ (there are $2\mathfrak h$ of them) and per-term errors derived above
	($\sum O(\mathfrak h^{-1}M^{-1/2})=O(M^{1/2}\mathfrak h^{-1})$), their accumulated contribution  being $o(M/\mathfrak h)$. Next, for the mean of the sum with $Z^{(2)}_l$, by \eqref{eq:dl-exact} and $\mathbb E (A_l^2+B_l^2)=O(\mathfrak h^2)$, we deduce that $\sum_{l\in\mathcal L_n}\mathbb E |c(T_{(l)})\,Z^{(2)}_l|=O(\mathfrak h^{1/2})$.
	By similar arguments, $\sum_l\mathbb E [c\,Z^{(3)}_l]=O(\mathfrak h^2M^{-1})$. Thus, the expectations of the sums with $Z^{(2)}_l$ and $Z^{(3)}_l$ are negligible compared with that with $Z^{(1)}_l$. This proves the first part of~\eqref{eq:mean_energy_sp}.
	
	For the second part concerning the conditional variance, assume without loss of generality that $c(\cdot)\geq 0$ and put $Z_l=c(T_{(l)})[(M+1)\mathfrak d_l]^2\ge0$, so that 
	$\mathbb E  Z_l=O(\mathfrak h^{-1})$, $\operatorname{Var}(Z_l)=O(\mathfrak h^{-2})$ (dominated by
	$\operatorname{Var}((\delta_l-\epsilon)^2)=O(\mathfrak h^{-2})$ because the $\mathfrak s$ terms are
	$O(\mathfrak h^2M^{-2})$). Then, we write  $\operatorname{Var}(\mathcal E_M)=\sum_{l,l'}\operatorname{Cov}(Z_l,Z_{l'})$ and bound separately the overlapping ($|l-l'|<2\mathfrak h$) and the non-overlapping ($|l-l'|\ge 2\mathfrak h$) pairs. There are at most $4\mathfrak h M$ ordered pairs, each of them can be bounded 
	$$|\operatorname{Cov}(Z_l,Z_{l'})|\le(\operatorname{Var} Z_l\operatorname{Var} Z_{l'})^{1/2}=O(\mathfrak h^{-2}),$$ so that in total their contribution at the variance is 
	$O(4\mathfrak h M\cdot\mathfrak h^{-2})=O(M/\mathfrak h)$.
	
	For the non-overlapping pairs ($|l-l'|\ge2\mathfrak h$), the sets of indices for  $W_l,W_{l'}$ are disjoint, so
	the raw sums $s_l=\sum_{j\in W_l}(E_j-1)=2\mathfrak h\delta_l$ and
	$s_{l'}=2\mathfrak h\delta_{l'}$ are independent, and writing
	$r=\sum_{j\notin W_l\cup W_{l'}}(E_j-1)$ we have $s_l,s_{l'},r$ mutually
	independent with $\operatorname{Var} s_l=\operatorname{Var} s_{l'}=2\mathfrak h$, $\operatorname{Var} r=M+1-4\mathfrak h$. (Recall that for
	$s\sim$ a sum of $m$ i.i.d.\ centered $\mathrm{Exp}(1)$ variables,
	$\mathbb E  s^2=m$, $\mathbb E  s^4=3m^2+6m$, so $\operatorname{Var}(s^2)=2m^2+6m$.) Setting
	$a=(2\mathfrak h)^{-1}- b$ with $b=(M+1)^{-1}$, 
	\[
	\delta_l-\epsilon=a\,s_l-b\,s_{l'}-b\,r,\qquad
	\delta_{l'}-\epsilon=a\,s_{l'}-b\,s_l-b\,r,
	\]
	whence $\zeta_l=(\delta_l-\epsilon)^2$ and $\zeta_{l'}=(\delta_{l'}-\epsilon)^2$
	are quadratic forms in the three independent variables $(s_l,s_{l'},r)$. Every
	term of $\operatorname{Cov}(\zeta_l,\zeta_{l'})$ must share at least one of these variables
	between the two factors; each shared variable enters $\zeta_{l'}$ (resp.\
	$\zeta_l$) only through a coefficient $b=(M+1)^{-1}$ (the coefficient $a$
	multiplies $s_l$ in $\zeta_l$ but $s_l$ appears in $\zeta_{l'}$ only via $b$).
	A direct expansion using the moments above gives
	\begin{equation}
		\left|\operatorname{Cov}(\zeta_l,\zeta_{l'})\right|\le C\,(M+1)^{-2}
		\qquad\text{uniformly over disjoint pairs } |l-l'|\ge 2\mathfrak h,
	\end{equation}
	for an absolute constant $C$. 
	Since $w_l=(c/g^2)(T_{(l)})$ is bounded and the
	$\zeta$–location coupling contributes at the same or smaller order (through
	$\operatorname{Cov}(w_l,w_{l'})=O(M^{-1})$ against $\mathbb E \zeta_l\mathbb E \zeta_{l'}=O(\mathfrak h^{-2})$, already accounted for in the first sum), the non-overlapping pairs total is
	\[
	\sum_{|l-l'|\ge2\mathfrak h}\!\!\operatorname{Cov}(Z_l,Z_{l'})
	=\underbrace{\sum_{l\ne l'}\mathbb E \zeta_l\mathbb E \zeta_{l'}\operatorname{Cov}(w_l,w_{l'})}_{=\,O(\mathfrak h^{-2})\cdot O(M)=O(M/\mathfrak h^2)}
	+\underbrace{\sum_{l\ne l'}\mathbb E [w_lw_{l'}]\operatorname{Cov}(\zeta_l,\zeta_{l'})}_{=\,O(M^2)\cdot O(M^{-2})=O(1)}
	=o(M/\mathfrak h),
	\]
	so the overlapping pairs dominate.
\end{proof}

\medskip


\begin{proof}[Proof of Lemma~\ref{lem:var_Q21}]
	Write
	$w^{(21)}_{ll'}=(M+1)^2\sum_{k\in\KM}\widehat D_{kl,1}\widehat D_{kl',1}$, so that
	$\mathcal Q_{21,n}=\sum_{(l,l')\in\In}w^{(21)}_{ll'}\mathcal Y_{l,l'}$ with $w^{(21)}$
	$\mathcal F^{\rm obs}$-measurable. Under $H_0$,
	$\E[\mathcal Y_{l,l'}\mid\mathcal F^{\rm obs}]=0$ for $(l,l')\in\In$, so that 
	\begin{equation}\label{eq:Q21-reduction} 
		\operatorname{Var}(\mathcal Q_{21,n}) = \E \Big[\! \operatorname{Var}(\mathcal Q_{21,n}\mid\mathcal F^{\rm obs})\Big]+ \operatorname{Var}\!\Big[\E(\mathcal Q_{21,n}\mid\mathcal F^{\rm obs})\Big] \le 2\bar\rho^{\,2}m_{\max}^2\,
		\E\!\sum_{(l,l')\in\In}\!\! \bigl(w^{(21)}_{ll'}\bigr)^{\!2} ,
	\end{equation}
	with 	$\bar\rho=\nu(\mathcal U)(\|c_X\|_\infty+\|\tau\|_\infty^2)$.
	Write $S_l:=T_{(l+\hh)}-T_{(l-\hh)}$. By \eqref{eq:neg_D1l},
	$$|\widehat D_{kl,1}|\le\|\phi_k'\|_\infty S_l^2/(2\hh),$$ hence by the Cauchy--Schwarz
	inequality in $k$ and $\sum_{k\in\KM}\|\phi_k'\|_\infty^2\lesssim|\KM|^3$,
	\begin{multline}\label{eq:Q21-perpair}
		\bigl(w^{(21)}_{ll'}\bigr)^2
		=(M+1)^4\Bigl(\sum_{k\in\KM}\! \widehat D_{kl,1}\widehat D_{kl',1}\Bigr)^{\!2}\!
		\\	\le (M+1)^4\Bigl(\sum_{k\in\KM}\! \|\phi_k'\|_\infty^2\Bigr)^{\!2}\frac{S_l^4S_{l'}^4}{(2\hh)^4}
		\lesssim (M+1)^4|\KM|^6\,\frac{S_l^4S_{l'}^4}{\hh^4}.
	\end{multline}
	The spacing conditional moments (given $M$) are controlled as in the proof of Lemma~\ref{lem:energy_sp}: since
	$Q'\le g_{\min}^{-1}$, one has $S_l\le g_{\min}^{-1}(U_{(l+\hh)}-U_{(l-\hh)})$, and the
	uniform spacing $U_{(l+\hh)}-U_{(l-\hh)}\sim\mathrm{Beta}(2\hh,M+1-2\hh)$ (conditionally given $M$) satisfies
	$\E[(U_{(l+\hh)}-U_{(l-\hh)})^8\mid M]\lesssim(\hh/M)^8$. Therefore, we get 
	$\E[S_l^8\mid M ]\lesssim g_{\min}^{-8}\hh^8/M^8$, and by the Cauchy--Schwarz inequality,
	\[
	\E[S_l^4S_{l'}^4\mid M]\le\bigl\{\E[S_l^8\mid M]\bigr\}^{1/2}\bigl\{\E[S_{l'}^8\mid M]\bigr\}^{1/2}
	\lesssim g_{\min}^{-8}\,\frac{\hh^8}{M^8}.
	\]
	Summing \eqref{eq:Q21-perpair} over the $|\In|\asymp M^2$ pairs,
	\[
	\E\Bigg[\sum_{(l,l')\in\In}\bigl(w^{(21)}_{ll'}\bigr)^2\mid M\bigg]
	\lesssim M^2\cdot M^4|\KM|^6\hh^{-4}\cdot g_{\min}^{-8}\hh^8 M^{-8}
	= g_{\min}^{-8}\,\frac{|\KM|^6\hh^4}{M^2},
	\]
	the overlapping pairs ($|\In^{\rm ov}|\le 4\hh M$) contributing at smaller order by the
	same crude bound. Since the $1\le m_i \le m_{\rm max}$ are i.i.d. distributed as $\mathfrak m$, by the law of large numbers and the dominated convergence theorem, $\E[M^{-2}] = \E[\mathfrak m]^{-2}n^{-2}\{1+o(1)\}$. Integrating out $M$  in the last display and combining with \eqref{eq:Q21-reduction}, yields
	\[
	\operatorname{Var}(\mathcal Q_{21,n})\ \le\ C\,\E[\mathfrak m]^{-2} |\KM|^6\hh^4 n^{-2}\{1+o(1)\},
	\]
	and all the variance inequalities in the statement follow.  
	
	We argue that if, instead of using the crude bound $|\widehat D_{kl,1}|\le\|\phi_k'\|_\infty S_l^2/(2\hh)$ in~\eqref{eq:Q21-perpair}, one expands the square of the double sum into a sum over four indices, then it is plausible that the condition $\hh\,|\KM|=o(n^{4/5})$ can be relaxed to $\hh^2\,|\KM|=o(n)$. Such a refinement would, however, require substantially longer calculations.	
\end{proof}

\medskip


\begin{proof}[Proof of Lemma~\ref{lem:var_Q23}] 
	Write $b_l:=(M+1)\mathfrak d_l$ ($\mathcal F^{\rm obs}$-measurable) and 
	$$
	\mathcal Q_{23,n}=\sum_{(l,l')\in\In}\sum_{k\in\KM}b_lb_{l'}\phi_k(T_{(l)})\phi_k(T_{(l')})\mathcal Y_{l,l'}
	=\sum_{(l,l')\in\In}a_{ll'}\,\mathcal Y_{l,l'},
	$$ 
	where
	$a_{ll'}:=b_lb_{l'}\Psi(T_{(l)},T_{(l')};\KM)$ is $\mathcal F^{\rm obs}$-measurable and symmetric.

	By definition, under $H_0$ ($\mu\equiv0$) we have $\E[\mathcal Y_{l,l'}\mid\mathcal F^{\rm obs}]=0$
	for $(l,l')\in\In$, and
	\[
	\operatorname{Var}(\mathcal Q_{23,n}\mid\mathcal F^{\rm obs})
	=\sum_{(l,l')\in\In}\sum_{(l'',l''')\in\In}a_{ll'}a_{l''l'''}\,
	\Cov\bigl(\mathcal Y_{l,l'},\mathcal Y_{l'',l'''}\mid\mathcal F^{\rm obs}\bigr).
	\]
	Recall that, in the context of Lemma~\ref{lem:var_Q23}, $\varrho_1^{-1}(l)$ denotes the subject index of rank $l$. Since a subject index
	appearing in exactly one of the two pairs contributes an independent mean-zero
	factor, $\Cov(\mathcal Y_{l,l'},\mathcal Y_{l'',l'''}\mid\mathcal F^{\rm obs})=0$ unless
	$\{\varrho_1^{-1}(l),\varrho_1^{-1}(l')\}=\{\varrho_1^{-1}(l''),\varrho_1^{-1}(l''')\}$ (the two pairs involve the same
	two subjects). For such matched pairs, writing $R_t$ for the conditional response
	covariance operator (kernel $\rho(u,u';t)$) and using
	$\E[Y_{l}(u)Y_{l''}(u')\mid\mathcal F^{\rm obs}]=\rho(u,u';T_{(l)})\mathbf 1\{\varrho_1^{-1}(l)=\varrho_1^{-1}(l''),\ \text{same visit}\}$
	plus the cross-covariance structure, one has the uniform bound
	\[
	\bigl|\Cov(\mathcal Y_{l,l'},\mathcal Y_{l'',l'''}\mid\mathcal F^{\rm obs})\bigr|\le \widetilde\rho^{\,2},
	\qquad \widetilde \rho:=\nu(\mathcal U)\bigl(\|c_X\|_\infty+\|\tau\|_\infty^2\bigr).
	\]
	Therefore
	\begin{equation}\label{eq:condvar-bound}
		\operatorname{Var}(\mathcal Q_{23,n}\mid\mathcal F^{\rm obs})
		\le \widetilde \rho^{\,2}\sum_{\substack{(l,l'),(l'',l''')\in\In\\ \{\varrho_1^{-1}(l),\varrho_1^{-1}(l')\}=\{\varrho_1^{-1}(l''),\varrho_1^{-1}(l''')\}}}
		|a_{ll'}|\,|a_{l''l'''}|.
	\end{equation}
	Next, group by the (unordered) subject-pair $\{i,i'\}$, $i\ne i'$. Each subject has at
	most $m_{\max}$ ranks (Condition~\eqref{cd_m_max}), so the ranks $(l,l')$ with
	$\{\varrho_1^{-1}(l),\varrho_1^{-1}(l')\}=\{i,i'\}$ number at most $2m_{\max}^2$, and likewise
	$(l'',l''')$. Bounding $|a_{ll'}||a_{l''l'''}|\le(a_{ll'}^2+a_{l''l'''}^2)/2$ and summing,
	\begin{equation}\label{eq:aux45}
		\operatorname{Var}(\mathcal Q_{23,n}\mid\mathcal F^{\rm obs})
		\le \widetilde \rho^{\,2}\,(2m_{\max}^2)\!\!\sum_{\{i,i'\}}\ \sum_{\substack{(l,l'):\\ \{\varrho_1^{-1}(l),\varrho_1^{-1}(l')\}=\{i,i'\}}}\!\! a_{ll'}^2
		\ \le\ 2\widetilde\rho^{\,2}m_{\max}^2\sum_{(l,l')\in\In} a_{ll'}^2 ,
	\end{equation}
	the last step re-expanding the subject-pair sum back to the rank-pair sum (each
	rank-pair lies in exactly one subject-pair). Now, we have to integrate over the design. 
	
	Recall that, by \eqref{eq:dl-exact},
	$$
	b_l=(M+1)\mathfrak d_l=\xi_l/g(T_{(l)})+(M+1)\mathfrak s_l \quad \text{  with } \ \E[b_l^2\mid T_{(l)}]=\frac{1+o(1)}{2\hh g^2(T_{(l)})}=:\pi_l \{1+o(1)\},
	$$
	and
	$$
	b_l^2\le \frac{C}{g_{\min}^2}\bigl\{ \xi_l^2+((M+1)\mathfrak s_l)^2\bigr\} \ \text{  on } \
	\mathcal G_0.
	$$
	Next, split $b_l^2=\pi_l+(b_l^2-\pi_l)$, so that
	\[
	b_l^2 b_{l'}^2=\pi_l\pi_{l'}
	+\pi_l\,(b_{l'}^2-\pi_{l'})+\pi_{l'}\,(b_l^2-\pi_l)
	+(b_l^2-\pi_l)(b_{l'}^2-\pi_{l'}).
	\]
	Next, let $K_{ll'}:=\Psi^2(T_{(l)},T_{(l')};\KM)$, a bounded, $\mathcal F^{\rm obs}$-measurable
	kernel (see also Lemma~\ref{lem:order-Gh2}) with 
	$$
	0\le K_{ll'}\le \Psi(T_{(l)},T_{(l)};\KM)\Psi(T_{(l')},T_{(l')};\KM)\le
	(\,|\KM|\sup_k\|\phi_k\|_\infty^2)^2.
	$$   
	Then, after integrating out with respect to the design (but not with respect to $M$) the bound~\eqref{eq:aux45}, we can decompose the $\sum_{(l,l')\in\In}$ as follows:
	\begin{equation}\label{eq:S3-split}
		\E \Bigg[\sum_{(l,l')\in\In} b_l^2 b_{l'}^2 K_{ll'} \mid M \Bigg]
		=\underbrace{\E\Bigg[\sum_{(l,l')\in\In}\pi_l\pi_{l'}K_{ll'}\mid M \Bigg]}_{=:T_0}
		+\ 2\,T_1+T_2,
	\end{equation}
	where $T_1=\E\big[\sum\pi_l(b_{l'}^2-\pi_{l'})K_{ll'}\mid M\big]$ and
	$T_2=\E\big[ \sum(b_l^2-\pi_l)(b_{l'}^2-\pi_{l'})K_{ll'}\mid M\big]$. We show $\E[T_0]$ is the stated
	leading term and $\E|T_1|$, $\E|T_2|$ are of smaller order.
	
	\medskip
	
	\emph{(a) Overlapping vs.\ disjoint windows.} Partition $\In$ into
	$\In^{\rm ov}=\{(l,l')\in\In:|l-l'|<2\hh\}$ and its complement
	$\In^{\rm dis}$. Since each subject carries at most $m_{\max}$ ranks and a fixed
	$l$ has at most $4\hh$ ranks $l'$ with $|l-l'|<2\hh$,
	$|\In^{\rm ov}|\le 4\hh M$, whereas $|\In^{\rm dis}|\asymp M^2$. Thus
	$|\In^{\rm ov}|/|\In^{\rm dis}|=O(\hh/M)$.
	
	\medskip
	
	\emph{(b) Leading term $T_0$.} On $\In^{\rm dis}$ the windows $W_l,W_{l'}$ are
	disjoint, so $\pi_l,\pi_{l'}$ (deterministic functions of $T_{(l)},T_{(l')}$) and
	$K_{ll'}$ are as in Lemma~\ref{lem:order-Gh2}. Thus, we get
	\begin{multline}
		\sum_{(l,l')\in\In^{\rm dis}}  \E[\pi_l\pi_{l'}K_{ll'}\mid M]
		=\frac{1}{(2\hh)^2}\sum_{l\ne l'}\E \left[\frac{\Psi^2(T_{(l)},T_{(l')};\mathcal K_M)}{g^2(T_{(l)})g^2(T_{(l')})}\mid M \right]\{1+o(1)\}
		\\ =\frac{(M+1)^2|\KM|}{(2\hh)^2}\!\int_{\mathcal T}\!\frac{{\rm d}t}{g^2}\,\{1+o(1)\},
	\end{multline}
	using $(M+1)^{-2}\sum_{l\ne l'}[\cdots\mid M ]\to\iint[\cdots ]\,g\,g$ and
	$\iint\Psi^2/(gg)=|\KM|\int g^{-2}\,\{1+o(1)\}$. Using $\pi_l\le\{2\hh g_{\min}^2\}^{-1}$ and
	$K_{ll'}\le|\KM|^2\sup_k\|\phi_k\|_\infty^4$, the overlapping part of $T_0$ can be bounded as follows:
	\[
	0\le\sum_{(l,l')\in\In^{\rm ov}}\!\!\E[\pi_l\pi_{l'}K_{ll'}\mid M]
	\le \frac{4\hh M}{(2\hh)^2 g_{\min}^4}\,|\KM|^2\sup_k\|\phi_k\|_\infty^4
	\le C \frac{M|\KM|^2}{\hh}.
	\]
	Integrating the bounds for 	$\sum_{(l,l')\in\In^{\rm dis}}$ and $\sum_{(l,l')\in\In^{\rm ov}}$ with respect to $M$ and using  $\hh^2|\KM|/n\to0$, we get
	$$
	\E[T_0]=\frac{n^2\E[\mathfrak m]^2|\KM|}{(2\hh)^2}\int_{\mathcal T}g^{-2}\{1+o(1)\}.
	$$
	
	\medskip
	
	\emph{(c) The centered terms $T_1,T_2$.} Both are shown to have the expectations of their absolute values of smaller order than that of $T_0$. The term $T_2$ and the overlapping part of $T_1$ are bounded by crude domination, using the
	deterministic bounds $0\le\pi_l\le\{2\hh g_{\min}^2\}^{-1}$ and $0\le K_{ll'}\le|\KM|^2\sup_k\|\phi_k\|_\infty^4$. The disjoint (non-overlapping part) of $T_1$ keeps the kernel $K_{ll'}$ localized. Throughout we use the centered-moment control $\E[b_l^2-\pi_l\mid T_{(l)}]=o(\pi_l)$ and $\operatorname{Var}(b_l^2\mid M)=O(\hh^{-2})$. See above and the proof of Lemma~\ref{lem:energy_sp}.
	
	For $T_2$, on $\In^{\rm dis}$ the windows $W_l,W_{l'}$ are disjoint, so the centered
	factors $b_l^2-\pi_l$ and $b_{l'}^2-\pi_{l'}$ are independent up to the shared
	normalizer $\epsilon$, whence
	$\E[(b_l^2-\pi_l)(b_{l'}^2-\pi_{l'})\mid T_{(l)},T_{(l')},M]=O(M^{-1}\hh^{-2})$
	(the residual coupling is through $\operatorname{Var}(\epsilon\mid M)=O(M^{-1})$, as in
	Lemma~\ref{lem:energy_sp}). Bounding $K_{ll'}\le|\KM|^2\sup_k\|\phi_k\|_\infty^4$ and
	summing over the $O(M^2)$ disjoint pairs,
	\[
	|T_2^{\rm dis}|\le O(M^{-1}\hh^{-2})\cdot O(M^2)\cdot|\KM|^2\sup_k\|\phi_k\|_\infty^4
	=O\Bigl(\frac{M|\KM|^2}{\hh^2}\Bigr).
	\]
	Integrating out $M$, we get $\E|T_2^{\rm dis}|=o(\E[T_0])=o(n^2|\KM|/\hh^2)$ by $\hh^2|\KM|/n\to0$. On $\In^{\rm ov}$, using
	$|b_l^2-\pi_l|\,|b_{l'}^2-\pi_{l'}|\le\{(b_l^2-\pi_l)^2+(b_{l'}^2-\pi_{l'})^2\}/2$,
	$\operatorname{Var}(b_l^2\mid M)=O(\hh^{-2})$ and $|\In^{\rm ov}|\le 4\hh M$,
	\[
	|T_2^{\rm ov}|\le 4\hh M\cdot O(\hh^{-2})\cdot|\KM|^2\sup_k\|\phi_k\|_\infty^4
	=O\Bigl(\frac{M|\KM|^2}{\hh}\Bigr),
	\]
	again $\E|T_2^{\rm ov}|=o(\E[T_0])$ by $\hh^2|\KM|/n\to0$.

	For $T_1=\E\big[\sum_{(l,l')\in\In} \pi_l(b_{l'}^2-\pi_{l'})K_{ll'}\mid M \big]$, we first investigate the centering of $b_{l'}^2$. By
	\eqref{eq:dl-exact}, on $\mathcal G_0$,
	\[
	b_{l'}^2=\frac{(\delta_{l'}-\epsilon)^2}{g^2(T_{(l')})(1+\epsilon)^2}
	+2\,\frac{\delta_{l'}-\epsilon}{g(T_{(l')})(1+\epsilon)}\,(M+1)\mathfrak s_{l'}
	+\big((M+1)\mathfrak s_{l'}\big)^2 .
	\]
	Take the conditional expectation given $U_{(l)},U_{(l')}$ with
	$(l,l')\in\In^{\rm dis}$, so that the window $W_{l'}$ is disjoint from the ranks
	determining $U_{(l)}$. The same moment computation as in the proof of
	Lemma~\ref{lem:energy_sp} which  controls $\operatorname{Cov}\big((\delta_{l'}-\epsilon)^2,\,g^{-2}(T_{(l')})\big)=O(\hh^{-1}M^{-1/2})$,
	the correction $|(1+\epsilon)^{-2}-1|\le 10|\epsilon|$ on $\mathcal G_0$, together with
	$\E[|\epsilon|(\delta_{l'}-\epsilon)^2]=O(\hh^{-1}M^{-1/2})$, and the $\mathfrak s_{l'}$
	terms (relative order $O(\hh^{3/2}M^{-1})$), combined with the disjoint-window
	near-independence through $\epsilon$ (the residual coupling is $O(M^{-1})$, exactly as
	for $T_2$ below), yields
	\begin{equation}\label{eq:centering-star}
		\big|\,\E[\,b_{l'}^2-\pi_{l'}\mid U_{(l)},U_{(l')},M\,]\,\big|\le C\,M^{-1/2}\,\pi_{l'},
		\qquad (l,l')\in\In^{\rm dis},
	\end{equation}
	with $\pi_{l'}=\{2\hh g^2(T_{(l')})\}^{-1}$. Now, split $T_1=T_1^{\rm dis}+T_1^{\rm ov}$ along $\In=\In^{\rm dis}\cup\In^{\rm ov}$. On
	$\In^{\rm dis}$ the weight $\pi_l K_{ll'}$ is $\sigma(U_{(l)},U_{(l')},M)$-measurable and
	non-negative, so by \eqref{eq:centering-star},
	\begin{multline}
		|T_1^{\rm dis}|
		=\Bigl|\sum_{(l,l')\in\In^{\rm dis}}\E\big[\pi_l K_{ll'}
		\E[b_{l'}^2-\pi_{l'}\mid U_{(l)},U_{(l')},M]\mid M\big] \Bigr| \\ 
		\le C\,M^{-1/2}\sum_{(l,l')\in\In^{\rm dis}}\!\!\E[\pi_l\pi_{l'}K_{ll'}\mid M]
		\le C\,M^{-1/2}\,T_0 .
	\end{multline}
	On $\In^{\rm ov}$ ($|\In^{\rm ov}|\le4\hh M$), the bounds
	$\pi_l\le\{2\hh g_{\min}^2\}^{-1}$, $K_{ll'}\le|\KM|^2\sup_k\|\phi_k\|_\infty^4$ and
	$\E|b_{l'}^2-\pi_{l'}|\le\{\operatorname{Var}(b_{l'}^2\mid M)\}^{1/2}
	+|\E b_{l'}^2-\pi_{l'}|=O(\hh^{-1})$ give
	\[
	|T_1^{\rm ov}|\le 4\hh M\cdot\frac{1}{2\hh g_{\min}^2}\cdot|\KM|^2\sup_k\|\phi_k\|_\infty^4
	\cdot O(\hh^{-1})\le C \frac{M|\KM|^2}{\hh},
	\]
	for some constant $C$. Integrating out $M$, $\E|T_1^{\rm ov}|=o(\E[T_0])$ by $\hh^2|\KM|/n\to0$. Since $M\asymp n$, $\E|T_1^{\rm dis}|\le C\,\E[M^{-1/2}T_0]
	=o(\E[T_0])$. Hence $\E|T_1|$ is negligible compared to $\E[T_0]$.
	
	\medskip
	
	Collecting the facts from (b) and (c) in \eqref{eq:S3-split}, we get 
	\begin{equation}\label{eq:sum-a2}
		\E \Big[\sum_{(l,l')\in\In} a_{ll'}^2\Big]
		=\E[T_0+2T_1+T_2]
		=\frac{n^2\E[\mathfrak m]^2|\KM|}{(2\hh)^2}\int_{\mathcal T}\frac{{\rm d}t}{g^2(t)}\,\{1+o(1)\}.
	\end{equation}
	Finally, by the law of total variance,
	\[
	\operatorname{Var}(\mathcal Q_{23,n}\mid M)
	=\E\bigl[\operatorname{Var}(\mathcal Q_{23,n}\mid\mathcal F^{\rm obs})\mid M\bigr]
	+\operatorname{Var}\bigl(\E[\mathcal Q_{23,n}\mid\mathcal F^{\rm obs}]\mid M\bigr).
	\]
	Under $H_0$, $\E[\mathcal Q_{23,n}\mid\mathcal F^{\rm obs}]
	=\sum_{(l,l')\in\In}a_{ll'}\,\E[\mathcal Y_{l,l'}\mid\mathcal F^{\rm obs}]=0$,
	so the second term vanishes. Combining the steps above,
	\[
	\operatorname{Var}(\mathcal Q_{23,n}\mid M)=\E\bigl[\operatorname{Var}(\mathcal Q_{23,n}\mid\mathcal F^{\rm obs})\mid M\bigr]
	\le 2\widetilde\rho^{\,2}m_{\max}^2\,\E\!\sum_{(l,l')\in\In}a_{ll'}^2
	\le C\,\frac{n^2|\KM|}{\hh^2},
	\]
	and dividing by $\{n(n-1)\}^2$ gives \eqref{eq:var_Q23H0}. The constant $C$ depends only on
	$g_{\min}$, $L_g$, $m_{\max}$, $\nu(\mathcal U)$, $\|c_X\|_\infty$, $\|\tau\|_\infty$, and
	$\max_{k\in\mathcal K_M}\|\phi_k\|_\infty$ (bounded by an absolute constant).
\end{proof}


\medskip

\begin{proof}[Proof of Lemma~\ref{lem:var_Q1}]
	Recall from \eqref{eq:bound_diff_spQ} that
	\[
	\mathcal Q_{1,n}=(M+1)\sum_{(l,l')\in\In}\sum_{k\in\KM}
	\frac{\phi_k(T_{(l)})}{g(T_{(l)})}\,\widehat D_{kl'}\,\mathcal Y_{l,l'}.
	\]
	Split $\widehat D_{kl'}=\widehat D_{kl',1}+\widehat D_{kl',2}$ as in \eqref{eq:def_frakdl},
	write $b_l:=(M+1)\mathfrak d_l$ and $\Psi(\cdot,\cdot):=\Psi(\cdot,\cdot;\KM)$, so that
	$\sum_{k\in\KM}\phi_k(T_{(l)})\phi_k(T_{(l')})=\Psi(T_{(l)},T_{(l')})$.
	Then
	\[
	\mathcal Q_{1,n}=\sum_{(l,l')\in\In}w_{ll'}\,\mathcal Y_{l,l'},
	\qquad w_{ll'}=w^{(1)}_{ll'}+w^{(2)}_{ll'},
	\]
	with the $\mathcal F^{\rm obs}$-measurable coefficients
	\[
	w^{(2)}_{ll'}=\frac{b_{l'}}{g(T_{(l)})}\,\Psi(T_{(l)},T_{(l')}),
	\qquad
	w^{(1)}_{ll'}=(M+1)\sum_{k\in\KM}\frac{\phi_k(T_{(l)})}{g(T_{(l)})}\,\widehat D_{kl',1}.
	\]
	Under $H_0$, $\E[\mathcal Y_{l,l'}\mid\mathcal F^{\rm obs}]=0$ for
	$(l,l')\in\In$, so the conditional-variance argument in the proof of
	Lemma~\ref{lem:var_Q23} applies in a similar way (only matched subject-pairs contribute, and
	$\E[\mathcal Q_{1,n}\mid\mathcal F^{\rm obs}]=0$ makes the second term of the law of total
	variance vanish). With $\widetilde\rho=\nu(\mathcal U)(\|c_X\|_\infty+\|\tau\|_\infty^2)$,
	\begin{equation}\label{eq:Q1-reduction}
		\operatorname{Var}(\mathcal Q_{1,n})
		\le 2\widetilde\rho^{\,2}m_{\max}^2\,\E\!\sum_{(l,l')\in\In}w_{ll'}^2
		\le 4\widetilde \rho^{\,2}m_{\max}^2\,\E\!\sum_{(l,l')\in\In}\bigl\{(w^{(2)}_{ll'})^2+(w^{(1)}_{ll'})^2\bigr\}.
	\end{equation}
	It remains to evaluate the two parts in the bound.
	
	\medskip

	\emph{Leading term.} Partition $\In=\In^{\rm ov}\cup\In^{\rm dis}$ as in the proof of
	Lemma~\ref{lem:var_Q23}. On $\In^{\rm ov}$ ($|\In^{\rm ov}|\le4\hh M$), with
	$\E[b_{l'}^2\mid M]\le\{2\hh g_{\min}^2\}^{-1}\{1+o(1)\}$ and
	$\Psi^2\le(|\KM|\sup_k\|\phi_k\|_\infty^2)^2$,
	\[
	\E\Bigg[\sum_{(l,l')\in\In^{\rm ov}}(w^{(2)}_{ll'})^2\mid M\Bigg]
	\lesssim\frac{4\hh M}{2\hh\,g_{\min}^4}\,|\KM|^2\sup_k\|\phi_k\|_\infty^4
	\lesssim\frac{M|\KM|^2}{g_{\min}^4}.
	\]
	On $\In^{\rm dis}$ the window $W_{l'}$ is disjoint from the ranks determining $T_{(l)}$; we split
	$b_{l'}^2=\pi_{l'}+(b_{l'}^2-\pi_{l'})$ in $(w^{(2)}_{ll'})^2$, with
	$\pi_{l'}=\{2\hh g^2(T_{(l')})\}^{-1}$. By the same argument as in part~(b) of the proof of
	Lemma~\ref{lem:var_Q23}, the $\pi_{l'}$-part yields 
	\begin{multline}
		\sum_{(l,l')\in\In^{\rm dis}}\E\!\Big[\frac{\Psi^2(T_{(l)},T_{(l')})}{g^2(T_{(l)})}\,\pi_{l'}\mid M\Big]
		=\frac{(M+1)^2}{2\hh}\iint_{\mathcal T^2}\frac{\Psi^2(t,s)}{g(t)g(s)}\,{\rm d}t\,{\rm d}s\,\{1+o(1)\}
		\\ =\frac{(M+1)^2|\KM|}{2\hh}\int_{\mathcal T}g^{-2}\,\{1+o(1)\},	
	\end{multline}
	while the centered remainder
	\[
	\Xi^{\rm dis}:=\E\!\sum_{(l,l')\in\In^{\rm dis}}
	\frac{\Psi^2(T_{(l)},T_{(l')})}{g^2(T_{(l)})}\,\bigl(b_{l'}^2-\pi_{l'}\bigr)
	\]
	is handled exactly as $T_1^{\rm dis}$ in the proof of Lemma~\ref{lem:var_Q23}: the weight
	$g^{-2}(T_{(l)})\,\Psi^2$ is non-negative and $\sigma(U_{(l)},U_{(l')},M)$-measurable, so
	by~\eqref{eq:centering-star},
	\[
	|\Xi^{\rm dis}|\le C\,M^{-1/2}\!\!\sum_{(l,l')\in\In^{\rm dis}}\!\!
	\E\!\Big[\frac{\pi_{l'}\,\Psi^2(T_{(l)},T_{(l')})}{g^2(T_{(l)})}\mid M\Big]
	=O(M^{-1/2})\cdot\frac{(M+1)^2|\KM|}{2\hh}\int_{\mathcal T}g^{-2}.
	\]
	Thus the overlapping contribution and the centered remainder are both
	$o\big(\E\sum_{\In}(w^{(2)}_{ll'})^2\big)$ under $\hh^2|\KM|/n\to0$, and integrating out $M$ yields
	\begin{equation}\label{eq:Q12-energy}
		\E\Bigg[\sum_{(l,l')\in\In}(w^{(2)}_{ll'})^2\Bigg]
		=\frac{\E[\mathfrak m]^2\,n^2\,|\KM|}{2\hh}\int_{\mathcal T} g^{-2}(t)\,{\rm d}t\,\{1+o(1)\}.
	\end{equation}
	
	\medskip
	
	\emph{Negligible term.} Next, summing over $k$, we get
	\[
	w^{(1)}_{ll'}
	=\frac{M+1}{2\hh\,g(T_{(l)})}\int_{T_{(l'-\hh)}}^{T_{(l'+\hh)}}
	\bigl\{\Psi(T_{(l)},t)-\Psi(T_{(l)},T_{(l')})\bigr\}\,{\rm d}t.
	\]
	Let $S_{l'}:=T_{(l'+\hh)}-T_{(l'-\hh)}$ and note that	
	$$
	\{\Psi(T_{(l)},t)-\Psi(T_{(l)},T_{(l')})\}^2\le S_{l'}\int_{T_{(l'-\hh)}}^{T_{(l'+\hh)}}
	\big\{\partial_2\Psi(T_{(l)},s)\big\}^2{\rm d}s,
	$$
	Therefore, by the Cauchy-Schwarz inequality 
	\[
	(w^{(1)}_{ll'})^2\le\frac{(M+1)^2 S_{l'}^3}{4\hh^2 g_{\min}^2}
	\int_{T_{(l'-\hh)}}^{T_{(l'+\hh)}}\big\{\partial_2\Psi(T_{(l)},s)\big\}^2\,{\rm d}s.
	\]
	Summing over $l$ at fixed $l'$ over disjoint windows and using
	$$
	\int_{\mathcal T}\big\{\partial_2\Psi(s',s)\big\}^2 g(s')\,{\rm d}s'
	\le g_{\max}\sum_{k\in\KM}\|\phi_k'\|_\infty^2\lesssim g_{\max}|\KM|^3,
	$$ 
	together with
	$\sum_{l'}\E[ S_{l'}^4 \mid M] \lesssim g_{\min}^{-4}\hh^4 M^{-3}$, we get
	\begin{equation}\label{eq:Q11-energy}
		\E\Bigg[\sum_{(l,l')\in\In}(w^{(1)}_{ll'})^2\mid M \Bigg] 
		\lesssim \frac{(M+1)^3|\KM|^3 g_{\max}}{\hh^2 g_{\min}^{2}}\sum_{l'}\E[ S_{l'}^4 \mid M]
		\lesssim \frac{g_{\max}}{g_{\min}^{6}}\,|\KM|^3\hh^2,
	\end{equation}
	the overlapping pairs ($|\In^{\rm ov}|\le4\hh M$) being handled by the inequality $|\partial_\theta\Psi|\lesssim|\KM|^2$ and contributing at smaller order.
	Integrating out $M$ in~\eqref{eq:Q11-energy} and using $\hh^2|\KM|/n\to0$ implies  the right-hand side of
	\eqref{eq:Q11-energy} is negligible against
	\eqref{eq:Q12-energy}.
	
	\medskip
	
	\emph{Collecting.} Inserting \eqref{eq:Q12-energy}-\eqref{eq:Q11-energy} into
	\eqref{eq:Q1-reduction} and bounding $\int_{\mathcal T}g^{-2}\le g_{\min}^{-2}$,
	\[
	\operatorname{Var}(\mathcal Q_{1,n})
	\le\Bigl\{2\,m_{\max}^2\,\nu(\mathcal U)^2\bigl(\|c_X\|_\infty+\|\tau\|_\infty^2\bigr)^2
	\!\int_{\mathcal T}\! g^{-2}+o(1)\Bigr\}\frac{n^2|\KM|}{\hh}
	\le C\,\frac{n^2|\KM|}{\hh}.
	\]
	Finally, $M\le n\,m_{\max}$ gives $M^2/\{n(n-1)\}^2\le C\,n^{-2}$, so
	$\operatorname{Var}(\{n(n-1)\}^{-1}\mathcal Q_{1,n})\le C\,|\KM|/(n^2\hh)$, and multiplying by
	$n^2|\KM|^{-1}$ yields
	$\operatorname{Var}(n|\KM|^{-1/2}\{n(n-1)\}^{-1}\mathcal Q_{1,n})\le C'/\hh$ with $C'$ independent of
	$|\KM|$. The constant $C$ depends only on $g_{\min}$, $g_{\max}$, $L_g$, $m_{\max}$,
	$\nu(\mathcal U)$, $\|c_X\|_\infty$, $\|\tau\|_\infty$ (and $\sup_k\|\phi_k\|_\infty$,
	equal to $4\sqrt2$ for the augmented cosine basis).
\end{proof}

\medskip

We now present the technical results used to prove the equivalence under the local alternatives between the spacings-based statistics and the one constructed with the known density and equal weights. 

\begin{proof}[Proof of Lemma~\ref{lem:mean_Qsp_H1}]
	Recall that
	\[
	\xi_l:=(M+1)\Bigl(\frac{U_{(l+\hh)}-U_{(l-\hh)}}{2\hh}-\frac1{M+1}\Bigr),
	\qquad 1\le l\le M .
	\]
	We abbreviate $\|\cdot\|:=\|\cdot\|_{L^2(\mathcal U)}$ and
	$\langle\cdot,\cdot\rangle:=\langle\cdot,\cdot\rangle_{L^2(\mathcal U)}$, and write
	$\mu_l:=\mu(\cdot,U_{(l)})\in L^2(\mathcal U)$. Let
	\[
	\mu_{\KM}(\cdot,t):=\sum_{k\in\KM}\beta_k(\cdot)\,\phi_k(t),\ \ \
	\bar\mu_1:=\sup_{t\in\mathcal T}\|\partial_t\mu(\cdot,t)\|,\ \ \
	F(t):=\Psi(t,t;\KM)\,\|\mu(\cdot,t)\|^2 .
	\]
	By the assumptions, the series $\sum_{k\ge1}\beta_k\phi_k(t)$ and
	$\sum_{k\ge1}\beta_k\phi_k'(t)$ converge in $L^2(\mathcal U)$, absolutely and uniformly in $t$,
	because $\sum_{k\ge1}k\|\beta_k\|\le Cr_n$ by~\eqref{eq:sums_bds2}. We identify $\mu(\cdot,t)$
	with the sum of its series, so that $t\mapsto\mu(\cdot,t)$ is continuously differentiable in
	$L^2(\mathcal U)$, and
	\begin{equation}\label{eq:mu_bars}
		\bar\mu+\bar\mu_1+\sup_{t\in\mathcal T}\|\mu_{\KM}(\cdot,t)\|
		\le 3C_\phi\sum_{k\ge1}k\,\|\beta_k\|\le C_1\,r_n ,
	\end{equation}
	with $C_1$ depending only on $C_\phi:=\max\{\sup_k\|\phi_k\|_\infty,\ \sup_k\|k^{-1}\phi_k'\|_\infty\}$ and $C_\Upsilon$. Consequently,
	\begin{equation}\label{eq:F_bounds}
		\|F\|_\infty\le C_\phi^2\,|\KM|\,\bar\mu^2,\qquad
		\|F'\|_\infty\le 2C_\phi^2  |\KM|^2\bar\mu^2+2C_\phi^2|\KM|\,\bar\mu\,\bar\mu_1
		\le C\,|\KM|^2r_n^2 .
	\end{equation}
	Finally, when there is no possible confusion, we omit the conditioning on $m_1,\ldots,m_n$ or on $M$ in the moments and write, for instance, $\E[\cdot]$ instead of $\E[\cdot\mid M]$.

	Since $\In$ and $\mathcal Y_{l,l'}$ are symmetric in $(l,l')$, the decomposition~\eqref{eq:bound_diff_spQ} gives
	\[
	2\mathcal Q_{1,n}+\mathcal Q_{2,n}=\!\!\sum_{(l,l')\in\In}w_{ll'}\mathcal Y_{l,l'}, \ \text{with } \ w_{ll'}:=\!\sum_{k\in\KM}\!\Bigl\{(M+1)^2\widehat\Phi_{kl}\widehat\Phi_{kl'}-\phi_k(U_{(l)})\phi_k(U_{(l')})\Bigr\}.
	\]
	The weights $w_{ll'}$ are functions of $(U_{(1)},\ldots,U_{(M)})$ only.
	By~\eqref{eq:orth_Yllprime}, $\E[\mathcal Y_{l,l'}\mid\mathcal F^{\rm obs}]
	=\langle\mu_l,\mu_{l'}\rangle$ for $(l,l')\in\In$.
	Given the $m_i$'s, the visit times $T_{ij}$ are i.i.d. with density $g$ (assumed uniform for simplicity), hence the 	vector of their ranks is uniformly distributed over the permutations of $\{1,\ldots,M\}$ and
	independent of the order statistics. Therefore, the random bijection $\varrho$ which maps $(i,j)$ in $l$ is independent
	of $(U_{(1)},\ldots,U_{(M)})$ and, 
	$$
	\forall l\ne l', \qquad 
	\PP\{(l,l')\in\In\mid m_1,\ldots,m_n,U_{(1)},\ldots,U_{(M)}\}=1-\pi_n.
	$$ 
	Hence (omitting the conditioning),
	\begin{equation}\label{eq:step0_exch}
		\E\bigl[2\mathcal Q_{1,n}+\mathcal Q_{2,n}\bigr]
		=(1-\pi_n)\sum_{1\le l\ne l'\le M}\E\bigl[w_{ll'}\langle\mu_l,\mu_{l'}\rangle\bigr].
	\end{equation}
	Set
	\[
	\mathcal B_k:=\sum_{l=1}^M\widehat\Phi_{kl}\,\mu_l,\qquad
	\mathcal G_k:=\sum_{l=1}^M\phi_k(U_{(l)})\,\mu_l,\qquad
	e_k:=\mathcal B_k-\beta_k ,
	\]
	all elements of $L^2(\mathcal U)$. Adding and removing the diagonal terms gives
	\begin{multline}
		\sum_{l\ne l'}w_{ll'}\langle\mu_l,\mu_{l'}\rangle
		=\sum_{k\in\KM}\Bigl\{(M+1)^2\|\mathcal B_k\|^2-\|\mathcal G_k\|^2
		+\sum_{l=1}^M\phi_k^2(U_{(l)})\|\mu_l\|^2\Bigr\}
		\\ -(M+1)^2\sum_{l=1}^M\sum_{k\in\KM}\widehat\Phi_{kl}^2\|\mu_l\|^2 .	
	\end{multline}
	A symmetric function of $(U_{(1)},\ldots,U_{(M)})$ is the same function of the unordered
	visit times. Hence, $\E[\|\mathcal G_k\|^2-\sum_l\phi_k^2(U_{(l)})\|\mu_l\|^2]=M(M-1)\|\beta_k\|^2$, by the representation~\eqref{eq:rep_bk} with $g\equiv1$. The same argument with $w_{ll'}$ replaced by
	$\Psi(U_{(l)},U_{(l')};\KM)$ proves \eqref{eq:mean_Qn_fixed_m}. Expanding
	$\|\mathcal B_k\|^2=\|\beta_k+e_k\|^2$ and using $(M+1)^2-M(M-1)=3M+1$, we obtain
	\begin{multline}\label{eq:step0_decomp}
		\frac{\E[2\mathcal Q_{1,n}+\mathcal Q_{2,n}]}{1-\pi_n}
		=(3M+1)\,\mathcal S(\KM;\mu)
		+2(M+1)^2\sum_{k\in\KM}\langle\beta_k,\E e_k\rangle
		\\	+(M+1)^2\sum_{k\in\KM}\E\|e_k\|^2-\mathfrak D_M,
	\end{multline}
	where $\mathfrak D_M:=(M+1)^2\sum_{l=1}^M\E\bigl[\sum_{k\in\KM}\widehat\Phi_{kl}^2\|\mu_l\|^2\bigr]$.
	Steps~1 to 4 below treat the four terms on the right-hand side.
	
	\medskip\noindent\textbf{Step 1 (the first term).} By Bessel's inequality and
	Assumption~\ref{ass:coeff-decay_H1n},
	$$\mathcal S(\KM;\mu)\le r_n^2\sum_{k\ge1}\|\beta_{\Upsilon,k}\|^2\le C_\Upsilon r_n^2,$$ so
	that $(3M+1)\,\mathcal S(\KM;\mu)\le4C_\Upsilon Mr_n^2$.
	
	\medskip\noindent\textbf{Step 2 (the diagonal term $\mathfrak D_M$).} 
	First, we consider boundary ranks. Let $l\notin\Ln$ (at most $2\hh$ ranks). By the definition of $\widehat\Phi_{kl}$ and Lemma~\ref{tb:mom} with $p=2$,
	\[
	(M+1)^2\sum_{l\notin\Ln}\E\Bigl[\sum_{k\in\KM}\widehat\Phi_{kl}^2\|\mu_l\|^2\Bigr]
	\lesssim \hh|\KM|\,\bar\mu^2 .
	\]
	For $l\in\Ln$, the window $[U_{(l-\hh)},U_{(l+\hh)}]\subset\mathcal T$, so that
	$\widehat\Phi_{kl}=(2\hh)^{-1}\int_{U_{(l-\hh)}}^{U_{(l+\hh)}}\phi_k$. Then,
	\[
	(M+1)\widehat\Phi_{kl}=(1+\xi_l)\bigl\{\phi_k(U_{(l)})+a_{kl}\bigr\},\quad
	a_{kl}:=\frac1{L_l}\int_{U_{(l-\hh)}}^{U_{(l+\hh)}}\bigl\{\phi_k(t)-\phi_k(U_{(l)})\bigr\}{\rm d}t,
	\]
	and thus  $|a_{kl}|\le C_\phi k\,L_l$. See also \eqref{eq:equiv_sp_g1} and \eqref{eq:def_frakdl}. Hence
	$$
	(M+1)^2\sum_{k\in\KM}\widehat\Phi_{kl}^2=(1+\xi_l)^2\{\Psi(U_{(l)},U_{(l)};\KM)+\mathfrak r_l\}
	$$
	with $\mathfrak r_l:=\sum_{k\in\KM}\{2\phi_k(U_{(l)})a_{kl}+a_{kl}^2\}$, which can be bounded as $|\mathfrak r_l|\le C\{|\KM|^2L_l+|\KM|^3L_l^2\}$. Since
	$L_l=2\hh(1+\xi_l)/(M+1)$, Lemma~\ref{tb:mom} with $p\in\{3,4\}$ gives
	$\E[(1+\xi_l)^2L_l]\le C\hh/M$ and $\E[(1+\xi_l)^2L_l^2]\le C\hh^2/M^2$, so that
	\[
	\sum_{l\in\Ln}\E\bigl[(1+\xi_l)^2|\mathfrak r_l|\,\|\mu_l\|^2\bigr]
	\le C\bar\mu^2\Bigl\{\hh|\KM|^2+\frac{\hh^2|\KM|^3}{M}\Bigr\}.
	\]
	It remains to handle
	$$
	\sum_{l\in\Ln}\E[(1+\xi_l)^2F(U_{(l)})]
	=\sum_{l\in\Ln}\bigl\{\E F(U_{(l)})+2\E[\xi_lF(U_{(l)})]+\E[\xi_l^2F(U_{(l)})]\bigr\}.
	$$
	Let $V_l=U_{(l-\hh)}/(1-L_l)$ be as in Lemma~\ref{tb:dec}. By the conditional independence
	of $V_l$ and $\xi_l$ and $\E[\xi_l]=0$, $\E[\xi_lF(V_l)]=0$. By the bound~\eqref{eq:dec_pointwise}, $|F(U_{(l)})-F(V_l)|\le\|F'\|_\infty L_l$. Hence, by the
	Cauchy-Schwarz inequality and Lemma~\ref{tb:mom},
	\begin{multline}
		\bigl|\E[\xi_lF(U_{(l)})]\bigr|=\bigl|\E[\xi_l\{F(U_{(l)})-F(V_l)\}]\bigr|
		\le\frac{2\hh}{M+1}\|F'\|_\infty\,\E\bigl[|\xi_l|(1+\xi_l)\bigr]
		\\ \le\frac{2\hh}{M+1}\|F'\|_\infty\bigl(\E\xi_l^2\bigr)^{1/2}\bigl(\E(1+\xi_l)^2\bigr)^{1/2}
		\le C\frac{\hh^{1/2}}{M}\|F'\|_\infty .
	\end{multline}
	Similarly, by the conditional independence of $V_l$ and $\xi_l$,
	\[
	\E[\xi_l^2F(U_{(l)})]=\operatorname{Var}(\xi_l)\,\E F(V_l)+\E\bigl[\xi_l^2\{F(U_{(l)})-F(V_l)\}\bigr].
	\]
	By~\eqref{eq:dec_pointwise}, we have
	$$
	|\E F(V_l)-\E F(U_{(l)})|\le\|F'\|_\infty\E L_l
	=\|F'\|_\infty 2\hh/(M+1),
	$$ 
	and
	$$
	|\E[\xi_l^2\{F(U_{(l)})-F(V_l)\}]|\le\|F'\|_\infty\E[\xi_l^2L_l],
	$$ 
	where
	$$
	\E[\xi_l^2L_l]=\frac{2\hh}{M+1}\E[\xi_l^2+\xi_l^3]
	\le\frac{2\hh}{M+1}\{(2\hh)^{-1}+(\E\xi_l^4)^{3/4}\}\le C/M.
	$$ 
	Together with
	$|\operatorname{Var}(\xi_l)-(2\hh)^{-1}|\le2/M$, see Lemma~\ref{tb:mom}, this gives
	\[
	\E[\xi_l^2F(U_{(l)})]=\frac1{2\hh}\E F(U_{(l)})+O\Bigl(\frac{\|F\|_\infty+\|F'\|_\infty}{M}\Bigr).
	\]
	Since $\sum_{l=1}^MF(U_{(l)})$ is a symmetric function of the visit times,
	$\sum_{l=1}^M\E F(U_{(l)})=M\int_{\mathcal T}F$, hence
	$\sum_{l\in\Ln}\E F(U_{(l)})=M\int_{\mathcal T}F+O(\hh\|F\|_\infty)$. Summing over $l\in\Ln$,
	\[
	\sum_{l\in\Ln}\E\bigl[(1+\xi_l)^2F(U_{(l)})\bigr]
	=\Bigl(1+\frac1{2\hh}\Bigr)M\int_{\mathcal T}F(t)\,{\rm d}t
	+O\bigl(\hh^{1/2}\|F'\|_\infty+\hh\|F\|_\infty\bigr).
	\]
	Collecting the three contributions and using~\eqref{eq:mu_bars}-\eqref{eq:F_bounds}, it follows
	\begin{multline}\label{eq:step2_diag}
		\mathfrak D_M=\Bigl(\!1+\frac1{2\hh}\Bigr)M\!\int_{\mathcal T}\!\Psi(t,t;\KM)\|\mu(\cdot,t)\|^2{\rm d}t
		+\mathcal R_{D,M},
		\\ \quad
		|\mathcal R_{D,M}|\le C\Bigl\{\hh|\KM|^2+\frac{\hh^2|\KM|^3}{M}\Bigr\}r_n^2 .
	\end{multline}
	By Lemma~\ref{cor:mean_Qsp_H1}, $\int_{\mathcal T}\Psi(t,t;\KM)\|\mu(\cdot,t)\|^2{\rm d}t=|\KM|\|\mu\|^2+O(r_n^2)$, which turns~\eqref{eq:step2_diag} into~\eqref{eq:replace_corr1}.

	\medskip  
	
	\noindent\textbf{Step 3 (linear term in $e_k$).} Exchanging the sum and the
	integral in~\eqref{eq:sp_est},
	\[
	\mathcal B_k=\int_{\mathcal T}\phi_k(t)\,\widetilde\mu(\cdot,t)\,{\rm d}t,\quad
	\widetilde\mu(\cdot,t):=\frac1{2\hh}\sum_{l=1}^M\chi_l(t)\,\mu_l,
	\quad\text{so that }\,
	e_k=\int_{\mathcal T}\phi_k(t)\,\gamma(\cdot,t)\,{\rm d}t,
	\]
	with $\gamma:=\widetilde\mu-\mu$.  By Lemma~\ref{tb:pou}, $\|\widetilde\mu(\cdot,t)\|\le\bar\mu$ for almost every $t$, so that
	$\|\gamma(\cdot,t)\|\le2\bar\mu$ on $\mathcal T\setminus\mathcal I_M$. On $\mathcal I_M$ the
	partition of unity is exact: for $t\in(U_{(r)},U_{(r+1)})$, $\hh\le r\le M-\hh$,
	$\widetilde\mu(\cdot,t)$ is the average of the $2\hh$ values $\mu_l$,
	$r-\hh+1\le l\le r+\hh$. Since $t$ and all these $U_{(l)}$ lie in
	$[U_{(r-\hh+1)},U_{(r+\hh)}]$, pointwise
	\begin{equation}\label{eq:gamma_interior}
		\|\gamma(\cdot,t)\|\le\bar\mu_1\,\Lambda_r, \,\, \Lambda_r:=U_{(r+\hh)}-U_{(r-\hh+1)},\,\, t\in(U_{(r)},U_{(r+1)}),\,\,
		\hh\le r\le M-\hh .
	\end{equation} 
	Now, since $\sum_{k\in\KM}\!\langle\beta_k,e_k\rangle
	\!=\!\int_{\mathcal T}\langle\mu_{\KM}(\cdot,t),\gamma(\cdot,t)\rangle{\rm d}t$, splitting the
	integral over the cells $(U_{(r)},U_{(r+1)})$, $\hh\le r\le M-\hh$, and over
	$\mathcal T\setminus\mathcal I_M$, then using \eqref{eq:mu_bars}, \eqref{eq:gamma_interior} and
	Lemma~\ref{tb:cells}, we get
	\begin{equation}\label{eq:step3_bias}
		2(M+1)^2\Bigl|\sum_{k\in\KM}\langle\beta_k,\E e_k\rangle\Bigr|
		\le CM^2r_n\Bigl\{\bar\mu_1\sum_{r=\hh}^{M-\hh}\E[\Delta_{r+1}\Lambda_r]
		+2\bar\mu\,\E|\mathcal T\setminus\mathcal I_M|\Bigr\}
		\le C\,\hh M\,r_n^2 .
	\end{equation}
	
	\medskip\noindent\textbf{Step 4 (quadratic term in $e_k$).} Split
	$\gamma=\gamma\mathbf 1_{\mathcal I_M}+\gamma\mathbf 1_{\mathcal T\setminus\mathcal I_M}$.
	For the first part, Bessel's inequality in $L^2(\mathcal T)$, applied for each $u$ and then
	integrated in $u$, gives
	$\sum_{k\in\KM}\|\int\phi_k\gamma\mathbf 1_{\mathcal I_M}\|^2
	\le\int_{\mathcal I_M}\|\gamma(\cdot,t)\|^2{\rm d}t$. For the second part, we use
	$\|\phi_k\|_\infty\le C_\phi$. Then, by~\eqref{eq:gamma_interior} and
	Lemma~\ref{tb:cells},
	\begin{multline}\label{eq:step4_var}
		(M+1)^2\sum_{k\in\KM}\E\|e_k\|^2
		\le 2(M+1)^2\Bigl\{\bar\mu_1^2\sum_{r=\hh}^{M-\hh}\E[\Delta_{r+1}\Lambda_r^2]
		+4C_\phi^2|\KM|\,\bar\mu^2\,\E|\mathcal T\setminus\mathcal I_M|^2\Bigr\}
		\\ 	\le C\,\hh^2|\KM|\,r_n^2 .
	\end{multline}
	
	As a conclusion, plugging Step~1, \eqref{eq:step2_diag},
	\eqref{eq:step3_bias} and \eqref{eq:step4_var} into~\eqref{eq:step0_decomp}
	yields~\eqref{eq:mean_Qsp_H1} with
	\[
	|\mathcal R_M|\le C\Bigl\{M+\hh M+\hh^2|\KM|+\hh|\KM|^2+\frac{\hh^2|\KM|^3}{M}\Bigr\}r_n^2 .
	\]
	When $\hh|\KM|\le M$, $\hh^2|\KM|\le\hh M$ and $\hh^2|\KM|^3/M\le\hh|\KM|^2$, which gives the
	stated bound. Finally, note that since $\hh|\KM|\le M$ necessarily holds for $n$ large, it follows that the remainder $\mathcal R_M$ satisfies $|\mathcal R_M|/(M|\KM|r_n^2)\le C(\hh/|\KM|+\hh|\KM|/M)\to0$. The representation in~\eqref{eq:replace_corr1} is now proved in detail. Finally, the bound~\eqref{eq:replace_corr2} follows by integrating both sides in~\eqref{eq:replace_corr1} with respect to $M$. Now the proof of the lemma is complete.  
\end{proof}

\medskip

\begin{lemma}[Diagonal-kernel evaluation]\label{cor:mean_Qsp_H1}
	Let $\{\phi_k\}$ be the cosine basis on $[0,1]$, or its Gram--Schmidt augmentation by $t,t^2$, and let
	$\KM$ be a set of consecutive indices. For every $w\in C^1([0,1])$ with $w\ge0$,
	\begin{equation}\label{eq:diag_eval}
		\Bigl|\int_0^1\Psi(t,t;\KM)\,w(t)\,{\rm d}t-|\KM|\int_0^1 w(t)\,{\rm d}t\Bigr|
		\le C\bigl(\|w'\|_{L^2([0,1])}+\|w\|_\infty\bigr),
	\end{equation}
	for an absolute constant $C$. In particular, with $w(t)=\|\mu(\cdot,t)\|^2_{L^2(\mathcal U)}$, so that
	$\int_0^1 w=\|\mu\|^2_{L^2(\mathcal U\times\mathcal T)}$ and
	$\|w\|_\infty+\|w'\|_{L^2}\le C r_n^2$ by \eqref{eq:mu_bars}-\eqref{eq:F_bounds},
	\begin{equation}\label{eq:diag_eval_mu}
		\int_0^1\Psi(t,t;\KM)\,\|\mu(\cdot,t)\|^2\,{\rm d}t
		=|\KM|\,\|\mu\|^2_{L^2(\mathcal U\times\mathcal T)}+O(r_n^2).
	\end{equation}
\end{lemma}

\begin{proof}
	Let $\Psi_K(t):=\sum_{k=1}^{K}\phi_k^2(t)$ be the diagonal of the kernel of the $L^2([0,1])$-orthogonal
	projection onto $\mathcal V_K=\operatorname{span}\{\phi_1,\dots,\phi_K\}$; it does not depend on the
	chosen orthonormal basis of $\mathcal V_K$. For consecutive $\KM=\{k_0,\dots,k_1\}$,
	$\Psi(t,t;\KM)=\Psi_{k_1}(t)-\Psi_{k_0-1}(t)$, so it suffices to bound
	$\int_0^1\Psi_K w-K\int_0^1 w$ uniformly in $K$.
	
	\emph{Cosine basis.} From $2\cos^2(j\pi t)=1+\cos(2j\pi t)$,
	$\Psi^{\cos}_K(t)=K+\sum_{j=1}^{K-1}\cos(2j\pi t)$, whence
	$\int_0^1\Psi^{\cos}_K w-K\int_0^1 w=\sum_{j=1}^{K-1}c_j$ with
	$c_j=\int_0^1\cos(2j\pi t)\,w(t)\,{\rm d}t$. Integrating by parts (the boundary terms vanish) and
	expanding on the orthonormal system $\{\sqrt2\sin(2j\pi\,\cdot)\}_{j\ge1}$,
	$c_j=-(2\sqrt2\,j\pi)^{-1}s_j$ with $s_j=\int_0^1\sqrt2\sin(2j\pi t)\,w'(t)\,{\rm d}t$ and
	$\sum_{j\ge1}s_j^2\le\|w'\|_{L^2}^2$; by the Cauchy--Schwarz inequality,
	\[
	\sum_{j\ge1}|c_j|\le\frac1{2\sqrt2\,\pi}\Bigl(\sum_{j\ge1}j^{-2}\Bigr)^{1/2}\|w'\|_{L^2}
	=\frac{\|w'\|_{L^2}}{4\sqrt3},
	\]
	uniformly in $K$.
	
	\emph{Augmented basis.} Here $\mathcal V_K$ contains $\mathcal W_K:=\operatorname{span}$ of the first
	$K-2$ cosines, with $\dim(\mathcal V_K\ominus\mathcal W_K)=2$; choosing an orthonormal basis
	$\varsigma_1,\varsigma_2$ of that complement, $\Psi_K=\Psi^{\cos}_{K-2}+\varsigma_1^2+\varsigma_2^2$.
	Since $\int\varsigma_i^2=1$ and $w\ge0$, both $\int(\varsigma_1^2+\varsigma_2^2)w$ and $2\int w$ lie in
	$[0,2\|w\|_\infty]$, so $|\int(\varsigma_1^2+\varsigma_2^2)w-2\int w|\le2\|w\|_\infty$; with the cosine
	bound this gives $|\int\Psi_K w-K\int w|\le\|w'\|_{L^2}/(4\sqrt3)+2\|w\|_\infty$, uniformly in $K$.
	
	Taking the difference at $K=k_1$ and $K=k_0-1$ yields \eqref{eq:diag_eval}. Specialising to
	$w=\|\mu(\cdot,\cdot)\|^2$ and using $\|w\|_\infty+\|w'\|_{L^2}\le Cr_n^2$ gives \eqref{eq:diag_eval_mu}.
\end{proof}

\medskip


\begin{proof}[Proof of Lemma \ref{prop:fluct_H1}]
	Again, by Minkowski's inequality, we have the inequality~\eqref{eq:mink_null}. Thus it suffices to prove 
	\begin{equation}\label{eq:piece-goal}
		\operatorname{Var}(\widetilde{\mathcal Q}_{n})=o(n^4 r_n^4),\qquad
		\widetilde{\mathcal Q}_{n}\in\{\mathcal Q_{1,n}, \mathcal Q_{21,n}, \mathcal Q_{22,n}, \mathcal Q_{23,n}\},
	\end{equation}
	because $\mathcal S(\KM;\mu)= r_n^2\mathcal S(\KM;\Upsilon)$. Now, $\widetilde{\mathcal Q}_{n}$ has the form 
	$$
	\widetilde{\mathcal Q}_n=\sum_{(l,l')\in\In}w_{ll'}\mathcal Y_{l,l'}
	$$
	with $w_{ll'}$ symmetric and
	$\mathcal F^{\rm obs}$-measurable. More precisely, with the uniform design and the definition~\eqref{eq:def_frakdl}, we have $\widehat D_{kl} =\widehat D_{kl,1} + \widehat D_{kl,2}= \widehat \Phi_{kl} - \phi_k(T_{(l)})/(M+1)$ and 
	$$
	w_{ll'} = 
	\begin{cases}
		w_{ll'}^{(1)} = \sum_{k\in\mathcal K_M}\{\phi_k (T_{(l)})/g(T_{(l)})\} \widehat D_{kl'}   & \text{if } \widetilde{\mathcal Q}_{n}=\mathcal Q_{1,n},\\
		w_{ll'}^{(21)}=  \sum_{k\in\mathcal K_M}\{(M+1)\widehat D_{kl,1}\}\{(M+1)\widehat D_{kl',1}\}  &  \text{if } \widetilde{\mathcal Q}_{n}=\mathcal Q_{21,n},\\
		w_{ll'}^{(22)}=  \sum_{k\in\mathcal K_M}\{(M+1)\widehat D_{kl,1}\}\{(M+1)\widehat D_{kl',2}\} &  \text{if } \widetilde{\mathcal Q}_{n}=\mathcal Q_{22,n},\\
		w_{ll'}^{(23)}= \sum_{k\in\mathcal K_M} \{(M+1)\widehat D_{kl,2}\}\{(M+1)\widehat D_{kl',2}\} & \text{if } \widetilde{\mathcal Q}_{n}=\mathcal Q_{23,n}.
	\end{cases}  
	$$
	If $\varepsilon_{(l)}$ is the error of the profile at the visiting time $T_{(l)}$, we have 	$\mathcal Y_{l,l'}=\langle \mu_l,\mu_{l'}\rangle +\int(\mu_l\varepsilon_{(l')}+\mu_{l'}\varepsilon_{(l)})
	+\int\varepsilon_{(l)}\varepsilon_{(l')}$. (Here, the integrals are with respect to $u$.) Hence
	\begin{equation}\label{eq:split}
		\widetilde{\mathcal Q}_n-\E[\widetilde{\mathcal Q}_n\mid\mathcal F^{\rm obs}]=\underbrace{2\sum_{(l,l')\in\In}w_{ll'}\!\int\mu_l\varepsilon_{(l')}}_{=:L}
		+\underbrace{\sum_{(l,l')\in\In}w_{ll'}\!\int\varepsilon_{(l)}\varepsilon_{(l')}}_{=:\widetilde{\mathcal Q}_n^{(0)}} .
	\end{equation}
	The term $\widetilde{\mathcal Q}_n^{(0)}$ is exactly $\widetilde{\mathcal Q}_n$ in the model with $\mu \equiv 0$ (under $H_0$).
	Note also that, with $G_r:=\sum_{l:(l,r)\in\In}w_{lr}\,\mu_l\in L^2(\mathcal U)$, it holds
	\begin{equation}\label{eq:A}
		\E\bigl[\operatorname{Var}(\widetilde{\mathcal Q}_n\mid\mathcal F^{\rm obs})\bigr]
		\ \le\ 2\,\operatorname{Var}(\widetilde{\mathcal Q}_n^{(0)})
		\ +\ 8\,\bar\rho\, m_{\max}\,\E\!\sum_{r=1}^M\|G_r\|^2 .
	\end{equation}
	Consequently, by the law of total variance,
	\begin{equation}\label{eq:totvar}
		\operatorname{Var}(\widetilde{\mathcal Q}_n)\le 2\,\operatorname{Var}(\widetilde{\mathcal Q}_n^{(0)})+8\bar\rho m_{\max}\,\E\!\sum_{r}\|G_r\|^2
		+\operatorname{Var}\bigl(\E[\widetilde{\mathcal Q}_n\mid\mathcal F^{\rm obs}]\bigr).
	\end{equation}
	Thus, for each of the values $\widetilde{\mathcal Q}_{n}$, we must bound three quantities: the \emph{noise} $\operatorname{Var}(\widetilde{\mathcal Q}_n^{(0)})$ (using the lemmas under 
	$H_0$ above), the \emph{linear term} $\E\sum_r\|G_r\|^2$, and the \emph{signal fluctuation}
	$\operatorname{Var}(\E[\widetilde{\mathcal Q}_n\mid\mathcal F^{\rm obs}])$. Let 
	$\widetilde{\mathcal Q}_{1,n}^{(0)}$, $\widetilde{\mathcal Q}_{21,n}^{(0)}$, $\widetilde{\mathcal Q}_{22,n}^{(0)}$ and $\widetilde{\mathcal Q}_{23,n}^{(0)}$ denote $\mathcal Q_{1,n}$, $\mathcal Q_{21,n}$, $\mathcal Q_{22,n}$, $\mathcal Q_{23,n}$ in the case $\mu\equiv 0$.
	
	\medskip
	
	\underline{\emph{The case $\widetilde{\mathcal Q}_{n}=\mathcal Q_{23,n}$}}. Here $w_{ll'}^{(23)}=a_{ll'}=b_lb_{l'}\Psi(U_{(l)},U_{(l')})$.

	\emph{Noise.} By Lemma~\ref{lem:var_Q23}, it follows $\operatorname{Var}(\widetilde{\mathcal Q}^{(0)}_{23,n})\le C\,n^2|\KM|/\hh^2$.

	\emph{Linear.}  We show that 
	\begin{equation}\label{eq:app_B23}
		\displaystyle \E\sum_{r=1}^M\Bigl\|\sum_{l}a_{lr}\mu_l\Bigr\|^2\ \le\ C\,r_n^2\,\frac{M^2|\KM|}{\hh^2}.
	\end{equation}
	Since $a_{lr}=b_lb_r\Psi(U_{(l)},U_{(r)})$, we have $G_r:=\sum_l a_{lr}\mu_l=b_r Z_r$ and  $\|G_r\|^2=b_r^2\|Z_r\|^2$, with
	$Z_r:=\sum_{l:(l,r)\in\In}b_l\Psi(U_{(l)},U_{(r)})\mu_l$ and
	\[
	\|Z_r\|^2=\sum_{l,\tilde l}b_lb_{\tilde l}\,\Psi(U_{(l)},U_{(r)})\Psi(U_{(\tilde l)},U_{(r)})\,
	\langle \mu_l,\mu_{\tilde l}\rangle \ \ge\ 0 .
	\]
	Let $D_r:=\{l:\ |l-r|\ge 2\hh\}$ be the ranks whose window is disjoint from that of $r$. Its
	complement has at most $4\hh$ elements. We bound the contributions of the diagonal ($l=\tilde l$), the
	disjoint off-diagonal ($l\ne\tilde l$, both in $D_r$), and the overlapping ranks separately.
	
	\emph{Diagonal, disjoint part.} For $l\in D_r$ the windows of $l$ and $r$ are disjoint, so by
	\eqref{eq:app_m:center} (with $s=2$ and the bounded kernel $\Psi(\cdot,U_{(r)})^2$) and $\|\mu_l\|\le\bar\mu$,
	\[
	\E\Bigl[b_r^2\sum_{l\in D_r}  b_l^2\,\Psi^2(U_{(l)},U_{(r)})\|\mu_l\|^2\Bigr]
	\le \bar\mu^2(2\hh)^{-2}\sum_{l}\E\bigl[\Psi^2(U_{(l)},U_{(r)})\bigr]\{1+o(1)\}.
	\]
	Now, $\sum_l\Psi(U_{(l)},s)^2=\sum_{i,j}\Psi(T_{ij},s)^2$ has conditional mean
	$$M \int_{\mathcal T}\Psi^2(t,s){\rm d}t=M\,\Psi(s,s)\le C_\phi^2 M|\KM|,$$ so that this term is at most
	$C\,r_n^2 (2\hh)^{-2}M|\KM|$ for each  $r$, and thus is bounded by $ C\,r_n^2 M^2|\KM|\hh^{-2}$ after summing over $r$.
	
	\emph{Off-diagonal, disjoint part.} For $l\ne\tilde l$ in $D_r$ the three windows are pairwise disjoint. By \eqref{eq:app_m:mean}-\eqref{eq:app_m:cov} and the near-independence of disjoint-window spacings from the global design (as in Lemma~\ref{lem:energy_sp}), $\E\big[b_r^2 b_l b_{\tilde l}\,[\cdot]\,\big]
	=\E[b_r^2]\,\Cov(b_l,b_{\tilde l})\,\E[\,\cdot\,]\{1+o(1)\}$ with $|\Cov(b_l,b_{\tilde l})|\le C/M$.
	Since
	\[
	\Bigl|\sum_{l\ne\tilde l}\E\bigl[\Psi(U_{(l)},U_{(r)})\Psi(U_{(\tilde l)},U_{(r)})\langle \mu_l,\mu_{\tilde l}\rangle \bigr]\Bigr|
	\le \E\Bigl\|\sum_{l}\Psi(U_{(l)},U_{(r)})\mu_l\Bigr\|^2\le C M^2\bar\mu^2
	\]
	(the last bound from $\E[\Psi(U_{(l)},U_{(r)})\mu_l\mid U_{(r)}]=\mu_{\KM}(\cdot,U_{(r)})$ and
	$\|\mu_{\KM}(\cdot,s)\|\le C r_n$), this part is $\le C(2\hh)^{-1}M^{-1}\cdot M^2\bar\mu^2=C r_n^2 M/\hh$
	per $r$, hence $\le C r_n^2 M^2/\hh$ in total --- of smaller order than the diagonal, by $\hh/|\KM|\to0$.
	
	\emph{Overlapping ranks.} There are at most $4\hh$ ranks $l\notin D_r$. For these we use the crude
	bounds $\E[b_r^2 b_l^2]\le(\E b_r^4\,\E b_l^4)^{1/2}\le C\hh^{-2}$, see \eqref{eq:app_m:mean}, and
	$$
	\Psi(U_{(l)},U_{(r)})^2\le\Psi(U_{(l)},U_{(l)})\Psi(U_{(r)},U_{(r)})\le C_\phi^4|\KM|^2,
	$$
	so their total (diagonal and off-diagonal) contribution is
	$O\bigl(M\cdot\hh\cdot\hh^{-2}\cdot|\KM|^2\bar\mu^2\bigr)=O\bigl(M|\KM|^2 r_n^2/\hh\bigr)$,
	again of smaller order than the diagonal since $\hh|\KM|=o(M)$. Collecting the three parts gives~\eqref{eq:app_B23}.

	\medskip

	\emph{Signal fluctuation.}  
	With $$
	S_{23}:=\E[\mathcal Q_{23,n}\mid\mathcal F^{\rm obs}]=\sum_{(l,l')\in\In}b_lb_{l'}\,h(U_{(l)},U_{(l')}), \qquad 
	h(t,s):=\Psi(t,s)\langle \mu(\cdot,t),\mu(\cdot,s)\rangle ,
	$$ we show
	\begin{equation}\label{eq:app_C23} 
		\operatorname{Var}(S_{23})\ \le\ C\,r_n^4\,\frac{M^2|\KM|}{\hh^2}\ +\ C\,M^2 r_n^4 .
	\end{equation}
	The kernel $h$ is symmetric and $\mathcal F^{\rm obs}$-measurable, with $|h|\le C_\phi^2|\KM|\bar\mu^2=O(|\KM|r_n^2)$
	and, crucially,
	\begin{equation}\label{eq:hHS}
		\iint_{\mathcal T^2}h(t,s)^2\,{\rm d}t\,{\rm d}s
		\le \bar\mu^4\!\iint_{\mathcal T^2}\Psi^2=\bar\mu^4|\KM|=r_n^4\bar\Upsilon^4|\KM| ,
	\end{equation}
	by the orthonormality of $\{\phi_k\}$. Expand
	\[
	\E[S_{23}^2]=\sum_{(l,l'),(p,p')\in\In}
	\E\bigl[b_lb_{l'}b_pb_{p'}\,h(U_{(l)},U_{(l')})\,h(U_{(p)},U_{(p')})\bigr],
	\]
	and classify the four-index terms by the partition of $\{l,l',p,p'\}$ induced by window
	coincidence. Each $b_l$ is mean-zero, see~\eqref{eq:app_m:mean}, with $\Cov(b_l,b_{l'})=O(M^{-1})$ on disjoint windows, see \eqref{eq:app_m:cov} and bounded fourth cumulant. Consequently, for any four ranks with pairwise disjoint 	windows,
	\begin{multline}\label{eq:4mom}
		\E[b_lb_{l'}b_pb_{p'}]
		=\Cov(b_l,b_{l'})\Cov(b_p,b_{p'})+\Cov(b_l,b_p)\Cov(b_{l'},b_{p'})
		\\ +\Cov(b_l,b_{p'})\Cov(b_{l'},b_p)+O(M^{-3}),
	\end{multline}
	the remainder being the fourth cumulant. We use this to organize the four-index sum.
	
	\emph{(A) Coincident pairs} ($\{l,l'\}=\{p,p'\}$ as index pairs). These contribute by a quantity bounded by $C\sum_{l\ne l'}\E[b_l^2b_{l'}^2\,h(U_{(l)},U_{(l')})^2]$. For the disjoint pairs $|l-l'|\ge2\hh$,
	\eqref{eq:app_m:center} (with $s=2$) and the independence of $U_{(l)},U_{(l')}$ give
	$\E[b_l^2b_{l'}^2h^2]=(2\hh)^{-2}\E[h(U_{(l)},U_{(l')})^2]\{1+o(1)\}
	=(2\hh)^{-2}\{1+o(1)\}\iint h^2$, so by \eqref{eq:hHS},
	\[
	\sum_{|l-l'|\ge2\hh}\!\!\E[b_l^2b_{l'}^2h^2]
	\le C\,(2\hh)^{-2}M^2\!\iint h^2\le C\,r_n^4\,\frac{M^2|\KM|}{\hh^2}.
	\]
	The $O(\hh M)$ overlapping pairs $|l-l'|<2\hh$ are bounded crudely, as in~\eqref{eq:app_B23}, by
	$\E[b_l^4]\le C\hh^{-2}$ and $|h|\le C|\KM|r_n^2$, contributing $O(M|\KM|^2r_n^4/\hh)$, of smaller order
	since $\hh|\KM|=o(M)$.
	
	\emph{(B) Distinct pairs with mutually disjoint windows.} Here $\E[b_lb_{l'}b_pb_{p'}\,hh]$ splits, by~\eqref{eq:4mom}, into the three products of covariances
	(plus the $O(M^{-3})$ remainder). Then, the pairing $\Cov(b_l,\!b_{l'})\!\Cov(b_p,\!b_{p'})$ reproduces $(\E S_{23})^2$ up to
	$o(\cdot)$ and cancels in $\operatorname{Var}(S_{23})=\E[S_{23}^2]-(\E S_{23})^2$. For each of the two \emph{cross}
	pairings, $\Cov(b_l,b_p)\Cov(b_{l'},b_{p'})$; since $\Cov(b_a,b_b)=-(M+2)^{-1}\{1+o(1)\}$ is
	asymptotically constant over disjoint pairs~\eqref{eq:app_m:cov}, the corresponding sum equals
	\begin{multline}
		(M+2)^{-2}\{1+o(1)\}\sum_{l,l',p,p'}\E\bigl[h(U_{(l)},U_{(l')})h(U_{(p)},U_{(p')})\bigr]
		\\	=(M+2)^{-2}\{1+o(1)\}\,\E\Bigl[\Bigl(\sum_{(l,l')}h(U_{(l)},U_{(l')})\Bigr)^{\!2}\Bigr].
	\end{multline}
	Now, $\sum_{l,l'} \! h(U_{(l)},U_{(l')})=\sum_{k\in\KM}\!\bigl\|\sum_l\phi_k(U_{(l)})\mu_l\bigr\|^2
	\! =O_{\PP}(M^2 r_n^2)$, because $\E[\phi_k(U_{(l)})\mu_l]=\beta_k$ and
	$\sum_k M^2\|\beta_k\|^2=M^2\mathcal S(\KM;\mu)=M^2 r_n^2\mathcal S(\KM;\Upsilon)$. Hence,
	$\E[(\sum h)^2]=O(M^4 r_n^4)$ and each cross pairing is $O(M^2 r_n^4)$.
	
	\emph{(C) All remaining partitions} (a lone window, or a triple) carry at least one unpaired factor
	$\E[b_\cdot]=0$ up to an $O(1/M)$ defect~\eqref{eq:app_m:mean}, hence are $O(M^{-1})$ relative to (A) and negligible.
	
	Collecting (A) to (C) and subtracting $(\E S_{23})^2$ yields the stated bound~\eqref{eq:app_C23}.

	Combining bounds, by \eqref{eq:totvar}  it holds
	\[
	\operatorname{Var}(\mathcal Q_{23,n})\le C\Bigl\{\frac{n^2|\KM|}{\hh^2}+r_n^2\frac{n^2|\KM|}{\hh^2}+r_n^4\frac{n^2|\KM|}{\hh^2}+n^2r_n^4\Bigr\}
	\le C\Bigl\{\frac{n^2|\KM|}{\hh^2}+n^2r_n^4\Bigr\},
	\]
	using $r_n$ bounded. It is easy to check that the rate of the bound is $o(n^4r_n^4)$, provided \eqref{eq:detect} holds. 
	
	\medskip

	\underline{\emph{The case $\widetilde{\mathcal Q}_{n}=\mathcal Q_{1,n}$}} Here, $w_{ll'}=w^{(1)}_{ll'}+w^{(2)}_{ll'}$ with $w^{(2)}_{ll'}=b_{l'}\Psi(U_{(l)},U_{(l')})$ (recall $g\equiv1$). The $w^{(1)}$ contribution is handled exactly as in the proof of Lemma~\ref{lem:var_Q1}, and it is negligible against that of $w^{(2)}$. We will therefore focus on that of $w^{(2)}$.
	
	\emph{Noise.} $\operatorname{Var}(\widetilde{\mathcal Q}_{1,n}^{(0)})\le C\,n^2|\KM|/\hh$ by Lemma~\ref{lem:var_Q1}.
	
	\emph{Linear.} Now $G_r=\sum_{l}w^{(2)}_{lr}\mu_l=b_r\,\widetilde\psi_r$,
	$\widetilde\psi_r:=\sum_{l:(l,r)\in\In}\Psi(U_{(l)},U_{(r)})\mu_l$. Let 
	$$
	\mu_{\KM}(\cdot,s)
	:=\sum_{k\in\KM}\beta_k(\cdot)\,\phi_k(s)
	=\int_{\mathcal T}\Psi(s,t;\KM)\,\mu(\cdot,t)\,{\rm d}t ,
	\qquad
	\beta_k(u)=\int_{\mathcal T}\mu(u,t)\,\phi_k(t)\,{\rm d}t ,
	$$
	Since
	$$
	\E\|\widetilde\psi_r\|^2=n^2\E[\mathfrak m]^2\E[\|\mu_{\KM}(\cdot,U_{(r)})\|^2]\{1+o(1)\}+O\bigl(n\,\E[\Psi^2]\bar\mu^2\bigr) \le C n^2\bar\mu^2
	$$ 
	(using $n\geq M/m_{\rm max}$, $n\gg|\KM|$ and $\|\mu_{\KM}(\cdot,s)\|\le C r_n$), and $\E[b_r^2]\le C/\hh$
	with $b_r$ independent of $\widetilde\psi_r$ up to the $O(\hh)$ overlapping ranks,
	\[
	\E\sum_{r}\|G_r\|^2=\E\sum_r b_r^2\|\widetilde\psi_r\|^2\le C\hh^{-1}\,n\cdot n^2\bar\mu^2
	=C\,\frac{n^3 r_n^2}{\hh}.
	\]
	Hence the linear contribution to $\operatorname{Var}(\mathcal Q_{1,n})$ is $\le C n^3 r_n^2/\hh=o(n^4r_n^4)$, because $n^3r_n^2/\hh=o(n^4r_n^4)$, implied by \eqref{eq:detect}.
	
	\emph{Signal fluctuation.} $S_1:=\E[\mathcal Q_{1,n}\mid\mathcal F^{\rm obs}]=\sum_{l'}b_{l'}\psi_{l'}$ with
	$\psi_{l'}:=\sum_{l}\Psi(U_{(l)},U_{(l')})\langle  \mu_l,\mu_{l'}\rangle$ an $\mathcal F^{\rm obs}$-kernel satisfying
	$\E[\psi_{l'}\mid U_{(l')},M]=M \langle \mu_{\KM}(\cdot,U_{(l')}),\mu_{l'}\rangle$. Hence, the expectation of the square is the variance plus the square of the expectation (identity valid also with conditional moments), we get
	$$
	\E[\psi_{l'}^2\mid M]\le C M^2\bar\mu^4=CM^2r_n^4\bar\Upsilon^4;
	$$ 
	(the conditional mean is $O(M\bar\mu^2)$ and
	$\operatorname{Var}(\psi_{l'}\mid U_{(l')},M)\le C M|\KM|\bar\mu^4$ is of smaller order since $M\gg|\KM|$). Integrating out $M$, we get $\E[\psi_{l'}^2]\le Cn^2r_n^4\E[\mathfrak m]^2\bar\Upsilon^4$. Next, exactly as for~\eqref{eq:app_C23} (now with a \emph{single} $b$ per term, so $\E[b_{l'}^2\mid M]\le C/\hh$ on the
	diagonal and $|\Cov(b_{l'},b_{p'}\mid M)|\le C/M$ off the diagonal), 
	\[
	\operatorname{Var}(S_1)\le \underbrace{C\hh^{-1}\!\sum_{l'}\E[\psi_{l'}^2]}_{\le CM^3r_n^4/\hh}
	+\underbrace{C M^{-1}\!\!\sum_{l'\ne p'}\bigl|\E[\psi_{l'}\psi_{p'}]\bigr|}_{\le CM^3r_n^4}
	\ \le\ C\,M^3 r_n^4 .
	\]
	Thus $\operatorname{Var}(\mathcal Q_{1,n})\le C\{n^2|\KM|/\hh+M^3r_n^2/\hh+M^3r_n^4\}=o(n^4r_n^4)$
	(the last term because $M^3=o(n^4)$).

	\smallskip
	
	\underline{\emph{The cases $\widetilde{\mathcal Q}_n=\mathcal Q_{21,n}$ and $\widetilde{\mathcal Q}_n=\mathcal Q_{22,n}$.}} 
	Write $v_{kl}:=(M+1)\widehat D_{kl,1}$, so that $w^{(21)}_{ll'}\! =\sum_{k\in\KM}\!\!v_{kl}v_{kl'}$, and recall
	$$
	S_l=\min\{U_{(l+\hh)},1\}-\max\{U_{(l-\hh)},0\}, \qquad |v_{kl}|\le(M+1)\|\phi_k'\|_\infty S_l^2/(2\hh).
	$$
	Moreover,
	\[
	G_r=\sum_l w^{(21)}_{lr}\mu_l,\qquad S_{21}=\E[\mathcal Q_{21,n}\mid\mathcal F^{\rm obs}]=\sum_{(l,l')\in\In}w^{(21)}_{ll'}\langle\mu_l,\mu_{l'}\rangle.
	\]
	Let $\mathcal E_{21}:=\E\sum_{(l,l')\in\In}(w^{(21)}_{ll'})^2$. The proof of Lemma~\ref{lem:var_Q21} shows $\mathcal E_{21}\le C|\KM|^6\hh^4/n^2$.
	
	\emph{Noise.} $\operatorname{Var}(\widetilde{\mathcal Q}^{(0)}_{21,n})\le 2\bar\rho^2 m_{\max}^2\,\mathcal E_{21}\le C|\KM|^6\hh^4/n^2$
	(Lemma~\ref{lem:var_Q21}).
	
	\emph{Linear.} By the Cauchy--Schwarz inequality in $l$, for each $r$,
	$$
	\|G_r\|^2\le\bar\mu^2\bigl(\sum_l|w^{(21)}_{lr}|\bigr)^2\le\bar\mu^2\,M\sum_l (w^{(21)}_{lr})^2,
	$$ 
	whence
	\[
	\E\sum_{r}\|G_r\|^2\le\bar\mu^2 M\,\E\!\sum_{(l,l')\in\In}(w^{(21)}_{ll'})^2
	=\bar\mu^2 M\,\mathcal E_{21}\le C\,r_n^2\,\frac{|\KM|^6\hh^4}{n}.
	\]
	
	\emph{Signal fluctuation.} Using $|\langle\mu_l,\mu_{l'}\rangle|\le\bar\mu^2$ and
	$$
	\sum_{(l,l')}|w^{(21)}_{ll'}|\le\sum_{k}\bigl(\sum_l|v_{kl}|\bigr)^2
	\le\bigl(\frac{M+1}{2\hh}\bigr)^2\bigl(\sum_l S_l^2\bigr)^2\sum_{k\in\KM}\|\phi_k'\|_\infty^2,
	$$
	with $\sum_{k\in\KM}\|\phi_k'\|_\infty^2\lesssim|\KM|^3$, we get
	$|S_{21}|\le C\,\bar\mu^2\,(M/\hh)^2\bigl(\sum_l S_l^2\bigr)^2|\KM|^3$. Hence, by
	$$
	\E\bigl[(\sum_l S_l^2)^4\mid M\bigr]\le C\,\hh^8/M^4
	$$ 
	(obtained from Hölder inequality and $\E[S_l^8\mid M]\lesssim(\hh/M)^8$), we get
	\[
	\operatorname{Var}(S_{21})\le\E[S_{21}^2]
	\le C\,\bar\mu^4\,\Bigl(\frac{M}{\hh}\Bigr)^{4}\,\E\Bigl[\bigl(\textstyle\sum_l S_l^2\bigr)^4\Bigr]\,|\KM|^6
	\le C\,r_n^4\,\hh^4|\KM|^6 .
	\]
	Collecting the bounds, we get by~\eqref{eq:totvar}
	\[
	\operatorname{Var}(\mathcal Q_{21,n})\le C\Bigl\{\frac{|\KM|^6\hh^4}{n^2}
	+r_n^2\,\frac{|\KM|^6\hh^4}{n}+r_n^4\,\hh^4|\KM|^6\Bigr\}.
	\]
	Under the condition $\hh|\KM|^{3/2}=o(n)$, the third term is
	$o(n^4r_n^4)$; the first two are $o(n^4r_n^4)$ as well, using in addition the detection condition
	\eqref{eq:detect}. Hence $\operatorname{Var}(\mathcal Q_{21,n})=o(n^4r_n^4)$.
	
	\smallskip

	For the \emph{case $\widetilde{\mathcal Q}_n=\mathcal Q_{22,n}$}, apply~\eqref{eq:totvar}, and bound each of the noise, linear and signal fluctuation contributions by the Cauchy-Schwarz inequality between the $\mathcal Q_{21,n}$- and $\mathcal Q_{23,n}$-type factors, exactly as in the derivation of \eqref{eq:bound_Q22H0}. Each is thereby dominated by the geometric mean of the corresponding contributions of $\mathcal Q_{21,n}$ and $\mathcal Q_{23,n}$, both $o(n^4r_n^4)$; hence $\operatorname{Var}(\mathcal Q_{22,n})=o(n^4r_n^4)$.
\end{proof}


\subsection{Additional simulation results}\label{append_simus_est}

In this section, we present additional simulation results that provide numerical evidence on the performance of the spacings-based mean function estimator compared to estimators constructed using OBS and MC weights, as well as those obtained by local linear smoothing and splines. 

Table~\ref{tab:cvbox} completes the numerical results from Section~\ref{sec:simulation}: it collects, for the three means and the three sample sizes, the
tunings that the criterion~\eqref{eq:cv} selected for the three data-driven estimators of Figure~\ref{fig:cvbox_mu1}, next to the
truncation $K$ of the rate-tuned competitor.  Note that the CV criterion with the spacings estimator does not select more basis elements than are necessary to accurately estimate $\mu_3$. Meanwhile, most of the time, it selects $r=2$ and a value of $K$ that increases with the sample size in the case of  $\mu_1$ and $\mu_2$ which indeed require $\mathcal B_{2,K}$ with a large $K$ for an accurate approximation.  

\begin{table}[ht!]
	\small 	\centering
	\begin{tabular}{llrccccc}
		\toprule
		& & & \multicolumn{2}{c}{$\widehat\mu^{\rm sp}$} & $\widehat\mu^{\rm LL}$ & $\widehat\mu^{\rm spl}$ & $\widehat\mu^{\rm rate}$ \\\cmidrule(lr){4-5}\cmidrule(lr){6-6}\cmidrule(lr){7-7}\cmidrule(lr){8-8}
		mean & $n$ & $\overline M$ & $\widehat r$ & $K_{\mathrm{CV}}$ & $\widehat h_t$ & $\widehat\lambda$ & $K$ \\
		\midrule
		\multirow{3}{*}{$\mu_1$} & $100$ & $702$ & $2$ (95\%) & $7$ & $0.065$ & $7.0\times 10^{-3}$ & $6$ \\
		& $200$ & $1405$ & $2$ (98\%) & $8$ & $0.057$ & $3.4\times 10^{-3}$ & $7$ \\
		& $500$ & $3511$ & $2$ (100\%) & $9$ & $0.050$ & $1.6\times 10^{-3}$ & $9$ \\
		\midrule
		\multirow{3}{*}{$\mu_2$} & $100$ & $702$ & $2$ (99\%) & $8$ & $0.050$ & $1.6\times 10^{-3}$ & $6$ \\
		& $200$ & $1405$ & $2$ (100\%) & $8$ & $0.050$ & $1.6\times 10^{-3}$ & $7$ \\
		& $500$ & $3511$ & $2$ (100\%) & $9$ & $0.050$ & $7.8\times 10^{-4}$ & $9$ \\
		\midrule
		\multirow{3}{*}{$\mu_3$} & $100$ & $702$ & $2$ (79\%) & $5$ & $0.074$ & $0.014$ & $6$ \\
		& $200$ & $1405$ & $2$ (94\%) & $5$ & $0.065$ & $7.0\times 10^{-3}$ & $7$ \\
		& $500$ & $3511$ & $2$ (99\%) & $5$ & $0.057$ & $7.0\times 10^{-3}$ & $9$ \\
		\bottomrule
	\end{tabular}
	\caption{\small Tuning parameters selected by criterion~\eqref{eq:cv} over $R=2000$ replications: for
		$\widehat\mu^{\rm sp}$, the most frequent augmentation $\widehat r$ with its selection
		frequency and the median truncation over the replications that selected that augmentation,
		the retained system being $\mathcal{B}_{\widehat r, K_{\mathrm{CV}}}$; for the two smoothers,
		the median selected bandwidth and penalty. The last column is the fixed-rule truncation $K$,
		determined by $r + K = \lceil 2\,(7 n)^{1/5} \rceil$ with $r=2$.}
	\label{tab:cvbox}
\end{table}


The two comparisons that follow concern $\mu_1$ only. Every replication reuses the design, the subject fields, the errors and the folds of the corresponding replication of Section~\ref{sec:simulation}, so the boxes
are paired with those of Figure~\ref{fig:cvbox_mu1}. They are paired log-ratios of integrated squared errors, whiskers extend to 1.5 times the interquartile range and outliers are omitted.

First, we study the performance of the spacings estimator $\widehat \mu^{\rm sp}_{r,K}$ obtained with  $\mathcal B_{2,K}$ compared to those obtained with $\mathcal B_{0,K}$ and $\mathcal B_{1,K}$. For the three values of $r$, corresponding to the three orthonormal systems, the value of $K$ is selected by minimizing the map $K\mapsto C(r,K)$ with $C(r,K)$ the CV criterion in~\eqref{eq:cv}. The endpoint derivatives of $\mu_1$ in $t$ are neither zero nor equal. Thus, $\widehat \mu^{\rm sp}_{r,K}$ with $r=2$ is expected to perform better, and this is clearly what Figure~\ref{fig:appG_fig8} reveals:  at the three sample sizes ($n\in \{100, 200, 500\}$) the boxes lie entirely above zero, so the quadratic augmentation is more accurate on nearly every replication. Table~\ref{tab:basis} reports the truncations the criterion selected in each system. Cross-validation attempts to compensate for the inappropriate constraints of $r=0$ and $r=1$ by making much larger choices for $K$. 

\begin{table}[ht!]
	\small 	\centering
	\begin{tabular}{lccc}
		\toprule
		$n$ & $\mathcal{B}_{0,K_{\rm CV}}$ & $\mathcal{B}_{1,K_{\rm CV}}$ & $\mathcal{B}_{2,K_{\rm CV}}$ \\
		\midrule
		$100$ & $15$ $[14,17]$ & $13$ $[11,13]$ & $7$ $[6,8]$ \\
		$200$ & $18$ $[16,20]$ & $15$ $[13,17]$ & $8$ $[7,9]$ \\
		$500$ & $22$ $[20,24]$ & $19$ $[17,21]$ & $9$ $[9,10]$ \\
		\bottomrule
	\end{tabular}
	\caption{\small Truncation selected by the criterion~\eqref{eq:cv}, with $r$ fixed on each column, on $\mu_1$: median and, in brackets, interquartile range of $K_{\rm CV}$ over the replications.}
	\label{tab:basis}
\end{table}

\begin{figure}[H]
	\centering
	\includegraphics[width=0.75\linewidth, height=9cm]{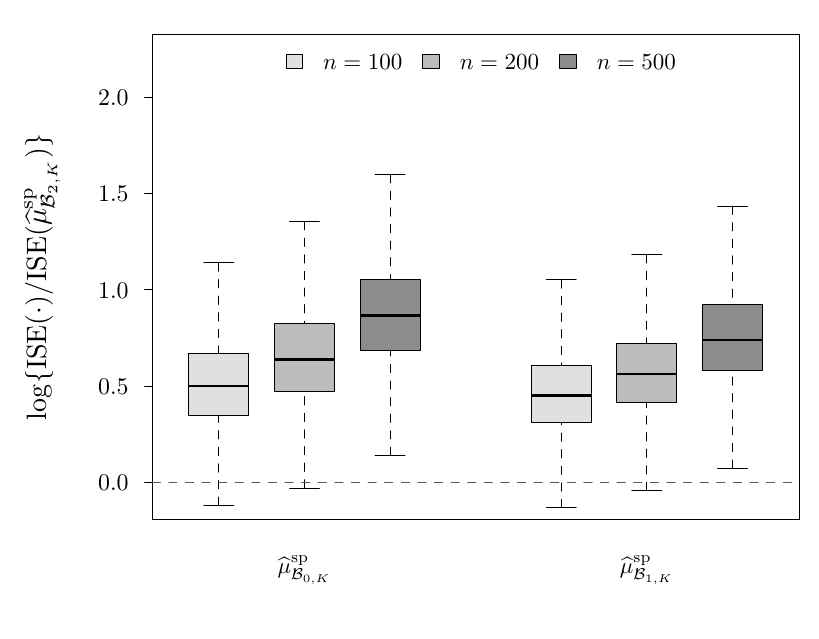}
	\vspace{-.7cm}\caption{\small Paired log-ratios of integrated squared errors on $\mu_1$, against $\widehat \mu^{\rm sp}_{r,K}$, one box per sample size. Positive values favor the quadratic augmentation. }
	\label{fig:appG_fig8}
\end{figure}

The next experiment is designed to compare the spacings estimator, which does not use $g$, with the competitors studied in Section~\ref{sec:series}:  the infeasible $\widehat \mu^{\rm OBS}$ using OBS weights and the true $g$ (which is the same as used in the simulations presented in Section~\ref{sec:simulation}), and the infeasible $\widehat \mu^{\rm MC}$ using the Monte Carlo linear integration with the true $g$. We add a third competitor, which is a feasible version of  $\widehat \mu^{\rm OBS}$ for which the density is estimated by the Parzen-Rosenblatt density estimator with a Gaussian kernel and reflected at the endpoints of $[0,1]$. The bandwidth of the density estimator is $b=M^{-1/3}$, so that the bias is negligible. All four select $(r, K)$ by~\eqref{eq:cv}, as reported in Table~\ref{tab:dens}.

\begin{table}[ht!]
	\small 	\centering
	\begin{tabular}{lcccccccc}
		\toprule
		& \multicolumn{2}{c}{$\widehat\mu^{\rm sp}$} & \multicolumn{2}{c}{$\widehat\mu^{\rm OBS}_{\widehat g}$} & \multicolumn{2}{c}{$\widehat\mu^{\rm OBS}$} & \multicolumn{2}{c}{$\widehat\mu^{\rm MC}$} \\\cmidrule(lr){2-3}\cmidrule(lr){4-5}\cmidrule(lr){6-7}\cmidrule(lr){8-9}
		$n$ & $\widehat r$ & $K_{\mathrm{CV}}$ & $\widehat r$ & $K_{\mathrm{CV}}$ & $\widehat r$ & $K_{\mathrm{CV}}$ & $\widehat r$ & $K_{\mathrm{CV}}$ \\
		\midrule
		$100$ & $2$ (95\%) & $7$ & $2$ (88\%) & $4$ & $2$ (95\%) & $4$ & $2$ (100\%) & $6$ \\
		$200$ & $2$ (98\%) & $8$ & $2$ (93\%) & $5$ & $2$ (95\%) & $5$ & $2$ (100\%) & $8$ \\
		$500$ & $2$ (100\%) & $9$ & $2$ (95\%) & $7$ & $2$ (97\%) & $7$ & $2$ (100\%) & $9$ \\
		\bottomrule
	\end{tabular}
	
	\caption{\small Pairs $(r,K)$ selected by the criterion~\eqref{eq:cv} on $\mu_1$: the most frequent $r$ with its selection frequency, and the median truncation over the replications that selected it, the retained system being $\mathcal{B}_{\widehat r,K_{\mathrm{CV}}}$.}
	\label{tab:dens}
\end{table}

In Figure~\ref{fig:appG_fig9} the boxes lie entirely above zero at the three sample sizes. Dispensing with the design
density is free here and provides a significant accuracy improvement over the three competing  estimators.

\begin{figure}[H]
	\centering
	\includegraphics[width=0.75\linewidth, height=9cm]{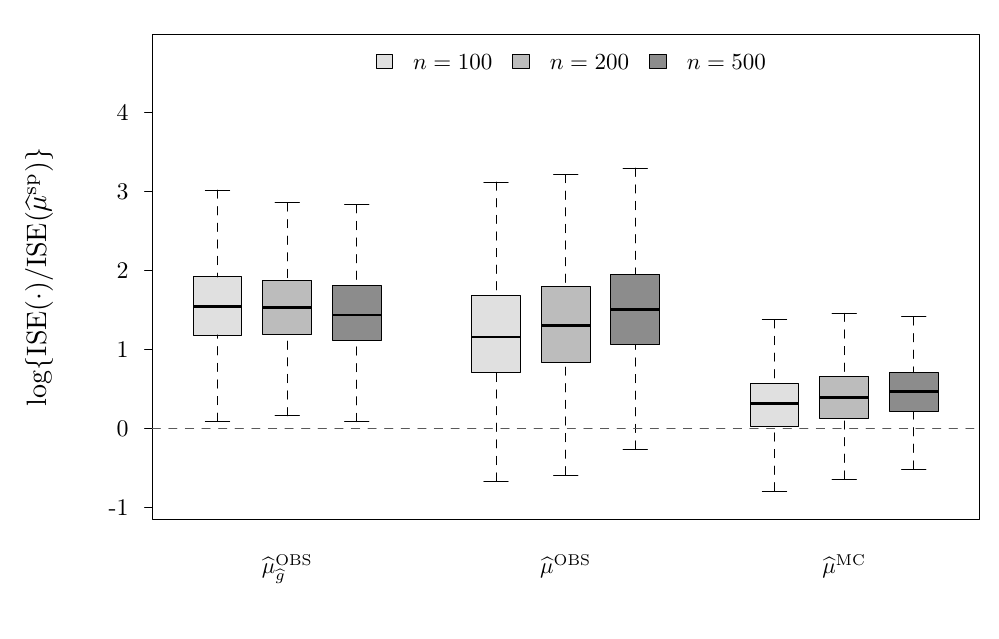}
	\vspace{-.7cm}	\caption{\small Paired log-ratios of integrated squared errors on $\mu_1$, against $\widehat \mu^{\rm sp}_{r,K}$, one box per sample size. Positive values favor the spacings estimator. }
	\label{fig:appG_fig9}
\end{figure}

\end{document}